\documentclass[12pt]{article}
\usepackage{verbatim}
\usepackage{amssymb,amsmath}
\usepackage{amsthm}
\usepackage{mathrsfs}
\usepackage{graphicx}
\usepackage{epsfig}
\usepackage{xcolor}
\usepackage{tikz}
\usetikzlibrary{arrows.meta,calc}
\newtheorem{theorem}{Theorem}
\newtheorem{corollary}{Corollary}[section]
\newtheorem{lemma}[corollary]{Lemma}
\newtheorem{proposition}[corollary]{Proposition}

\newcommand{\beq}{\begin{eqnarray*}}
\newcommand{\eeq}{\end{eqnarray*}}
\newcommand{\beqn}{\begin{eqnarray}}
\newcommand{\eeqn}{\end{eqnarray}}

\newcommand{\Prob} {{\mathbb P}}
\newcommand{\Pro}{{\bf P}}
\newcommand{\Z}{{\mathbb Z}}
\newcommand{\E}{{\mathbb E}}
\newcommand{\Ex}{{\bf E}}
\newcommand{\I}{{\mathbb I}}

\newcommand{\R}{{\mathbb{R}}}
\newcommand{\C}{{\mathbb C}}
\newcommand{\rad}{{\rm rad}}

\newcommand{\dist}{{\rm dist}}
\newcommand{\x}{{\bf x}}

\def \al {{\alpha}}
\def \de {{\delta}}
\def \ep {{\epsilon}}
\def \ga {{\gamma}}

\def \Im {{\rm Im}}
\def \Re {{\rm Re}}
\def \p {\partial}

\def \Half {{\mathbb H}}

\def \be{{\bf e}}

\def \Loop {{\cal L}}

\def \soup {{\cal C}}

\newenvironment{remark}[1][Remark]{\begin{trivlist}
\item[\hskip \labelsep {\bfseries #1}]}{\end{trivlist}}
\newenvironment{definition}[1][Definition]{\begin{trivlist}
\item[\hskip \labelsep {\bfseries #1}]}{\end{trivlist}}

\begin{document}
\title{$SLE_6$ and 2-d critical bond percolation on the square
lattice}

\author{Wang Zhou\thanks{ Department of Statistics and Data Science, National University of Singapore, Singapore 117546. Email: stazw@nus.edu.sg}\\
{\it \small National University of Singapore}}

\date{}

\maketitle

\begin{abstract}
Through the rotational invariance of the 2-d critical bond percolation exploration path on the  square lattice we express Smirnov's edge parafermionic observable as a sum of two new edge observables. With the help of these two new edge observables  we can apply the discrete harmonic analysis and conformal mapping theory to prove the convergence of the 2-d critical bond percolation exploration path on the  square lattice to the trace of $SLE_6$ as the mesh size of the lattice tends to zero.
\end{abstract}


\tableofcontents

\section{Introduction}

\subsection{Background}
The pioneer work in percolation dates back to~\cite{BH57}, where Broadbent and  Harnmerslev proposed an easy random model for the probability of the venter of a large porous stone getting wet after being immersed into a bucket of water. Since then, more and more papers and books have been devoted to this area. Among these work, we should mention Kestern's famous result that $1/2$ is the critical probability in bond percolation on the two dimensional square lattice. In other words, if every edge on the square lattice is removed with probability larger than $1/2$, there is no infinite open cluster; if every edge on the square lattice is removed with probability smaller than $1/2$, one can observe infinite open cluster.  Therefore, a phase transition occurs at the critical probability $1/2$ in percolation. This interesting phase transition attracts not only physicists, but also mathematicians into the area of percolation. In fact, it has
become a major stream of mathematics. One can refer to Grimmett's book~\cite{Grim99} for more background.

Although it is well-developed, there are still many open problems in percolation theory. The most well-known is perhaps the long-standing ``conformal invariance" conjecture of Aizenman and
Langlands, Pouliotthe and Saint-Aubin~\cite{LPS94}, which states that  the probabilities of some macroscopic events have conformally invariant limits. Interestingly, Schramm rephrased this conjecture in his seminal paper~\cite{Schramm}, conjecturing that the percolation exploration path converges to his conformally invariant
Schramm-Loewner evolution ($SLE$) curve as the mesh of the lattice tends to zero. $SLE$ is a one-parameter family of random processes which is obtained as the solution to Loewner's differential equation with driving function a one parameter Brownian motion. The success of $SLE$ is two-fold. It is not only a powerful tool in the determination of the intersection exponent of two dimensional Brownian motion~\cite{LSW01A, LSW01B, LSW02}, but also a unique candidate for the scaling limit in various lattice models in statistical physics at their critical point  such as the critical site percolation on the triangular lattice~\cite{Smirnov}, ~\cite{CN}, loop erased random walks and uniform spanning tree Peano paths~\cite{LSWlerw}, the level lines of the discrete Gaussian free field~\cite{SSfree},
the interfaces of the random cluster model associated with the Ising model~\cite{Smi10}.
There are good surveys on $SLE$, see for example,~\cite{wernernotes},~\cite{GKnotes},~\cite{Cardynotes},~\cite{lawlerbook}.

Although the exploration path in the critical percolation on the two-dimensional triangular lattice converges to $SLE_6$~\cite{Smirnov}, ~\cite{CN}, the picture on the general lattice, especially on the square lattice, is still unclear. This  is unexpected since the whole situation is always affected by moving one part slightly. Perhaps the main reason that there is  no proof of the critical percolation exploration path on the general lattice is due to the fact that the proof for site percolation on the triangular lattice depends heavily on the nice properties of equilateral triangle.

In this paper, we will prove that the exploration path in the critical bond percolation on the two dimensional square lattice converges to $SLE_6$. We use two inputs for critical bond percolation on the square lattice.
The first is Smirnov's local parafermionic relation~\cite{Smi10}.
The second is the half-plane one-arm estimate with exponent \(1/3\).
The transfer-matrix calculation of Ikhlef and
Ponsaing~\cite{IkPon12} predicted this exponent. The estimate used
here is proved in~\cite{zho26} by deriving an exact strip
boundary-passage formula and comparing the strip event with the
half-plane arm event using the Russo-Seymour-Welsh (RSW) method and quasi-multiplicativity.

\subsection{Percolation setup}
We identify $\R^2$ with $\C$.  The square lattice $(\Z^2,\E)$ has vertex set $\Z^2$ and edge set
\[
 \E=\{\langle x,y\rangle:x,y\in\Z^2,\ |x-y|=1\}.
\]
Its vertices and edges are called primal.  The dual lattice has vertex set
$(\Z^2)^*=(1/2,1/2)+\Z^2$ and edge set $\E^*$.  A dual edge $E^*$ crosses exactly one primal edge $E$.  We say that $E^*$ is dual-open when $E$ is closed.  The medial lattice $(\Z^2)^\diamond$ has one vertex at the center of each primal edge and one medial edge between the centers of two consecutive primal edges around a face.  Every medial edge is given its canonical orientation, counterclockwise around the primal vertex which it borders.  Primal, dual and medial objects on the lattice of mesh $\de$ are obtained by multiplication by $\de$.

Let $a^\diamond$ and $b^\diamond$ be two distinct medial vertices.  Let $\p_{ab}^\diamond$ and $\p_{ba}^\diamond$ be two simple edge-avoiding medial paths from $a^\diamond$ to $b^\diamond$, with disjoint interiors, such that their union is the boundary of a finite simply connected medial domain $\Omega^\diamond$.  The first path is traversed counterclockwise from $a^\diamond$ to $b^\diamond$ and the second clockwise.  We call $(\Omega^\diamond,a^\diamond,b^\diamond)$ a medial Dobrushin domain.  Its medial edge and vertex sets are denoted by $E_{\Omega^\diamond}$ and $V_{\Omega^\diamond}$.  The two exterior medial edges at $a^\diamond$ and $b^\diamond$ are denoted by $e_a$ and $e_b$ respectively.

The associated primal domain $(\Omega,a,b)$ and dual domain $(\Omega^*,a^*,b^*)$ are obtained from the primal and dual bonds adjacent to $\Omega^\diamond$.  Their boundary arcs are denoted by
\[
 \p_{ab},\quad \p_{ba},\qquad
 \p_{ab}^*,\quad \p_{ba}^*.
\]
For mesh $\de$, we use the notation
\[
 (\Omega_\de^\diamond,a_\de^\diamond,b_\de^\diamond),
 \qquad (\Omega_\de,a_\de,b_\de),
 \qquad (\Omega_\de^*,a_\de^*,b_\de^*).
\]
The primal boundary arc $\p_{ba,\de}$ is wired open, and the dual boundary arc $\p_{ab,\de}^*$ is wired dual-open.  Every remaining primal bond is independently open or closed with probability $1/2$.  We write $\Pro$ and $\Ex$ for the resulting probability and expectation.

For every percolation configuration there is a unique medial path $\ga_\de$ from $a_\de^\diamond$ to $b_\de^\diamond$ which separates the primal-open cluster attached to $\p_{ba,\de}$ from the dual-open cluster attached to $\p_{ab,\de}^*$.  At each medial vertex the path turns by $\pi/2$ and does not cross the primal-open or dual-open bond through that vertex.  This is the percolation exploration path.  We index $\ga_\de$ by its successive medial vertices.

For two oriented medial edges $e_1,e_2$ visited in this order by $\ga_\de$, let
$W_{\ga_\de}(e_1,e_2)$ be the total signed rotation of the path from $e_1$ to $e_2$.  Smirnov's edge parafermionic observable is
\begin{equation}\label{para}
 F(e)=\Ex\left[
  \exp\left(\frac i3W_{\ga_\de}(e,e_b)\right)
  \I(e\in\ga_\de)
 \right].
\end{equation}
The symbol $e_b$ also denotes the unit complex number in the direction of the exterior medial edge at $b_\de^\diamond$.

Let $v$ be an interior medial vertex.  The four oriented medial edges incident to $v$ are denoted by $A,B,C,D$ in clockwise order, with $A$ and $C$ pointing towards $v$ and $B$ and $D$ pointing away from $v$.  Smirnov's local relation is
\begin{equation}\label{CR}
 F(A)-F(C)=i\big(F(B)-F(D)\big).
\end{equation}
For an interior medial vertex we also put
\beq
F(v)&=&e_b^{-1/3}\left(
 e^{-i\pi/4}\frac{F(A)+F(C)}2
 +e^{i\pi/4}\frac{F(B)+F(D)}2\right),\\
F^*(v)&=&e_b^{-1/3}\left(
 e^{i\pi/4}\frac{F(A)+F(C)}2
 +e^{-i\pi/4}\frac{F(B)+F(D)}2\right).
\eeq
If $v$ is a boundary medial vertex on $\p_{ba,\de}^{\diamond}$ and only $A,D$ are present in the domain, we set
\[
 F(v)=e_b^{-1/3}\big(e^{-i\pi/4}F(A)+e^{i\pi/4}F(D)\big).
\]
If only $B,C$ are present, we use
\[
 F(v)=e_b^{-1/3}\big(e^{-i\pi/4}F(C)+e^{i\pi/4}F(B)\big).
\]
The other boundary orientations are defined by rotation.

\subsection{Main results}
The boundary regularity assumption used in the proof is imposed at the intermediate scale
\[
 r_\de=(\log\de^{-1})^{-1}.
\]

\medskip
\noindent{\it Condition} $\mathrm C$.
The primal-open boundary arc is a concatenation
\[
 \p_{ba,\de}=\mathrm L_1\cup\cdots\cup\mathrm L_{\mathrm n}
\]
of straight lattice segments in counterclockwise order from $b_\de$ to $a_\de$, where
$\mathrm L_j=[\mathrm v_j,\mathrm v_{j+1}]$ and its length
$\mathrm l_j$ is at least $r_\de$.  The dual-open boundary arc is a concatenation
\[
 \p_{ab,\de}^*=\mathrm L_1^*\cup\cdots\cup\mathrm L_{\mathrm n^*}^*
\]
of straight dual-lattice segments in counterclockwise order from $a_\de^*$ to $b_\de^*$, where
$\mathrm L_{j^*}^*=[\mathrm v_{j^*}^*,\mathrm v_{j^*+1}^*]$ and its length
$\mathrm l_{j^*}^*$ is at least $r_\de$.  At $a_\de^\diamond$ and $b_\de^\diamond$ the terminal primal and dual segments meet at right angles,
\[
 \arg\frac{\mathrm v_{\mathrm n^*+1}^*-\mathrm v_{\mathrm n^*}^*}
                {\mathrm v_1-\mathrm v_2}
 =\arg\frac{\mathrm v_{\mathrm n}-\mathrm v_{\mathrm n+1}}
                {\mathrm v_2^*-\mathrm v_1^*}
 =\frac\pi2.
\]
For a point $u$ in the relative interior of $\mathrm L_j$, let $\mathbf n_j$ be the unit normal which points into the primal domain on this traversal of the boundary, and put $u'=u+r_\de\mathbf n_j$.  The segment $(u,u']$ is contained in the domain and is disjoint from
\[
 \bigcup_{k\ne j}\mathrm L_k\cup\p_{ab,\de}^*.
\]
The analogous condition holds for every dual segment $\mathrm L_{j^*}^*$, with the inward normal in the dual domain.  If two consecutive segments overlap geometrically, the normal is chosen separately for the two boundary traversals.

\begin{remark}
Condition $\mathrm C$ is a technical condition at this intermediate scale.  It can be arranged by first taking a square-lattice approximation with mesh comparable with $r_\de$ and then dividing its bonds into bonds of length $\de$.  If the areas of the domains are uniformly bounded, the condition on the inward normal segments gives
$\mathrm n+\mathrm n^*=O(r_\de^{-2})$.
\end{remark}

\begin{remark}
Two consecutive boundary segments are allowed to overlap.  Thus a geometric segment may be traversed twice when one follows a Dobrushin boundary arc.  The angle between consecutive traversals is allowed to be $\pi/2$, $3\pi/2$, or $2\pi$.
\end{remark}

Our first theorem identifies the scaling limit of the parafermionic observable.  Let
$\mathbf F:\Omega_0\to\R\times(0,1)$ be a conformal map such that
$\mathbf F(a)=-\infty$ and $\mathbf F(b)=\infty$.  The choice of which Dobrushin boundary arc is mapped to the lower side fixes $\mathbf F$ up to a real additive constant.

\begin{theorem}\label{paraconverge}
Suppose that $\Omega_0$ is a bounded simply connected domain with two marked boundary prime ends $a,b$, and that
$(\Omega_\de^\diamond,a_\de^\diamond,b_\de^\diamond)$ converges to
$(\Omega_0,a,b)$ in the Carath\'eodory sense.  Under Condition $\mathrm C$,
\[
 \de^{-1/3}\mathrm c_\de^{-1/3}F(v)
 \longrightarrow \big(\mathbf F'(v)\big)^{1/3}
\]
uniformly for $v$ in every compact subset of $\Omega_0$.  Here
$\mathrm c_\de$ is defined in \eqref{c-de-ci},
\[
 0<\liminf_{\de\to0}\mathrm c_\de
 \leq\limsup_{\de\to0}\mathrm c_\de<\infty,
\]
and the cube-root branch is fixed by the orientation of the medial edges and the boundary winding convention.
\end{theorem}

Let $\mathcal M$ be the space of continuous curves in $\C$.  Curves are considered modulo increasing reparametrization.  We use the Fr\'echet metric
\[
 \mathrm d(\ga_1,\ga_2)
 =\inf_{\varphi}\sup_{t\in[0,1]}
   |\ga_1(t)-\ga_2(\varphi(t))|,
\]
where the infimum is over increasing homeomorphisms
$\varphi:[0,1]\to[0,1]$ after the two curves have been continuously parametrized on $[0,1]$.

\begin{theorem}\label{pathconverge}
Under the assumptions of Theorem~\ref{paraconverge}, the exploration path
$\ga_\de$ converges weakly in $(\mathcal M,\mathrm d)$ to chordal $SLE_6$ in
$\Omega_0$ from $a$ to $b$.
\end{theorem}

Let $c_\de^\diamond\in\p_{ba,\de}^\diamond$ and
$d_\de^\diamond\in\p_{ab,\de}^\diamond$ converge respectively to boundary points $c,d$ distinct from $a,b$.  Let $\p_{1,\de}$ be the medial boundary arc from $a_\de^\diamond$ to $d_\de^\diamond$ contained in $\p_{ab,\de}^\diamond$, and let $\p_{2,\de}$ be the medial boundary arc from $b_\de^\diamond$ to $c_\de^\diamond$ contained in $\p_{ba,\de}^\diamond$.

\begin{corollary}\label{card}
Under the assumptions of Theorem~\ref{pathconverge}, the probability that a dual-open cluster joins $\p_{1,\de}$ to the dual vertices adjacent to $\p_{2,\de}$ converges to Cardy's formula for the marked domain $(\Omega_0;a,b,c,d)$.
\end{corollary}

\subsection{Proof strategy}
The main idea is to compare whole percolation configurations rather than to seek another exact local identity for $F$.  We select finite portions of the exploration path, together with their adjacent primal-open and dual-open paths and the attached medial loop, and then exchange the states of finitely many reflected or translated primal bonds.  At criticality these changes preserve the Bernoulli measure.  The primal-open and dual-open paths determine the modified exploration path, and the construction is chosen so that the winding of the distinguished medial edge is unchanged.  If the required primal-open and dual-open paths cannot be selected, the failure either gives a bulk multi-arm event or is first mapped injectively, by reflection or folding, to a standard boundary arm event.  The corresponding arm estimate is applied only after that reduction.

The proof has the following structure.

\medskip
\noindent{\it Step 1.}
Section~\ref{TranInv} compares the contribution of two parallel oriented medial edges at two neighbouring medial vertices.  The event is divided into the two possible two-edge visits and the event in which all four incident edges are used.  In the last case, the central bond is reversed first.  The open side, dual-open side, meeting points and return points are then selected by fixed rules.  The selected exploration path and the attached loop are shifted by one lattice spacing.  This change preserves probability.  The configurations for which the shift cannot be made are controlled by four-arm, five-arm and close-return estimates.

\medskip
\noindent{\it Step 2.}
Section~\ref{rot} develops the corresponding rotation constructions.  Rotation around a medial vertex produces the four local passage correspondences used later.  Rotation around the center of a medial edge is first proved after the radius of the attached loop has been revealed, so this is a conditional estimate.  The revisit construction of Subsection~\ref{partIII} combines the two rotation constructions in Subsections \ref{partI}-\ref{partII}. This provides the basis for defining the modified edge observable.

\medskip
\noindent{\it Step 3.}
Section~\ref{modi} uses the rotation constructions to decompose
\[
 F=\mathrm F+\mathrm F_i.
\]
The two terms are defined by one-passage and two-passage percolation events at a medial vertex.  They satisfy the discrete Cauchy--Riemann relation inherited from \eqref{CR}, together with the additional approximate identities
\[
 \mathrm F(A)+\mathrm F(C)
 =\mathrm F(B)+\mathrm F(D)+O(\de^2)
\]
and
\[
 \mathrm F_i(A)+\mathrm F_i(C)
 =-\mathrm F_i(B)-\mathrm F_i(D)+O(\de^2)
\]
away from the boundary.

\medskip
\noindent{\it Step 4.}
Section~\ref{confinvar} studies the cubic discrete integrals
\[
 \int\de^{-1}\mathrm F(z)^3\,dz,
 \qquad
 \int\de^{-1}\mathrm F_i(z)^3\,dz.
\]
The cube is dictated by the parafermionic spin $1/3$.  Reflection near a
straight boundary segment gives the approximate boundary phase of the
modified observables.  The half-plane one-arm estimate from~\cite{zho26},
together with the estimates near the $\pi/2$, $3\pi/2$, and $2\pi$
corner geometries, controls the cubic integral along the inner contour.
The local identities from Sections~\ref{TranInv}--\ref{modi} imply that
the discrete integrals of $\mathrm F^3$ and $\mathrm F_i^3$ along sides of lattice face
and the discrete Laplacian of $\mathrm F^3$ and $\mathrm F_i^3$  are summable.
Discrete harmonic measure then identifies the limiting primitive and the
limiting observable.  Finally, $\de^{-1/3}\mathrm F_i$ is shown to
converge to zero.

\medskip
\noindent{\it Step 5.}
Section~\ref{ConSLE6} reveals an initial part of the exploration path and applies the same observable in the remaining percolation domain.  A nearby domain with the required boundary shape is used only to apply the results of Section~\ref{confinvar}.  The conditional observable is a martingale.  Its scaling limit gives two martingales for the Loewner driving function, which show that its quadratic variation is $6t$.  Section~\ref{ProofMain} removes the extra boundary condition and completes the proofs of the main theorems.

\subsection{Review of arm events}
We present the arm estimates used throughout the paper.
Let $\Lambda(n)$ be the subgraph of $\Z^2$ induced by the square $[-n,n]^2$ or the open ball $B(0;n)$ with center $0$ and radius $n$. Changing between these choices only changes constants. A self-avoiding path of primal open edges or dual-open edges connecting the inner boundary $\p_r$ to the outer boundary $\p_R$ of an annulus $\Lambda_R\setminus\Lambda_{r-1}$ is called an arm. An arm is of type $1$ if it consists of open primal edges, and of type $0$ if it consists of dual-open edges.
  For
$\sigma=(\sigma_1,\ldots,\sigma_k)\in\{0,1\}^k$, let
$A_\sigma(r,R)$ be the event that there are $k$ disjoint arms from $\p_r$ to $\p_R$, appearing in counterclockwise order with types
$\sigma_1,\ldots,\sigma_k$.

The superscript $+$ means that the arms are contained in the upper half-plane, and the superscript $++$ means that they are contained in the first quadrant.  Thus $A_1^+(0,R)$ is the boundary one-arm event,
$A_{01}^+(r,R)$ is the half-plane two-arm event,
$A_{010}^+(r,R)$ is the half-plane three-arm event, and
$A_{01}^{++}(r,R)$ is the quarter-plane two-arm event.  There are constants
$c_0<\infty$ and $\alpha_1,\alpha_2>0$ such that
\beqn
&&\Pro(A_1(r,R))\leq c_0 (r/R)^{\alpha_1}, \label{one} \\
&&\Pro(A_{0101}(r,R))\leq c_0(r/R)^{1+\alpha_2}, \label{four}\\
&&\Pro(A_{0011}(r,R)) \leq  c_0(r/R)^{1+\alpha_2}, \label{four-a}\\
&&\Pro(A_{01011}(r,R))\leq c_0(r/R)^{2}, \label{five}\\
&&\Pro(A_{010101}(r,R))\leq c_0(r/R)^{2+\alpha_1}, \label{six}\\
&&\Pro(A_{01}^+(r,R))\leq c_0(r/R), \label{halftwo}\\
&&\Pro(A_{010}^+(r,R))\leq c_0(r/R)^{2}, \label{halfthree}\\
&&c_0^{-1}R^{-1/3}\leq \Pro(A_1^+(0,R))\leq c_0R^{-1/3}, \label{halfone}\\
&&\Pro(A_{01}^{++}(r,R))\leq c_0(r/R)^{1+\alpha_2}. \label{quarter}
\eeqn
In \eqref{four-a}, the two adjacent arms of the same type are required to be vertex-disjoint.  Quasi-multiplicativity and \eqref{halfone} give, uniformly for $1\leq r<R$,
\[
 \Pro(A_1^+(r,R))\asymp (r/R)^{1/3}.
\]

The bulk one-arm estimate \eqref{one} follows from the RSW method; see Chapter 7 of~\cite{BR12}.
The bound \eqref{four-a} follows from the comparison of
polychromatic arm probabilities in
Reeves and Sosoe~\cite[Proposition~1.1]{RS20}:
\[
  \Pro(A_{0011}(r,R))
  \preceq
  \Pro(A_{0101}(r,R)).
\]
Here their dual-closed arms are our dual-open arms. The estimate
\eqref{four} now gives \eqref{four-a}.

The five-arm estimate \eqref{five} follows from the standard planar five-arm argument in~\cite{KSZ98}.  Combining it with \eqref{one} gives \eqref{six}.  The half-plane estimates \eqref{halftwo}--\eqref{halfthree} and the quarter-plane estimate \eqref{quarter} are contained in~\cite{DMT20}.  The boundary one-arm estimate \eqref{halfone} is proved in~\cite{zho26}, and its earlier transfer-matrix prediction is in~\cite{IkPon12}.

From now on, put $\alpha=\min(\alpha_1,\alpha_2,1)$ and choose positive
numbers $\ep,\varepsilon$ such that
\[
 \ep\leq\alpha/100,
 \qquad
 \varepsilon\leq\alpha\ep/100.
\]

\subsection{Notation}
We use $\dist(\cdot,\cdot)$ for Euclidean distance.  The open ball centered at $v$ with radius $r$ is denoted by $B(v;r)$, and the open annulus with radii $r<R$ by $B(v;r,R)$.  The relations $a\preceq b$ and $a=O(b)$ mean that $|a|\leq Cb$ for a universal constant $C<\infty$.  The line segment with endpoints $z_1,z_2$ is denoted by $[z_1,z_2]$.

For a percolation configuration $\omega$, we write $\ga_\de(\omega)$ for its exploration path.  If $\mathrm E\subseteq E_{\Omega_\de}$, let $\omega_{\setminus\mathrm E}$ be the configuration obtained by reversing the state of every bond in $\mathrm E$:
\beq
 \omega_{\setminus\mathrm E}(E)&=&\omega(E),
       \qquad E\notin\mathrm E,\\
 \omega_{\setminus\mathrm E}(E)&=&1-\omega(E),
       \qquad E\in\mathrm E.
\eeq
For a single bond we write $\omega_{\setminus E}$.

We use the same symbol for a lattice vertex and its position in $\R^2$ or $\C$, and put
$\be_1=(1,0)$ and $\be_2=(0,1)$.  All constants implicit in $O(\cdot)$ or $\preceq$ are uniform in the mesh unless a dependence is explicitly indicated.

\subsection{Outline of the paper}
Section~\ref{TranInv} proves the bulk and boundary translation estimates. Its proof uses explicit
probability-preserving maps of percolation configurations. Section~\ref{rot} also constructs the rotations around a medial vertex and removes the remaining radius conditioning by arm arguments.

Section~\ref{modi} defines the modified observables and proves their local and translation relations.  Section~\ref{confinvar} establishes the approximate boundary phase of the modified observables, controls the cubic discrete integrals, and identifies the scaling limit of the observable.  Section~\ref{ConSLE6} uses the conditional observable to show that every limit of a convergent subsequence is chordal $SLE_6$.

Section~\ref{ProofMain} combines these results and removes the extra boundary condition used near the marked point.  The Appendix contains the three discrete random-walk estimates needed in Section~\ref{confinvar}.

Finally, Section~\ref{sec:universality} discusses universality
and the possible extension of our result and method to critical
bond percolation on isoradial graphs.

\section{Translational invariance} \label{TranInv}

We prove that the parafermionic observable changes very little when a medial edge is translated by one lattice spacing.  The construction is made directly on percolation configurations.  It moves the loop attached to the bond through the first medial vertex, together with a terminal portion of the incoming exploration path.  All choices are determined by the configuration.  The resulting map preserves the Bernoulli probability and the winding phase.

Let $v_0$ and $v_1$ be two medial vertices in $\Omega_\de^{\diamond}$ with $\dist(v_0,v_1)=\de$.  The four medial edges incident to $v_j$ are denoted by $A_j,B_j,C_j,D_j$ in \emph{clockwise} order.  The edges $A_j,C_j$ point towards $v_j$, while $B_j,D_j$ point away from $v_j$.  We assume that $A_1$ is a translate of $A_0$ by one lattice spacing.  The same convention is used for $B_j,C_j,D_j$.  Let $E_j$ be the primal bond whose centre is $v_j$, and put
\[
 d_0=\dist\big(v_0,\p_{ba,\de}\cup\p_{ab,\de}^*\big).
\]

For the proof of the estimate for $A_0,A_1$, we separate the three possible local situations. In each of the events below, $\ga_\de$ is followed from $a_\de^\diamond$ to $b_\de^\diamond$, and $A_j$ is the first of the four medial edges incident to $v_j$ that is visited by $\ga_\de$. Put
\beqn
 \mathcal C_j^D
 &=&\{A_j,D_j\in\ga_\de,\ B_j,C_j\notin\ga_\de\},\\
 \mathcal C_j^B
 &=&\{A_j,B_j\in\ga_\de,\ C_j,D_j\notin\ga_\de\},\\
 \mathcal C_j^4
 &=&\{A_j,B_j,C_j,D_j\in\ga_\de, \ \text{and} \ \ga_\de \ \text{visits} \ A_j \ \text{first}\}.
\eeqn
Then $\{\ga_\de \ \text{visits} \ A_j \ \text{first}\}$ is the disjoint union of these three events.  If $\omega\in\mathcal C_j^4$, reverse the state of $E_j$.  The two passages through $v_j$ are then reconnected into one passage through $v_j$ and one loop attached to $E_j$.  The new configuration belongs to $\mathcal C_j^D\cup\mathcal C_j^B$.  At $p=1/2$, this local change preserves probability.  It also leaves $W_{\ga_\de}(A_j,e_b)$ unchanged.  Thus it is enough to construct the translation map on $\mathcal C_j^D$ and $\mathcal C_j^B$.  We give the construction for $\mathcal C_j^D$.  The other case follows by reflection.

\subsection{The bulk translation estimate}

\begin{proposition} \label{Tran}
Given any $\ep>0$, if $\de$ is small enough and
\[
 \min\big(\dist(v_0,a_\de^\diamond),
           \dist(v_0,b_\de^\diamond)\big)\geq \de^{1-\ep},
\]
then
\beqn
|F(A_0)-F(A_1)| &\preceq& (\de/d_0)^{1+\al-\ep}, \label{TrA}\\
|F(B_0)-F(B_1)| &\preceq& (\de/d_0)^{1+\al-\ep}, \label{TrB}\\
|F(C_0)-F(C_1)| &\preceq& (\de/d_0)^{1+\al-\ep}, \label{TrC}\\
|F(D_0)-F(D_1)| &\preceq& (\de/d_0)^{1+\al-\ep}. \label{TrD}
\eeqn
\end{proposition}

\begin{proof}
We assume that $v_1=v_0-\de\be_1$.  We first compare the contributions of $\mathcal C_0^D$ and $\mathcal C_1^D$.  Let $\omega\in\mathcal C_0^D$ and write $\ga_0=\ga_\de(\omega)$.  There is a unique time $n_0$ at which $\ga_0$ visits $v_0$.  The path uses $A_0,D_0$ at this visit.  The other two incident medial edges, $B_0,C_0$, belong to the loop attached to $E_0$.  We denote this loop by $\Loop_0$ and put
\[
 r_0=8\de+\max\{|z-v_0|:z\in\Loop_0\}.
\]
The construction below moves $\Loop_0$ and a terminal portion of the incoming exploration path by one lattice spacing to the left.

If $\Loop_0$ reaches distance $d_0/8$ from $v_0$, the two sides of the exploration path and the two sides of the loop give four arms from a fixed multiple of $\de$ to a scale comparable with $d_0$.  After the usual separation of arms, \eqref{four} and \eqref{four-a} give
\begin{equation}
 \Pro\big(\mathcal C_0^D,\ r_0\geq d_0/8\big)
 \preceq (\de/d_0)^{1+\al}.
 \label{trans-macro-cont}
\end{equation}

If $d_0\leq100\de$, the proposition follows after increasing the implicit constant.  Hence we assume
\[
 t:=\de/d_0<1/100.
\]
Fix
\[
 \kappa=\frac{\ep}{4(1+\al)}.
\]
If $r_0>d_0t^\kappa$, the exploration path and $\Loop_0$ again give four arms from scale $\de$ to scale $d_0t^\kappa$.  Therefore
\begin{equation}
 \Pro\big(\mathcal C_0^D,\ r_0>d_0t^\kappa\big)
 \preceq t^{(1-\kappa)(1+\al)}
 \preceq t^{1+\al-\ep/2}.
 \label{trans-large-r0}
\end{equation}
It remains to work on the event $r_0\leq d_0t^\kappa$.

\medskip
\noindent
{\bf Step 1: annuli and the primal-open and dual-open sides.}

Put
\[
 R=\log(d_0/\de),\qquad
 m_0=\left\lfloor\frac{\log(d_0/(16Rr_0))}{\log R}\right\rfloor,
 \qquad
 k_0=1+\left\lfloor2+\frac{4}{\al}\right\rfloor.
\]
If $R$ is bounded, the conclusion is again immediate after changing the constant.  We take $R$ sufficiently large.  Since $r_0\leq d_0t^\kappa$ and $t=e^{-R}$,
\[
 m_0\geq
 \left\lfloor
 \frac{\kappa R-\log(16R)}{\log R}
 \right\rfloor.
\]
Hence $m_0\geq k_0$ when $\de$ is sufficiently small.  The above choice of $k_0$ is also large enough for the arm estimates below.  For $k_0\leq k\leq m_0$, let
\[
 \mathcal A_k=B(v_0;r_0R^k,r_0R^{k+1}),
 \qquad
 D_k=B(v_0;2r_0R^k),
 \qquad
 B_k=B(v_0;r_0R^{k+1}/2).
\]
The annuli $\mathcal A_k$ are pairwise disjoint, apart from their
boundary circles.  The set $B_k\setminus D_k$ is the central subannulus used in the
construction.  Moreover, for every $k\leq m_0$,
\[
 r_0R^{k+1}\leq\frac{d_0}{16},
 \qquad
 \rad(B_k)=\frac{r_0R^{k+1}}2\leq\frac{d_0}{32}.
\]
Thus every full annulus $\mathcal A_k$, and hence every central
subannulus $B_k\setminus D_k$, stays away from
$\p_{ba,\de}\cup\p_{ab,\de}^*$.

Parameterize $\ga_0$ by its successive visits to medial vertices.  Let
\[
 \tau_k=\min\{j\leq n_0:\ga_0(j)\in B_k\}
\]
be its first entrance into $B_k$, and let
\[
 s_k=\sup\{j\leq n_0:\ga_0(j-1)\notin B_k,
             \ \ga_0[j,n_0]\subset B_k\}
\]
be its final entrance before reaching $v_0$.  Thus $\ga_0[s_k,n_0]$ is the terminal part of the incoming exploration path which stays in $B_k$ until it reaches $v_0$.

For a medial vertex $z$, let $E_z$ be the primal bond centred at $z$.  We locate where the earlier part of $\ga_0$ first meets the primal-open and dual-open boundary paths adjacent to $\ga_0[s_k,n_0]$.  Define
\begin{align}
 t_{o,k}
 &=\min\left\{j\leq s_k:\begin{array}{l}
E_{\ga_0(j)}\text{ shares an endpoint with an open bond}\\[-1mm]
\text{on the primal-open side of }\ga_0[s_k,n_0]
 \end{array}\right\},\label{trans-tok}\\
 T_{o,k}
 &=\min\{j\geq s_k:E_{\ga_0(j)}\text{ shares an endpoint with }
                         E_{\ga_0(t_{o,k})}\},\label{trans-Tok}\\
 t_{d,k}
 &=\min\left\{j\leq s_k:\begin{array}{l}
E_{\ga_0(j)}\text{ shares an endpoint with a closed bond}\\[-1mm]
\text{whose dual is on the dual-open side of }\ga_0[s_k,n_0]
 \end{array}\right\},\label{trans-tdk}\\
 T_{d,k}
 &=\min\{j\geq s_k:E_{\ga_0(j)}\text{ shares an endpoint with }
                         E_{\ga_0(t_{d,k})}\}.\label{trans-Tdk}
\end{align}
Figure \ref{translation} illustrates the order of the two contacts $s_k<T_{d,k}<T_{o,k}$.

\begin{figure}[ht]
\centering
\resizebox{0.6\textwidth}{!}{
\begin{tikzpicture}[
  x=1cm,y=1cm,
  line cap=round,
  line join=round,
  every node/.style={font=\small},
  annulus/.style={draw=gray!72,line width=1.15pt},
  path/.style={draw=black,line width=0.72pt},
  special point/.style={circle,fill=blue!72!black,inner sep=1.35pt}
]
  \path[use as bounding box] (-6.35,-5.02) rectangle (3.62,4.86);

  \coordinate (v0) at (-5.14000,0.81000);
  \coordinate (B) at (-5.18000,-1.27000);
  \coordinate (C) at (-3.92000,2.51000);
  \coordinate (D) at (-4.98000,-2.75000);
  \coordinate (E) at (-2.88000,3.59000);
  \coordinate (F) at (-4.86000,-4.35000);
  \coordinate (G) at (-1.64000,4.53000);
  \coordinate (H) at (-1.20836,-2.53563);
  \coordinate (I) at (-0.05739,-0.10965);

  \coordinate (tok) at (-1.53000,-1.32000);
  \coordinate (Tok) at (-1.44000,-1.14000);
  \coordinate (tdk) at (-0.30000,0.73000);
  \coordinate (Tdk) at (-0.44000,0.69000);

  \draw[annulus] (B) arc[start angle=-91.102,end angle=54.335,radius=2.08038];
  \draw[annulus] (D) arc[start angle=-87.427,end angle=50.891,radius=3.56359];
  \draw[annulus] (F) arc[start angle=-86.894,end angle=46.745,radius=5.16759];

  \draw[path] plot coordinates {
      (1.760,-4.830) (1.760,-4.750) (1.760,-4.450) (1.760,-4.270) (1.740,-4.190) (1.720,-4.090)
      (1.720,-4.010) (1.650,-3.860) (1.600,-3.740) (1.560,-3.650) (1.480,-3.550) (1.410,-3.440)
      (1.270,-3.310) (1.150,-3.240) (1.040,-3.210) (0.940,-3.190) (0.860,-3.150) (0.760,-3.150)
      (0.650,-3.120) (0.380,-3.070) (0.240,-3.050) (0.140,-3.050) (0.040,-3.010) (-0.060,-3.010)
      (-0.260,-2.970) (-0.420,-2.950) (-0.520,-2.910) (-0.600,-2.890) (-0.740,-2.830) (-0.910,-2.750)
      (-1.000,-2.710) (-1.240,-2.510) (-1.420,-2.360) (-1.480,-2.240) (-1.540,-2.130) (-1.580,-2.050)
      (-1.620,-1.930) (-1.640,-1.850) (-1.640,-1.710) (-1.640,-1.570) (-1.600,-1.470) (-1.530,-1.320)
      (-1.420,-1.250) (-1.340,-1.230) (-1.180,-1.230) (-1.000,-1.230) (-0.820,-1.250) (-0.640,-1.290)
      (-0.410,-1.350) (-0.290,-1.440) (-0.200,-1.510) (-0.120,-1.610) (-0.040,-1.710) (0.060,-1.810)
      (0.130,-1.960) (0.220,-2.090) (0.300,-2.250) (0.390,-2.420) (0.460,-2.540) (0.530,-2.660)
      (0.620,-2.770) (0.710,-2.890) (0.840,-2.990) (0.940,-3.070) (1.050,-3.130) (1.180,-3.170)
      (1.300,-3.210) (1.400,-3.210) (1.500,-3.210) (1.800,-3.230) (2.020,-3.230) (2.160,-3.210)
      (2.340,-3.130) (2.500,-3.090) (2.620,-3.030) (2.740,-2.970) (2.900,-2.810) (3.040,-2.680)
      (3.110,-2.560) (3.180,-2.410) (3.240,-2.270) (3.260,-2.130) (3.300,-1.990) (3.300,-1.250)
      (3.300,-1.050) (3.260,-0.950) (3.240,-0.870) (3.220,-0.770) (3.180,-0.650) (3.080,-0.450)
      (3.000,-0.260) (2.900,-0.090) (2.840,0.050) (2.580,0.390) (2.440,0.590) (2.400,0.670)
      (1.960,1.250) (1.780,1.490) (1.680,1.570) (1.560,1.670) (1.480,1.750) (1.400,1.790)
      (1.260,1.870) (1.100,1.950) (0.980,1.990) (0.880,2.030) (0.800,2.050) (0.700,2.070)
      (0.420,2.110) (0.160,2.130) (-0.020,2.150) (-0.560,2.150) (-0.740,2.150) (-0.840,2.130)
      (-0.940,2.090) (-1.070,1.980) (-1.140,1.870) (-1.140,1.690) (-1.140,1.530) (-1.120,1.450)
      (-1.060,1.320) (-1.020,1.230) (-0.940,1.120) (-0.830,1.000) (-0.710,0.900) (-0.560,0.810)
      (-0.420,0.750) (-0.300,0.730) (0.040,0.630) (0.180,0.590) (0.260,0.590) (0.520,0.530)
      (0.740,0.490) (0.820,0.450) (1.000,0.410) (1.190,0.350) (1.310,0.300) (1.450,0.270)
      (1.560,0.180) (1.630,0.060) (1.640,-0.150) (1.640,-0.290) (1.580,-0.370) (1.520,-0.470)
      (1.420,-0.530) (1.310,-0.620) (1.220,-0.670) (1.110,-0.680) (0.980,-0.710) (0.900,-0.730)
      (0.480,-0.730) (0.290,-0.710) (0.150,-0.620) (0.070,-0.500) (0.010,-0.380) (0.000,-0.270)
      (-0.040,-0.170) (-0.060,-0.090) (-0.070,0.020) (-0.080,0.150) (-0.120,0.270) (-0.160,0.430)
      (-0.200,0.540) (-0.310,0.630) (-0.440,0.690) (-0.520,0.710) (-0.900,0.710) (-1.040,0.710)
      (-1.150,0.680) (-1.270,0.580) (-1.340,0.460) (-1.400,0.350) (-1.420,0.250) (-1.420,-0.070)
      (-1.420,-0.230) (-1.400,-0.330) (-1.360,-0.430) (-1.350,-0.580) (-1.340,-0.850) (-1.350,-1.040)
      (-1.440,-1.140) (-1.550,-1.200) (-1.640,-1.250) (-2.120,-1.250) (-2.300,-1.250) (-2.380,-1.230)
      (-2.500,-1.190) (-2.670,-1.160) (-2.790,-1.110) (-2.910,-1.060) (-3.030,-0.980) (-3.120,-0.930)
  };

  \foreach \P in {B,C,D,E,F,G,v0}
    \fill[black] (\P) circle (0.65pt);
  \foreach \P in {H,I}
    \node[special point] at (\P) {};
  
  \foreach \P in {tok,tdk}
    \node[special point] at (\P) {};
 
  \foreach \P in {Tok,Tdk}
    \draw[blue!72!black,line width=0.75pt,fill=white] (\P) circle (1.65pt);

  \node[anchor=east] at (-5.940,-1.290) {$r_0R^k$};
  \node[anchor=east] at (-5.780,-2.710) {$2r_0R^k$};
  \node[anchor=east] at (-5.527,-4.364) {$r_0R^{k+1}/2$};
  \node[anchor=north west] at ($(v0)+(-0.05,-0.04)$) {$v_0$};

  \node[anchor=north east,fill=white,inner sep=1.5pt] (taulab)
    at (-1.48,-2.78) {$\gamma_0(\tau_k)$};
  \draw[-{Stealth[length=1.4mm]},thin] (taulab.east) -- (H);
  \node[anchor=north west,inner sep=1.5pt] at ($(I)+(0.02,-0.01)$)
    {$\gamma_\delta(s_k)$};
  \node[anchor=north west,fill=white,inner sep=1.5pt] (toklab)
    at (-0.42,-1.68) {$\gamma_0(t_{o,k})$};
  \draw[-{Stealth[length=1.4mm]},black!70,thin]
    (toklab.west) -- (tok);

  \node[anchor=south east,fill=white,inner sep=1.5pt] (Toklab)
    at (-2.12,-0.72) {$\gamma_0(T_{o,k})$};
  \draw[-{Stealth[length=1.4mm]},black!70,thin]
    (Toklab.east) -- (Tok);

  \node[anchor=south west,fill=white,inner sep=1.5pt] (tdklab)
    at (0.02,1.22) {$\gamma_0(t_{d,k})$};
  \draw[-{Stealth[length=1.4mm]},black!70,thin]
    (tdklab.west) -- (tdk);

  \node[anchor=north east,fill=white,inner sep=1.5pt] (Tdklab)
    at (-1.28,0.24) {$\gamma_\delta(T_{d,k})$};
  \draw[-{Stealth[length=1.4mm]},black!70,thin]
    (Tdklab.east) -- (Tdk);

  \node[font=\scriptsize] at (1.440,-4.290) {$+$};
  \node[font=\scriptsize] at (1.920,-4.190) {$-$};
  \node[font=\scriptsize] at (-0.420,-3.210) {$+$};
  \node[font=\scriptsize] at (-0.280,-2.830) {$-$};
\end{tikzpicture}
}
\caption{The case $s_k<T_{d,k}<T_{o,k}$.  The filled blue points are the first contacts $\ga_0(t_{o,k})$ and $\ga_0(t_{d,k})$ on the earlier part of the exploration path.  The open blue points are the corresponding points $\ga_0(T_{o,k})$ and $\ga_0(T_{d,k})$ on the terminal part of the exploration path.  The signs $+$ and $-$ indicate the primal-open and dual-open sides.}
\label{translation}
\end{figure}

The sets in these four definitions are nonempty.  The terminal part of the exploration path has a primal-open boundary on one side and a dual-open boundary on the other.  The indices $t_{o,k},t_{d,k}$ specify the first contacts of the earlier part of the exploration path with these two boundary paths.  Once such a contact is fixed, the corresponding boundary path meets the terminal part, so $T_{o,k}$ or $T_{d,k}$ exists.

If one of the first contacts occurs at $s_k$, we use the first later
visit to the path segment determined by the other contact.  More
precisely, on $\{T_{o,k}=s_k\}$ define
\[
 t'_{o,k}=\inf\{j\geq s_k:\ga_0(j)\in\ga_0[T_{d,k},n_0]\},
 \qquad
 T'_{o,k}=\inf\{j\geq T_{d,k}:\ga_0(j)=\ga_0(t'_{o,k})\},
\]
and define $t'_{d,k},T'_{d,k}$ by interchanging $o$ and $d$.  If both $T_{o,k}$ and $T_{d,k}$ equal $s_k$, the scale is declared bad.

There are now four possible orders of the two contacts.  We present the corresponding indices explicitly.  Define
\begin{align*}
 \sigma_{o,k}&=\sup\{j\geq T_{o,k}:\ga_0(j)\in\ga_0[T_{d,k},T_{o,k}]\},
 &&\text{if }s_k<T_{d,k}<T_{o,k},\\
 \sigma_{d,k}&=\sup\{j\geq T_{d,k}:\ga_0(j)\in\ga_0[T_{o,k},T_{d,k}]\},
 &&\text{if }s_k<T_{o,k}<T_{d,k},\\
 \sigma'_{o,k}&=\sup\{j\geq T'_{o,k}:\ga_0(j)\in\ga_0[T_{d,k},T'_{o,k}]\},
 &&\text{if }s_k=T_{o,k}\text{ and }T_{d,k}<T'_{o,k},\\
 \sigma'_{d,k}&=\sup\{j\geq T'_{d,k}:\ga_0(j)\in\ga_0[T_{o,k},T'_{d,k}]\},
 &&\text{if }s_k=T_{d,k}\text{ and }T_{o,k}<T'_{d,k}.
\end{align*}
Thus each $\sigma$ is the final visit of the later part of the exploration path to the segment between the two contacts.

In each case we take one self-avoiding primal-open path and one self-avoiding dual-open path along the two outer sides of the indicated portions of $\ga_0$.  They are denoted by $\ga_k^o$ and $\ga_k^d$, indexed towards $v_0$.  Their reversals, indexed from $v_0$ outwards, are denoted by $\ga_k^{o-}$ and $\ga_k^{d-}$.  The four definitions are as follows.
\begin{enumerate}
\item If $s_k<T_{d,k}<T_{o,k}$, let $\ga_k^o$ be the outer primal-open boundary of $\ga_0[T_{o,k},n_0)$ and let $\ga_k^d$ be the outer dual-open boundary of $\ga_0[\sigma_{o,k},n_0)$.
\item If $s_k<T_{o,k}<T_{d,k}$, let $\ga_k^d$ be the outer dual-open boundary of $\ga_0[T_{d,k},n_0)$ and let $\ga_k^o$ be the outer primal-open boundary of $\ga_0[\sigma_{d,k},n_0)$.
\item If $s_k=T_{o,k}$ and $T_{d,k}<T'_{o,k}$, let $\ga_k^o$ be the outer primal-open boundary of $\ga_0[T'_{o,k},n_0)$ and let $\ga_k^d$ be the outer dual-open boundary of $\ga_0[\sigma'_{o,k},n_0)$.
\item If $s_k=T_{d,k}$ and $T_{o,k}<T'_{d,k}$, let $\ga_k^d$ be the outer dual-open boundary of $\ga_0[T'_{d,k},n_0)$ and let $\ga_k^o$ be the outer primal-open boundary of $\ga_0[\sigma'_{d,k},n_0)$.
\end{enumerate}

Let
\[
 W_k=B\left(v_0;\frac{3r_0R^k}{2},
                  \frac{3r_0R^k}{2}+\de R^{3/2}\right),
\]
and let $W_k^{\rm sel}$ be its inner half as follows,
\[
 W_k^{\mathrm{sel}}
 =B\left(v_0;\frac{3r_0R^k}{2},
                  \frac{3r_0R^k}{2}+\frac{\delta R^{3/2}}2\right).
\]
Put
\[
 a_k=\frac{3r_0R^k}{2}.
\]
Since the reversed paths are indexed away from $v_0$, define their
final exits from $B(v_0;a_k)$ by
\begin{align*}
 j_{ok}
 &=\inf\{j\geq1:\ga_k^{o-}[j,\infty)
             \subset B^c(v_0;a_k)\},\\
 j_{dk}
 &=\inf\{j\geq1:\ga_k^{d-}[j,\infty)
             \subset B^c(v_0;a_k)\}.
\end{align*}
Thus, after the index $j_{ok}$, respectively $j_{dk}$, the
corresponding reversed boundary path never enters $B(v_0;a_k)$
again.  Since these paths are finite, the notation
$\ga_k^{x-}[j,\infty)$ means the segment from index $j$ to the
last edge of the path.  Figure \ref{translation-annuli-new}
shows the objects used at one scale.

\begin{figure}[!htbp]
\begin{center}
\resizebox{0.4\textwidth}{!}{
\begin{tikzpicture}[x=1.02cm,y=1.02cm,>=Stealth,line cap=round,line join=round]
  \coordinate (vzero) at (0,0);
  \coordinate (vone) at (-0.55,0.02);

  \draw[gray!70,line width=0.9pt] (vzero) circle (3.52);
  \draw[gray!70,line width=0.9pt] (vzero) circle (0.66);
  \draw[thick] (vzero) circle (3.05);
  \draw[thick] (vzero) circle (1.18);
  \draw[dashed] (vzero) circle (2.05);
  \draw[dashed] (vzero) circle (2.34);
  \node at (-2.98,2.73) {$\mathcal A_k$};
  \node at (-2.35,1.72) {$B_k\setminus D_k$};
  \node at (0.50,0.73) {$D_k$};
  \node at (2.12,-0.28) {$W_k$};

  \draw[thick]
    (-3.32,-0.28) .. controls (-2.15,-0.55) and (-1.82,0.16) ..
    (-1.28,0.42) .. controls (-0.92,0.58) and (-0.75,0.18) .. (vzero);
  \draw[-{Stealth[length=2.2mm]},thick] (-0.82,0.24)--(-0.16,0.05);
  \node[anchor=east] at (-2.55,-0.78) {$\ga_0[0,n_0]$};

  \draw[very thick]
    (vzero) .. controls (-0.82,0.58) and (-1.52,1.36) .. (-3.00,2.20);
  \draw[-{Stealth[length=2.2mm]},very thick] (-2.34,1.78)--(-2.65,1.99);
  \node[anchor=south east] at (-2.65,2.03) {$\ga_k^{o-}$};

  \draw[very thick,dashed]
    (vzero) .. controls (-0.78,-0.45) and (-1.50,-1.40) .. (-3.04,-2.30);
  \draw[-{Stealth[length=2.2mm]},very thick,dashed] (-2.34,-1.86)--(-2.68,-2.08);
  \node[anchor=north east] at (-2.65,-2.13) {$\ga_k^{d-}$};

  \draw[-{Stealth[length=1.8mm]},very thick] (-1.60,1.31)--(-1.87,1.43);
  \node[anchor=east] (weststepo) at (-2.28,1.16) {westward step};
  \draw[-{Stealth[length=1.35mm]},thin] (weststepo.east) -- (-1.73,1.37);
  \draw[-{Stealth[length=1.8mm]},very thick,dashed] (-1.57,-1.30)--(-1.83,-1.42);
  \node[anchor=east] (weststepd) at (-2.28,-1.14) {westward step};
  \draw[-{Stealth[length=1.35mm]},thin] (weststepd.east) -- (-1.70,-1.36);

  \fill (vzero) circle (1.6pt) node[below right] {$v_0$};
  \fill (vone) circle (1.6pt) node[below left] {$v_1$};
  \draw[-{Stealth[length=1.5mm]}] (vone)--(vzero);
  \draw[<->] (2.53,-2.07)--(2.72,-2.25);
  \node[anchor=north] at (2.54,-2.58) {$\de R^{3/2}$};
\end{tikzpicture}}
\end{center}
\caption{One full annulus $\mathcal A_k$ used in the translation construction.  The two thick circles bound the central subannulus $B_k\setminus D_k$.  The incoming exploration path enters $B_k$ for the last time and then reaches $v_0$.  The solid and dashed curves represent the reversed primal-open and dual-open boundary paths.  A westward step of each reversed path is selected in the thin annulus $W_k$. }
\label{translation-annuli-new}
\end{figure}

\medskip
\noindent
{\bf Step 2: bad annuli.}

We first describe the exceptional configurations.  A scale $k$ is bad if one of the following obstructions occurs.
\begin{enumerate}
\item The incoming exploration path enters $D_k$ and later returns to $\p B_k$.
\item The first contacts with the primal-open and dual-open sides both occur at the final entrance point $\ga_0(s_k)$.
\item A close approach relevant to the translation occurs.  There are two possibilities.  First, at some point within distance $4\de$ of $W_k$, there are four vertex-disjoint arms of types $0,0,1,1$ or $0,1,0,1$ from scale $2\de$ to scale $r_0R^k/2$.  Second, there exists a medial vertex $z\in D_k\setminus B(v_0;r_0R^k)$ such that, if
\[
 d_z=\dist\big(z,B(v_0;r_0R^k)\big),
\]
there are six disjoint arms from a $2\de$-neighbourhood of $z$ to $\p B(z;d_z)$, and four of these arms continue from $\p B(z;d_z)$ to distance comparable with $r_0R^k$.  All these arms are contained in $\mathcal A_k$.
\item One of the path portions between the two contacts returns to $D_k$, or a later part comes within distance $2\de$ of the earlier part used to define it.
\item The exploration path through $v_0$, or an exploration path through $v_1$ in the same bond configuration, makes a surrounding return in the outer part of $B_k$.
\end{enumerate}

We now give the precise definitions.  First reveal $\Loop_0$ and the
bonds incident to it, so that $r_0$ and the annuli above are fixed.  For
$k_0\leq k\leq m_0$, let $\mathcal F_k$ be the sigma-field generated by
the states of all primal bonds whose centres lie outside the full annulus
$\mathcal A_k$.  Conditional on $\mathcal F_k$, the random bonds whose
centres lie in $\mathcal A_k$ remain independent critical bonds.  The
exposed configuration only induces boundary connections on the two
boundary components of $\mathcal A_k$.  All conditional estimates below
are uniform over these induced connections.

Let $\mathcal E_k$ be the event that $\ga_0[0,n_0]$ touches or crosses $\p B_k$ after it has touched or crossed $\p D_k$.  Put
\[
 \mathcal G_{1k}=\{T_{o,k}=T_{d,k}=s_k\}.
\]
Put
\[
 W_k^+=B\left(v_0;a_k-4\de,
                    a_k+\de R^{3/2}+4\de\right).
\]
Let $\mathcal G_{2k}$ be the event that there exists $z\in W_k^+$
for which four vertex-disjoint arms, in cyclic order of types
$0,0,1,1$ or $0,1,0,1$, join $\p B(z;2\de)$ to
$\p B(z;r_0R^k/2)$.  The fixed enlargement by $4\de$ accounts for the lattice-scale
enlargement of the set of bonds used in the translation below.

Let $\mathcal G_{3k}$ be the event that there exists a medial vertex
$z\in D_k\setminus B(v_0;r_0R^k)$ with
\[
 d_z=\dist\big(z,B(v_0;r_0R^k)\big)>2\de
\]
such that six disjoint arms join $\p B(z;2\de)$ to
$\p B(z;d_z)$ and four disjoint arms continue from
$\p B(z;d_z)$ to distance comparable with $r_0R^k$.  We require all
these arms to lie in $\mathcal A_k$.  The four arms in the second part
are four of the six arms in the first part.

For the four possible orders of the two contacts, define
\begin{align*}
 \mathcal E_{ok1}
 &=\{s_k<T_{d,k}<T_{o,k},\
       \ga_0[T_{o,k},\sigma_{o,k}]\cap D_k\ne\emptyset\},\\
 \mathcal E_{dk1}
 &=\{s_k<T_{o,k}<T_{d,k},\
       \ga_0[T_{d,k},\sigma_{d,k}]\cap D_k\ne\emptyset\},\\
 \mathcal E'_{ok1}
 &=\{s_k=T_{o,k},\ T_{d,k}<T'_{o,k},\
       \ga_0[T'_{o,k},\sigma'_{o,k}]\cap D_k\ne\emptyset\},\\
 \mathcal E'_{dk1}
 &=\{s_k=T_{d,k},\ T_{o,k}<T'_{d,k},\
       \ga_0[T'_{d,k},\sigma'_{d,k}]\cap D_k\ne\emptyset\}.
\end{align*}
Let $\mathcal E_{ok2}$ be the event that $s_k<T_{d,k}<T_{o,k}$ and there is an index $j\in[\sigma_{o,k},n_0)$ such that
\[
 \ga_0[\sigma_{o,k},j]\cap D_k\ne\emptyset,
 \qquad
 \dist\big(\ga_0(j),\ga_0[0,\sigma_{o,k}]\big)\leq2\de.
\]
Define $\mathcal E_{dk2}$, $\mathcal E'_{ok2}$ and $\mathcal E'_{dk2}$ in the same way, with $\sigma_{o,k}$ and the corresponding contact-order condition replaced respectively by $\sigma_{d,k}$, $\sigma'_{o,k}$ and $\sigma'_{d,k}$ and their corresponding conditions.  Set
\begin{align*}
 \mathcal E_{ok}&=\mathcal E_{ok1}\cup\mathcal E_{ok2},&
 \mathcal E_{dk}&=\mathcal E_{dk1}\cup\mathcal E_{dk2},\\
 \mathcal E'_{ok}&=\mathcal E'_{ok1}\cup\mathcal E'_{ok2},&
 \mathcal E'_{dk}&=\mathcal E'_{dk1}\cup\mathcal E'_{dk2},
\end{align*}
and
\[
 \mathfrak E_k
 =\mathcal E_{ok}\cup\mathcal E_{dk}
  \cup\mathcal E'_{ok}\cup\mathcal E'_{dk}.
\]

We next define the surrounding-return events.  Let $\mathcal F_{ok}$ be the union of the following two events.
\begin{enumerate}
\item The time $\sigma_{o,k}$ is defined and there are $j_1,j_2\in[\sigma_{o,k},n_0]$ such that
\[
 \dist\big(\ga_0(j_1),\ga_0(j_2)\big)=\de,
\]
both points belong to
\[
 B_k\setminus B\left(v_0;\frac{3r_0R^k}{2}\right),
\]
and the union of $\ga_0[j_1,j_2]$ and the straight segment joining its endpoints surrounds $v_0$.
\item In the same bond configuration, there is an exploration path $\ga_1$ through $v_1$ and indices $j_1,j_2$ such that
\[
 \dist\big(\ga_1(j_1),\ga_1(j_2)\big)=\de,
\]
both points belong to
\[
 B_k\setminus B\left(v_0;\frac{3r_0R^k}{2}\right),
\]
and the union of $\ga_1[j_1,j_2]$ and the straight segment joining its endpoints surrounds $v_0$.
\end{enumerate}
Define $\mathcal F_{dk}$, $\mathcal F'_{ok}$ and $\mathcal F'_{dk}$ in the same way, replacing $\sigma_{o,k}$ in the first event by $\sigma_{d,k}$, $\sigma'_{o,k}$ and $\sigma'_{d,k}$, respectively.  Put
\[
 \mathfrak F_k
 =\mathcal F_{ok}\cup\mathcal F_{dk}
  \cup\mathcal F'_{ok}\cup\mathcal F'_{dk}.
\]

Put
\[
 \mathcal B_k
 =\mathcal E_k\cup\mathcal G_{1k}\cup\mathcal G_{2k}\cup\mathcal G_{3k}
  \cup\mathfrak E_k\cup\mathfrak F_k.
\]
We call the scale $k$ bad precisely on $\mathcal B_k$.

We estimate these bad events in turn.  The event $\mathcal E_k$ gives four alternating arms across the central subannulus $B_k\setminus D_k\subset\mathcal A_k$, whose outer-to-inner radius ratio is $R/4$.  Hence
\[
 \Pro(\mathcal E_k\mid\mathcal F_k)\preceq R^{-1-\al}.
\]

For $\mathcal G_{1k}$, put $z=\ga_0(s_k)$.  The two portions $\ga_0[0,s_k]$ and $\ga_0[s_k,n_0]$, together with their primal-open and dual-open sides, give four arms from a $2\de$-neighbourhood of $z$ to distance $r_0R^k$.  There are at most $Cr_0R^{k+1}\de^{-1}$ possible positions for $z$.  Thus, by \eqref{four} and \eqref{four-a},
\[
 \Pro(\mathcal G_{1k}\mid\mathcal F_k)
 \preceq r_0R^{k+1}\de^{-1}
 \left(\frac{\de}{r_0R^k}\right)^{1+\al}
 \preceq R^{-1-\al}.
\]

For $\mathcal G_{2k}$, cover $W_k^+$ by at most
\[
 C(r_0R^k)(\de R^{3/2})\de^{-2}
\]
lattice balls of radius $2\de$.  At each possible centre the four arms run to distance $r_0R^k/2$.  Therefore
\[
 \Pro(\mathcal G_{2k}\mid\mathcal F_k)
 \preceq C(r_0R^k)(\de R^{3/2})\de^{-2}
 \left(\frac{\de}{r_0R^k}\right)^{1+\al}
 \preceq R^{-1-\al}.
\]

We next estimate $\mathcal G_{3k}$.  Put $r_k=r_0R^k$ and divide the
possible values of $d_z$ into dyadic intervals
\[
 2^j\de\leq d_z<2^{j+1}\de,
 \qquad 1\leq j\leq
 \left\lfloor\log_2(r_k/\de)\right\rfloor.
\]
For a fixed $j$, the possible vertices $z$ lie in a strip of width of
order $2^j\de$ around $\p B(v_0;r_k)$.  Hence their number is at most
\[
 C r_k(2^j\de)\de^{-2}.
\]
For each such $z$, arm separation, the six-arm estimate and the
four-arm estimate give
\[
 C\left(\frac{\de}{2^j\de}\right)^{2+\alpha_1}
  \left(\frac{2^j\de}{r_k}\right)^{1+\alpha_2}.
\]
It follows that the contribution of this dyadic interval is at most
\[
 C\left(\frac{\de}{r_k}\right)^{\alpha_2}
 2^{j(\alpha_2-\alpha_1)}.
\]
We can select $\al_1$ and $\al_2$ such that $\al_1=\al_2$. Summing over $j$ and recalling that
$\al=\min(\alpha_1,\alpha_2)$, we obtain
\[
 \Pro(\mathcal G_{3k}\mid\mathcal F_k)
 \preceq
 C\log(r_k/\de)\left(\frac{\de}{r_k}\right)^{\al}.
\]
Since $r_0\geq8\de$ and $k\geq k_0$, the choice of $k_0$ in
Step~1 is large enough to give, for $R$ large,
\[
 \Pro(\mathcal G_{3k}\mid\mathcal F_k)
 \preceq R^{-1-\al}.
\]

We next estimate $\mathfrak E_k$.  Since the events
$\mathcal E_k$, $\mathcal G_{1k}$ and $\mathcal G_{2k}$ have already
been estimated, it is enough to consider
\[
 \mathfrak E_k\cap
 \mathcal E_k^c\cap
 \mathcal G_{1k}^c\cap
 \mathcal G_{2k}^c.
\]
On this event, one of the four nondegenerate orders of the two contacts holds.
Moreover, the two contact points, the corresponding final return point,
and, for an event of type $\mathcal E_{xk2}$, the first later
close-approach point used below lie in the central subannulus
$B_k\setminus D_k\subset\mathcal A_k$, up to the $2\de$
neighbourhoods appearing in the definitions.
 Put $\rho=\log R$ and divide the
radial interval from $2r_0R^k$ to $r_0R^{k+1}/2$ into the intervals
\[
 I_q=[2r_0R^k\rho^q,2r_0R^k\rho^{q+1}],
 \qquad
 0\leq q\leq
 \left\lfloor\frac{\log(R/4)}{\log\rho}\right\rfloor.
\]

We first consider $\mathcal E_{ok1}$.  On this event,
$s_k<T_{d,k}<T_{o,k}$ and
$\ga_0[T_{o,k},\sigma_{o,k}]$ meets $D_k$.  Put
\[
 d_{ok,1}=|\ga_0(t_{o,k})-v_0|,
 \qquad
 d_{ok,2}=|\ga_0(t_{d,k})-v_0|,
 \qquad
 d_{ok,3}=|\ga_0(\sigma_{o,k})-v_0|.
\]
Fix the order of these three radii and the intervals $I_q$ which
contain them.  By the definitions of the two contacts and of
$\sigma_{o,k}$, the corresponding portions of the exploration path,
together with their primal-open and dual-open sides, give four
disjoint arms across each annulus whose boundary radii are
consecutive terms of
\[
 2r_0R^k,\quad
 d_{ok,j_1},\quad d_{ok,j_2},\quad d_{ok,j_3},\quad
 \frac{r_0R^{k+1}}2,
\]
where
\[
 d_{ok,j_1}\leq d_{ok,j_2}\leq d_{ok,j_3}.
\]
Arm separation, quasi-multiplicativity and
\eqref{four}--\eqref{four-a} give
\[
 C\rho^{3+3\al}R^{-1-\al}
\]
for every fixed order and every fixed choice of the three radial
intervals.  Summing over the six orders and over the possible
intervals gives
\[
 \Pro(\mathcal E_{ok1}\cap
 \mathcal E_k^c\cap
 \mathcal G_{1k}^c\cap
 \mathcal G_{2k}^c\mid\mathcal F_k)
 \preceq C(\log R)^C R^{-1-\al}.
\]

We next consider $\mathcal E_{ok2}$.  Let $j_*$ be the smallest index
strictly larger than $\sigma_{o,k}$ such that
\[
 \ga_0[\sigma_{o,k},j_*]\cap D_k\neq\emptyset,
 \qquad
 \dist\big(\ga_0(j_*),\ga_0[0,\sigma_{o,k}]\big)\leq2\de.
\]
Let $i_*\leq\sigma_{o,k}$ be the largest index such that
\[
 |\ga_0(i_*)-\ga_0(j_*)|\leq2\de,
\]
and put
\[
 d_{ok,4}=|\ga_0(j_*)-v_0|.
\]
By the minimality of $j_*$, the two portions
\[
 \ga_0[i_*,\sigma_{o,k}]
 \quad\hbox{and}\quad
 \ga_0[\sigma_{o,k},j_*]
\]
do not come within distance $2\de$ before their last approach.
Their primal-open and dual-open sides therefore give four disjoint
arms across every annulus whose boundary radii are consecutive
terms of the ordered list
\[
 2r_0R^k,\quad
 d_{ok,1},\quad d_{ok,2},\quad d_{ok,3},\quad d_{ok,4},\quad
 \frac{r_0R^{k+1}}2.
\]
Since
\[
 |\ga_0(i_*)-\ga_0(j_*)|\leq2\de,
\]
the two corresponding radial distances belong to the same interval
$I_q$ and require only one radial choice.  Summing over the orders
and over the four radial intervals gives
\[
 \Pro(\mathcal E_{ok2}\cap
 \mathcal E_k^c\cap
 \mathcal G_{1k}^c\cap
 \mathcal G_{2k}^c\mid\mathcal F_k)
 \preceq C(\log R)^C R^{-1-\al}.
\]

The same proof applies to $\mathcal E_{dk1}$ and
$\mathcal E_{dk2}$ after interchanging $o$ and $d$.  It applies to
the primed events after replacing the unprimed stopping times by
the corresponding primed stopping times.  Consequently,
\begin{align*}
 \Pro(\mathfrak E_k\mid\mathcal F_k)
 &\leq
 \Pro(\mathcal E_k\mid\mathcal F_k)
 +\Pro(\mathcal G_{1k}\mid\mathcal F_k)
 +\Pro(\mathcal G_{2k}\mid\mathcal F_k)\\
 &\quad+
 \Pro\left(
 \mathfrak E_k\cap
 \mathcal E_k^c\cap
 \mathcal G_{1k}^c\cap
 \mathcal G_{2k}^c
 \,\middle|\,\mathcal F_k
 \right)\\
 &\preceq C(\log R)^C R^{-1-\al}.
\end{align*}

Now, consider $\mathfrak F_k$.  Choose the first pair of nearby
points in its definition.  The two path portions, together with their
primal-open and dual-open sides, give six arms from scale $2\de$ to
scale $r_0R^k/2$.  The number of possible positions is at most $C(r_0R^k)(r_0R^{k+1})\de^{-2}$.  By \eqref{six},
\[
 \Pro(\mathfrak F_k\mid\mathcal F_k)
 \preceq C(r_0R^k)(r_0R^{k+1})\de^{-2}
 \left(\frac{\de}{r_0R^k}\right)^{2+\al}
 \preceq R^{-1-\al}.
\]

Thus
\begin{equation}
 \Pro(\mathcal B_k\mid\mathcal F_k)
 \preceq C(\log R)^C R^{-1-\al}.
 \label{trans-bad-scale}
\end{equation}
The estimate is uniform over all admissible boundary connections induced by the already exposed configuration.

\medskip
\noindent
{\bf Step 3: westward steps of the reversed boundary paths.}

Recall that $\ga_k^x$ is indexed towards $v_0$, whereas
$\ga_k^{x-}$ is the same path indexed away from $v_0$, where
$x=o,d$.  When $\ga_k^{x-}(j)$ is regarded as a point in the plane,
it denotes the centre of the $j$th bond of the path.  We say that
$\ga_k^{x-}$ has a \emph{westward step}
at the index $j$ if
\[
 \Re\ga_k^{x-}(j+1)<\Re\ga_k^{x-}(j).
\]
Equivalently, the corresponding step of $\ga_k^x$, followed towards
$v_0$, is eastward.  We use ``eastward'' and ``westward'' for the
absolute horizontal directions, and ``left turn'' and ``right turn''
only relative to the orientation of the path.  Thus, when $\ga_k^x$
arrives downward, a left turn is eastward and a right turn is
westward; when it arrives upward, a right turn is eastward and a left
turn is westward.  In particular, a westward step of $\ga_k^{x-}$
may be a left turn or a right turn, according to its incoming vertical
direction.

On $\mathcal B_k^c$, each path $\ga_k^{x-}$ crosses
$W_k^{\rm sel}$.  In the notation introduced in Step 1, define the
outer ends of the two final crossings by
\begin{align*}
 J_{ok}
 &=\sup\left\{j\geq j_{ok}:
   \ga_k^{o-}[j_{ok},j]\subset
   B\left(v_0;a_k+\frac{\de R^{3/2}}2\right)\right\},\\
 J_{dk}
 &=\sup\left\{j\geq j_{dk}:
   \ga_k^{d-}[j_{dk},j]\subset
   B\left(v_0;a_k+\frac{\de R^{3/2}}2\right)\right\}.
\end{align*}
Thus $\ga_k^{o-}[j_{ok},J_{ok}]$ and
$\ga_k^{d-}[j_{dk},J_{dk}]$ are the final outward crossings of
$W_k^{\rm sel}$ by the two reversed boundary paths.

Put
\[
 b_k=a_k+\frac{\de R^{3/2}}2,
 \qquad
 M=\left\lfloor\frac{R^{3/2}}{40}\right\rfloor,
\]
and, for $0\leq\ell\leq20M$, let
\[
 \mathbb B_\ell=B(v_0;b_k-\ell\de).
\]
For $0\leq m<M$, set
\[
 \mathcal Q_m=\mathbb B_{20m}\setminus
                  \overline{\mathbb B}_{20m+20}.
\]
These annuli are disjoint and have width $20\de$.

We first treat the primal-open boundary $\ga_k^o$.  For a fixed
$m$, follow $\ga_k^o$ from its outer end towards $v_0$.  Let $E_z$ be
the first bond which crosses the circle $\p\mathbb B_{20m+5}$, and let
$z$ be its centre.

Let $\mathcal S$ be the set of primal bonds touched or crossed by the
part of $\ga_0$ whose outer primal-open boundary is $\ga_k^o$.  We
reveal
\[
 \big(\mathcal S\cap\mathbb B_{20m+5}^c\big)
 \cup
 \left\{E\in\mathcal S:
 E\cap\p\mathbb B_{20m+5}\text{ contains a non-endpoint of }E
 \right\},
\]
together with the states of these bonds.  This information determines
$\ga_k^o$ from its outer end up to $E_z$ as the outermost
self-avoiding primal-open boundary.  It does not reveal the
states of bonds on the unexamined side of the stopping circle.

Put
\[
 \begin{array}{lll}
 z_N=z+i\de,&
 z_{NW}=z+\dfrac{\de e^{3\pi i/4}}{\sqrt2},&
 z_{NNW}=z+\sqrt2\de e^{3\pi i/4},\\[2mm]
 z_W=z-\de,&
 z_{SW}=z+\dfrac{\de e^{-3\pi i/4}}{\sqrt2}.&
 \end{array}
\]
For $\xi=N,NW,NNW,W,SW$, let $E_\xi$ be the primal bond centred at
$z_\xi$.  We now bound the conditional probability that
$\ga_k^{o-}$ has no westward step while crossing $\mathcal Q_m$.

Suppose first that $z$ lies below the horizontal line through $v_0$.
Assume that the stopping bond $E_z$ is upward when $\ga_k^o$ is
followed towards $v_0$.  We first check that $E_N$ and $E_{NW}$ have
not been examined before $E_z$ is located.  If $E_{NW}$ were open and
the relevant part of $\ga_0$ touched it before reaching $E_z$, then
$E_z$ would not belong to the outermost primal-open boundary
$\ga_k^o$.  If $E_{NW}$ were closed and the dual edge crossing it was touched
before $E_z$, then either $E_z\notin\ga_k^o$ or the local
four-arm event $\mathcal G_{2k}$ would occur.  Finally, if the relevant
part of $\ga_0$ passed through the centre $z_N$ before reaching $E_z$,
then it would already have passed through $z_{NW}$, or again
$E_z\notin\ga_k^o$.  Since $E_z\in\ga_k^o$ and
$\mathcal B_k^c\subset\mathcal G_{2k}^c$, neither $E_N$ nor $E_{NW}$
has been touched or crossed before $E_z$.  Their states have therefore not been revealed.  Conditional on the revealed bonds, they are independent fair Bernoulli bonds.

If both $E_N$ and $E_{NW}$ are dual-open, the outer primal-open
boundary is forced to make an eastward step, which is a right turn in
this upward case, as it is followed towards $v_0$, before leaving
$\mathcal Q_m$.  Reversing the orientation, this is a westward step
of $\ga_k^{o-}$.  Hence the conditional probability of no such step
is at most
\[
 1-2^{-2}=\frac34.
\]

Assume next that the stopping bond $E_z$ is westward when $\ga_k^o$
is followed towards $v_0$.  The same outermost-boundary argument shows
that the relevant part of $\ga_0$ has not passed through $z_{NW}$ or
$z_W$ before reaching $E_z$.  Indeed, an earlier contact with the
corresponding open bond would remove $E_z$ from $\ga_k^o$, while an
earlier contact with the corresponding dual-open bond would either do
the same or force $\mathcal G_{2k}$.  Thus $E_{NW}$ and $E_W$ have not
been examined.

If the bond $E_N$, together with its endpoints, is not contained in
$\overline{\mathbb B}_{20m+5}$, choose
$\varrho\in\{1,2,3,4,5\}$ so that the same stopping bond is obtained at
$\p\mathbb B_{20m+\varrho}$ and $E_N$, together with its endpoints, is
contained in $\overline{\mathbb B}_{20m+\varrho}$.  We then use this
new stopping circle.  The local bonds used below can now be revealed
successively.  If $E_{SW}$ lies on the already examined side of the
stopping circle, the fourth possibility below cannot occur and is
omitted.

The event that no eastward step of $\ga_k^o$ is forced is contained in the
union of the following four disjoint possibilities:
\begin{enumerate}
\item At least one of $E_N$ and $E_{NW}$ is dual-open.

\item The bonds $E_N,E_{NW},E_{NNW}$ are open, and at least one
      of $E_N,E_{NW}$ does not belong to $\ga_k^o$.

\item The bond $E_{NNW}$ is dual-open, the bonds
      $E_W,E_{NW},E_N$ are open, and at least one of
      $E_N,E_{NW}$ does not belong to $\ga_k^o$.

\item The bonds $E_{NNW},E_W$ are dual-open, the bonds
      $E_{SW},E_{NW},E_N$ are open, and at least one of
      $E_N,E_{NW}$ does not belong to $\ga_k^o$.
\end{enumerate}
The conditional probabilities of these four events are bounded above,
respectively, by
\[
 1-2^{-2},\qquad 2^{-3},\qquad 2^{-4},\qquad 2^{-5}.
\]
Indeed, in the last three cases we have dropped the additional
restriction that at least one of $E_N,E_{NW}$ does not belong to
$\ga_k^o$.  Dropping this restriction only enlarges the event.
Consequently,
\beq
&& \Pro(\text{no westward step in this local configuration}
       \mid\text{the revealed bonds})\\
 &\leq&
 1-2^{-2}+2^{-3}+2^{-4}+2^{-5}
 =\frac{31}{32}.
\eeq
\begin{figure}[!htbp]
\begin{center}
\resizebox{0.7\textwidth}{!}{
\begin{tikzpicture}[x=1.15cm,y=1.15cm,>=Stealth,line cap=round]
  \coordinate (z) at (0,0);
  \draw[dashed] (1.92,-1.48) arc[start angle=-54,end angle=62,radius=2.78];
  \node[align=left,anchor=west] at (3.32,0.45) {outside the\\stopping circle};
  \node[align=left,anchor=east] at (-1.52,-1.34) {inside the\\stopping circle};

  \fill (z) circle (1.8pt) node[below right] {$z$};
  \fill ($(z)+(0,1)$) circle (1.8pt) node[above right] {$i\de$};
  \fill ($(z)+(-0.78,0.78)$) circle (1.8pt)
    node[left] {$\de e^{3\pi i/4}/\sqrt2$};
  \fill ($(z)+(-1.40,1.40)$) circle (1.8pt)
    node[left] {$\sqrt2\de e^{3\pi i/4}$};
  \fill ($(z)+(-1,0)$) circle (1.8pt) node[above left] {$-\de$};
  \fill ($(z)+(-0.78,-0.78)$) circle (1.8pt)
    node[above left] {$\de e^{-3\pi i/4}/\sqrt2$};

  \draw[-{Stealth[length=2.2mm]},thick] ($(z)+(1.10,0)$)--($(z)+(-0.72,0)$);
\end{tikzpicture}}
\end{center}
\caption{The local bond centres in the case which gives the bound $31/32$.  The arrow shows $\ga_k^o$ followed towards $v_0$; the reverse path $\ga_k^{o-}$ is followed in the opposite direction.  Here $z$ lies below the horizontal line through $v_0$, and the stopping bond points westward towards $v_0$.}
\label{translation-local-step-new}
\end{figure}

Suppose now that $z$ lies above the horizontal line through $v_0$.
If the stopping bond is downward, the reflected version of the first
case gives the bound $1-2^{-2}$; in this downward case the forced
eastward step is a left turn of $\ga_k^o$.  If the stopping bond is
westward, the event that no eastward step of $\ga_k^o$ is forced is
contained in the
following three disjoint possibilities:
\begin{enumerate}
\item the bond centred at $z_W$ is open;
\item the bond centred at $z_W$ is dual-open and the bond centred at
      $z_{NW}$ is open;
\item the bonds centred at $z_W,z_{NW}$ are dual-open and the bond
      centred at $z_{SW}$ is open.
\end{enumerate}
Their total probability is
\[
 2^{-1}+2^{-2}+2^{-3}=\frac78.
\]
The same change of the stopping circle is made when one of the
adjacent bonds is not on its inner side.

Finally, suppose that $z$ lies on the horizontal line through $v_0$.
Use the next stopping circle $\p\mathbb B_{20m+6}$.  If its stopping
bond has centre off the horizontal line, one of the preceding cases
applies.  If the new centre is still on the line, one of the two states of an unrevealed adjacent bond forces the required step.  The
conditional probability of no westward step is then at most
$1/2$.

The largest of the four bounds is
\[
 \max\left\{\frac34,\frac{31}{32},\frac78,\frac12\right\}
 =\frac{31}{32}.
\]
For the dual-open boundary $\ga_k^d$, interchange primal-open and
dual-open bonds.  The same bound follows.

For a fixed possible first bond of the crossing, reveal the annuli
$\mathcal Q_m$ successively.  The bonds used in the next annulus have not been revealed, since the annuli are disjoint.  The conditional
chain rule therefore gives a factor $(31/32)^M$.  There are at most
$C(d_0/\de)^2$ possible first bonds.  Hence, for $x=o,d$,
\begin{equation}
\begin{split}
 &\Pro\left(\begin{array}{c}
 \mathcal B_k^c\hbox{ occurs and }\ga_k^{x-}\hbox{ has no westward step}\\
 \hbox{when it is followed away from }v_0\hbox{ across }W_k^{\rm sel}
 \end{array}\middle|\mathcal F_k\right)\\
 &\hspace{18mm}\preceq C(d_0/\de)^2(31/32)^M
 \preceq C\exp(-cR^{3/2}).
\end{split}
 \label{trans-no-left-step}
\end{equation}

By \eqref{trans-no-left-step} and a union bound over $k_0\leq k\leq m_0$ and $x=o,d$, the probability that one of the westward steps of the two reversed boundary paths needed below does not exist at some scale is at most
\[
 C m_0\exp(-cR^{3/2})\preceq C\exp(-c'R^{3/2}).
\]
Since $R=\log(d_0/\de)$, this is smaller than any fixed power of $\de/d_0$.  We add this event to the exceptional set and ignore it.  Thus, in the rest of the proof, we may assume that for every $k_0\leq k\leq m_0$ with $\mathcal B_k^c$, the two reversed boundary paths have the required westward steps in $W_k^{\rm sel}$.

For each fixed \(k\), let \(\omega^{L,k}\) be the configuration obtained by the construction in Steps 4 and 5, with \(k\) in place of \(K\). This construction is performed on \(\mathcal B_k^c\), with the required westward steps supplied above. We call \(\mathcal A_k\) good if \(\mathcal B_k^c\) holds both in \(\omega\) and in \(\omega^{L,k}\). (On a first reading, the reader may disregard the condition on \(\omega^{L,k}\).)
 Note that from the proof of one-to-one property in Step 6, the translation map $\omega \to \omega^{L,k}$ is one-to-one. By
\eqref{trans-bad-scale},
\begin{equation}
 \Pro(\mathcal A_k\hbox{ is not good}\mid\mathcal F_k)
 \preceq C(\log R)^C R^{-1-\al}.
 \label{trans-good-one-scale}
\end{equation}
For $R$ large, the right hand side is at most
$R^{-1-\al+\ep/4}$.  Reveal the full annuli from the outside towards
$v_0$.  The bound above is uniform over the information already revealed,
so it can be iterated.  Hence
\begin{equation}
 \Pro(\hbox{there is no good annulus}\mid r_0)
 \preceq C_\ep\left(\frac{r_0}{d_0}\right)^{1+\al-\ep/2}.
 \label{trans-no-good-cond}
\end{equation}
To remove the conditioning on $r_0$, sum over
$2^\ell\de\leq r_0<2^{\ell+1}\de$.  If $r_0$ is of this order, the
exploration path and the attached loop give four arms from scale $\de$
to scale $2^\ell\de$.  Arm separation and quasi-multiplicativity give
\begin{align}
 &\Pro\big(\mathcal C_0^D,
           \ \hbox{there is no good annulus}\big)\notag\\
 &\quad\preceq C_\ep\sum_\ell
 \left(\frac{\de}{2^\ell\de}\right)^{1+\al}
 \left(\frac{2^{\ell+1}\de}{d_0}\right)^{1+\al-\ep/2}
 \preceq C_\ep(\de/d_0)^{1+\al-\ep}.
 \label{trans-no-good-sum}
\end{align}

On the complement of these events, let
\[
 K=\max\{k\in[k_0,m_0]:\mathcal A_k\hbox{ is good}\}.
\]
We give the construction in the case
$s_K<T_{d,K}<T_{o,K}$.  The other three orders are treated in the same
way, with the corresponding primed or unprimed paths from Step~1.

\par\medskip
\noindent
{\bf Step 4: the two auxiliary paths and the choice of the translated portions.}

We first define two complete auxiliary curves.  These curves are defined
before the two cutting indices are chosen.  Let $\mathcal E_{K,c}^d$ be
the set of primal edges which are on the left hand side when one traverses
the whole path $\ga_K^{o-}$ from $v_0$ to its outer endpoint and which
share one vertex with an edge of $\ga_K^{o-}$.  After making every edge in
$\mathcal E_{K,c}^d$ dual-open, the resulting auxiliary exploration path
from the outer endpoint of $\ga_K^{o-}$ to $v_0$ is denoted by
$\Gamma_{1c}$.  In the order $s_K<T_{d,K}<T_{o,K}$, its initial point is
$\ga_0(T_{o,K})$. Note that $\Gamma_{1c}$ consists of medial edges, but it is not an exploration path in $\omega$.

Similarly, let $\mathcal E_{K,c}^o$ be the set of dual edges which are on
the right hand side when one traverses the whole path $\ga_K^{d-}$ from
$v_0$ to its outer endpoint and which share one vertex with a dual edge of
$\ga_K^{d-}$.  After making every edge whose dual is in $\mathcal E_{K,c}^o$ open, the
resulting auxiliary exploration path from the outer endpoint of
$\ga_K^{d-}$ to $v_0$ is denoted by $\Gamma_{2c}$.  In the present order,
its initial point is $\ga_0(\sigma_{o,K})$.  The definitions in the other
three contact orders are the same, with the corresponding endpoints from
Step~1.  We index $\Gamma_{1c}$ and $\Gamma_{2c}$ towards $v_0$.

Put
\[
 h=\de R^{3/2},\qquad
 W_K^{\rm in}=B(v_0;a_K),\qquad
 W_K^{\rm out}=B^c(v_0;a_K+h).
\]
We use the following common occurrences of the two complete auxiliary
curves.  Let $q_{K,\rm out}^o$ and $q_{K,\rm out}^d$ be the largest
indices such that
\[
 \Gamma_{1c}(q_{K,\rm out}^o)
 =\Gamma_{2c}(q_{K,\rm out}^d)
 =:z_{K,\rm out}\in W_K^{\rm out}.
\]
There is only one such pair. If they don't exist,put
\[
 (q_{K,\rm out}^o,q_{K,\rm out}^d)=(0,0).
\]
  Since both complete auxiliary curves end at
$v_0$, they have a common occurrence in $W_K^{\rm in}$ after these two
indices.  Let $q_{K,\rm in}^o$ and $q_{K,\rm in}^d$ be the smallest
indices larger than $q_{K,\rm out}^o$ and $q_{K,\rm out}^d$, respectively,
such that
\[
 \Gamma_{1c}(q_{K,\rm in}^o)
 =\Gamma_{2c}(q_{K,\rm in}^d)
 =:z_{K,\rm in}\in W_K^{\rm in}.
\]

We next choose the crossings on which the cuts will be made.  For $R$
large, define
\[
 W_K^\circ
 =B\left(v_0;a_K+4\de,
                a_K+\frac h2-4\de\right)
 \subset W_K^{\rm sel}.
\]
List in their order along $\Gamma_{1c}$ all complete inward crossings of
$W_K^\circ$ made by
\[
 \Gamma_{1c}[q_{K,\rm out}^o,q_{K,\rm in}^o]
\]
as
\[
 C_{K,1}^o,C_{K,2}^o,\ldots,C_{K,N_o}^o.
\]
Define
\[
 C_{K,1}^d,C_{K,2}^d,\ldots,C_{K,N_d}^d
\]
in the same way from
$\Gamma_{2c}[q_{K,\rm out}^d,q_{K,\rm in}^d]$.  These two lists are nonempty.
We begin with the first crossing in each list.

The primal-open side of a selected $o$-crossing determines a corresponding
outward crossing of $W_K^\circ$ by $\ga_K^{o-}$.  The dual-open side of a
selected $d$-crossing determines a corresponding outward crossing by
$\ga_K^{d-}$.  On each of these two boundary crossings choose the first
westward step when the boundary path is followed away from $v_0$.  Denote
the two indices by $j_o,j_d$.  The exponentially small event on which one
of the required steps does not exist was already added to the exceptional
set in Step~3.

Whenever a selected boundary crossing is changed below, we also impose
the following one-sided requirement.  On the $o$-side, if there is a crossing of $W_K^\circ$
from its left-hand side by $\ga_K^{o-}$, pass to this complete crossing in the
$o$-crossing list and choose the first westward step there.  On the
$d$-side, if there is a crossing of $W_K^\circ$ from its
right-hand side by $\ga_K^{d-}$, do the same in the $d$-crossing list.  Continue until
\begin{equation}
\begin{split}
 &\text{there is no crossing of \(W_K^\circ\) by \(\ga_K^{o-}\) that enters from the left-hand side of \(\ga_K^{o-}\)}\\
 &\qquad \qquad\qquad\qquad \qquad\qquad\qquad \qquad\qquad\text{and lies to the right of \(\ga_K^{d-}\)},\\[2mm]
 &\text{there is no crossing of \(W_K^\circ\) by \(\ga_K^{d-}\) that enters from the right-hand side of \(\ga_K^{o-}\)}\\
 &\qquad \qquad\qquad\qquad \qquad\qquad\qquad \qquad\qquad\text{and lies to the left of \(\ga_K^{d-}\)}.
\end{split}
\label{trans-one-sided-tail}
\end{equation}
A crossing of $W_K^\circ$ by $\ga_K^{o-}$ from its
right-hand side, or by $\ga_K^{d-}$ from its left-hand
side, is allowed.  On $\mathcal G_{3K}^c\cap\mathfrak F_K^c$, such a
crossing does not interfere with the part used in the translation.  Every
change moves in one of the two finite
crossing lists, so this adjustment stops after finitely many changes. In other words, we choose a crossing of \(W_k^\circ\) by the outer primal-open boundary and a crossing by the outer dual-open boundary which are adjacent, in the sense that the region between their facing sides contains no other complete crossing of either boundary.

Now denote the set of primal edges which are on the left hand side when
one traverses along $\ga_K^{o-}$ from $v_0$ to $\ga_K^{o-}(j_o)$ and
which share one vertex with an edge in $\ga_K^{o-}[1,j_o]$ by
$\mathcal E_K^d$.  Denote the set of dual edges which are on the right
hand side when one traverses along $\ga_K^{d-}$ from $v_0$ to
$\ga_K^{d-}(j_d)$ and which share one vertex with a dual edge in
$\ga_K^{d-}[1,j_d]$ by $\mathcal E_K^o$.

If we make all edges in $\mathcal E_K^d$ dual-open, there is an
exploration path from $\ga_K^{o-}(j_o)$ to $v_0$.  Denote it by
$\Gamma_1$.  If we make all edges whose dual edges are in $\mathcal E_K^o$ open, there is an
exploration path from $\ga_K^{d-}(j_d)$ to $v_0$.  Denote it by
$\Gamma_2$.  These are terminal portions of $\Gamma_{1c}$ and
$\Gamma_{2c}$ after the two cuts.  In the original configuration they
need not be exploration paths, but they are paths of medial edges.  We
index them towards $v_0$ and put
\[
 \Gamma_1(0)=\ga_K^{o-}(j_o),\qquad
 \Gamma_2(0)=\ga_K^{d-}(j_d).
\]
From now on, a superscript $l$ denotes translation to the left by $\de$.

We list possible close approaches between a portion to be translated on one side
and the unshifted boundary tail on the other side.  Define
\beq
 \mathcal{EX}_K^o
 &=&\{E^*\in\mathcal E_K^o\cup\ga_K^{d-}[1,j_d]:
 \dist(v_{E^*},\ga_K^{o-}[j_o+1,\infty))\leq\de\},\\
 \mathcal{EX}_K^d
 &=&\{E\in\mathcal E_K^d\cup\ga_K^{o-}[1,j_o]:
 \dist(v_E,\ga_K^{d-}[j_d+1,\infty))\leq\de\}.
\eeq
By the definition of $q_{K,\rm out}^o,q_{K,\rm out}^d$ and
\eqref{trans-one-sided-tail}, we have
$\mathcal{EX}_K^o=\mathcal{EX}_K^d=\emptyset$.

We now define the matching indices.  An index $j\geq0$ is called an
$o$-matching index if
\[
 \Gamma_1^l(j)\in\ga_K^{o-}[j_o+1,\infty)
 \quad\hbox{and}\quad
 \omega(E_{\Gamma_1(j)})=0.
\]
An index $j\geq0$ is called a $d$-matching index if
\[
 \Gamma_2^l(j)\in\ga_K^{d-}[j_d+1,\infty)
 \quad\hbox{and}\quad
 \omega(E_{\Gamma_2(j)})=1.
\]
The word matching refers to both the equality of the two medial vertices
and the required state of the bond.  We continue to use ``touches'' and
``comes within distance $\de$'' for geometric relations between paths.
Define
\beq
 t_K^o&=&\max\{j\geq0:
 \Gamma_1^l(j)\in\ga_K^{o-}[j_o+1,\infty),
 \omega(E_{\Gamma_1(j)})=0\},\\
 t_K^d&=&\max\{j\geq0:
 \Gamma_2^l(j)\in\ga_K^{d-}[j_d+1,\infty),
 \omega(E_{\Gamma_2(j)})=1\}.
\eeq
If one of these sets is empty, its maximum is understood as $-1$ and
$\Gamma_i[-1,\infty)$ denotes the whole path $\Gamma_i$.

It remains to prove that the last matching points are inside $W_K$.  Suppose, for
example, that $t_K^o\geq0$ and
$\Gamma_1(t_K^o)\in W_K^{\rm out}$.  Follow the complete curve
$\Gamma_{1c}$ from this point toward $z_{K,\rm in}$.  Before reaching
$W_K^{\rm in}$ it makes another complete inward crossing of $W_K^\circ$.
However, this contradicts \eqref{trans-one-sided-tail}.
 The same argument is applied on
the $d$-side.

We have therefore proved
\[
 t_K^o=-1\quad\hbox{or}\quad\Gamma_1(t_K^o)\in W_K,
\]
and
\[
 t_K^d=-1\quad\hbox{or}\quad\Gamma_2(t_K^d)\in W_K.
\]
We translate the medial portions
\[
 \Gamma_1[t_K^o,\infty),\qquad
 \Gamma_2[t_K^d,\infty)
\]
one lattice spacing to the left. Thus, when a matching index is nonnegative, the bond at the matching point is translated
as well.

We finally explain why the two terminal portions can be translated.  The
maximality of $t_K^o$ and $t_K^d$ excludes any later matching index on the
same side.  The identities
$\mathcal{EX}_K^o=\mathcal{EX}_K^d=\emptyset$, together with the exclusions
of $\mathcal G_{3K}$, $\mathfrak E_K$, and $\mathfrak F_K$, exclude a
later close approach to the opposite unshifted boundary tail outside
$W_K^o$.  Any local excursion produced when a translated terminal portion
is joined to the unchanged configuration starts in $W_K^o$.  On
$\mathcal G_{2K}^c$, it returns to the required region before reaching
distance $r_0R^K/2$ from its starting point.  Otherwise, its two sides
together with the translated and unshifted primal-open and dual-open
portions give four disjoint arms, of types $0,0,1,1$ or $0,1,0,1$, from
scale $2\de$ to scale $r_0R^K/2$.  This is the event $\mathcal G_{2K}$.
The same argument applies on the $d$-side.  The excursion is not required
to stay inside $W_K^o$.

Let $\mathcal P_K$ be the set of pairs $(r_o,r_d)$ for which all the
preceding requirements of this step are satisfied after the corresponding
auxiliary paths and matching indices have been redefined.  Choose a pair
in $\mathcal P_K$ which is maximal for the coordinatewise order.  If
there is more than one maximal pair, use a fixed deterministic
tie-breaking rule.  The final selected crossings are
$C_{K,r_o}^o$ and $C_{K,r_d}^d$.  Thus neither selected crossing can be
replaced by a later crossing in its ordered list, with the other crossing
kept fixed, while all the preceding requirements remain satisfied.

\medskip
\noindent
{\bf Step 5: construction of the translated configuration.}

Now we construct the exploration path $\ga_\de^L$ passing through $v_1$.
Fix a horizontal row.  Let $\mathcal E$ be the set of primal bonds on this
row which are touched or crossed by
\[
 \Gamma_1[t_K^o,\infty)\cup\Gamma_2[t_K^d,\infty)\cup\Loop_0.
\]
So if $t_K^o\geq0$, the bond centred at $\Gamma_1(t_K^o)$ is in
$\mathcal E$.  If $t_K^d\geq0$,  the bond centred at
$\Gamma_2(t_K^d)$ is in $\mathcal E$. Also note that if $t_K^o=-1$, the bond centred at $\Gamma_1(0)$ is in
$\mathcal E$. Similarly for the case where $t_K^d=-1$.

Suppose
\[
 \mathcal E=\{E_{(1)},E_{(2)},\ldots,E_{(n)}\},
\]
where the bonds are ordered from left to right, and put
\[
 \mathcal E^l=\{E_{(1)}^l,E_{(2)}^l,\ldots,E_{(n)}^l\}.
\]
First define
\begin{equation}
 \omega^L(E_{(j)}^l)=\omega(E_{(j)}),
 \qquad j=1,2,\ldots,n.
 \label{trans-row-copy}
\end{equation}
The two sets $\mathcal E\setminus\mathcal E^l$ and
$\mathcal E^l\setminus\mathcal E$ have the same cardinality.  Write them
from left to right as
\[
 \mathcal E\setminus\mathcal E^l=\{G_1,\ldots,G_m\},
 \qquad
 \mathcal E^l\setminus\mathcal E=\{H_1,\ldots,H_m\}.
\]
For the remaining bonds set
\begin{equation}
 \omega^L(G_j)=\omega(H_j),
 \qquad j=1,2,\ldots,m.
 \label{trans-row-close}
\end{equation}
When $m=0$, this second assignment is empty.  Equations
\eqref{trans-row-copy} and \eqref{trans-row-close} assign to the bonds in
$\mathcal E\cup\mathcal E^l$ exactly the old states of these same bonds,
each state being used once.  Thus they give a permutation of the bond
states on this finite set.  We perform this construction on every
horizontal row containing at least one selected bond and leave every other
primal bond unchanged.  Thus $\omega$ and $\omega^L$ have the same Bernoulli product weight.
Together with the one-to-one property proved in Step~6, this is the
probability comparison used in Step~7.

  In particular, if $E\in\mathcal E$, then
\[
 \omega^L(E^l)=\omega(E).
\]
If a bond in $\mathcal E^l$ also belongs to $\mathcal E$, this is because
it satisfies the definition of $\mathcal E$ independently.

The portions $\Gamma_1[t_K^o,\infty)$ and
$\Gamma_2[t_K^d,\infty)$, together with the attached loop, are translated
one lattice spacing to the left.  The portions before $t_K^o,t_K^d$ are
unchanged.
Outside $W_K$, the translated terminal portions do not
come within distance $\de$ of the opposite unshifted boundary tail.
Therefore the primal-open and dual-open sides in $\omega^L$ determine an
exploration path $\ga_\de^L$ from $a_\de^\diamond$ which passes through
$v_1$ and uses the corresponding local passage $A_1,D_1$.  In the case
$s_K<T_{d,K}<T_{o,K}$, the part of this path before the modified region is
unchanged.  In particular, by $\mathcal E_{oK2}^c$ it agrees with the
corresponding part of $\ga_0$ up to the corresponding joining point.

Translation does not change the turns of the moved terminal portions.
Outside
$W_K$ the old and new paths stay between the corresponding primal-open
and dual-open boundary paths.  Hence
\begin{equation}
 W_{\ga_0}(A_0,e_b)=W_{\ga_\de^L}(A_1,e_b).
 \label{trans-winding}
\end{equation}

\medskip
\noindent
{\bf Step 6: the one-to-one property.}

It remains to show that the construction is one-to-one.  Suppose two
configurations $\omega^{(1)}$ and $\omega^{(2)}$ on which the construction
succeeds give the same translated configuration
\[
 (\omega^{(1)})^L=(\omega^{(2)})^L=\eta.
\]
Attach a superscript $(n)$ to all objects constructed from
$\omega^{(n)}$.  The translated attached loop is the same for the two
configurations, and hence $r_0^{(1)}=r_0^{(2)}$.  Let $K_1$ and $K_2$ be
the outermost good indices of $\omega^{(1)}$ and $\omega^{(2)}$.

We first show that $K_1=K_2$.  Suppose that $K_1<K_2$.  The modification
made at scale $K_1$ is contained in $B_{K_1}$ and does not change anything
in the scale-$K_2$ annulus.  Hence $\mathcal B_{K_2}^c$ holds for $\omega^{(1)}$ if and only if it
holds for $\eta$, noting the definition of good annulus.  On the other
hand, the exceptional event on which the scale-$K_2$ translation changes
this condition has already been discarded.  Since $K_2$ is good for
$\omega^{(2)}$, $\mathcal B_{K_2}^c$ holds for $\omega^{(2)}$, hence for
$\eta$, and therefore for $\omega^{(1)}$.  Thus $\mathcal A_{K_2}$ is
good for $\omega^{(1)}$, contradicting the maximality of $K_1$.
Interchanging $\omega^{(1)}$ and $\omega^{(2)}$ rules out $K_2<K_1$.
Therefore
\[
 K_1=K_2=:K.
\]

The complete auxiliary curves $\Gamma_{1c},\Gamma_{2c}$, their selected
common occurrences, and the two ordered crossing lists are fixed before
matching is introduced.  The final selected crossings and terminal
portions are then obtained by the deterministic updating rule of Step~4,
including the matching-point adjustment and the final maximality
convention.
 We now show directly that the two translated terminal
portions are determined uniquely.

We first compare the two proposed $o$-terminal portions.  Write
\[
 P_n=\Gamma_1^{(n)}(t_K^{(n)o},\infty),\qquad n=1,2,
\]
and orient both paths toward $v_0$.

Since $\Gamma_1^{(n)}[t_K^{(n)o},\infty)$ is translated and $\omega^{(1)L}=\omega^{(2)L}$, we have
\[
P_1\subseteq P_2 \qquad \text{or} \qquad P_2\supseteq P_2.
\]
From now on, we suppose $P_1\subseteq P_2$. It remains to exclude strict inclusion.

We first consider $t_K^{(1)o}\leq 0$. The definition of $t_K^o$ shows that
 $t_K^{(1)o}=-1$. In this case, the translated part of $\ga_K^{(1)o-}$ and the fixed part of $\ga_K^{(1)o-}$ are concatenated immediately. If $P_1\not=P_2$, there are two cases. The first case is that $\Gamma_{c}^{(2)}$ crosses $\gamma_K^{(1)o-}$.
Hence in $(\omega^{(2)})^L$, the exploration path must cross the fixed part of $\ga_K^{(1)o-}$. So, $(\omega^{(1)})^L\not=(\omega^{(2)})^L$.
The second case is that $\Gamma_{1c}^{(2)}$ crosses the opposite side of $\gamma_K^{(1)o-}$. This implies that the
primal-open bonds touched on the open side of the translated $\Gamma_1^{(2)}[t_K^{(2)o},\infty)$
together with the corresponding part of $\gamma_K^{(1)o-}$,
contain an
open circuit.  This circuit surrounds an open bond of $\gamma_K^{(1)o-}$.  This is
impossible because $\ga_K^{(1)o-}$ is a portion of the outermost primal-open boundary. Therefore, we only have $P_1=P_2$.

Next, we suppose $t_K^{(1)o}>0$. Let
\[
 x=\Gamma_1^{(1)}(t_K^{(1)o}).
\]
Since $P_1\subseteq P_2$ and $\omega^{(1)L}=\omega^{(2)L}$, we have $\omega^{(1)}(E_x)=\omega^{(2)}(E_x)=0$. Now we change the state of $E_x $ in both $\omega^{(1)}$ and $\omega^{(2)}$ to open. The consequence is that the state of $E_x^l$ in $\eta$ is open now. So we are back to the case where $t_K^{(1)o}\leq 0$. Hence, we have $P_1=P_2$.

The argument for the $d$-terminal portions is the same with primal-open and
dual-open interchanged.  Hence
\[
 \Gamma_2^{(1)}(t_K^{(1)d},\infty)
 =\Gamma_2^{(2)}(t_K^{(2)d},\infty).
\]
On every horizontal row, the rule defining $\omega^L$ is a permutation
of the states of the bonds in $\mathcal E\cup\mathcal E^l$.  The set
$\mathcal E$ is determined by the two terminal portions and by $\Loop_0$.
This permutation
has a unique inverse.  All other bonds are unchanged.  Hence
\[
 \omega^{(1)}=\omega^{(2)}.
\]
Thus the map $\omega\mapsto\omega^L$ is one-to-one.

The same construction, with left and right interchanged, maps the good part
of $\mathcal C_1^D$ into $\mathcal C_0^D$.  It is again one-to-one and
preserves the winding.

\medskip
\noindent
{\bf Step 7: summing the configurations.}

Let $G_0^D$ be the part of $\mathcal C_0^D$ on which the construction
succeeds.  The preceding steps give an injection
\[
 G_0^D\longrightarrow\mathcal C_1^D
\]
which preserves probability and the winding phase.  Interchanging left and
right gives an injection from the corresponding good part of
$\mathcal C_1^D$ into $\mathcal C_0^D$.  Therefore the difference between
the contributions of $\mathcal C_0^D$ and $\mathcal C_1^D$ is bounded by
\[
 O\big((\de/d_0)^{1+\al-\ep}\big).
\]
The reflected construction gives the same estimate for
$\mathcal C_0^B$ and $\mathcal C_1^B$.

It remains to consider $\mathcal C_0^4$ and $\mathcal C_1^4$.  Reverse the
state of $E_0$ before applying the construction, and reverse the state of
$E_1$ afterwards.  The construction moves the attached loop together with the corresponding
terminal portions of the exploration path.  Thus reversing $E_1$ reconnects
the translated loop to the translated terminal portions and restores the two
passages through $v_1$.  The two
local bond changes are involutions.  They preserve probability and do not
change the winding of $A_0$ or $A_1$.  Therefore the same bound holds for
the all-four-edge contributions.

The discarded configurations are those in
\eqref{trans-macro-cont}, \eqref{trans-large-r0} and
\eqref{trans-no-good-sum}, together with the exponentially small failures in
\eqref{trans-no-left-step}.  Their total probability is
$O((\de/d_0)^{1+\al-\ep})$. Rotating the construction, and exchanging primal-open and
dual-open bonds when needed, provides the translation map on $\{A_j\in \ga_\de, \ \text{which visits} \ C_j \ \text{first}\}$.
This proves \eqref{TrA}, and \eqref{TrB}--\eqref{TrD}.
\end{proof}

\subsection{The estimate near a straight boundary segment}

For a special Dobrushin domain, the same translation map gives an estimate in terms of the distance to the remaining part of the boundary even when the medial vertices are close to a straight boundary segment.

\begin{proposition} \label{Tranimp}
Suppose $\mathrm{L}$ is a horizontal line segment contained in $\p\Omega_{ba,\de}$ or $\p \Omega^*_{ab,\de}$.  Denote the Euclidean distances from $v_0$ to $\mathrm{L}$ and to
$(\p\Omega_{ba,\de}\cup \p \Omega^*_{ab,\de})\setminus\mathrm{L}$ by $d_V$ and $d_0$, respectively, and assume that $d_0\geq d_V$.  Let $v_1$ be on the left-hand side of $v_0$.  Given any $\ep>0$, if $\de$ is small enough and
\[
 \min\big(\dist(v_0,a_\de^\diamond),
           \dist(v_0,b_\de^\diamond)\big)\geq\de^{1-\ep},
\]
then
\beqn
|F(A_0)-F(A_1)| &\preceq& (\de/d_0)^{1+\al-\ep}, \label{TrAs}\\
|F(B_0)-F(B_1)| &\preceq& (\de/d_0)^{1+\al-\ep}, \label{TrBs}\\
|F(C_0)-F(C_1)| &\preceq& (\de/d_0)^{1+\al-\ep}, \label{TrCs}\\
|F(D_0)-F(D_1)| &\preceq& (\de/d_0)^{1+\al-\ep}. \label{TrDs}
\eeqn
\end{proposition}

\begin{proof}
We first assume that $\mathrm L\subset\p\Omega_{ba,\de}$.
After a translation and, if needed, a reflection, we may suppose that
$\mathrm L$ is contained in the real axis, that the domain lies above it,
and that $v_1=v_0-\de\be_1$.  We use the configuration map in the proof
of Proposition~\ref{Tran}.  We only have to reconsider the arm estimates
for annuli which meet $\mathrm L$ and to check that the rowwise state
permutation is still admissible.

We use the notation of that proof.  We may increase the fixed lower index
$k_0$, which removes only a fixed number of inner annuli.  Since
\[
 r_0R^{k+1}\leq d_0/16,
 \qquad k_0\leq k\leq m_0,
\]
every annulus stays away from the boundary outside the relative
interior of $\mathrm L$.  Conditional on $\mathcal F_k$, its unrevealed
interior bonds are independent critical bonds.  All the estimates below
are uniform over the boundary connections produced by the revealed
configuration.

We first provide the estimate for the attached loop.  Its two sides and
the two sides of the exploration path give four arms up to distance
$r_0$.  Beyond that distance one open arm may use the wired boundary,
while the other three arms remain disjoint in the half-plane.  Hence, by
arm separation and quasi-multiplicativity,
\[
 \Pro(\mathcal C_0^D,r_0\geq r)
 \preceq
 \begin{cases}
  (\de/r)^{1+\al},&r\leq d_V,\\[1mm]
  (\de/d_V)^{1+\al}(d_V/r)^2,&d_V<r.
 \end{cases}
 \preceq (\de/r)^{1+\al}.
\]
When $d_V$ is comparable with $\de$, the second line follows directly
from \eqref{halfthree}.  Thus the large-loop estimates
\eqref{trans-macro-cont} and \eqref{trans-large-r0} remain valid with
$d_0$ as the outer scale.

Put $r_k=r_0R^k$.  The event $\mathcal E_k$ gives four alternating arms
across the central subannulus.  If these arms reach $\mathrm L$, three
half-plane arms continue to the outer radius.  The bulk four-arm estimate
up to the boundary scale and \eqref{halfthree} beyond it give
\[
 \Pro(\mathcal E_k\mid\mathcal F_k)\preceq R^{-1-\al}.
\]
The same conclusion follows directly from \eqref{halfthree} when the
inner circle already meets $\mathrm L$.

The events $\mathcal G_{1k}$ and $\mathcal G_{2k}$ require a separate
argument because their four arms may have type $0011$.  For
$\mathcal G_{1k}$, split according as the distance of the vertex from
$\mathrm L$ is at most or greater than $r_kR^{-\al}$.  In the first
case, take the separating exploration path closest to $\mathrm L$ and
shift the primal bonds touched by it by half a lattice spacing.  This
gives the half-plane two-arm event in \eqref{halftwo}, with probability
at most $CR^{-1-\al}$.  In the second case, the bulk four-arm estimate
up to the distance from the vertex to $\mathrm L$, together with the same
count of possible roots as in Proposition~\ref{Tran}, gives the same
bound.  For $\mathcal G_{2k}$ use the identical split with threshold
$r_kR^{-1-\al}$.  In the part nearer $\mathrm L$, \eqref{halftwo}
gives
\[
 \frac{r_kR^{-1-\al}}{r_k}=R^{-1-\al}.
\]
In the other part, the bulk four-arm estimate and the covering of
$W_k^+$ used in Proposition~\ref{Tran} again give $CR^{-1-\al}$.

This is also the treatment of the possible $1001$ orientation.  In that
orientation the two boundary-most arms are primal-open and may join the
wired boundary.  We do not count them as a third half-plane arm.  We use
the two vertex-disjoint dual-open arms and the shifted separating path,
which give the half-plane two-arm event above.

The same distinction applies to every four-arm factor in
$\mathfrak E_k$.  Alternating factors are treated as for
$\mathcal E_k$.  A factor of type $0011$ is treated by the same
thin-strip split and the half-plane two-arm estimate.  The choices of the
first return, the first close approach, and the stopping radii are
unchanged.  Therefore the sums in Step~2 of Proposition~\ref{Tran} give
\[
 \Pro(\mathfrak E_k\mid\mathcal F_k)
 \preceq C(\log R)^C R^{-1-\al}.
\]

For $\mathcal G_{3k}$, fix a dyadic value of $d_z$ as in that step.  If
the six-arm part reaches $\mathrm L$, the continuation contains the
disjoint half-plane events $A_{01}^+$ and $A_{010}^+$.  If only the
outer four-arm part reaches $\mathrm L$, use the preceding alternating
or $1001$ argument.  The van den Berg-Kestern-Reimer inequality,
\eqref{halftwo}--\eqref{halfthree}, and the same strip count give the
estimate obtained in Proposition~\ref{Tran}.  The same
$A_{01}^+$ and $ A_{010}^+$ estimate applies to $\mathfrak F_k$ after
the first nearby pair has been chosen.  We have therefore proved, at
every bulk or boundary scale,
\begin{equation}
 \Pro(\mathcal B_k\mid\mathcal F_k)
 \preceq C(\log R)^C R^{-1-\al}.
 \label{trans-boundary-bad-scale-short}
\end{equation}

The westward-step argument is local.  Every bond tested there is an
unrevealed interior bond on the domain side of the exploration path.
Thus \eqref{trans-no-left-step} is unchanged.  Combining it with
\eqref{trans-boundary-bad-scale-short}, revealing the annuli from the
outside inwards, and using the loop estimate above gives
\[
 \Pro\big(\mathcal C_0^D,\ \hbox{there is no good annulus}\big)
 \preceq C_\ep(\de/d_0)^{1+\al-\ep}.
\]
The selected-scale stability estimate is proved from the same first
four-arm or six-arm certificate.  By the definition of good annulus, the
selected index $K$ is still good after translation for the same family
of annuli centred at $v_0$.  Every larger annulus is outside the modified
region, so $K$ remains the outermost good index in this family.

We can now apply Steps~4--7 of Proposition~\ref{Tran}.  The translation
is parallel to $\mathrm L$.   The proof of the image
path, the winding identity, and the one-to-one property are unchanged.

The discarded configurations have total probability
$O((\de/d_0)^{1+\al-\ep})$.  The three local passage cases and the
all-four-edge case are now handled as in Step~7 of
Proposition~\ref{Tran}.  This proves \eqref{TrAs}.  The remaining three
estimates follow exactly as at the end of that proof.  When
$\mathrm L\subset\p\Omega^*_{ab,\de}$, the same proof applies after
exchanging primal-open and dual-open bonds.
\end{proof}

\section{Rotational invariance} \label{rot}
  Subsection \ref{partI} changes the local visit at one medial vertex by reflecting selected primal bonds and interchanging primal-open and dual-open states.  This leads to the representation of the modified parafermionic observable used in Section \ref{modi}. Subsection \ref{partII} moves a local visit across one medial edge by a half-turn of selected bond states.  Subsection \ref{partIII} combines these maps and removes the conditioning on the radius of the attached loop in Subsection \ref{partII}.  In each case we specify the selected primal-open and dual-open paths, the multi-arm events on which the bond change is discarded, and the comparison proving that two original configurations cannot give the same modified configuration. The whole section provides the basis that the two representations of the same modified edge observable around adjacent medial vertices are essentially the same.

\subsection{Rotation around medial vertex} \label{partI}

We first explain the purpose of this subsection.  Let $v$ be a medial vertex and suppose that the exploration path $\ga_\de$ visits exactly two of the four medial edges incident to $v$.  There are four possible local visits.  We will divide the configurations producing each visit into two classes and construct a measure-preserving bijection from the second class of one visit to the first class of the next visit.  The bijection changes the percolation configuration only in a region traced by the exploration path and by the loop attached to the bond through $v$.  The common medial edge of the two local visits keeps the same winding.  The final statement is Proposition \ref{rot-inv}.

The key operation is a reflection--duality switch.  We reflect a finite set of primal bonds across the line through the opposite medial edges $A$ and $C$ and interchange primal-open and dual-open states.  The selected primal-open and dual-open path segments then force the reflected exploration-path segment.  The stopping bonds are chosen so that no unused part of the exploration path enters between these two paths and so that the one-to-one property can be proved by comparing two original configurations with the same modified configuration.

The proof has four steps.  First we define the reflection--duality switch and the stopping points of the exploration path and the attached loop.  Second we construct the switch when the loop stays inside the ball of radius $d$.  Third we truncate the same construction at radius $d$.  The reflection maps in Lemma~\ref{special1} send each truncation failure injectively to the intersection of two standard half-plane two-arm events supported on disjoint sets of random bonds.  Applying the half-plane two-arm estimate to the resulting half-plane events gives probability $O((\de/d)^2)$.  Finally we rotate the construction to the other three local visits and join the four maps.

Let
\[
d= d_v:=\dist\big(v,\p_{ba,\de}\cup\p_{ab,\de}^*\big).
\]
The four medial edges incident to $v$ are denoted by $A,B,C,D$ in \emph{clockwise} order, with $A$ and $C$ oriented towards $v$ and $B$ and $D$ oriented away from $v$.  We use the following notation for the four two-edge visits:
\beq
\mathcal C_1^v&=&\{\omega:A,D\in\ga_\de,\ B,C\notin\ga_\de\},\qquad
\mathcal C_2^v=\{\omega:A,B\in\ga_\de,\ C,D\notin\ga_\de\},\\
\mathcal C_3^v&=&\{\omega:B,C\in\ga_\de,\ A,D\notin\ga_\de\},\qquad
\mathcal C_4^v=\{\omega:C,D\in\ga_\de,\ A,B\notin\ga_\de\}.
\eeq
They are displayed in Figure \ref{config}.  This convention will be used throughout the subsection.

\begin{figure}[!htbp]
\begin{center}
\resizebox{0.98\textwidth}{!}{
\begin{tikzpicture}[x=1cm,y=1cm,>=Stealth,line cap=round]
  \foreach \x/\p/\solidA/\solidB/\solidC/\solidD in {
    0/1/1/0/0/1,
    4.9/2/1/1/0/0,
    9.8/3/0/1/1/0,
    14.7/4/0/0/1/1}
  {
    \coordinate (O) at (\x,0);
    \fill (O) circle (1.6pt);

    \ifnum\solidA=1
      \draw[very thick,-{Stealth[length=2mm]}] ($(O)+(0,1.38)$)--(O);
    \else
      \draw[gray,thick,dashed,-{Stealth[length=2mm]}] ($(O)+(0,1.38)$)--(O);
    \fi
    \ifnum\solidB=1
      \draw[very thick,-{Stealth[length=2mm]}] (O)--($(O)+(1.38,0)$);
    \else
      \draw[gray,thick,dashed,-{Stealth[length=2mm]}] (O)--($(O)+(1.38,0)$);
    \fi
    \ifnum\solidC=1
      \draw[very thick,-{Stealth[length=2mm]}] ($(O)+(0,-1.38)$)--(O);
    \else
      \draw[gray,thick,dashed,-{Stealth[length=2mm]}] ($(O)+(0,-1.38)$)--(O);
    \fi
    \ifnum\solidD=1
      \draw[very thick,-{Stealth[length=2mm]}] (O)--($(O)+(-1.38,0)$);
    \else
      \draw[gray,thick,dashed,-{Stealth[length=2mm]}] (O)--($(O)+(-1.38,0)$);
    \fi

    \node[above] at ($(O)+(0,1.50)$) {$A$};
    \node[right] at ($(O)+(1.48,0)$) {$B$};
    \node[below] at ($(O)+(0,-1.50)$) {$C$};
    \node[left] at ($(O)+(-1.48,0)$) {$D$};
    \node[below] at ($(O)+(0,-1.90)$) {Possibility \p};
  }
\end{tikzpicture}}
\end{center}
\caption{The four possible two-edge visits of $\ga_\de$ at $v$.  The edges $A,B,C,D$ are ordered clockwise.  Solid medial edges belong to the exploration path and dashed medial edges do not.  The arrows give the fixed orientations.}\label{config}
\end{figure}

For a primal bond $E$, we use the convention
\[
\omega(E)=1 \quad\Longleftrightarrow\quad E\hbox{ is primal-open},
\]
and hence the dual bond $E^*$ is dual-open exactly when $\omega(E)=0$.  Let $E_v$ be the primal bond whose midpoint is $v$.  Let $L_v^{AC}$ be the straight line through $A$ and $C$, and let $\rho_{AC}$ denote reflection in this line.  The reflected segment $\rho_{AC}(E)$ is a dual bond.  We denote by $E^{AC}$ the primal bond crossed by this reflected dual bond.  Consequently,
\[
\omega(E^{AC})=0
\]
means that the reflected dual bond is dual-open.  If $E$ is crossed by $L_v^{AC}$, then $E^{AC}=E$. In each construction below, the modified percolation configuration
obtained from $\omega$ is denoted by $\omega'$.  The exploration path
from $a_\de^\diamond$ to $b_\de^\diamond$ in the configuration
$\omega'$ is denoted by $\ga_\de'$.  On a pair $E,E^{AC}$, the reflection--duality switch is
\beq
\omega'(E)=1-\omega(E^{AC}),\qquad
\omega'(E^{AC})=1-\omega(E).
\eeq
Thus a primal-open bond is reflected into a dual-open bond and conversely.  Since $p=1/2$, every bijection of percolation configurations obtained by such switches preserves probability.

Let $\Loop_v$ be the medial loop attached to $E_v$, and put
\[
r=\max\{|z-v|:z\in\Loop_v\}.
\]
We parameterize $\Loop_v$ counterclockwise, starting at $v$, and write
\[
\Loop_v(0)=v,\qquad
n_\Loop=\inf\{j>0:\Loop_v(j)=v\}.
\]
The same loop parameterized in the reverse direction is denoted by $\Loop_v^{-1}$.  Its reflections in $L_v^{AC}$ are denoted by $\Loop_v^{AC}$ and $\Loop_v^{-AC}$, respectively.  Let
\[
n_v=\min\{j:\ga_\de(j)=v\}.
\]
We refer to $\ga_\de[0,n_v]$ as the past of the exploration path at $v$ and to $\ga_\de[n_v,\infty)$ as its future.

We shall say that two medial paths share an edge if they traverse the same unoriented medial edge.  For compatibility with the notation used later in the paper, the expression $\eta_1[j,j+1]\subseteq_c\eta_2$ means only that the single unoriented edge $\eta_1[j,j+1]$ is also an edge of $\eta_2$. It is not a symmetric notation for geometric crossing.  Define
\beq
t^v&=&\inf\{j>0:\Loop_v^{AC}[j,j+1]
       \hbox{ is an edge of }\ga_\de[0,n_v]\},\\
t_v&=&\inf\{k>0:\ga_\de(k)=\Loop_v^{AC}(t^v)\}.
\eeq
If the first set is empty, we put $t^v=n_\Loop$ and $t_v=n_v$.  This convention makes $t^v$ and $t_v$ well defined in every configuration.

We also use the set of primal bonds whose midpoint is visited both by the past of the exploration path and by the loop before its return to $v$:
\beq
\mathfrak E^v
 =\big\{E_{\ga_\de(j)}:
       0\leq j<n_v,\ \ga_\de(j)\in\Loop_v[0,n_\Loop)\big\}.
\eeq
Because $j<n_v$, the central bond $E_v$ does not belong to $\mathfrak E^v$.

We now carry out the construction for the local visit $(A,D)$, that is,
\[
A,D\in\ga_\de,
\qquad B,C\notin\ga_\de.
\]
The other three visits are obtained from this construction by rotation and by exchanging the roles of primal-open and dual-open bonds.  The new visit is either $(A,B)$ or $(C,D)$.  In the first case,
\beq
W_{\ga_\de'}(B,e_b)&=&W_{\ga_\de}(A,e_b)-\pi/2,\qquad
W_{\ga_\de'}(A,e_b)=W_{\ga_\de}(D,e_b)-\pi/2,
\eeq
and in the second case,
\beq
W_{\ga_\de'}(D,e_b)&=&W_{\ga_\de}(A,e_b)+\pi/2,\qquad
W_{\ga_\de'}(C,e_b)=W_{\ga_\de}(D,e_b)+\pi/2.
\eeq
The signs follow from the clockwise ordering in Figure \ref{config}: the first replacement makes a clockwise quarter-turn at $v$, while the second makes a counterclockwise quarter-turn.

\medskip
\noindent
{\bf Case I: the loop stays inside $B(v;d)$, that is, $r<d$.}

In this case the entire loop and its reflection are contained in the domain away from the wired and dual-wired boundary arcs.  The construction has three subcases.  The first two occur when every bond in $\mathfrak E^v$ has the complementary state of its reflected partner.  They are distinguished by the state of the bond at the first shared edge of $\Loop_v^{AC}$ and the past of $\ga_\de$.  The third subcase starts at the last place where this complementary relation fails.  Figure \ref{rotation-surgery} shows the three cut points schematically.

\begin{figure}[!htbp]
\begin{center}
\resizebox{0.5\textwidth}{!}{
\begin{tikzpicture}[
  x=0.43cm,
  y=0.43cm,
  line cap=round,
  line join=round,
  exploration path/.style={black,line width=1.5pt},
  attached loop/.style={black,line width=1.05pt,dash pattern=on 3.5pt off 2.3pt},
  reflected loop/.style={blue!70!black,line width=1.3pt},
  local edge/.style={gray!70,line width=0.85pt},
  local arrow/.style={gray!70,line width=0.85pt,-{Stealth[length=1.8mm,width=1.25mm]}},
  panel title/.style={font=\normalsize},
  small label/.style={font=\scriptsize,inner sep=.8pt},
  path label/.style={font=\scriptsize,inner sep=1.1pt,fill=white,fill opacity=.86,text opacity=1}
]
  \begin{scope}
    \draw[local edge] (4.35,20.4) -- (4.35,18.4);
    \draw[attached loop] (7.35,19.4) circle[radius=3];
    \draw[reflected loop] (4.35,16.4) circle[radius=3];
    \draw[local arrow] (3.35,20.4) -- (4.35,19.4);
    \draw[local arrow] (4.35,19.4) -- (3.35,18.4);
    \draw[exploration path] plot coordinates {(3.34,18.34) (3.25,18.25) (3.18,18.18) (3.14,18.1) (3.14,18.02) (3.12,17.92) (3.12,17.84) (3.12,17.76) (3.16,17.62) (3.22,17.51) (3.34,17.42) (3.46,17.36) (3.57,17.33) (3.74,17.26) (3.86,17.2) (3.97,17.17) (4.08,17.16) (4.16,17.12) (4.26,17.06) (4.41,17.03) (4.54,16.92) (4.67,16.87) (4.79,16.8) (4.88,16.7) (4.96,16.58) (5,16.5) (5.04,16.37) (5.1,16.22) (5.14,16.04) (5.18,15.9) (5.24,15.8) (5.32,15.7) (5.43,15.57) (5.61,15.4) (5.98,15.2) (6.2,15.08) (6.32,15.06) (6.84,14.88) (7.1,14.8) (7.2,14.8) (7.34,14.74) (8.2,14.74) (8.78,14.74) (9.02,14.74) (9.16,14.72) (9.36,14.72) (9.42,14.7)};
    \draw[exploration path] plot coordinates {(8.04,13.28) (7.95,13.33) (7.84,13.42) (7.73,13.47) (7.61,13.53) (7.5,13.54) (7.42,13.56) (7.34,13.58) (7.26,13.58) (7.15,13.63) (7.06,13.68) (6.92,13.72) (6.8,13.74) (6.69,13.76) (6.55,13.81) (6.41,13.83) (6.27,13.85) (6.16,13.86) (6.04,13.85) (5.91,13.77) (5.84,13.66) (5.82,13.54) (5.79,13.42) (5.67,13.27) (5.56,13.26) (5.46,13.24) (5.33,13.25) (5.21,13.31) (5.09,13.36) (4.97,13.43) (4.83,13.5) (4.69,13.55) (4.58,13.6) (4.48,13.66) (4.37,13.73) (4.22,13.8) (4.08,13.9) (3.94,13.98) (3.84,14.02) (3.71,14.07) (3.56,14.14) (3.48,14.18) (3.38,14.2) (3.21,14.26) (3.08,14.32) (3,14.34) (2.92,14.38) (2.82,14.42) (2.74,14.44) (2.62,14.5) (2.3,14.62) (2.12,14.68) (2.01,14.74) (1.56,15.06) (1.24,15.29) (1.1,15.5) (1,15.66) (0.94,15.82) (0.88,15.98) (0.84,16.06) (0.82,16.2) (0.82,16.3) (0.82,16.4) (0.8,16.56) (0.78,17.3) (0.78,17.56) (0.8,17.76) (0.82,17.96) (0.86,18.2) (0.88,18.38) (0.92,18.52) (0.94,18.62) (0.94,18.7) (1.04,19.06) (1.1,19.3) (1.2,19.72) (1.28,19.98) (1.42,20.34) (1.54,20.58) (1.63,20.75) (1.71,20.86) (1.84,20.92) (1.94,20.92) (2.22,20.94) (2.48,20.94) (2.6,20.92) (2.7,20.88) (2.81,20.79) (2.92,20.74) (3.05,20.65) (3.16,20.55) (3.24,20.48) (3.3,20.44)};
    \node[small label] at (3.92,20.06) {$A$};
    \node[small label] at (3.96,18.52) {$D$};
    \node[path label,anchor=west] at (5.88,13.72) {$\ga_\de(s_v)$};
    \node[path label,anchor=east] at (4.22,13.42) {$\ga_\de(t_v)$};
    \node[small label] at (0.3,18.66) {$+$};
    \node[small label] at (1.3,18.74) {$-$};
    \node[small label] at (10.54,20.02) {$+$};
    \node[small label] at (9.72,20.06) {$-$};
    \node[small label] at (4.54,19.04) {$v$};
    \node[panel title] at (3.4711,23.8871) {\textbf{Subcase (1)}};
  \end{scope}

  \begin{scope}
    \draw[local edge] (17.9,20.7) -- (17.9,18.7);
    \draw[attached loop] (20.9,19.7) circle[radius=3];
    \draw[reflected loop] (17.9,16.7) circle[radius=3];
    \draw[local arrow] (16.9,20.7) -- (17.9,19.7);
    \draw[local arrow] (17.9,19.7) -- (16.9,18.7);
    \draw[exploration path] plot coordinates {(16.9044,18.692) (16.846,18.653) (16.7293,18.5947) (16.6028,18.5071) (16.4958,18.4098) (16.4569,18.3223) (16.379,18.2152) (16.3207,18.1082) (16.2039,18.0109) (16.1261,17.9137) (16.0677,17.8066) (16.0288,17.6996) (15.9704,17.6023) (15.8926,17.5051) (15.8537,17.4272) (15.8148,17.3494) (15.7759,17.2326) (15.7175,17.1548) (15.5813,16.9992) (15.484,16.8824) (15.4256,16.7851) (15.27,16.6295) (15.0949,16.4544) (14.9878,16.4057) (14.803,16.3376) (14.4722,16.2987) (13.9858,16.2598) (13.5772,16.2209) (12.8378,16.2209) (12.4487,16.2209)};
    \draw[exploration path] plot coordinates {(16.8849,20.7739) (16.8947,20.8809) (16.9433,21.0268) (16.9628,21.1825) (17.0211,21.3382) (17.099,21.4938) (17.1963,21.6884) (17.2935,21.8246) (17.3714,21.9413) (17.5076,22.2332) (17.6438,22.4278) (17.7313,22.6029) (17.8773,22.7586) (18.0427,22.8656) (18.1302,22.8558) (18.2081,22.8948) (18.3929,22.9531) (18.4999,23.0115) (18.5777,23.0115) (18.8112,23.0699) (19.0058,23.1088) (19.2198,23.1088) (19.356,23.1477) (19.4339,23.1282) (19.6673,23.1477) (19.8036,23.1672) (19.9787,23.2061) (20.3094,23.2255) (20.6597,23.245) (20.8932,23.2839) (21.2823,23.3228) (21.8077,23.3228) (22.4887,23.3228) (23.0335,23.3228) (23.5199,23.2645) (24.0647,23.2061) (24.7652,23.1088) (25.096,23.0504) (25.1738,22.992) (25.2516,22.9337) (25.3586,22.8169) (25.4268,22.6807) (25.6019,22.5056) (25.6992,22.3791) (25.8937,21.9803) (26.0299,21.63) (26.0688,21.3965) (26.1272,21.0463) (26.205,20.5599) (26.244,20.2291) (26.2829,19.8594) (26.2829,18.9644) (26.2829,18.6141) (26.2829,18.1861) (26.1856,17.7969) (26.0688,17.4272) (25.991,17.1743) (25.9132,16.9992) (25.8354,16.863) (25.7381,16.7268) (25.5824,16.5127) (25.4657,16.4349) (25.3489,16.3182) (24.9987,16.143) (24.7847,16.0263) (24.4928,15.9485) (24.2204,15.8901) (24.1231,15.8901) (23.5394,15.8512) (22.7805,15.8123) (21.9633,15.8123) (21.6909,15.8123) (21.4769,15.8123) (21.185,15.7928) (20.9515,15.7539) (20.8543,15.7539) (20.7764,15.7344) (20.6986,15.7344) (20.6208,15.6761) (20.4651,15.6761) (20.3192,15.6274) (20.0954,15.5982) (19.8814,15.5593) (19.609,15.5204) (19.2782,15.462) (19.1615,15.462) (19.0058,15.4231) (18.7334,15.3453) (18.4804,15.3064) (18.3053,15.248) (18.0524,15.1896) (17.8578,15.1118) (17.7022,15.0534) (17.6049,14.995) (17.5076,14.9367) (17.4103,14.8978) (17.2352,14.8005) (17.0309,14.6935) (16.846,14.5864) (16.7098,14.5086) (16.5736,14.3919) (16.4374,14.2751) (16.3012,14.1389) (16.1261,14.0222) (16.0094,13.9249) (15.9218,13.8373) (15.7759,13.7498) (15.698,13.6914) (15.6202,13.6136) (15.5229,13.5747) (15.4159,13.4676) (15.2992,13.3995) (15.1824,13.3509) (15.1338,13.2828)};
    \node[small label] at (17.47,20.36) {$A$};
    \node[small label] at (17.51,18.82) {$D$};
    \node[small label] at (25.8358,20.0302) {$+$};
    \node[small label] at (13.5269,16.4911) {$-$};
    \node[small label] at (13.4414,15.9052) {$+$};
    \node[small label] at (24.09,20.32) {$+$};
    \node[small label] at (23.27,20.36) {$-$};
    \node[small label] at (18.09,19.34) {$v$};
    \node[small label] at (26.6331,20.1123) {$-$};
    \node[path label,anchor=west] at (20.6986,15.5788) {$\ga_\de(t_v)$};
    \node[path label,anchor=east] at (15.8926,13.9443) {$\ga_\de(\mathrm t_v)$};
    \node[panel title] at (17.4492,23.8871) {\textbf{Subcase (2)}};
  \end{scope}

  \begin{scope}
    \draw[local edge] (11.45,6.95) -- (11.45,4.95);
    \draw[attached loop] (14.45,5.95) circle[radius=3];
    \draw[reflected loop] (11.45,2.95) circle[radius=3];
    \draw[local arrow] (10.45,6.95) -- (11.45,5.95);
    \draw[local arrow] (11.45,5.95) -- (10.45,4.95);
    \draw[exploration path] plot coordinates {(10.5268,4.9698) (10.4444,4.8874) (10.362,4.7226) (10.3149,4.546) (10.3031,4.4165) (10.2913,3.9339) (10.2913,3.6043) (10.3384,3.463) (10.3855,3.3217) (10.5032,3.0628) (10.5268,2.9215) (10.621,2.7096) (10.7387,2.4742) (10.8093,2.3565) (11.0683,2.0033) (11.3037,1.7443) (11.3979,1.6972) (11.4921,1.6501) (11.6804,1.5089) (11.8452,1.3441) (11.9394,1.2735) (12.1984,1.2028) (12.4573,1.1322) (12.6457,0.9674) (12.8105,0.8497) (13.0459,0.7084) (13.1872,0.6378) (13.4344,0.6025) (13.6816,0.5436) (13.9406,0.473) (14.0348,0.4259) (14.1289,0.3788) (14.2937,0.3553) (14.435,0.3317) (15.1648,0.3317) (15.7534,0.3082) (16.2243,0.3082) (16.9542,0.3082) (17.5898,0.3082) (17.8253,0.3082)};
    \draw[exploration path] plot coordinates {(5.6768,2.5683) (5.8299,2.5801) (6.4538,2.5919) (6.7363,2.5919) (6.9011,2.5683) (7.0894,2.5213) (8.196,2.5213) (8.455,2.5213) (8.5962,2.4977) (8.7963,2.5095) (8.9611,2.5801) (9.0789,2.7096) (9.1848,2.7567) (9.3025,2.8509) (9.4202,2.9215) (9.5144,3.0863) (9.6086,3.204) (9.6086,3.3217) (9.6321,3.4159) (9.6321,3.5336) (9.6557,3.6513) (9.6557,3.9574) (9.6557,4.1222) (9.6321,4.2399) (9.5968,4.393) (9.538,4.4989) (9.4909,4.6166) (9.3143,4.7814) (9.1848,4.8874) (9.0435,5.0404) (8.9494,5.1346) (8.8552,5.2523) (8.761,5.2994) (8.6551,5.4289) (8.4785,5.5937) (8.2902,5.7467) (8.1253,5.888) (8.0076,6.0057) (7.8428,6.194) (7.7487,6.2647) (7.7016,6.3824) (7.6074,6.5236) (7.5368,6.6178) (7.4426,6.8297) (7.419,6.9239) (7.372,7.0416) (7.372,7.3477) (7.372,7.5596) (7.3955,7.6537) (7.4073,7.7832) (7.4426,7.8892) (7.5485,8.0657) (7.6309,8.207) (7.7722,8.36) (7.9841,8.5602) (8.1018,8.7132) (8.2313,8.8191) (8.3255,8.9604) (8.4079,9.1134) (8.5374,9.2429) (8.5844,9.396) (8.6904,9.4666) (8.7963,9.5725) (8.9376,9.6432) (9.0435,9.6785) (9.1377,9.7256) (9.2319,9.7962) (9.3496,9.7962) (9.4438,9.8433) (9.538,9.8904) (9.6557,9.9139) (9.8322,9.961) (9.9853,10.0316) (10.1265,10.0787) (10.2207,10.1493) (10.3855,10.1964) (10.4797,10.22) (10.668,10.3142) (10.8799,10.3612) (10.9741,10.4083) (11.1625,10.4554) (11.3273,10.479) (11.445,10.4554) (11.5391,10.4554) (11.6333,10.479) (11.751,10.5025) (12.01,10.526) (12.7634,10.526) (13.5874,10.526) (13.8464,10.526) (14.8588,10.5496) (15.2355,10.5496) (16.2008,10.5731) (17.166,10.5731) (17.5192,10.5731) (17.6369,10.5496) (17.8723,10.5496) (18.1078,10.526) (18.2726,10.479) (18.4727,10.4672) (18.6375,10.4319) (18.9083,10.3377) (19.026,10.2906) (19.179,10.2082) (19.3674,10.1258) (19.6146,10.0316) (19.7323,9.9845) (19.8971,9.8904) (19.9913,9.8433) (20.1325,9.7727) (20.2502,9.6785) (20.4857,9.5372) (20.6505,9.4195) (20.7682,9.29) (20.9212,9.0546) (21.0743,8.6896) (21.1684,8.5013) (21.2037,8.3718) (21.2862,8.2188) (21.3332,8.1246) (21.3803,7.9716) (21.4039,7.8421) (21.4274,7.7479) (21.451,7.5831) (21.451,7.4418) (21.451,7.3477) (21.4039,7.1358) (21.3803,6.8768) (21.3332,6.6885) (21.2155,6.3824) (21.1449,6.147) (21.0507,5.9586) (20.9801,5.8174) (20.8624,5.629) (20.7917,5.5113) (20.6976,5.3936) (20.6152,5.2405) (20.4857,5.1581) (20.3444,5.0875) (20.1325,4.9462) (20.0148,4.8991) (19.8735,4.8285) (19.7087,4.7814) (19.6146,4.7814) (19.5204,4.7579) (19.3791,4.7579) (19.2614,4.7343) (18.9318,4.7343) (18.7317,4.7226) (18.5551,4.7108) (18.4138,4.6873) (18.3197,4.7108) (18.1902,4.7814) (18.0842,4.8285) (17.9547,4.8756) (17.8135,4.958) (17.6722,5.0522) (17.531,5.1346) (17.4486,5.2288) (17.4486,5.3936) (17.4721,5.5701) (17.5663,5.629) (17.6958,5.7232) (17.837,5.7938) (18.0018,5.8762) (18.2373,6.0057) (18.4021,6.1116) (18.5316,6.2529) (18.6257,6.3353) (18.6964,6.4295) (18.7435,6.5236) (18.7435,6.6178) (18.8023,6.7473) (18.8376,6.8768) (18.8376,6.971) (18.8376,7.1122) (18.8612,7.23) (18.8612,7.4654) (18.8729,7.6891) (18.8612,7.8185) (18.8259,7.9951) (18.767,8.1246) (18.6493,8.3718) (18.5198,8.5837) (18.4138,8.7367) (18.3668,8.8545) (18.2726,8.9251) (18.1431,9.0546) (17.8017,9.3018) (17.6369,9.396) (17.4957,9.443) (17.272,9.4901) (17.0954,9.5372) (16.9071,9.6078) (16.7658,9.6314) (16.5775,9.6549) (16.4833,9.6549) (16.342,9.702) (16.2008,9.702) (16.0477,9.7138) (15.8712,9.7491) (15.6828,9.7727) (15.4003,9.7727) (14.2937,9.7962) (13.1637,9.7962) (12.6222,9.7962) (12.4338,9.7962) (12.2455,9.7727) (12.0571,9.7256) (11.9158,9.7256) (11.751,9.702) (11.6098,9.6549) (11.4685,9.6078) (11.3273,9.5843) (11.2095,9.5608) (11.1154,9.5137) (10.9623,9.443) (10.774,9.3371) (10.5503,9.1841) (10.362,9.0075) (10.2913,8.8545) (10.2207,8.6661) (10.2207,8.5719) (10.1972,8.4307) (10.1972,8.313) (10.1736,8.1952) (10.1736,8.1011) (10.1501,8.0069) (10.1501,7.7008) (10.1501,7.4889) (10.1736,7.3477) (10.1854,7.1476) (10.2443,7.0063) (10.3502,6.8886) (10.4562,6.8533)};
    \node[small label] at (11.02,6.61) {$A$};
    \node[small label] at (11.06,5.07) {$D$};
    \node[small label] at (17.7377,8.5168) {$+$};
    \node[small label] at (7.0769,2.7411) {$+$};
    \node[small label] at (6.9914,2.1552) {$-$};
    \node[small label] at (17.64,6.57) {$+$};
    \node[small label] at (16.82,6.61) {$-$};
    \node[small label] at (11.64,5.59) {$v$};
    \node[small label] at (18.488,9.0934) {$-$};
    \node[path label,anchor=west] at (17.5427,5.7703) {$\ga_\de(\tau_v)$};
    \node[path label,anchor=east] at (17.3578,5.8644) {$\Loop_v^{-1}(\tau_-^v)$};
    \node[path label,anchor=east] at (8.6938,5.2523) {$\Loop_v^{-AC}(\varsigma^v)$};
    \node[panel title] at (11.1154,11.609) {\textbf{Subcase (3)}};
  \end{scope}
\end{tikzpicture}}
\end{center}
\caption{The three subcases for the local visit $(A,D)$.  Solid black curves are portions of the exploration path $\ga_\de$, the dashed black circle is the attached loop $\Loop_v$, and the blue circle is its reflection in $L_v^{AC}$.  The labelled points are the stopping points used in the three bond changes.}
\label{rotation-surgery}
\end{figure}

\medskip
\noindent
{\bf Subcase (1).}
Assume that
\[
\omega(E)=1-\omega(E^{AC})
   \quad\hbox{for every }E\in\mathfrak E^v,
\]
and, when $t_v<n_v$, that
\[
\omega(E_{\ga_\de(t_v)})=0.
\]
Define
\beq
s^v&=&\inf\{j>0:\Loop_v^{AC}(j)
       \in\ga_\de[0,n_v]\},\\
s_v&=&\inf\{k\leq n_v:\ga_\de(k)=\Loop_v^{AC}(s^v)\}.
\eeq
Thus $s^v$ is the first vertex at which the reflected loop touches the
past, whereas $t^v$ is the first index at which the reflected loop and
the past share an edge.  Consequently
\[
 s^v\leq t^v,
\]
and equality may occur.  Under the present state condition, in
particular when $t_v<n_v$ the bond $E_{\ga_\de(t_v)}$ is dual-open.
The local primal-open/dual-open pairing at the first shared edge then
implies that $\Loop_v^{AC}$ meets the future of the exploration path
before it returns to $v$.

We apply the reflection--duality switch to every primal bond touched or crossed by $\Loop_v[0,s^v)$:
\beqn
 && \omega'(E)=1-\omega(E^{AC}),\qquad
    \omega'(E^{AC})=1-\omega(E),
   \quad\hbox{if }\Loop_v[0,s^v)\hbox{ touches or crosses }E; \nonumber\\
 && \omega'(E)=\omega(E),\qquad\hbox{otherwise}. \label{subcase(1)}
\eeqn

We verify the resulting exploration path.  Before reaching $v$, it follows the same primal-open/dual-open boundary as $\ga_\de$.  At $v$, the switched bonds put the dual-open side along the reflected loop segment and the primal-open side along the original segment.  Therefore the interface leaves $v$ through $B$, follows $\Loop_v^{AC}[0,s^v]$, and rejoins the old future.  No other exit is possible because a primal-open path and a dual-open path lie on the two sides of the displayed exploration-path segment.  Thus
\[
A,B\in\ga_\de',\qquad C,D\notin\ga_\de'.
\]
Outside the finite set of switched bond pairs, $\ga_\de'$ agrees with $\ga_\de$.  The local quarter-turn gives the two winding identities displayed above.

The one-to-one property of this construction is proved below.

\medskip
\noindent
{\bf Subcase (2).}
Assume that
\[
\omega(E)=1-\omega(E^{AC})
   \quad\hbox{for every }E\in\mathfrak E^v,
\]
that $t_v<n_v$, and that
\[
\omega(E_{\ga_\de(t_v)})=1.
\]
Since $\Loop_v^{-AC}$ traverses the same reflected loop as
$\Loop_v^{AC}$ in the opposite direction, it also has a shared edge
with the past $\ga_\de[0,n_v)$.  Define
\beq
\mathrm t^v
&=&\inf\{j>0:\Loop_v^{-AC}[j,j+1]
       \hbox{ is an edge of }\ga_\de[0,n_v)\},\\
\mathrm t_v
&=&\inf\{ k>0:
       \ga_\de(k)=\Loop_v^{-AC}(\mathrm t^v)\}.
\eeq
Thus $\mathrm t_v$ is the corresponding time on the past of the
exploration path.  We switch the bonds touched or crossed by the two loop segments
$\Loop_v[0,t^v]$ and $\Loop_v^{-1}[0,\mathrm t^v)$:
\beqn
 && \omega'(E)=1-\omega(E^{AC}),\qquad
    \omega'(E^{AC})=1-\omega(E), \nonumber\\
 &&\hspace{25mm}
    \hbox{if }\Loop_v[0,t^v]\cup\Loop_v^{-1}[0,\mathrm t^v)
    \hbox{ touches or crosses }E;\nonumber\\
 && \omega'(E)=\omega(E),\qquad\hbox{otherwise}. \label{subcase(2)}
\eeqn

The segment $\ga_\de(t_v,n_v]$ and the reflected loop segment $\Loop_v^{AC}[0,t^v]$ now have the primal-open and dual-open bonds on the required sides.  Hence the new exploration path enters $v$ through $C$ and leaves through $D$:
\[
C,D\in\ga_\de',\qquad A,B\notin\ga_\de'.
\]
Let
\[
w=\ga_\de(\mathrm t_v)
  =\Loop_v^{-AC}(\mathrm t^v).
\]
At the shared edge selected by $\mathrm t^v$, the two reflected primal bonds have the same state.  Thus either
\beqn
\omega'(E_w)=1,\qquad
\omega'(E_w^{AC})=1, \label{stat1-mathtv}
\eeqn
or
\beqn
\omega'(E_w)=0,\qquad
\omega'(E_w^{AC})=0. \label{stat0-mathtv}
\eeqn
The change at $v$ is a counterclockwise quarter-turn and gives the stated winding identities.  The one-to-one property of this construction is proved below.

\medskip
\noindent
{\bf Subcase (3).}
Assume that there exists a bond $E\in\mathfrak E^v$ such that
\[
\omega(E)\neq1-\omega(E^{AC}).
\]
We begin at the last such mismatch along the past of the exploration path:
\beqn
\tau_v&=&\sup\{j<n_v:E_{\ga_\de(j)}\in\mathfrak E^v,
   \ \omega(E_{\ga_\de(j)})\neq1-\omega(E_{\ga_\de(j)}^{AC})\},\\
\tau^v&=&\inf\{j>0:\Loop_v(j)=\ga_\de(\tau_v)\}.
\eeqn
By maximality of $\tau_v$, every later common bond of
$\ga_\de(\tau_v,n_v)$ and $\Loop_v(\tau^v,n_\Loop)$ has complementary reflected state.

Let
\[
\tau_-^v=\inf\{j>0:\Loop_v^{-1}(j)=\Loop_v(\tau^v)\}.
\]
Starting from $v$ along the reverse reflected loop, let $\varsigma^v$ be the first index at which an edge is shared with the past of $\ga_\de$ or with $\Loop_v[0,\tau^v]$:
\beq
\varsigma^v=\inf\big\{j\leq\tau_-^v:
   \Loop_v^{-AC}[j,j+1]\hbox{ is an edge of }
   \ga_\de[0,n_v)\cup\Loop_v[0,\tau^v]\big\}.
\eeq
If this set is empty, we put $\varsigma^v=\tau_-^v$.

The original and reverse portions of the loop may cross the same primal bonds.  We denote these overlaps by
\beq
\mathcal E_{coi}
 =\big\{E_{\Loop_v^{-1}(j)}:j\leq\varsigma^v,
      \ \Loop_v^{-1}(j)\in\Loop_v(0,\tau^v)\big\}.
\eeq
Among them, $\mathcal E_{exc}$ consists of the bonds for which the reflected bond is already fixed by the past or by the first loop segment with the state needed by the exploration path:
\beq
\mathcal E_{exc}
&=&\big\{E_{\Loop_v^{-1}(j)}\in\mathcal E_{coi}:
 \Loop_v^{-AC}(j)\in\ga_\de[0,\tau_v]\cup\Loop_v(0,\tau^v],
 \ \omega(E_{\Loop_v^{-AC}(j)})=1\big\}\\
&&\ \cup
\big\{E_{\Loop_v^{-1}(j)}\in\mathcal E_{coi}:
 \Loop_v^{-AC}(j)\in\ga_\de[\tau_v,n_v),
 \ \omega(E_{\Loop_v^{-AC}(j)})=0\big\}.
\eeq
Order the bonds of $\mathcal E_{coi}\setminus\mathcal E_{exc}$ in the order in which they are crossed by $\Loop_v^{-1}$.  Consecutive bonds with the same reflected state form a block:
\begin{equation}\label{decom-Ecoi}
\mathcal E_{coi}\setminus\mathcal E_{exc}
 =\{E_j:j\in\cup_{l=0}^{l_0}[k_l,k_{l+1}),\ j\in\Z^+\}.
\end{equation}
The block decomposition will be used in the one-to-one comparison below: the reflected state is constant inside each block and alternates from one block to the next.

The modified configuration is
\beqn
 && \omega'(E_{\ga_\de(\tau_v)})=1-\omega(E_{\ga_\de(\tau_v)}),
    \qquad \omega'(E_v)=1-\omega(E_v); \nonumber\\
 && \omega'(E)=1-\omega(E^{AC}),\qquad
    \omega'(E^{AC})=1-\omega(E), \nonumber\\
 &&\hspace{20mm}
    E\in\{E:v_E\in\Loop_v^{-1}(0,\varsigma^v)\}
      \setminus\mathcal E_{coi}; \nonumber\\
 && \omega'(E)=\omega(E),\qquad
    \omega'(E^{AC})=1-\omega(E^{AC}),
    \quad E\in\mathcal E_{coi}\setminus\mathcal E_{exc};\nonumber\\
 && \omega'(E^{AC}_{\ga_\de(\tau_v)})
       =1-\omega(E^{AC}_{\ga_\de(\tau_v)}), \nonumber\\
 &&\hspace{13mm}
   \hbox{if }\Loop_v^{-AC}(0,\varsigma^v]
      \hbox{ does not share an edge with }
      \ga_\de[0,n_v)\cup\Loop_v[0,\tau^v] \nonumber\\
 &&\hspace{13mm}
   \hbox{and }\ga_\de^{AC}(\tau_v)
       \notin\ga_\de[0,\tau_v]\cup\Loop_v[0,\tau^v];\nonumber\\
 && \omega'(E)=\omega(E),\qquad\hbox{otherwise}. \label{subcase(3)}
\eeqn

The first two flips change the local passage at $v$.  Away from overlap bonds, the reflected primal-open and dual-open path segments determine the new exploration-path segment.  On an overlap bond, changing only the reflected state prevents the two loop traversals from producing an additional branch.  The optional final flip resolves the remaining local pairing at the starting mismatch.

We now verify the exploration path.  The dual-open bonds touched by
$\ga_\de(\tau_v,n_v]$ and by $\Loop_v[\tau^v,n_\Loop]$ form one side of the new interface, while the complementary primal-open bonds form the other side.  Therefore the new exploration path passes through $v$ using $C$ and $D$, and it cannot use $A$ or $B$.  Outside the finite set of switched bond pairs it agrees with the old exploration path.  Hence
\[
C,D\in\ga_\de',\qquad A,B\notin\ga_\de',
\]
and the two windings are increased by $\pi/2$ as stated above.

The one-to-one property of this construction is proved below.

\medskip
\noindent
{\bf The one-to-one property when $r<d$.}
For $k=1,2$, a superscript $(k)$ denotes the object constructed from
$\omega^{(k)}$.  Suppose that both $\omega^{(1)}$ and $\omega^{(2)}$
have the local visit $(A,D)$ and that their modified configurations are the same:
\[
 (\omega^{(1)})'=(\omega^{(2)})'=:\eta.
\]
Let $\Loop_v'$ denote the loop attached to $E_v^*$ in the common modified configuration
 $\eta$.

\medskip
\noindent
{\it Comparison within Subcase~(1).}
Suppose that both original configurations belong to Subcase~(1).  In the
common modified configuration, $\Loop_v^{AC(k)}(s^{v(k)})$ is the first point at which
the past of $(\ga_\de^{(k)})'$ touches $\Loop_v'$.  Hence
\[
 \Loop_v^{(1)}[0,s^{v(1)}]
 =\Loop_v^{(2)}[0,s^{v(2)}].
\]
The two constructions therefore switch the same reflected bond pairs.
Since their modified configurations are equal and every other bond is unchanged, we
obtain $\omega^{(1)}=\omega^{(2)}$.

\medskip
\noindent
{\it Comparison within Subcase~(2).}
Suppose that both original configurations belong to Subcase~(2).  Since
their modified configurations are equal, we may assume, without loss of generality, that
\[
 \Loop_v^{AC(1)}[0,t^{v(1)}]
 \subseteq
 \Loop_v^{AC(2)}[0,t^{v(2)}]
\]
as sets of medial edges.  If the inclusion were strict, the two past
segments
\[
 (\ga_\de^{(1)})[t_v^{(1)},n_v^{(1)})
 \quad\hbox{and}\quad
 (\ga_\de^{(2)})[t_v^{(2)},n_v^{(2)})
\]
would cross.  At their first crossing, the two switching rules would
assign different states to at least one primal bond in the common modified configuration,
contrary to $(\omega^{(1)})'=(\omega^{(2)})'$.  Therefore
\[
 \Loop_v^{AC(1)}[0,t^{v(1)}]
 =
 \Loop_v^{AC(2)}[0,t^{v(2)}].
\]
The same comparison, using the reverse parametrization of the loop,
gives
\[
 (\Loop_v^{(1)})^{-1}[0,\mathrm t^{v(1)}]
 =
 (\Loop_v^{(2)})^{-1}[0,\mathrm t^{v(2)}].
\]
The definitions of $t^v$ and $\mathrm t^v$ then give
\[
 \Loop_v^{(1)}=\Loop_v^{(2)},
 \qquad
 (\ga_\de^{(1)})[0,n_v^{(1)})
 =
 (\ga_\de^{(2)})[0,n_v^{(2)}).
\]
The two original configurations use the same switched bond pairs, and their common modified configuration
therefore implies $\omega^{(1)}=\omega^{(2)}$.

\medskip
\noindent
{\it Comparison within Subcase~(3).}
Suppose that both original configurations belong to Subcase~(3).  We may
assume
\[
 \Loop_v^{(1)}[0,\tau^{v(1)}]
 \subseteq
 \Loop_v^{(2)}[0,\tau^{v(2)}].
\]
If the inclusion were strict, the past segments
\[
 (\ga_\de^{(1)})[\tau_v^{(1)},n_v^{(1)})
 \quad\hbox{and}\quad
 (\ga_\de^{(2)})[\tau_v^{(2)},n_v^{(2)})
\]
would cross, and at the first crossing the two constructions would give
different states to a primal bond in the common modified configuration.  Thus
\[
 \Loop_v^{(1)}[0,\tau^{v(1)}]
 =
 \Loop_v^{(2)}[0,\tau^{v(2)}].
\]

Consider now the reverse-reflected segments
\[
 (\Loop_v^{(k)})^{-AC}(0,\varsigma^{v(k)}],
 \qquad k=1,2.
\]
If both segments share an edge with
\[
 (\ga_\de^{(k)})[0,n_v^{(k)})
 \cup\Loop_v^{(k)}[0,\tau^{v(k)}],
\]
or if neither segment does, the definitions of the bond switch,
$\tau^v$, and $\varsigma^v$ imply
\[
 \Loop_v^{(1)}=\Loop_v^{(2)},
 \qquad
 (\ga_\de^{(1)})[0,n_v^{(1)})
 =
 (\ga_\de^{(2)})[0,n_v^{(2)}).
\]
The common modified configuration then gives $\omega^{(1)}=\omega^{(2)}$.  If exactly one
of the two reverse-reflected segments has such a shared edge, the two
constructions prescribe different states either on a bond of
$\mathcal E_{coi}\setminus\mathcal E_{exc}$ or on
$E_{\ga_\de(\tau_v)}^{AC}$, again contradicting equality of the modified configurations.
Thus the map is one-to-one within Subcase~(3).

\medskip
\noindent
{\it Comparison of Subcase~(1) with Subcases~(2)--(3).}
The modified configuration from Subcase~(1) has local visit $(A,B)$,
whereas those from Subcases~(2) and~(3) have local visit $(C,D)$.
Hence Subcase~(1) cannot give the same modified configuration as either
of the other two subcases.

\medskip
\noindent
{\it Comparison of Subcases~(2) and~(3).}
Suppose, for contradiction, that $\omega^{(1)}$ belongs to Subcase~(3),
$\omega^{(2)}$ belongs to Subcase~(2), and
\[
 (\omega^{(1)})'=(\omega^{(2)})'.
\]
We may assume
\[
 \Loop_v^{(1)}[0,\tau^{v(1)}]
 \subseteq
 \Loop_v^{AC(2)}[0,t^{v(2)}].
\]
If the inclusion were strict, the two past segments
\[
 (\ga_\de^{(1)})[\tau_v^{(1)},n_v^{(1)})
 \quad\hbox{and}\quad
 (\ga_\de^{(2)})[t_v^{(2)},n_v^{(2)})
\]
would cross, and the two switch rules would assign different states to
a primal bond at their first crossing.  Therefore
\[
 \Loop_v^{(1)}[0,\tau^{v(1)}]
 =
 \Loop_v^{AC(2)}[0,t^{v(2)}].
\]

First suppose that
$(\Loop_v^{(1)})^{-AC}(0,\varsigma^{v(1)}]$ shares an edge with
\[
 (\ga_\de^{(1)})[0,n_v^{(1)})
 \cup\Loop_v^{(1)}[0,\tau^{v(1)}],
\]
or that it does not share such an edge but
\[
 (\ga_\de^{(1)})^{AC}(\tau_v^{(1)})
 \in
 (\ga_\de^{(1)})[0,\tau_v^{(1)}]
 \cup\Loop_v^{(1)}[0,\tau^{v(1)}].
\]
Then
\[
 E_{(\ga_\de^{(1)})(\tau_v^{(1)})}
 =E_{(\ga_\de^{(2)})(t_v^{(2)})},
\]
and the two bond states at this bond in the common modified configuration are both zero.  However,
the Subcase~(3) rule gives
\[
 (\omega^{(1)})'
 \big(E_{(\ga_\de^{(1)})(\tau_v^{(1)})}^{AC}\big)=1,
\]
whereas the Subcase~(2) rule gives
\[
 (\omega^{(2)})'
 \big(E_{(\ga_\de^{(2)})(t_v^{(2)})}^{AC}\big)=0.
\]
This contradicts equality of the modified configurations.

It remains to consider the case in which the reverse-reflected segment
does not share an edge with the displayed past-and-loop set and
\[
 (\ga_\de^{(1)})^{AC}(\tau_v^{(1)})
 \notin
 (\ga_\de^{(1)})[0,\tau_v^{(1)}]
 \cup\Loop_v^{(1)}[0,\tau^{v(1)}].
\]
If there is a bond
\[
 E\in\mathcal E_{coi}^{(1)}\setminus\mathcal E_{exc}^{(1)}
\]
with
\[
 \omega^{(1)}(E)+\omega^{(1)}(E^{AC})=1,
\]
choose the first such bond crossed by
$\Loop_v^{(1)}(0,\tau^{v(1)}]$.  The equality of the initial loop
segments implies that $E$ is also crossed by
$\Loop_v^{AC(2)}[0,t^{v(2)})$ and by $\Loop_v^{(2)}$.  The two switch
rules then give
\[
 (\omega^{(1)})'(E)+(\omega^{(1)})'(E^{AC})\neq1,
 \qquad
 (\omega^{(2)})'(E)+(\omega^{(2)})'(E^{AC})=1,
\]
which is impossible.  If there is no such bond, equality of the modified configurations
forces a bond $\widehat E\in\mathfrak E^{v(2)}$ satisfying
\[
 \omega^{(2)}(\widehat E)
 +\omega^{(2)}(\widehat E^{AC})\neq1,
\]
contrary to the defining assumption of Subcase~(2).  Thus Subcases~(2) and~(3) give disjoint sets of modified configurations.

We have proved that the map $\omega\mapsto\omega'$ is one-to-one on the
union of the three subcases when $r<d$.

\medskip
\noindent
{\bf Case II: the loop reaches $\p B(v;d)$, that is, $r\geq d$.}

Only the portions of the exploration path and of the attached loop before
leaving $B(v;d)$ will be changed.  We first discard one elementary
exceptional event.  If
\[
 \ga_\de[0,n_v)\cap\Loop_v=\emptyset,
\]
then the two percolation sides of the outgoing exploration path, the two sides of the clockwise loop portion; and the primary open side of $\ga_\de[0,n_v]$ give five disjoint arms in
$B(v;\de,d)$.  By \eqref{five}, this event has probability
$O((\de/d)^2)$.  We henceforth assume that the past of the exploration path
meets the attached loop.

Let
\beq
 \tau_\p&=&\inf\{j<n_v:\ga_\de[j,n_v]\subseteq B(v;d)\},\\
 \tau^d&=&\sup\{j<n_\Loop:\Loop_v^{-1}[0,j]\subseteq B(v;d)\},\\
 t^d&=&\sup\{j<n_\Loop:\Loop_v[0,j]\subseteq B(v;d)\}.
\eeq
Thus $\ga_\de[\tau_\p,n_v]$ is the last part of the past contained in
$B(v;d)$, while $\Loop_v[0,t^d]$ and
$\Loop_v^{-1}[0,\tau^d]$ are the two maximal initial loop segments contained
in this ball.

\medskip
\noindent
{\it Subcases~(1) and~(2).}
If the assumptions of Subcase~(1) hold and $s^v\leq t^d$, we use
\eqref{subcase(1)}.  If the assumptions of Subcase~(2) hold and
\[
 t^v\leq t^d,\qquad \mathrm t^v\leq\tau^d,
\]
we use \eqref{subcase(2)}.  In either situation every changed reflected bond
pair is contained in $B(v;d)$.  The primal-open and dual-open paths used in
the verification for $r<d$ are therefore unchanged up to their joining
points, and the same argument gives the required local visit and winding.

\medskip
\noindent
{\it Subcase~(3).}
The stopping point of the reverse-reflected loop is now defined before that
loop leaves $B(v;d)$:
\beq
 \varsigma^v
 =\inf\big\{j\leq\min(\tau^d,\tau_-^v):
 \Loop_v^{-AC}[j,j+1]\hbox{ is an edge of }
 \ga_\de[0,n_v)\cup\Loop_v[0,\tau^v]\big\}.
\eeq
If the set is empty, put
$\varsigma^v=\min(\tau^d,\tau_-^v)$.  When
$\tau_-^v\leq\tau^d$, the whole reverse loop portion needed in
Subcase~(3) remains in $B(v;d)$, and we use \eqref{subcase(3)}.  When
$\tau_-^v>\tau^d$ and $\varsigma^v<\tau^d$, the reverse-reflected loop has a
shared edge with the past or with $\Loop_v[0,\tau^v]$ before it leaves the
ball.  We stop the change at that first shared edge and put
\beqn
 && \omega'(E_{\ga_\de(\tau_v)})=1-\omega(E_{\ga_\de(\tau_v)}),
    \qquad \omega'(E_v)=1-\omega(E_v); \nonumber\\
 && \omega'(E)=1-\omega(E^{AC}),\qquad
    \omega'(E^{AC})=1-\omega(E), \nonumber\\
 &&\hspace{20mm}
    E\in\{E:v_E\in\Loop_v^{-1}(0,\varsigma^v)\}
       \setminus\mathcal E_{coi}; \nonumber\\
 && \omega'(E)=\omega(E),\qquad
    \omega'(E^{AC})=1-\omega(E^{AC}),
    \quad E\in\mathcal E_{coi}\setminus\mathcal E_{exc};\nonumber\\
 && \omega'(E)=\omega(E),\qquad\hbox{otherwise}.
 \label{subcase(3)-II}
\eeqn
The same primal-open and dual-open paths as in Subcase~(3) force the modified
exploration path to pass through $v$ using $C,D$.  In the truncated case it
joins the unchanged part at the shared edge corresponding to $\varsigma^v$.
No bond outside $B(v;d)$ is changed, and the winding identities are the same
as for $r<d$.

The remaining failures are contained in the three events below.  The first
covers the failure of Subcase~(1), or the failure of the first loop portion
in Subcase~(2); the second covers the reverse-loop failure in Subcase~(3);
and the third covers the second loop portion in Subcase~(2).

\begin{lemma}\label{special1}
Let
\beq
 \mathfrak E
 &=&\{t^v>t^d,\ \Loop_v^{AC}(0,t^d]
       \cap\ga_\de[\tau_\p,n_v)=\emptyset\},\\
 \mathfrak E'
 &=&\{\tau_-^v>\tau^d,\ \varsigma^v\geq\tau^d\},\\
 \mathfrak E''
 &=&\{t^v\leq t^d,\ \mathrm t^v>\tau^d,
      \ \omega(E_{\ga_\de(t_v)})=1,\\
 &&\hspace{34mm}
      \omega(E)=1-\omega(E^{AC})
       \hbox{ for every }E\in\mathfrak E^v\}.
\eeq
Then
\beqn
 \Pro(\mathfrak E)&\preceq&(\de/d)^2, \label{spe1}\\
 \Pro(\mathfrak E')&\preceq&(\de/d)^2, \label{spe12}\\
 \Pro(\mathfrak E'')&\preceq&(\de/d)^2. \label{spe2}
\eeqn
\end{lemma}

\begin{proof}
After a rotation and translation, suppose that $v=0$, that the medial edge
$A$ points towards $v$ in the negative imaginary direction, and that
\[
 \Half_L=\{z:\Re z\leq0\},\qquad
 \Half_R=\{z:\Re z\geq0\},\qquad
 L_v^{AC}=\{z:\Re z=0\}.
\]
Let
\[
 \mathcal B_d=\{E\in E_{\Omega_\de}:v_E\in B(v;d)\}.
\]

We begin with the proof of \eqref{spe1}.  We keep the successive bond
changes explicit because they are also used to prove the one-to-one
property.

Let
\beq
 \mathrm F^d
 &=&\{E:\ga_\de[\tau_\p,n_v]\hbox{ crosses }E,
        \ v_E\notin L_v^{AC}\},\\
 \mathrm F^o
 &=&\{E:\ga_\de[\tau_\p,n_v]\hbox{ touches }E,
        \ v_E\notin L_v^{AC}\},\\
 \mathrm F_L^d
 &=&\{E:\ga_\de[\tau_\p,n_v]\hbox{ crosses }E,
        \ v_E\in L_v^{AC}\},\\
 \mathrm F_L^o
 &=&\{E:\ga_\de[\tau_\p,n_v]\hbox{ touches }E,
        \ v_E\in L_v^{AC}\}.
\eeq
Write
\[
 \mathrm F^d=\{E_j^d:1\leq j\leq\mathfrak n_d\},\qquad
 \mathrm F^o=\{E_k^o:1\leq k\leq\mathfrak n_o\},
\]
and
\[
 \mathrm F_L^d=\{F_m^d:1\leq m\leq\mathfrak m\},\qquad
 \mathrm F_L^o=\{F_n^o:1\leq n\leq\mathfrak n\}.
\]
Within each family, the bonds are listed in the order in which
$\ga_\de[\tau_\p,n_v]$ first crosses or touches them while moving towards
$v$.  A bond is listed only once, even if the exploration path meets it
twice.

On $\mathfrak E$,
\[
 \Loop_v^{AC}(0,t^d]\cap\ga_\de[\tau_\p,n_v)=\emptyset.
\]
After reflecting the whole picture in $L_v^{AC}$ if necessary, we may
suppose that $\ga_\de[\tau_\p,n_v)$ does not cross the negative imaginary
half-axis and that $\Loop_v(0,t^d]$ does not cross the positive imaginary
half-axis.  This is the key planar observation used below.

We first construct $\varpi_d$ from $\omega$ by successive changes of the
bonds in $\mathcal B_d$.  The construction retains a dual-open path in
$\Half_L-\sqrt2\de/4$ from $\p B(v;2\de)$ to $\p B(v;d)$.

\noindent
{\bf Step $0$}. For every primal bond $E$ touched by $\ga_\de[\tau_\p,n_v]$ (so $E$ is a primal-open bond touched by $\ga_\de[\tau_\p,n_v]$), define
\beq
\omega_{0}(E)=1-\omega(E^{AC}), \ \omega_{0}(E^{AC})=1-\omega(E).
\eeq
For every other primal bond $E\in \mathcal B_d$, $\omega_{0}(E)=\omega(E)$.

{\bf Step $0$} produces the reflected dual-open path required in $\Half_L-\sqrt2\de/4$.  The remaining steps retain this path and make the change one-to-one.

Next, we define $\omega_j$ by iteration. Given $\omega_{j-1}$, we define $\omega_j$ as follows.

\noindent
{\bf Step $j$}. Consider $E_j^d \in\mathrm{F}^d$.

(1)-$j$. $v_{E_j^d}\in \Half_L\setminus L_v^{AC}$.  Define $\omega_j=\omega_{j-1}$.

(2)-$j$. $v_{E_j^d}\in \Half_R\setminus L_v^{AC}$ and there is no $j_0$ such that $E_j^{dAC}=E_{j_0}^d\in\mathrm{F}^d$.  If there is a dual-open path from $\p B(v;2\de)$ to $\p B(v;d)$  in $\Half_L-\sqrt2\de/4$ after the states of $E_j^d$ and $E_j^{dAC}$ are changed to $1-\omega_{j-1}(E_j^{dAC})$ and $1-\omega_{j-1}(E_j^d)$ respectively in $\omega_{j-1}$, define
 \beq
\omega_j(E_j^d)&=&1-\omega_{j-1}(E_j^{dAC}), \ \omega_j(E^{dAC}_j)=1-\omega_{j-1}(E_j^d),\\
\omega_j(E)&=&\omega_{j-1}(E), \ \text{every other bond} \ E\in \mathcal B_d.
\eeq
 If there is no dual-open path from $\p B(v;2\de)$ to $\p B(v;d)$  in $\Half_L-\sqrt2\de/4$ after the states of $E_j^d$ and $E_j^{dAC}$ are changed to $1-\omega_{j-1}(E_j^{dAC})$ and $1-\omega_{j-1}(E_j^d)$ respectively in $\omega_{j-1}$,
define $\omega_j=\omega_{j-1}$.

(3)-$j$. $v_{E_j^d}\in \Half_R\setminus L_v^{AC}$ and there exists $j_0$ such that $E_j^{dAC}=E_{j_0}^d\in\mathrm{F}^d$.   If there is a dual-open path from $\p B(v;2\de)$ to $\p B(v;d)$  in $\Half_L-\sqrt2\de/4$  after the state of $E_j^d$ is changed to $1-\omega_{j-1}(E_j^d)$ in $\omega_{j-1}$, define
 \beq
\omega_j(E_j^d)&=&1-\omega_{j-1}(E_j^d),\\
\omega_j(E)&=&\omega_{j-1}(E), \ \text{every other bond} \ E\in \mathcal B_d.
\eeq
If there is no dual-open path from $\p B(v;2\de)$ to $\p B(v;d)$  in $\Half_L-\sqrt2\de/4$ after the state of $E_j^d$ is changed to $1-\omega_{j-1}(E_j^d)$ in $\omega_{j-1}$, define $\omega_j=\omega_{j-1}$.

\noindent
{\bf Step $\mathfrak{n}_d+k$}. Consider $E_k^o \in\mathrm{F}^o$.

(1)-$\mathfrak{n}_d+k$. $v_{E_k^o}\in \Half_L\setminus L_v^{AC}$ and there is no $k_0$ such that $E_k^{oAC}=E_{k_0}^o\in\mathrm{F}^o$.  If there is a dual-open path from $\p B(v;2\de)$ to $\p B(v;d)$  in $\Half_L-\sqrt2\de/4$  after the states of $E_k^o$ and $E_k^{oAC}$ are changed to $1-\omega_{\mathfrak{n}_d+k-1}(E_k^{oAC})$ and $1-\omega_{\mathfrak{n}_d+k-1}(E_k^o)$ respectively in $\omega_{\mathfrak{n}_d+k-1}$, define
 \beq
\omega_{\mathfrak{n}_d+k}(E_k^o)&=&1-\omega_{\mathfrak{n}_d+k-1}(E_k^{oAC}), \ \omega_{\mathfrak{n}_d+k}(E_k^{oAC})=1-\omega_{\mathfrak{n}_d+k-1}(E_k^o),\\
\omega_{\mathfrak{n}_d+k}(E)&=&\omega_{\mathfrak{n}_d+k-1}(E), \ \text{every other bond} \ E\in \mathcal B_d.
\eeq
  If there is no dual-open path from $\p B(v;2\de)$ to $\p B(v;d)$  in $\Half_L-\sqrt2\de/4$ after the states of $E_k^o$ and $E_k^{oAC}$ are  changed to $1-\omega_{\mathfrak{n}_d+k-1}(E_k^{oAC})$ and $1-\omega_{\mathfrak{n}_d+k-1}(E_k^o)$ respectively in $\omega_{\mathfrak{n}_d+k-1}$,
define $\omega_{\mathfrak{n}_d+k}=\omega_{\mathfrak{n}_d+k-1}$.

(2)-$\mathfrak{n}_d+k$. $v_{E_k^o}\in \Half_L\setminus L_v^{AC}$ and there exists $k_0$ such that $E_k^{oAC}=E_{k_0}^o\in\mathrm{F}^o$. If there is a dual-open path from $\p B(v;2\de)$ to $\p B(v;d)$  in $\Half_L-\sqrt2\de/4$  after the state of $E_k^o$ is changed to $1-\omega_{\mathfrak{n}_d+k-1}(E_k^o)$ in $\omega_{\mathfrak{n}_d+k-1}$, define
 \beq
\omega_{\mathfrak{n}_d+k}(E_k^o)&=&1-\omega_{\mathfrak{n}_d+k-1}(E_k^o), \\
\omega_{\mathfrak{n}_d+k}(E)&=&\omega_{\mathfrak{n}_d+k-1}(E), \ \text{every other bond} \ E\in \mathcal B_d.
\eeq
  If there is no dual-open path from $\p B(v;2\de)$ to $\p B(v;d)$  in $\Half_L-\sqrt2\de/4$ after the state of $E_k^o$ is changed to  $1-\omega_{\mathfrak{n}_d+k-1}(E_k^o)$ in $\omega_{\mathfrak{n}_d+k-1}$,
define $\omega_{\mathfrak{n}_d+k}=\omega_{\mathfrak{n}_d+k-1}$.

(3)-$\mathfrak{n}_d+k$. $v_{E_k^o}\in \Half_R\setminus L_v^{AC}$. Define $\omega_{\mathfrak{n}_d+k}=\omega_{\mathfrak{n}_d+k-1}$.

\noindent
{\bf Step $\mathfrak{n}_d+\mathfrak{n}_o+m$}. Consider $F_m^d \in\mathrm{F}_L^d$.

Define
 \beq
\omega_{\mathfrak{n}_d+\mathfrak{n}_o+m}(F_m^d)&=&1-\omega_{\mathfrak{n}_d+\mathfrak{n}_o+m-1}(F_m^d),\\
\omega_{\mathfrak{n}_d+\mathfrak{n}_o+m}(E)&=&\omega_{\mathfrak{n}_d+\mathfrak{n}_o+m-1}(E), \ \text{every other bond} \ E\in \mathcal B_d.
\eeq

After {\bf Step $\mathfrak{n}_d+\mathfrak{n}_o+\mathfrak{m}$}, we obtain $\omega_{\mathfrak{n}_d+\mathfrak{n}_o+\mathfrak{m}}$. Let $\varpi_d=\omega_{\mathfrak{n}_d+\mathfrak{n}_o+\mathfrak{m}}$.

We also apply the same method to $\ga_\de[\tau_\p,n_v)$, swapping the roles of primal-open bonds and dual-open bonds and replacing $\Half_L-\sqrt2\de/4$ with $\Half_R+\sqrt2\de/4$, to obtain the configuration $\varpi_o$. To avoid any confusion, we still put down the detailed construction as follows. In the following construction, we use {\bf Step$_o$} to specify each step.

\noindent
{\bf Step$_o$ $0$}. For every primal bond $E$ crossed by $\ga_\de[\tau_\p,n_v]$ (so $E^*$ is a dual-open bond touched by $\ga_\de[\tau_\p,n_v]$), define
\beq
\omega'_{0}(E)=1-\omega(E^{AC}), \ \omega'_{0}(E^{AC})=1-\omega(E).
\eeq
For every other primal bond $E\in \mathcal B_d$, $\omega'_{0}(E)=\omega(E)$.

Next, we define $\omega'_j$ by iteration. Given $\omega'_{j-1}$, we define $\omega'_j$ as follows.

\noindent
{\bf Step$_o$ $j$}. Consider $E_j^d \in\mathrm{F}^d$.

(1)-$j$. $v_{E_j^d}\in \Half_L\setminus L_v^{AC}$.  Define $\omega'_j=\omega'_{j-1}$.

(2)-$j$. $v_{E_j^d}\in \Half_R\setminus L_v^{AC}$ and there is no $j_0$ such that $E_j^{dAC}=E_{j_0}^d\in\mathrm{F}^d$.  If there is a primal-open path from $\p B(v;2\de)$ to $\p B(v;d)$  in $\Half_R+\sqrt2\de/4$ after the states of $E_j^d$ and $E_j^{dAC}$ are changed to $1-\omega'_{j-1}(E_j^{dAC})$ and $1-\omega'_{j-1}(E_j^d)$ respectively in $\omega'_{j-1}$, define
 \beq
\omega'_j(E_j^d)&=&1-\omega'_{j-1}(E_j^{dAC}), \ \omega'_j(E^{dAC}_j)=1-\omega'_{j-1}(E_j^d),\\
\omega'_j(E)&=&\omega'_{j-1}(E), \ \text{every other bond} \ E\in \mathcal B_d.
\eeq
 If there is no primal-open path from $\p B(v;2\de)$ to $\p B(v;d)$  in $\Half_R+\sqrt2\de/4$ after the states of $E_j^d$ and $E_j^{dAC}$ are changed to $1-\omega'_{j-1}(E_j^{dAC})$ and $1-\omega'_{j-1}(E_j^d)$ respectively in $\omega'_{j-1}$,
define $\omega'_j=\omega'_{j-1}$.

(3)-$j$. $v_{E_j^d}\in \Half_R\setminus L_v^{AC}$ and there exists $j_0$ such that $E_j^{dAC}=E_{j_0}^d\in\mathrm{F}^d$.   If there is a primal-open path from $\p B(v;2\de)$ to $\p B(v;d)$  in $\Half_R+\sqrt2\de/4$  after the state of $E_j^d$ is changed to $1-\omega'_{j-1}(E_j^d)$ in $\omega'_{j-1}$, define
 \beq
\omega'_j(E_j^d)&=&1-\omega'_{j-1}(E_j^d),\\
\omega'_j(E)&=&\omega'_{j-1}(E), \ \text{every other bond} \ E\in \mathcal B_d.
\eeq
If there is no primal-open path from $\p B(v;2\de)$ to $\p B(v;d)$  in $\Half_R+\sqrt2\de/4$ after the state of $E_j^d$ is changed to $1-\omega'_{j-1}(E_j^d)$ in $\omega'_{j-1}$, define $\omega'_j=\omega'_{j-1}$.

\noindent
{\bf Step$_o$ $\mathfrak{n}_d+k$}. Consider $E_k^o \in\mathrm{F}^o$.

(1)-$\mathfrak{n}_d+k$. $v_{E_k^o}\in \Half_L\setminus L_v^{AC}$ and there is no $k_0$ such that $E_k^{oAC}=E_{k_0}^o\in\mathrm{F}^o$.  If there is a primal-open path from $\p B(v;2\de)$ to $\p B(v;d)$  in $\Half_R+\sqrt2\de/4$  after the states of $E_k^o$ and $E_k^{oAC}$ are changed to $1-\omega'_{\mathfrak{n}_d+k-1}(E_k^{oAC})$ and $1-\omega'_{\mathfrak{n}_d+k-1}(E_k^o)$ respectively in $\omega'_{\mathfrak{n}_d+k-1}$, define
 \beq
\omega'_{\mathfrak{n}_d+k}(E_k^o)&=&1-\omega'_{\mathfrak{n}_d+k-1}(E_k^{oAC}), \ \omega'_{\mathfrak{n}_d+k}(E_k^{oAC})=1-\omega'_{\mathfrak{n}_d+k-1}(E_k^o),\\
\omega'_{\mathfrak{n}_d+k}(E)&=&\omega'_{\mathfrak{n}_d+k-1}(E), \ \text{every other bond} \ E\in \mathcal B_d.
\eeq
  If there is no primal-open path from $\p B(v;2\de)$ to $\p B(v;d)$  in $\Half_R+\sqrt2\de/4$ after the states of $E_k^o$ and $E_k^{oAC}$ are  changed to $1-\omega'_{\mathfrak{n}_d+k-1}(E_k^{oAC})$ and $1-\omega'_{\mathfrak{n}_d+k-1}(E_k^o)$ respectively in $\omega'_{\mathfrak{n}_d+k-1}$,
define $\omega'_{\mathfrak{n}_d+k}=\omega'_{\mathfrak{n}_d+k-1}$.

(2)-$\mathfrak{n}_d+k$. $v_{E_k^o}\in \Half_L\setminus L_v^{AC}$ and there exists $k_0$ such that $E_k^{oAC}=E_{k_0}^o\in\mathrm{F}^o$. If there is a primal-open path from $\p B(v;2\de)$ to $\p B(v;d)$  in $\Half_R+\sqrt2\de/4$  after the state of $E_k^o$ is changed to $1-\omega'_{\mathfrak{n}_d+k-1}(E_k^o)$ in $\omega'_{\mathfrak{n}_d+k-1}$, define
 \beq
\omega'_{\mathfrak{n}_d+k}(E_k^o)&=&1-\omega'_{\mathfrak{n}_d+k-1}(E_k^o), \\
\omega'_{\mathfrak{n}_d+k}(E)&=&\omega'_{\mathfrak{n}_d+k-1}(E), \ \text{every other bond} \ E\in \mathcal B_d.
\eeq
  If there is no primal-open path from $\p B(v;2\de)$ to $\p B(v;d)$  in $\Half_R+\sqrt2\de/4$ after the state of $E_k^o$ is changed to  $1-\omega'_{\mathfrak{n}_d+k-1}(E_k^o)$ in $\omega'_{\mathfrak{n}_d+k-1}$,
define $\omega'_{\mathfrak{n}_d+k}=\omega'_{\mathfrak{n}_d+k-1}$.

(3)-$\mathfrak{n}_d+k$. $v_{E_k^o}\in \Half_R\setminus L_v^{AC}$. Define $\omega'_{\mathfrak{n}_d+k}=\omega'_{\mathfrak{n}_d+k-1}$.

\noindent
{\bf Step$_o$ $\mathfrak{n}_d+\mathfrak{n}_o+n$}. Consider $F_n^o \in\mathrm{F}_L^o$.

Define
 \beq
\omega'_{\mathfrak{n}_d+\mathfrak{n}_o+n}(F_n^o)&=&1-\omega'_{\mathfrak{n}_d+\mathfrak{n}_o+n-1}(F_n^o),\\
\omega'_{\mathfrak{n}_d+\mathfrak{n}_o+n}(E)&=&\omega'_{\mathfrak{n}_d+\mathfrak{n}_o+n-1}(E), \ \text{every other bond} \ E\in \mathcal B_d.
\eeq

After {\bf Step$_o$ $\mathfrak{n}_d+\mathfrak{n}_o+\mathfrak{n}$}, we obtain $\omega'_{\mathfrak{n}_d+\mathfrak{n}_o+\mathfrak{n}}$. Let $\varpi_o=\omega'_{\mathfrak{n}_d+\mathfrak{n}_o+\mathfrak{n}}$.

Next we express $\varpi_d$ as the product of its restrictions to the two shifted half-planes:
\[
 \varpi_d=\varpi_{Ld}\times\varpi_{Rd},\qquad
 \varpi_{Ld}=\varpi_d\big|_{\Half_L-\sqrt2\de/4},\qquad
 \varpi_{Rd}=\varpi_d\big|_{\Half_R-\sqrt2\de/4}.
\]
By construction, $\varpi_{Ld}$ contains a dual-open path in
$\Half_L-\sqrt2\de/4$ from $\p B(v;2\de)$ to $\p B(v;d)$.
Similarly,
\[
 \varpi_o=\varpi_{Lo}\times\varpi_{Ro},\qquad
 \varpi_{Lo}=\varpi_o\big|_{\Half_L-\sqrt2\de/4},\qquad
 \varpi_{Ro}=\varpi_o\big|_{\Half_R-\sqrt2\de/4}.
\]
The restriction $\varpi_{Ro}$ contains a primal-open path in the smaller
half-plane $\Half_R+\sqrt2\de/4$ from $\p B(v;2\de)$ to
$\p B(v;d)$.

Using $\varpi_{Ld}$ and $\varpi_{Ro}$,  we can define
$$
\varpi=\varpi_{Ld}\times \varpi_{Ro},
$$
which contains a dual-open path in $\Half_L-\sqrt2\de/4$  from $\p B(v;2\de)$ to $\p B(v;d)$ and a primal-open path in $\Half_R+\sqrt2\de/4$  from $\p B(v;2\de)$ to $\p B(v;d)$.

We next prove the one-to-one property of this map $\omega \rightarrow \varpi$.

Suppose $\omega^{(k)}$, $k=1,2$, are two configurations on $\mathcal B_d$ such that $\varpi^{(1)}=\varpi^{(2)}$, and $\ga_\de^{(1)}(\tau_\p^{(1)})\not=\ga_\de^{(2)}(\tau_\p^{(2)})$, or $\ga_\de^{(1)}(\tau_\p^{(1)})=\ga_\de^{(2)}(\tau_\p^{(2)})$ with $\omega^{(1)}(E_{\ga_\de^{(1)}(\tau_\p^{(1)})})\not=\omega^{(2)}(E_{\ga_\de^{(2)}(\tau_\p^{(2)})})$. Since both $\ga_\de^{(1)}$ and $\ga_\de^{(2)}$ move towards $v$, we can define
\beq
\mathbf{t}_1&=&\inf\{j\geq \tau_\p^{(1)}: \ga_\de^{(1)}[j,j+1]\subseteq_c \ga_\de^{(2)}[\tau_\p^{(2)},n_v^{(2)}]\},\\
\mathbf{t}_2&=&\inf\{j\geq \tau_\p^{(2)}: \ga_\de^{(2)}(j)=\ga_\de^{(1)}(\mathbf{t}_1)\}.
\eeq
Without loss of generality, suppose $\omega^{(1)}(E_{\ga_\de^{(1)}(\mathbf{t}_1)})=0$ and $\omega^{(2)}(E_{\ga_\de^{(2)}(\mathbf{t}_2)})=1$.

First, suppose $\ga_\de^{(1)}(\mathbf{t}_1)\in \Half_R\setminus L_v^{AC}$. We have to consider four cases.

Case 1.  $\omega^{(1)}(E^{AC}_{\ga_\de^{(1)}(\mathbf{t}_1)})=1$ and $\omega^{(2)}(E^{AC}_{\ga_\de^{(2)}(\mathbf{t}_2)})=0$. Then
\beqn
\varpi_d^{(1)}(E^{AC}_{\ga_\de^{(1)}(\mathbf{t}_1)})=1\not=0=\varpi_d^{(2)}(E^{AC}_{\ga_\de^{(2)}(\mathbf{t}_2)}). \label{varpi}
\eeqn

Case 2. $\omega^{(1)}(E^{AC}_{\ga_\de^{(1)}(\mathbf{t}_1)})=1$ and $\omega^{(2)}(E^{AC}_{\ga_\de^{(2)}(\mathbf{t}_2)})=1$. We directly have $\varpi_d^{(1)}(E_{\ga_\de^{(1)}(\mathbf{t}_1)})=\varpi_o^{(1)}(E_{\ga_\de^{(1)}(\mathbf{t}_1)})=0$.  If $E^{AC}_{\ga_\de^{(2)}(\mathbf{t}_2)}\notin \mathrm{F}^{o(2)}$, then we have \eqref{varpi}. If $E^{AC}_{\ga_\de^{(2)}(\mathbf{t}_2)}\in \mathrm{F}^{o(2)}$, we consider $\varpi_o^{(2)}$. Noting the primal-open bonds touched by $\ga_\de^{(1)}$ and $\ga_\de^{AC(2)}$, one can see that $E^{AC}_{\ga_\de^{(2)}(\mathbf{t}_2)}$ and $E_{\ga_\de^{(2)}(\mathbf{t}_2)}$ are irrelevant in the construction of $\varpi_o^{(2)}$. So $\varpi_o^{(2)}(E_{\ga_\de^{(2)}(\mathbf{t}_2)})=1$.

Case 3. $\omega^{(1)}(E^{AC}_{\ga_\de^{(1)}(\mathbf{t}_1)})=0$ and $\omega^{(2)}(E^{AC}_{\ga_\de^{(2)}(\mathbf{t}_2)})=0$. We directly have $\varpi_d^{(2)}(E_{\ga_\de^{(2)}(\mathbf{t}_2)})=\varpi_o^{(2)}(E_{\ga_\de^{(2)}(\mathbf{t}_2)})=1$.  If $E^{AC}_{\ga_\de^{(1)}(\mathbf{t}_1)}\notin \mathrm{F}^{d(1)}$, then we have \eqref{varpi}.   If $E^{AC}_{\ga_\de^{(1)}(\mathbf{t}_1)}\in \mathrm{F}^{d(1)}$, we consider $\varpi_o^{(1)}$. Noting the primal-open bonds touched by $\ga_\de^{(1)}$, one can see that $E_{\ga_\de^{(1)}(\mathbf{t}_1)}$ is irrelevant in the construction of $\varpi_o^{(1)}$. So $\varpi_o^{(1)}(E_{\ga_\de^{(1)}(\mathbf{t}_1)})=0$.

Case 4.  $\omega^{(1)}(E^{AC}_{\ga_\de^{(1)}(\mathbf{t}_1)})=0$ and $\omega^{(2)}(E^{AC}_{\ga_\de^{(2)}(\mathbf{t}_2)})=1$.

If $E^{AC}_{\ga_\de^{(1)}(\mathbf{t}_1)}\notin \mathrm{F}^{d(1)}$ and $E^{AC}_{\ga_\de^{(2)}(\mathbf{t}_2)}\notin \mathrm{F}^{o(2)}$, then
$$\varpi_d^{(2)}(E_{\ga_\de^{(2)}(\mathbf{t}_2)})=\varpi_d^{(2)}(E^{AC}_{\ga_\de^{(2)}(\mathbf{t}_2)})=0.$$
Note that  the dual-open bonds touched by $\ga_\de^{AC(1)}[\tau_\p^{(1)},n_v^{(1)}]$ obtained in {\bf Step $0$} and the dual-open bonds touched by $\ga_\de^{(1)}$ and $\ga_\de^{(2)}$  guarantee that both $E^{AC}_{\ga_\de^{(1)}(\mathbf{t}_1)}$ and $E_{\ga_\de^{(1)}(\mathbf{t}_1)}$ are irrelevant in the construction of $\varpi_d^{(1)}$. In other words,  there is a dual-open path from $\p B(v;2\de)$ to $\p B(v;d)$  in $\Half_L-\sqrt2\de/4$ if  we change the state of $E_{\ga_\de^{(1)}(\mathbf{t}_1)}$ and $E^{AC}_{\ga_\de^{(1)}(\mathbf{t}_1)}$ from dual-open to open in the construction of $\varpi_d^{(1)}$ from $\omega^{(1)}$ before {\bf Step $\mathfrak{n}_d+\mathfrak{n}_o+1$}. Hence
$$\varpi_d^{(1)}(E_{\ga_\de^{(1)}(\mathbf{t}_1)})=\varpi_d^{(1)}(E^{AC}_{\ga_\de^{(1)}(\mathbf{t}_1)})=1.$$
The same argument shows the following \eqref{varpi41}, \eqref{varpi42} and \eqref{varpi43}.

\noindent
If $E^{AC}_{\ga_\de^{(1)}(\mathbf{t}_1)}\in \mathrm{F}^{d(1)}$ and $E^{AC}_{\ga_\de^{(2)}(\mathbf{t}_2)}\in \mathrm{F}^{o(2)}$, then
\beqn
\varpi_d^{(1)}(E^{AC}_{\ga_\de^{(1)}(\mathbf{t}_1)})=0, \varpi_d^{(2)}(E^{AC}_{\ga_\de^{(2)}(\mathbf{t}_2)})=1. \label{varpi41}
\eeqn
If $E^{AC}_{\ga_\de^{(1)}(\mathbf{t}_1)}\notin \mathrm{F}^{d(1)}$ and $E^{AC}_{\ga_\de^{(2)}(\mathbf{t}_2)}\in \mathrm{F}^{o(2)}$, then
\beqn
\varpi_o^{(1)}(E_{\ga_\de^{(1)}(\mathbf{t}_1)})=0, \varpi_o^{(2)}(E_{\ga_\de^{(2)}(\mathbf{t}_2)})=1.\label{varpi42}
\eeqn
If $E^{AC}_{\ga_\de^{(1)}(\mathbf{t}_1)}\in \mathrm{F}^{d(1)}$ and $E^{AC}_{\ga_\de^{(2)}(\mathbf{t}_2)}\notin \mathrm{F}^{o(2)}$, then
\beqn
\varpi_o^{(1)}(E_{\ga_\de^{(1)}(\mathbf{t}_1)})=0, \varpi_o^{(2)}(E_{\ga_\de^{(2)}(\mathbf{t}_2)})=1.\label{varpi43}
\eeqn

The four cases show that $\varpi^{(1)}\not=\varpi^{(2)}$ if $\ga_\de^{(1)}(\mathbf{t}_1)\in \Half_R\setminus L_v^{AC}$

Second, suppose $\ga_\de^{(1)}(\mathbf{t}_1)\in \Half_L\setminus L_v^{AC}$. By the symmetric property of the respective construction of $\varpi_d$ and $\varpi_o$, we have $\varpi^{(1)}\not=\varpi^{(2)}$.

Finally, suppose $\ga_\de^{(1)}(\mathbf{t}_1)\in L_v^{AC}$.  {\bf Step$_o$ $0$} and {\bf Step$_o$ $\mathfrak{n}_d+\mathfrak{n}_o+n$} show that  $\varpi^{(1)}\not=\varpi^{(2)}$.

Hence we can suppose $\varpi^{(1)}=\varpi^{(2)}$ and $\ga_\de^{(1)}(\tau_\p^{(1)})=\ga_\de^{(2)}(\tau_\p^{(2)})$ with $\omega^{(1)}(E_{\ga_\de^{(1)}(\tau_\p^{(1)})})=\omega^{(2)}(E_{\ga_\de^{(2)}(\tau_\p^{(2)})})$. To simplify the parametrization of $\ga_\de^{(1)}$ and $\ga_\de^{(2)}$, suppose $\tau_\p^{(1)}=\tau_\p^{(2)}$. If $\ga_\de^{(1)}[\tau_\p^{(1)},n_v^{(1)}]=\ga_\de^{(2)}[\tau_\p^{(2)},n_v^{(2)}]$, then we can conclude that $\omega^{(1)}=\omega^{(2)}$. Therefore we suppose $\ga_\de^{(1)}[\tau_\p^{(1)},n_v^{(1)}]\not=\ga_\de^{(2)}[\tau_\p^{(2)},n_v^{(2)}]$ and define
\beq
\mathbf{t}&=&\inf\{j>\tau_\p^{(1)}=\tau_\p^{(2)}: \omega^{(1)}(E_{\ga_\de^{(1)}(j)})\not=\omega^{(2)}(E_{\ga_\de^{(2)}(j)})\}, \\
 d_\mathbf{t}&=&\max\{|w-v|: w\in \ga_\de^{(1)}[\mathbf{t},n_v^{(1)}]\cap \ga_\de^{(2)}[\mathbf{t},n_v^{(2)}], \omega^{(1)}(E_w)\not=\omega^{(2)}(E_w)\}\\
 \mathbf{t}_3&=&\sup\{j\in [\mathbf{t}, n_v^{(1)}]: \ga_\de^{(1)}(j)\in\ga_\de^{(2)}[\tau_\p^{(2)},n_v^{(2)}], |\ga_\de^{(1)}(j)-v|=d_\mathbf{t} \},\\
\mathbf{t}_4&=&\sup\{j\geq \mathbf{t}: \ga_\de^{(2)}(j)=\ga_\de^{(1)}(\mathbf{t}_3)\}.
\eeq
So, $\ga_\de^{(1)}(\mathbf{t}_3)\not=\ga_\de^{(2)}(\mathbf{t}_4)$ or $\ga_\de^{(1)}(\mathbf{t}_3)=\ga_\de^{(2)}(\mathbf{t}_4)$ with $\omega^{(1)}(E_{\ga_\de^{(1)}(\mathbf{t}_3)})\not=\omega^{(2)}(E_{\ga_\de^{(2)}(\mathbf{t}_4)})$.
The same argument in the case where $\varpi^{(1)}=\varpi^{(2)}$ and $\ga_\de^{(1)}(\tau_\p^{(1)})\not=\ga_\de^{(2)}(\tau_\p^{(2)})$ shows that
$\ga_\de^{(1)}(\mathbf{t}_3)=\ga_\de^{(2)}(\mathbf{t}_4)$ with $\omega^{(1)}(E_{\ga_\de^{(1)}(\mathbf{t}_3)})=\omega^{(2)}(E_{\ga_\de^{(2)}(\mathbf{t}_4)})$. This is a contradiction. Therefore, $\mathbf{t}$ does not exist. In other words, $\ga_\de^{(1)}[\tau_\p^{(1)},n_v^{(1)}]=\ga_\de^{(2)}[\tau_\p^{(2)},n_v^{(2)}]$, which implies that $\omega^{(1)}=\omega^{(2)}$.

This proves that the map $\omega\mapsto\varpi$ is one-to-one.

Similarly, we apply the same method to $\Loop_v(0,t^d]$, swapping the roles of primal-open bonds and dual-open bonds,  to get the configuration $\varpi_{\Loop}$ which contains a primal-open path in $\Half_L-\sqrt2\de/4$  from $\p B(v;2\de)$ to $\p B(v;d)$, and a dual-open path in $\Half_R+\sqrt2\de/4$  from $\p B(v;2\de)$ to $\p B(v;d)$.

  Since $t^v> t^d$ and $\Loop_v^{AC}(0,t^d]\cap\ga_\de[\tau_\p,n_v)=\emptyset$, we can construct $\varpi$ and $\varpi_{\Loop}$ simultaneously. Therefore, in $\Half_R+\sqrt2\de/4$ there are primal-open and dual-open paths from $\p B(v;2\de)$ to  $\p B(v;d)$. By \eqref{halftwo}, the corresponding event has probability at most of order $\de/d$.  The same estimate holds in $\Half_L-\sqrt2\de/4$.

The left and right arm events depend on disjoint sets of primal bonds.  Hence they are independent, and \eqref{halftwo} gives \eqref{spe1}.

\medskip
\noindent
We next prove \eqref{spe12}.  On the event $\{\tau_-^v>\tau^d\}$, the
incoming exploration-path portion
\[
 \ga_\de[\tau_\p,n_v)
\]
and the clockwise loop portion
\[
 \Loop_v^{-1}(0,\tau^d]
\]
are disjoint away from $v$ and both reach $\p B(v;d)$.  The primal-open
and dual-open sides of the incoming exploration path, the primal-open and
dual-open sides of the clockwise loop, and the arm separating these two
interfaces give five disjoint arms from scale $O(\de)$ to scale $d$.
Therefore
\[
 \mathfrak E'\subset A_5(v;c_1\de,c_2d)
\]
for universal constants $c_1,c_2>0$.  By \eqref{five},
\[
 \Pro(\mathfrak E')\preceq(\de/d)^2.
\]
The condition $\varsigma^v\geq\tau^d$ is needed in the definition of the
failure of the truncated Subcase~(3), but it is not needed for this
probability estimate.

\medskip
\noindent
It remains to prove \eqref{spe2}.  We use the same successive changes as
above, now for the two path portions
\[
 \ga_\de[\tau_\p,n_v)
 \qquad\hbox{and}\qquad
 \Loop_v^{-1}(0,\tau^d],
\]
with primal-open and dual-open interchanged for the loop portion.  The
condition $\mathrm t^v>\tau^d$ means that
$\Loop_v^{-AC}(0,\tau^d]$ has no shared edge with the past before leaving
$B(v;d)$.

We only have to check that the two successive changes are compatible when
their bond families overlap.  If the incoming exploration path meets a
bond $E$ and the clockwise loop meets $E^{AC}$, then the reflected loop
and the past have a shared edge before time $\tau^d$, contradicting
$\mathrm t^v>\tau^d$.  Thus an overlap can occur only when both path
portions meet the same primal bond $E$.  In this case $E\in\mathfrak E^v$,
and the definition of $\mathfrak E''$ gives
\[
 \omega(E)=1-\omega(E^{AC}).
\]
The reflection--duality change therefore leaves this reflected pair in
the same two states.  Hence the two successive constructions prescribe
the same states on every common bond.

After these common pairs have been fixed, apply the preceding Step~$0$,
Step~$j$, and Step$_o$ constructions to the remaining ordered families.
The one-to-one proof is the same first-disagreement proof as for
$\mathfrak E$: at the first bond where two original configurations or
their two path portions differ, one of the cases considered above gives
different states in the common modified configuration.  Thus the
simultaneous change is one-to-one.

The modified configuration contains a primal-open arm and a dual-open arm
in each of the two half-annuli.  The left and right arm events are disjoint.  Applying the half-plane two-arm estimate
\eqref{halftwo} on the two sides gives
\[
 \Pro(\mathfrak E'')\preceq(\de/d)^2.
\]
This proves \eqref{spe2} and completes the proof.
\end{proof}

After removing the five-arm event and the three events in
Lemma~\ref{special1}, one of \eqref{subcase(1)}, \eqref{subcase(2)},
\eqref{subcase(3)}, or \eqref{subcase(3)-II} is defined.  Every changed bond
lies in $B(v;d)$, and each change has the local visit and winding stated in
Case~I.

\medskip
\noindent
{\bf The one-to-one property in Case II.}
For $k=1,2$, attach a superscript $(k)$ to every object defined from
$\omega^{(k)}$.  Suppose that $\omega^{(1)}$ and $\omega^{(2)}$ are two
nonexceptional configurations in Case~II with local visit $(A,D)$ and that
they give the same modified configuration:
\[
 (\omega^{(1)})'=(\omega^{(2)})'=:\eta.
\]
Since no bond outside $B(v;d)$ is changed, both original configurations agree
with $\eta$ outside this ball.  It remains to compare the reflected bond
pairs changed inside the ball. This indeed follows from the arguments in Case~I.

\medskip
\noindent
{\bf Comparison of Cases~I and~II.}
Suppose that one original configuration belongs to Case~I and the other to
Case~II, and that their modified configurations are equal.  If they belong
to the same subcase, the corresponding comparison above uses the changed loop
portions together with exploration-path portions that are unchanged outside
$B(v;d)$, and gives equality of the two original configurations.  Their loop
radii would then be equal, contradicting
$r<d$ for one and $r\geq d$ for the other.  If they belong to different
subcases, the preceding different-subcase comparison gives a contradiction.
Therefore a configuration from Case~I and one from Case~II cannot give the
same modified configuration.

Consequently, after deleting an event of probability $O((\de/d)^2)$, the
map $\omega\mapsto\omega'$ is one-to-one on the union of Cases~I and~II.
This is a proof of injectivity.

\medskip
\noindent
{\bf The four local visits.}

We now rotate the preceding construction to the other three local visits.
For the visit $(A,D)$, let $\mathcal E_{1,2}^v$ be the configurations
satisfying the conditions of Subcase~(1), in either Case~I or Case~II,
and put
\[
 \mathcal E_{1,1}^v=\mathcal C_1^v\setminus\mathcal E_{1,2}^v.
\]
Thus $\mathcal E_{1,1}^v$ consists of Subcases~(2)--(3).

For the visit $(A,B)$, exchange the roles of $B$ and $D$ and use the
loop attached to $E_v^*$, parameterized clockwise.  Let
$\mathcal E_v^{2,1}$ be the Subcase~(1) class for this rotated
construction and put
\[
 \mathcal E_v^{2,2}=\mathcal C_2^v\setminus\mathcal E_v^{2,1}.
\]
Rotating the construction by $\pi$, let $\mathcal E_{3,2}^v$ be the
Subcase~(1) class in $\mathcal C_3^v$ and put
\[
 \mathcal E_{3,1}^v=\mathcal C_3^v\setminus\mathcal E_{3,2}^v.
\]
Likewise, let $\mathcal E_v^{4,1}$ be the Subcase~(1) class in
$\mathcal C_4^v$ and put
\[
 \mathcal E_v^{4,2}=\mathcal C_4^v\setminus\mathcal E_v^{4,1}.
\]
Whenever one of the four changes fails because of an exceptional event
from Case~II, that configuration is included in the exceptional set of
the corresponding map.  Each such event has probability
$O((\de/d)^2)$.

If changing the state of $E_v$ makes all four incident medial edges
belong to the exploration path, we retain the states of all other bonds
and write $\omega_{\setminus E_v}$.  For any one of the classes above,
the corresponding class with the state of $E_v$ reversed is obtained by
replacing $\omega$ with $\omega_{\setminus E_v}$.

\medskip
\noindent
{\it The four one-to-one maps.}
The preceding proofs in Cases~I and~II, together with their rotated
versions, give the following one-to-one maps outside exceptional events
of probability $O((\de/d)^2)$:
\beq
 \mathrm f_{AD}:\mathcal E_{1,1}^v
 &\longrightarrow&\mathcal E_v^{4,1}\cup\mathcal E_v^{4,2},\qquad
 \mathrm f_{AD}:\mathcal E_{1,2}^v
 \longrightarrow\mathcal E_v^{2,1}\cup\mathcal E_v^{2,2},\\
 \mathrm f_{AB}:\mathcal E_v^{2,1}
 &\longrightarrow&\mathcal E_{1,1}^v\cup\mathcal E_{1,2}^v,\qquad
 \mathrm f_{AB}:\mathcal E_v^{2,2}
 \longrightarrow\mathcal E_{3,1}^v\cup\mathcal E_{3,2}^v,\\
 \mathrm f_{BC}:\mathcal E_{3,1}^v
 &\longrightarrow&\mathcal E_v^{2,1}\cup\mathcal E_v^{2,2},\qquad
 \mathrm f_{BC}:\mathcal E_{3,2}^v
 \longrightarrow\mathcal E_v^{4,1}\cup\mathcal E_v^{4,2},\\
 \mathrm f_{CD}:\mathcal E_v^{4,1}
 &\longrightarrow&\mathcal E_{3,1}^v\cup\mathcal E_{3,2}^v,\qquad
 \mathrm f_{CD}:\mathcal E_v^{4,2}
 \longrightarrow\mathcal E_{1,1}^v\cup\mathcal E_{1,2}^v.
\eeq
Each change preserves the critical percolation measure because it
permutes finitely many bond states.

We next prove that the two maps starting from the visits $(A,D)$ and
$(B,C)$ can be put together.  Within each of $\mathrm f_{AD}$ and
$\mathrm f_{BC}$, the one-to-one property has already been proved.
Most pairs of modified configurations are distinguished by the local
visit of the modified exploration path.  The only two image sets with
local visit $(A,B)$ that still need to be distinguished are
\[
 \mathrm f_{AD}(\mathcal E_{1,2}^v)
 \qquad\hbox{and}\qquad
 \mathrm f_{BC}(\mathcal E_{3,1}^v).
\]
It therefore remains to prove that these two sets are disjoint:
\beqn
 \mathrm f_{AD}(\mathcal E_{1,2}^v)
 \cap
 \mathrm f_{BC}(\mathcal E_{3,1}^v)
 =\emptyset.
 \label{ADnotBC}
\eeqn
The corresponding disjointness statement for the two image sets with
modified local visit $(C,D)$ follows from \eqref{ADnotBC} by rotating
all labels through $\pi$.  Hence it suffices to prove
\eqref{ADnotBC}.

\medskip
\noindent
{\it Proof of \eqref{ADnotBC}.}
Let $\omega^{(1)}\in\mathcal E_{1,2}^v$ and
$\omega^{(2)}\in\mathcal E_{3,1}^v$, and suppose that
\[
 (\omega^{(1)})'=(\omega^{(2)})'=:\omega'.
\]
A superscript $(k)$ denotes the exploration path, the attached loop, and
all stopping indices associated with $\omega^{(k)}$.  For
$\omega^{(2)}$, the notation is obtained from that for $\omega^{(1)}$
by replacing the visit $(A,D)$ with the visit $(C,B)$.  Thus
$\omega^{(1)}$ belongs to Subcase~(1), whereas $\omega^{(2)}$ belongs
to Subcase~(2) or Subcase~(3).

\medskip
\noindent
{\it Subcase~(1) in Case~I versus Subcase~(2) in Case~I.}
Suppose first that $\omega^{(2)}$ belongs to Subcase~(2).  Equality of
the modified configurations implies that
\[
 \Loop_v^{AC(1)}[0,s^{v(1)}]
 \subsetneqq
 (\ga_\de^{(2)})[t_v^{(2)},n_v^{(2)}],
 \qquad\hbox{or}\qquad
 \Loop_v^{AC(1)}[0,s^{v(1)}]
 \supseteq
 (\ga_\de^{(2)})[t_v^{(2)},n_v^{(2)}].
\]
Assume the first possibility.  Then, for some $s>t_v^{(2)}$,
\[
 \Loop_v^{AC(1)}[0,s^{v(1)}]
 = (\ga_\de^{(2)})[s,n_v^{(2)}].
\]
The reflected path
$\ga_\de^{AC(2)}[s,n_v^{(2)}]$ does not cross
$(\Loop_v^{(2)})^{-AC}(0,\mathrm t^{v(2)}]$.  On the other hand,
$\omega^{(1)}(E)=1-\omega^{(1)}(E^{AC})$ for every
$E\in\mathfrak E^{v(1)}$, and the point
$\ga_\de^{(2)}(\mathrm t_v^{(2)})$ is the first equal-state contact
specified by \eqref{stat1-mathtv}--\eqref{stat0-mathtv}.  Hence
$\Loop_v^{(1)}[0,s^{v(1)}]$ must cross
$(\Loop_v^{(2)})^{-AC}(0,\mathrm t^{v(2)}]$, which is part of the
modified exploration path for $\omega^{(1)}$.  This contradicts
$(\omega^{(1)})'=(\omega^{(2)})'$.

We therefore have
\[
 \Loop_v^{AC(1)}[0,s^{v(1)}]
 \supseteq
 (\ga_\de^{(2)})[t_v^{(2)},n_v^{(2)}].
\]
In $(\omega^{(1)})'$, the point
$\Loop_v^{AC(1)}(s^{v(1)})$ is the first point at which the loop in the
modified configuration touches the past of the modified exploration
path.  Consequently,
\[
 (\ga_\de^{(1)})[0,s_v^{(1)}]
   =(\ga_\de^{(2)})[0,t_v^{(2)}],
\]
\[
 (\ga_\de^{(1)})[s_v^{(1)},n_v^{(1)}]
   =\Loop_v^{AC(2)}[0,t^{v(2)}],
 \qquad
 \Loop_v^{AC(1)}[0,s^{v(1)}]
   =(\ga_\de^{(2)})[t_v^{(2)},n_v^{(2)}].
\]
The state condition in Subcase~(2) also gives
$s_v^{(1)}=t_v^{(1)}$.  The reflection in $L_v^{AC}$ of the part of the
loop in $(\omega^{(2)})'$ from $\ga_\de^{(2)}(t_v^{(2)})$ to $v$ is
$\ga_\de^{AC(2)}[t_v^{(2)},n_v^{(2)}]$, and this path does not cross
$(\Loop_v^{(2)})^{-AC}(0,\mathrm t^{v(2)}]$.  By the complementary-state
condition for $\omega^{(1)}$ and the definition of the first equal-state
contact, $\Loop_v^{(1)}[0,t^{v(1)}]$ must cross the corresponding part
of the modified future and leave it through a primal-open bond.  Since
these are the same paths in the common modified configuration, we again
obtain a contradiction.  Thus Subcase~(1) and Subcase~(2) cannot give
the same modified configuration.

\medskip
\noindent
{\it Subcase~(1) in Case~I versus Subcase~(3) in Case~I.}
Suppose next that $\omega^{(2)}$ belongs to Subcase~(3).  Equality of
the modified configurations implies that
\[
 \Loop_v^{AC(1)}[0,s^{v(1)}]
 \subseteq
 (\ga_\de^{(2)})[\tau_v^{(2)},n_v^{(2)}],
 \qquad\hbox{or}\qquad
 \Loop_v^{AC(1)}[0,s^{v(1)}]
 \supsetneqq
 (\ga_\de^{(2)})[\tau_v^{(2)},n_v^{(2)}].
\]
The definition of $s^{v(1)}$ excludes the second possibility.  Hence
there is $s_0\geq\tau_v^{(2)}$ such that
\[
 (\ga_\de^{(2)})[s_0,n_v^{(2)}]
 =\Loop_v^{AC(1)}[0,s^{v(1)}].
\]
Since
$(\Loop_v^{AC(1)}[0,s^{v(1)}])^{AC}
 =\Loop_v^{(1)}[0,s^{v(1)}]$, the reverse-reflected path
$(\Loop_v^{(2)})^{-AC}(0,\varsigma^{v(2)}]$ does not cross
$(\ga_\de^{(2)})(\tau_v^{(2)},n_v^{(2)})$.

Suppose first that
$(\Loop_v^{(2)})^{-AC}(0,\varsigma^{v(2)}]$ shares an edge with
\[
 (\ga_\de^{(2)})[0,\tau_v^{(2)}]
 \cup\Loop_v^{(2)}[0,\tau^{v(2)}].
\]
Then
\[
 \ga_\de^{AC(2)}[s_0,n_v^{(2)})
 \subseteq
 \ga_\de^{AC(2)}[\tau_v^{(2)},n_v^{(2)}]
\]
does not cross the part of $(\ga_\de^{(2)})'$ from $v$ to
$(\Loop_v^{(2)})^{-AC}(\varsigma^{v(2)})$.  However, the
complementary-state condition for $\omega^{(1)}$ and the definition of
$\varsigma^{v(2)}$ imply that $\Loop_v^{(1)}[0,s^{v(1)}]$ crosses that
same part of the modified exploration path.  This is impossible.

It remains to consider the case in which
$(\Loop_v^{(2)})^{-AC}(0,\varsigma^{v(2)}]$ has no shared edge with the
preceding union.  If the primal-open or dual-open side of this path
crosses the opposite side of
\[
 \Loop_v^{(1)}[s^{v(1)},n_\Loop^{(1)})
 \setminus\Loop_v^{(1)}(0,s^{v(1)}),
\]
then a primal bond has different states in $(\omega^{(1)})'$ and
$(\omega^{(2)})'$.  If there is no such crossing, the definitions of the
bond change on
$\mathcal E_{coi}^{(2)}\setminus\mathcal E_{exc}^{(2)}$ and at
$E_{\ga_\de^{(2)}(\tau_v^{(2)})}^{AC}$ again give different states in
the two modified configurations.  More precisely, the contradiction is
at the first bond $E_1^{AC(2)}$ when
$\mathcal E_{coi}^{(2)}\setminus\mathcal E_{exc}^{(2)}\neq\emptyset$,
and otherwise at
$E_{\ga_\de^{(2)}(\tau_v^{(2)})}^{AC}$, unless the reverse-reflected
path has the shared edge excluded above.  Thus Subcase~(1) and
Subcase~(3) cannot give the same modified configuration.

\medskip
\noindent
{\it Comparisons involving Case~II.}
After the exceptional events in Lemma~\ref{special1} are removed, the
same two arguments apply when one or both configurations belong to
Case~II, with the loop portions stopped at $t^d$ or $\tau^d$.  We omit
the repeated details.  This proves \eqref{ADnotBC}.  Rotating the argument through $\pi$ excludes
the other possible collision.  Consequently
\[
 \mathrm f_+
 :=\mathrm f_{AD}\cup\mathrm f_{BC}
\]
is one-to-one outside an event $\mathcal B_+\subset
\mathcal C_1^v\cup\mathcal C_3^v$ satisfying
\[
 \Pro(\mathcal B_+)\preceq(\de/d)^2.
\]
The same argument, with the roles of the local visits reversed, shows
that
\[
 \mathrm f_-:=\mathrm f_{AB}\cup\mathrm f_{CD}
\]
is one-to-one outside an event $\mathcal B_-\subset
\mathcal C_2^v\cup\mathcal C_4^v$ satisfying
\[
 \Pro(\mathcal B_-)\preceq(\de/d)^2.
\]

\medskip
\noindent
{\it The unmatched configurations.}
Put
\[
 \mathcal X=\mathcal C_1^v\cup\mathcal C_3^v,
 \qquad
 \mathcal Y=\mathcal C_2^v\cup\mathcal C_4^v,
\]
and let
\[
 \mathcal U_+
 =\mathcal Y\setminus
   \mathrm f_+(\mathcal X\setminus\mathcal B_+).
\]
At $p=1/2$ every complete bond configuration has the same probability.
Since $\mathrm f_+$ is one-to-one,
\[
 |\mathcal U_+|
 =|\mathcal Y|-|\mathcal X\setminus\mathcal B_+|.
\]
The one-to-one map $\mathrm f_-$ gives
\[
 |\mathcal Y\setminus\mathcal B_-|\leq|\mathcal X|.
\]
Consequently
\[
 |\mathcal U_+|\leq|\mathcal B_+|+|\mathcal B_-|,
\]
and hence
\begin{equation}\label{four-visits-unmatched}
 \Pro(\mathcal U_+)
 \preceq
 \Pro(\mathcal B_+)+\Pro(\mathcal B_-)
 \preceq(\de/d)^2.
\end{equation}

Set
\[
 \mathcal E_{v,0}=\mathcal B_+\cup\mathcal U_+,
 \qquad
 \mathcal E_{v,1}
 =\{\omega_{\setminus E_v}:\omega\in\mathcal E_{v,0}\},
 \qquad
 \mathcal E_v=\mathcal E_{v,0}\cup\mathcal E_{v,1}.
\]
The passage from complete to incomplete configurations changes the
probability by at most a universal factor, and therefore
\[
 \Pro(\mathcal E_v)\preceq(\de/d)^2.
\]

Define the final classes for the visits $(A,B)$ and $(C,D)$ by
\beq
 \mathcal E_{2,1}^v
 &=&\mathrm f_{AD}
       (\mathcal E_{1,2}^v\setminus\mathcal E_{v,0}), \qquad
 \mathcal E_{2,2}^v
 =\mathcal C_2^v\setminus
       (\mathcal E_{2,1}^v\cup\mathcal E_{v,0}),\\
 \mathcal E_{4,1}^v
 &=&\mathrm f_{BC}
       (\mathcal E_{3,2}^v\setminus\mathcal E_{v,0}),\qquad
 \mathcal E_{4,2}^v
 =\mathcal C_4^v\setminus
       (\mathcal E_{4,1}^v\cup\mathcal E_{v,0}).
\eeq
For $j=1,2,3,4$ and $m=1,2$, put
$
 \mathcal E_{j,m}^{v,a}
 =\{\omega_{\setminus E_v}:\omega\in\mathcal E_{j,m}^v\}$.
Because $\mathrm f_+$ is a bijection from
$\mathcal X\setminus\mathcal E_{v,0}$ onto
$\mathcal Y\setminus\mathcal E_{v,0}$, its restrictions to the four
classes give
\beq
 \mathrm f_{AD}:\mathcal E_{1,1}^v\setminus\mathcal E_v
 &\longrightarrow&\mathcal E_{4,2}^v\setminus\mathcal E_v, \qquad
 \mathrm f_{AD}:\mathcal E_{1,2}^v\setminus\mathcal E_v
 \longrightarrow\mathcal E_{2,1}^v\setminus\mathcal E_v,\\
 \mathrm f_{BC}:\mathcal E_{3,1}^v\setminus\mathcal E_v
 &\longrightarrow&\mathcal E_{2,2}^v\setminus\mathcal E_v,\qquad
 \mathrm f_{BC}:\mathcal E_{3,2}^v\setminus\mathcal E_v
 \longrightarrow\mathcal E_{4,1}^v\setminus\mathcal E_v,
\eeq
and each displayed map is onto its stated range.

\begin{proposition} \label{rot-inv}
After deleting $\mathcal E_v$, there is a measure-preserving bijection
\[
 \mathrm f_v:
 \mathcal E_{j,2}^v\setminus\mathcal E_v
 \longrightarrow
 \mathcal E_{j+1,1}^v\setminus\mathcal E_v,
 \qquad j=1,2,3,4,
\]
where $\mathcal E_{5,1}^v$ means $\mathcal E_{1,1}^v$.  Moreover,
\[
 \Pro(\mathcal E_v)\preceq(\de/d)^2,
\]
and the winding at the common medial edge of the two local visits is
unchanged:
\beq
 W_{\mathrm f_v(\ga_\de)}(A,e_b)
 &=&W_{\ga_\de}(A,e_b),
 \qquad \omega\in\mathcal E_{1,2}^v\setminus\mathcal E_v,\\
 W_{\mathrm f_v(\ga_\de)}(B,e_b)
 &=&W_{\ga_\de}(B,e_b),
 \qquad \omega\in\mathcal E_{2,2}^v\setminus\mathcal E_v,\\
 W_{\mathrm f_v(\ga_\de)}(C,e_b)
 &=&W_{\ga_\de}(C,e_b),
 \qquad \omega\in\mathcal E_{3,2}^v\setminus\mathcal E_v,\\
 W_{\mathrm f_v(\ga_\de)}(D,e_b)
 &=&W_{\ga_\de}(D,e_b),
 \qquad \omega\in\mathcal E_{4,2}^v\setminus\mathcal E_v.
\eeq
Here $\mathrm f_v(\ga_\de)$ denotes the exploration path in the
configuration $\mathrm f_v(\omega)$.
\end{proposition}

\begin{proof}
On $\mathcal E_{1,2}^v$ and $\mathcal E_{3,2}^v$, let $\mathrm f_v$
be $\mathrm f_{AD}$ and $\mathrm f_{BC}$, respectively.  On
$\mathcal E_{2,2}^v$ and $\mathcal E_{4,2}^v$, let $\mathrm f_v$ be
the inverse of the corresponding one-to-one restriction
\[
 \mathrm f_{BC}:\mathcal E_{3,1}^v\to\mathcal E_{2,2}^v,
 \qquad
 \mathrm f_{AD}:\mathcal E_{1,1}^v\to\mathcal E_{4,2}^v.
\]
These inverse symbols refer only to the inverse of a bijection between
its domain and range, so no assertion is made that the inverse is obtained
by applying the same bond change to an arbitrary modified configuration.
The winding identities for the inverse maps follow from the same
identities for the forward branches.  Measure preservation follows from
the finite permutation of bond states in the forward construction.
\end{proof}

\begin{remark}
The proposition remains true with $\mathcal E_{v,0}$ in place of
$\mathcal E_v$.  The configurations with the state of $E_v$ reversed are
included because the modified parafermionic observable in Section
\ref{modi} uses the same exceptional set for two-edge and four-edge
visits.
\end{remark}

\begin{remark}
The bijection above is based on $\mathrm f_{AD}\cup\mathrm f_{BC}$.
Repeating the argument with $\mathrm f_{AB}\cup\mathrm f_{CD}$ gives a
second bijection with an exceptional event of the same order.  Reversing
the exploration path from $b_\de^\diamond$ to $a_\de^\diamond$ and
reflecting in $L_v^{BD}$ gives two further bijections.  In those
constructions the past of $\ga_\de$ is replaced by the past of the
reversed exploration path, and the attached loop is parameterized in
the reverse direction.
\end{remark}

\subsection{Rotation around center of medial edge} \label{partII}

Throughout this subsection, two medial paths are said to \emph{touch} if
they have a common medial vertex, and to \emph{cross} if they have a
common unoriented medial edge.  Thus every crossing is a touch, but a
touch need not be a crossing.  The phrase ``a crossing of an annulus''
keeps its usual meaning: a connected path joining the two boundary
components of the annulus.

We next compare the exploration path at two adjacent medial vertices.  The purpose of this subsection is to move a local visit from one endpoint of a medial edge to the other.  We select finite primal-open and dual-open path segments, rotate the intervening primal-bond states through $180$ degrees, and leave all other bonds unchanged.  The construction is made after the loop attached to the first vertex has been explored.  Its error estimate is therefore a \emph{fixed-loop} estimate, which depends on the radius of that loop.  This distinction will be important in Section \ref{partIII}, where a five-arm event is used to remove the conditioning.

Let $w$ and $v$ be adjacent medial vertices such that
\[
 \Re(v-w)=\de/2,\qquad \Im(v-w)=-\de/2.
\]
Thus the primal bond $E_w$ is horizontal, the primal bond $E_v$ is vertical, and the medial edge joining $w$ to $v$ is denoted both by $B_w$ and by $A$.  The four medial edges incident to $w$ are denoted by $A_w,B_w,C_w,D_w$ in clockwise order, with $A_w,C_w$ oriented towards $w$.  The four edges incident to $v$ are the edges $A,B,C,D$ fixed in Section \ref{partI}.  Put
\[
 c=\frac{v+w}{2},\qquad R(z)=2c-z.
\]
The map $R$ is the half-turn, also the reflection about the center of the common medial edge.  If $E$ is a primal bond, $E^R$ denotes the primal bond whose midpoint is $R(v_E)$.  If $\eta$ is a medial path, write
\[
 R(\eta)=\{R(z):z\in\eta\}.
\]
For the exploration path we use the original notation
\[
 \ga_w^R(j)=R(\ga_w(j)),\qquad
 \ga_w^R[i,j]=R(\ga_w[i,j]).
\]
Thus $\ga_w^R[i,j]$ is the half-turn of $\ga_w[i,j]$.

\begin{figure}[!htbp]
\begin{center}
\resizebox{0.38\textwidth}{!}{
\begin{tikzpicture}[
  x=1cm,y=1cm,
  >=Stealth,
  line cap=round,
  line join=round,
  every node/.style={font=\normalsize}
]
  \definecolor{ggbgray}{RGB}{110,109,115}
  \definecolor{ggbblue}{RGB}{21,101,192}
  \path[use as bounding box] (-3.45,-3.95) rectangle (3.45,4.15);

  \draw[ggbgray,line width=1.1pt] (-3,1)--(1,1);
  \draw[ggbgray,line width=1.1pt] (1,1)--(1,-3);

  \draw[ggbgray,line width=1.1pt,-{Stealth[length=2.4mm,width=1.65mm]}]
    (-1,1)--(-3,3);
  \draw[ggbgray,line width=1.1pt,-{Stealth[length=2.4mm,width=1.65mm]}]
    (-3,-1)--(-1,1);
  \draw[ggbgray,line width=1.1pt,-{Stealth[length=2.4mm,width=1.65mm]}]
    (1,3)--(-1,1);
  \draw[ggbgray,line width=1.1pt,-{Stealth[length=2.4mm,width=1.65mm]}]
    (-1,1)--(1,-1);
  \draw[ggbgray,line width=1.1pt,-{Stealth[length=2.4mm,width=1.65mm]}]
    (1,-1)--(3,1);
  \draw[ggbgray,line width=1.1pt,-{Stealth[length=2.4mm,width=1.65mm]}]
    (3,-3)--(1,-1);
  \draw[ggbgray,line width=1.1pt,-{Stealth[length=2.4mm,width=1.65mm]}]
    (1,-1)--(-1,-3);

  \fill[ggbblue] (-1,1) circle (2.25pt);
  \fill[ggbblue] (1,-1) circle (2.25pt);
  \fill[ggbblue] (0,0) circle (2.25pt);

  \node[anchor=west] at (-1.22,1.67) {$w$};
  \node[anchor=west] at (1.30,-0.75) {$v$};
  \node[anchor=west] at (-0.46,-0.01) {$c$};

  \node[anchor=west] at (-2.10,1.40) {$D_w$};
  \node[anchor=west] at (-2.10,0.78) {$C_w$};
  \node[anchor=west] at (-0.59,1.45) {$A_w$};
  \node[anchor=west] at (-0.60,0.78) {$B_w$};

  \node[anchor=west] at (0.49,-0.4) {$A$};
  \node[anchor=west] at (1.02,-0.4) {$B$};
  \node[anchor=west] at (0.49,-1.58) {$D$};
  \node[anchor=west] at (1.02,-1.58) {$C$};

  \node at (-1.00,3.62) {original visit $(A_w,D_w)$};
  \node at (0.95,-3.63) {new visit $(C,D)$};
\end{tikzpicture}}
\end{center}
\caption{The half-turn about $c$ changes the local visit $(A_w,D_w)$ at $w$ into $(C,D)$ at $v$.  The arrows show the fixed orientations of the medial edges.}
\label{edge-center-rotation}
\end{figure}

We first treat the local visit
\[
 A_w,D_w\in\ga_w,
 \qquad B_w,C_w\notin\ga_w.
\]
The new configuration will be denoted by $\omega_v$, and its exploration path by $\ga_v$.  We shall obtain
\[
 C,D\in\ga_v,
 \qquad A,B\notin\ga_v,
\]
and
\beqn
 W_{\ga_w}(A_w,e_b)&=&W_{\ga_v}(D,e_b), \label{edge-centre-wind-1}\\
 W_{\ga_w}(D_w,e_b)&=&W_{\ga_v}(C,e_b). \label{edge-centre-wind-2}
\eeqn
The other three local visits are obtained by rotating this construction.

Let
\[
 d_v=\dist\big(v,\p_{ba,\de}\cup\p_{ab,\de}^*\big),\qquad
 d_w=\dist\big(w,\p_{ba,\de}\cup\p_{ab,\de}^*\big),
 \qquad d_{v,w}=\min(d_v,d_w).
\]
Let $\Loop_w$ be the medial loop attached to $E_w$, and put
\[
 r_w=\max\{|z-c|:z\in\Loop_w\},
 \qquad n_w=\min\{j:\ga_w(j)=w\}.
\]
We parameterize $\Loop_w$ counterclockwise, starting at $w$, and denote the reverse parametrization by $\Loop_w^{-1}$.  In the remainder of this subsection we explore $\Loop_w$, the states of the primal bonds touched or crossed by it, and the portions of $\ga_w$ needed to determine the first exits from $B(c;r_w)$.  We then condition on this revealed information.  Thus $r_w$ and the local arrangement of the loop are fixed, while the bonds in the annuli outside $B(c;r_w)$ remain independent critical Bernoulli bonds.

\medskip
\noindent
{\bf Good annuli.}
We assume first that $r_w<d_{v,w}$.  Fix a large constant $\mathrm R_0>16$, to be chosen depending only on $\varepsilon$, and set
\[
 \rho_k=r_w\mathrm R_0^k,
 \qquad
 m_c=\max\{k:4\rho_k<d_{v,w}\}.
\]
For $1\leq k\leq m_c$, consider the annulus
\[
 \mathcal A_k=B(c;\rho_k/2)\setminus B(c;2\rho_{k-1}).
\]
The annuli $\mathcal A_k$ are disjoint.  Write
\[
 \ga_w^-=\ga_w[0,n_w],\qquad
 \ga_w^+=\ga_w[n_w,\infty).
\]
Let $r(\ga_w^+)$ be the reversed $\ga_w^+$ from $b_\de^\diamond$ to $w$.
We first define the crossing event
\[
 \mathcal H_{k,0}
=\left\{
 \begin{array}{l}
 \text{one of $\ga_w^-$ and $r(\ga_w^+)$ has exactly one crossing of
 $\mathcal A_k$ from outer to inner,}\\[-1mm]
 \text{and the other has at most two crossings of $\mathcal A_k$ from outer to inner}
 \end{array}\right\}.
\]
On $\mathcal H_{k,0}$, let
\[
 \ga_w[t_{l,k},\tau_{l,k}]
\]
be the last crossing of $\mathcal A_k$ made by $\ga_w^-$ as the past
approaches $w$, and let
\[
 \ga_w[t_{r,k},\tau_{r,k}]
\]
be the first crossing of $\mathcal A_k$ made by $\ga_w^+$ as the future
leaves $w$.  Thus
\[
 t_{l,k}<\tau_{l,k}<n_w<t_{r,k}<\tau_{r,k}.
\]
These are called the selected incoming and outgoing crossings.

On $\mathcal H_{k,0}$, define
\[
 \mathcal H_{k,1}
 =\left\{\ga_w^R[t_{r,k},\tau_{r,k}]
        \text{ crosses }
        \ga_w[t_{l,k},\tau_{l,k}]\right\}.
\]
Here and below the notation $\subseteq_c$ has the meaning fixed in
Section~\ref{partI}: the indicated unoriented medial edge is also an edge
of the other path.  On $\mathcal H_{k,1}$ define the first pair of crossing
times by
\beqn
 \mathrm{s}_{r,k}
 &=&\min\big\{j\in[t_{r,k},\tau_{r,k}]:
       \ga_w^R[j,j+1]\subseteq_c
       \ga_w[t_{l,k},\tau_{l,k}]\big\}, \label{rmsr}\\
 \mathrm{s}_{l,k}
 &=&\min\big\{j\in[t_{l,k},\tau_{l,k}]:
       \ga_w(j)=\ga_w^R(\mathrm{s}_{r,k})\big\}. \label{rmsl}
\eeqn
Define the second pair by
\beqn
 \mathfrak{s}_{l,k}
 &=&\max\big\{j\in[t_{l,k},\tau_{l,k}]:
       \ga_w^R[j-1,j]\subseteq_c
       \ga_w[t_{r,k},\tau_{r,k}]\big\}, \label{fraksl}\\
 \mathfrak{s}_{r,k}
 &=&\max\big\{j\in[t_{r,k},\tau_{r,k}]:
       \ga_w(j)=\ga_w^R(\mathfrak{s}_{l,k})\big\}. \label{fraksr}
\eeqn
The vertices used to mark the two prescribed common edges are
\[
 x_k^1=\ga_w(\mathrm{s}_{l,k})
      =\ga_w^R(\mathrm{s}_{r,k}),
 \qquad
 x_k^2=\ga_w(\mathfrak{s}_{r,k})
      =\ga_w^R(\mathfrak{s}_{l,k}).
\]

We next define two local events using these crossing times.  Put
\[
 \mathcal H_{k,2}
 =\left\{
 \begin{array}{l}
 \ga_w[t_{l,k},\mathrm{s}_{l,k}]
   \cap\ga_w[t_{r,k},\mathrm{s}_{r,k}]=\emptyset,\\[1mm]
 \ga_w[\mathfrak{s}_{r,k},\tau_{r,k})
   \cap\ga_w[\mathfrak{s}_{l,k},\tau_{l,k}]=\emptyset
 \end{array}\right\}.
\]
Thus the two selected crossings do not touch before the first prescribed
common edge or after the second prescribed common edge.

On $\bigcap_{j=0}^2\mathcal H_{k,j}$ define the reverse crossing times
\beqn
 \mathrm{s}_{l,k,2}
 &=&\min\big\{j\in[0,\mathrm{s}_{l,k}]:
       \ga_w[j,j+1]\subseteq_c\ga_w^R(n_w,\mathrm{s}_{r,k}]\big\}, \label{rmsl2}\\
  \mathrm{s}_{r,k,2}
 &=&\min\big\{j\in[n_w,\mathrm{s}_{r,k}]:
       \ga_w^R(j)=
       \ga_w(\mathrm{s}_{l,k,2})\big\}, \label{rmsr2}\\
 \mathfrak{s}_{r,k,2}
 &=&\max\big\{j\in[\mathfrak{s}_{r,k},\infty):
       \ga_w[j-1,j]\subseteq_c\ga_w^R[\mathfrak{s}_{l,k},n_w)\big\}, \label{fraksr2}\\
        \mathfrak{s}_{l,k,2}
 &=&\max\big\{j\in[\mathfrak{s}_{l,k},n_w]:
       \ga_w^R(j)=
       \ga_w(\mathfrak{s}_{r,k,2})\big\}. \label{fraksl2}
\eeqn
By the definition of $\mathrm{s}_{l,k}$, $\mathrm{s}_{r,k}$, $\mathfrak{s}_{l,k}$ and $\mathfrak{s}_{r,k}$, we can regard these sets as non-empty.
Finally, put
\[
 \mathcal H_{k,3}
 =\left\{
 \begin{array}{l}
 \ga_w[0,\mathrm{s}_{l,k,2}]
   \cap\ga_w[n_w,\mathrm{s}_{r,k,2}]=\emptyset, \ (\ga_w[0,\mathrm{s}_{l,k,2}]\cup \ga_w[\mathrm{s}_{r,k,2},\infty))
   \cap (\Loop_w\cup\Loop_w^R)=\emptyset \\[1mm]
 \ga_w[\mathfrak{s}_{r,k,2},\infty)
   \cap\ga_w[\mathfrak{s}_{l,k,2},n_w]=\emptyset, \ (\ga_w[0,\mathfrak{s}_{l,k,2}]\cup\ga_w[\mathfrak{s}_{r,k,2},\infty))
   \cap(\Loop_w\cup\Loop_w^R)=\emptyset
 \end{array}\right\}.
\]
This is the separation condition for the part of the path which is used in
the half-turn.

\medskip
\noindent
{\bf Probability of finding a good annulus.}
For the local estimates in the $k$th annulus, fix the previously
revealed bond states in $B(c;\rho_{k-1})$. The other unprescribed
bonds remain independent critical Bernoulli bonds. In the proof of
Lemma~\ref{Hk2}, we additionally condition on the bonds in
$B(c;5\rho_k/8)$ only while keeping the selected paths and their
primal-open and dual-open connections fixed during arm separation
and joining. These additional bond states are then averaged out.
No bonds in the subsequent annuli are conditioned on.

The complement of $\mathcal H_{k,0}$ produces six alternating arms across a subannulus of ratio comparable to $\mathrm R_0$.  Hence, by \eqref{six},
\beqn
 \Pro(\mathcal H_{k,0}^c\mid\hbox{the previously revealed bonds})
 \preceq \mathrm R_0^{-2-\al}. \label{Hk0}
\eeqn
Moreover, Lemma \ref{rotation-event} gives
\beqn
 \Pro(\mathcal H_{k,1}^c\cap\mathcal H_{k,0}
       \mid\hbox{the previously revealed bonds})
 \preceq \mathrm R_0^{-2}. \label{Hk1}
\eeqn
The remaining local geometric estimate is presented next.

\begin{lemma}\label{Hk2}
For every $1\le k\le m_c$ and every fixed configuration
$\eta$ of the previously revealed original bonds in
$B(c;\rho_{k-1})$, we have
\begin{equation}\label{Hk2-logarithmic-target}
 \Pro\bigl(\mathcal H_{k,2}^{c}\cap\mathcal H_{k,0}
                 \cap\mathcal H_{k,1}\mid\eta\bigr)
 \le C(1+\log \mathrm R_0)\mathrm R_0^{-2}.
\end{equation}
The constant $C$ is independent of the prescribed configuration,
the lattice mesh and $k$.
\end{lemma}

\noindent\textit{Proof strategy.}
We first distinguish the two ways in which the selected incoming and
outgoing paths can touch.  We group the configurations according to how
far the outgoing path travels before its first common edge with the
half-turn of the incoming path.  The inner estimate follows from the
twisted-arm bound in Lemma \ref{rotation-event}..  At the contact, the complete return of the exploration
supplies the open and dual-open connections between the two paths.  These
connections are kept while arm separation, the RSW method and the conditional FKG inequality give
separated outer landing intervals.  The usual comparison of two disks
gives comparable probabilities at different reference squares.
The first-common-edge count then bounds their total number, and we
sum over the distance scales.

\

\noindent
{\bf Proof of Lemma \ref{Hk2}}.

Put $r=\rho_{k-1}$, $R_1=\rho_k$, and write
$\mu=\Pro(\,\cdot\mid\eta)$.
Only bonds in the current annulus outside the previously revealed
disk are newly examined.  The auxiliary reference configuration does
not examine any later original annulus.  The lattice mesh is always $\delta$.
Changes of a circular radius by $C\delta$ are understood to
include whole lattice bonds.  Constants can be enlarged to cover bounded
values of $R_1/r$.

\par\medskip\noindent
\textit{The touching event and the selected paths.}
With the crossing times defined above, put
\[
 P=\ga_w[t_{l,k},\tau_{l,k}],\qquad Q=\ga_w[t_{r,k},\tau_{r,k}],\qquad
 t_{l,k}<\tau_{l,k}<n_w<t_{r,k}<\tau_{r,k}.
\]
These are the selected incoming and outgoing crossings of
$B(c;2r,R_1/2)$.  For a medial path
$\xi$, let $\mathcal E(\xi)$ denote its set of unoriented edges.
Using the marked times in \eqref{rmsr}--\eqref{fraksr},
the two events to be estimated are
\[
\begin{aligned}
 S_1={}&\mathcal H_{k,0}\cap\mathcal H_{k,1}\cap
  \{\ga_w[t_{l,k},\mathrm{s}_{l,k}]\cap\ga_w[t_{r,k},\mathrm{s}_{r,k}]\ne\varnothing\},\\
 S_2={}&\mathcal H_{k,0}\cap\mathcal H_{k,1}\cap
  \{\ga_w[\mathfrak{s}_{r,k},\tau_{r,k})
                    \cap\ga_w[\mathfrak{s}_{l,k},\tau_{l,k}]\ne\varnothing\}.
\end{aligned}
\]
Their union is $\mathcal H_{k,2}^{c}\cap\mathcal H_{k,0}
\cap\mathcal H_{k,1}$.

On $S_1$, choose the last vertex of $\ga_w[t_{l,k},\mathrm{s}_{l,k}]$ which is also
visited by $\ga_w[t_{r,k},\mathrm{s}_{r,k}]$, and write
\[
 x=\ga_w(i_x)=\ga_w(j_x),\qquad
 t_{l,k}\le i_x\le \mathrm{s}_{l,k}<n_w<j_x\le \mathrm{s}_{r,k}.
\]
The whole excursion $\Lambda_x=\ga_w[i_x,j_x]$ contains the visit to
$w$.  Put
\begin{equation}\label{Hk2-sc-first}
 g=\ga_w[\mathrm{s}_{r,k},\mathrm{s}_{r,k}+1],\qquad Q^-=\ga_w[t_{r,k},\mathrm{s}_{r,k}),\qquad
 \mathcal E(Q^-)\cap\mathcal E(RP)=\varnothing.
\end{equation}
The half-turned edge $Rg$ belongs to $P$.  The order of the marked visits is
$x,Rg,w,x,g$.  Throughout the construction keep
\[
 \mathcal K_x=\Lambda_x\cup\ga_w[j_x,\mathrm{s}_{r,k}+1]
                     =\ga_w[i_x,\mathrm{s}_{r,k}+1],
\]
together with both percolation sides of $P,Q$ and all their half-turns.
The open and dual-open connections along $\mathcal K_x$ are kept.

The definitions of the last incoming and first outgoing crossings imply
\begin{equation}\label{Hk2-sc-confinement}
 \ga_w[t_{l,k},\tau_{r,k}]\subset B(c;R_1/2+C\delta).
\end{equation}
Thus $\mathcal K_x$ and $R\mathcal K_x$, including their adjacent bonds, lie in
\[
                         D=B(c;5R_1/8).
\]
The excursion $\Lambda_x$ may enter $B(c;r)$ and is not required to
stay in the annulus $B(c;r,R_1)$.  All bonds it uses there are already
revealed by $\eta$.  We keep those portions together with the rest
of the excursion.  No new connection inside $B(c;r)$ is requested.

Here and below, keeping $\mathcal K_x$ fixed during arm separation means
conditioning on the original bond configuration in the deterministic
disk $D$.  Write $\mathscr D$ for the sigma-field generated by $\eta$
and the states of the original bonds in $D$, with the whole-bond
convention already specified.  This conditioning fixes the entire
incoming path in $D$, both percolation sides of these paths,
and their half-turns, as well as $\mathcal K_x$.  The additional random
crossings use bonds outside $D$ and retain their critical product law
under this conditioning.  After the conditional comparison we integrate
over the states in $D$ which were not prescribed by $\eta$.  Thus the
extra conditioning is used only for arm separation and joining.  It is
not a requirement that different reference squares contain the same
realized loop.

\par\medskip\noindent
\textit{Dyadic scales and the inner estimate.}
Let $c_b$ be a fixed constant large enough to include the bonds
adjacent to the marked edges, and set
\[
 d_s=\max_{t_{r,k}\le j\le \mathrm{s}_{r,k}}|\ga_w(j)-c|,
 \qquad \overline d_s=\max\{16r,d_s+c_b\delta\}.
\]
For each dyadic $a=16r\,2^m$, $m\ge0$, put
\begin{equation}\label{Hk2-sc-Sa}
                      S_a=S_1\cap\{a\le\overline d_s<2a\}.
\end{equation}
We partition by $a$ but do not
condition the product law on the random value of $d_s$.
On $S_a$, $x,Rx,g,Rg$ and their adjacent bonds lie
in $B(c;2a)$.  The entire loop $\Lambda_x$ need not lie in this smaller disk.
Its confinement is \eqref{Hk2-sc-confinement}.

If $a/r$ is larger than a fixed constant, the portion $Q^-$ reaches
radius at least $a/2$.
Let
\[
 F_a=\mathcal A_{010}^{\rm tw}(8r,a/8)
                 \cup\mathcal A_{101}^{\rm tw}(8r,a/8).
\]
The notation $\mathcal A_{101}^{\rm tw}(u,v)$ denotes the ordered
open--dual-open--open arm event in the folded annulus of Lemma~\ref{rotation-event}.
The order $010$ denotes the colour-reversed event.
Lemma~\ref{rotation-event} gives
\begin{equation}\label{Hk2-sc-inner}
       S_a\subseteq F_a, \qquad           \mu(S_a)\preceq (r/a)^2.
\end{equation}
For bounded $a/r$, take $F_a$ to be the whole space and enlarge the implicit constant.
For the nontrivial values of $a/r$, this event uses only original
bonds between radii $8r$ and $a/8$, with an adjustment of the radii
by $C\delta$.  It is independent, under $\mu$, of the unrevealed
original bonds outside $B(c;a/4)$.

\par\medskip\noindent
\textit{The inner connections supplied by $x$.}
The contact $x$ is not merely a point used to locate a reference square.
Its two visits, $x=\ga_w(i_x)=\ga_w(j_x)$ with
$i_x<n_w<j_x$, give the complete excursion
$\Lambda_x=\ga_w[i_x,j_x]$ from $x$ back to $x$ through $w$.
Together with the continuation from the second visit to $x$ up to $g$,
this gives the path $\mathcal K_x=\ga_w[i_x,\mathrm{s}_{r,k}+1]$.
The open and dual-open paths along this excursion provide the inner
connections between the sides leading to $g$ and $Rg$.
Following $P$ and $Q$ to the outer circle provides their outer
continuations, but does not replace these inner connections.
We now specify where the contact and the complete excursion are used.

Copy $0$ denotes the original configuration, and copy $1$ its
half-turn.  If $e$ is a primal bond with midpoint
$v_e$, let $e^R$ be the primal bond with midpoint $R(v_e)$.  The map $R$
sends primal bonds to primal bonds.  On the two copies set
\[
 \omega_0(e)=\omega(e),\qquad
 \omega_1(e^R)=1-\omega(e).
\]
Thus the half-turn of an original open path is dual-open in copy $1$,
and conversely.  The two copies use the same original bond variables,
including the states prescribed by $\eta$.  At $g,Rg$ retain
the original matching of the corresponding open and dual-open sides.
Other coincident geometric positions in different copies are not joined.

\smallskip\noindent\textit{The contact joins the inner paths.}
Suppose the primal bond $E_x$ with midpoint $x$ is open.
For a medial edge $h$, write $p_h,d_h$ for its adjacent primal and
dual vertices.  Follow the open sides from $p_{Rg}$ and
$p_g$ to the two endpoints of $E_x$.  Joining them through $E_x$
gives an open connection
\[
                         U:p_{Rg}\longrightarrow p_g
\]
which uses $E_x$ once.  Let $d_+,d_-$ be the endpoints of $E_x^*$.
The complete excursion $\Lambda_x$ gives a dual-open return
\[
                              V_x:d_+\longrightarrow d_-.
\]
Since $E_x$ is open, this return does not use $E_x^*$: it is the
connection along the dual-open side of the kept excursion.
The corresponding sides from the contact to the common edges also give
\[
 D_l:d_+\longrightarrow d_{Rg},\qquad
 D_r:d_-\longrightarrow d_g.
\]
Thus $D_l$ followed backwards, then $V_x$, then $D_r$, connects
$d_{Rg}$ to $d_g$ by dual-open bonds.  The open connection
$U$ and this dual-open connection are both present in the original
configuration.  In particular, the return $V_x$ is supplied by the
whole $x$-to-$x$ excursion, not by the existence of the three outer
arms alone.

At the common edges the corresponding vertices are joined by
\[
\begin{array}{ll}
 (p_g)_0\sim(R d_{Rg})_1,
 & (d_g)_0\sim(R p_{Rg})_1,\\
 (p_{Rg})_0\sim(R d_g)_1,
 & (d_{Rg})_0\sim(R p_g)_1.
\end{array}
\]
The half-turn of the complete construction supplies the corresponding
inner connections at $Rx$ from the same original bond states.

\smallskip\noindent\textit{The existing outer arms retain these connections.}
Follow the incoming exploration backwards from $Rg$ and the
outgoing exploration forwards from $g$.  Their selected crossings
reach $\partial B(c;R_1/2)$, and their kept outer continuations
reach $\partial B(c;R_1)$.  Keep both open and dual-open sides.
Let $O_l^{\rm o},O_l^{\rm d}$ denote the resulting open and dual-open
outer connections starting at $Rg$, and let
$O_r^{\rm o},O_r^{\rm d}$ denote those starting at $g$.
With the right to left convention of path products, the kept connections
include
\begin{equation}\label{Hk2-sc-outer-connections}
\begin{aligned}
 A&=(RO_l^{\rm d})_1*U_0*
                       (RO_r^{\rm d})_1^{-1},\\
 M_-&=(RO_l^{\rm o})_1*(D_r)_0,
       &M_+&=M_-*(V_x)_0.
\end{aligned}
\end{equation}
Here an inverse reverses a path.  The open connection $A$ uses $U$,
and hence the contact bond $E_x$, to join the two open outer
connections.  The dual-open connection $M_-$ starts at $d_-$.
The path $M_+$ starts at $d_+$ and first follows $V_x$ before continuing
along $M_-$.  Both therefore reach the same middle outer connection.
This is the specific use of the complete return in keeping the middle
connection attached to either of the two dual vertices at $x$.
The other side $RO_r^{\rm o}$ and all half-turned connections
are kept too.  Outside $D$, the portions displayed in this formula
lie in copy $1$.  Their open and dual-open crossing requirements can
therefore all be pulled back to the original lattice by the same
half-turn.  For closed $E_x$, interchange open and dual-open throughout.

The three paths used for the outer connections are
are
\begin{equation}\label{Hk2-sc-three}
       \alpha_Q,\qquad R\alpha_{P,2}, \qquad
                          R\alpha_{P,3}.
\end{equation}
Here $\alpha_Q$ is the chosen side from $g$ along the outgoing path,
and $\alpha_{P,2},\alpha_{P,3}$ are the two sides followed backwards
from $Rg$ along the incoming path, each continued to the
outer circle $\partial B(c;R_1)$.  These paths are used together with the inner
connections just described.

\par\medskip\noindent
\textit{Separation of the outer endpoints and conditional FKG.}
Fix a dyadic $a<R_1/2^{12}$.  We
condition on the bond states in $D$ as above and use
\[
 \mathcal W_1=B(c;5R_1/8,13R_1/16),\qquad
 \mathcal W_2=B(c;3R_1/4,15R_1/16).
\]
Both annuli are outside $D$.  Their overlap is
$B(c;3R_1/4,13R_1/16)$.  All widths and radii in this construction
are fixed positive multiples of $R_1$.

\smallskip\noindent\textit{The finite landing choices.}
Fix $\varepsilon>0$ sufficiently small and a fixed finite family
of triples of disjoint arcs on the outer circle, in the
required cyclic order.  Each arc has length at least
$\varepsilon R_1$, and different arcs have distance at least
$\varepsilon R_1$.  The outer circle may first be taken at
$L=15R_1/16$.  A fixed-ratio RSW extension continues the paths to
$\partial B(c;R_1)$ when that convention is used.
We use a single index
\[
                            j\in\{1,\ldots,M\}
\]
to specify the colour order, the three outer arcs, the separated
intermediate arcs, and the fixed rectangles joining them.  The number
$M$ is a fixed constant, independent of $a,\eta$ and $\delta$.
It includes the colour-reversed case.  For a fixed $j$, omit $j$ from
the three arc names and write them as $J_1,J_2,J_3$.

Apply the usual relative arm-separation argument to the three existing
connections supplied by the contact, conditional on the original states in $D$.
When the two open paths remain in different components up to the next
circle, use the three-arm separation argument.  When they
have joined, keep both inner attachments to their open component
and the intervening dual-open connection.  Use two-arm separation for
that component and the dual-open path, followed by the RSW crossings that
connect the open component to two outer intervals on the two sides of
the middle interval.  Those extra crossings lie on the side not occupied
by the middle path.  No connection along the original common portion
is removed.  The dual-open case is identical after interchanging colours.

Let $T_j$ denote the resulting separated event for choice $j$.
It requires the inner connections supplied by $S_a$,  the event $F_a$,
and the three connections with the separated intermediate crossings and
outer arcs specified by $j$.  The configuration in $D$ includes
$x,g,\Lambda_x,\mathcal K_x$, the incoming portion in $D$, and all their
half-turns.  For each admissible value
of $\mathscr D$, the relative separation argument gives
\[
 \mu\left(\bigcup_{j=1}^{M}T_j\,\middle|\,\mathscr D\right)
       \ge c_{\rm sep}\,\mu(S_a\mid\mathscr D).
\]
The constant is uniform in these conditioned states, including the
shape of the kept excursion and its connections through
$B(c;r)$.  Integrating over the additional states in $D$, with
$\eta$ still prescribed, gives
\begin{equation}\label{Hk2-sc-terminal-separation}
             \mu\left(\bigcup_{j=1}^{M}T_j\right)
                         \ge c_{\rm sep}\,\mu(S_a),
             \qquad T_j\subset S_a.
\end{equation}
This is the relative arm-separation estimate with the inner connections
kept.

\smallskip\noindent\textit{The connections across the two annuli.}
Fix $j$.  For $i=1,2,3$, let $\Gamma_i^-,\Gamma_i^0,\Gamma_i^+$
be connected open or dual-open crossings in the specified rectangles
near the circles of radii
\[
        11R_1/16,\qquad25R_1/32,\qquad7R_1/8,
\]
respectively.  Since
\[
\frac{5}{8}<\frac{11}{16}<\frac{25}{32}<\frac{13}{16},
\qquad
\frac34<\frac{25}{32}<\frac78<\frac{15}{16},
\]
the crossings $\Gamma_i^-$ and $\Gamma_i^0$ can be joined inside
$\mathcal W_1$, and $\Gamma_i^0$ and $\Gamma_i^+$ can be joined inside
$\mathcal W_2$.  The crossing $\Gamma_i^0$ lies in $\mathcal W_1\cap\mathcal W_2$ and is
\emph{the same connected crossing} in both joins.  In particular, the
two joins are not treated as independent events.
Let $A_i$ be the inner endpoint set of the $i$th connection in
\eqref{Hk2-sc-three}, interpreted on its prescribed copy.  We require
\begin{equation}\label{Hk2-sc-joins}
 A_i\longleftrightarrow\Gamma_i^-
    \longleftrightarrow\Gamma_i^0
    \longleftrightarrow\Gamma_i^+
    \longleftrightarrow J_i,\qquad i=1,2,3.
\end{equation}
The rectangle crossings are oriented so that the connecting crossings
intersect $\Gamma_i^-$ and $\Gamma_i^0$ inside $\mathcal W_1$, and intersect
$\Gamma_i^0$ and $\Gamma_i^+$ inside $\mathcal W_2$.  Hence each join is
between connected paths of the required colour.  Only a bounded number
of rectangles of bounded aspect ratio is used, and in copy $1$ every
crossing is expressed in terms of the corresponding original bonds.

\smallskip\noindent\textit{Definitions for the FKG calculation.}
Let $\mathcal U_j^+$ be the set of original primal bonds that the
rectangle crossings require to be open, and let $\mathcal U_j^-$ be
the set they require to be closed.  These sets are disjoint and are
chosen before their states are examined.  In copy $0$, a primal-open
path uses original open bonds and a dual-open path uses dual edges
crossing original closed bonds.  In copy $1$, the convention
$\omega_1(e^R)=1-\omega(e)$ reverses these requirements: a
primal-open bond $e^R$ corresponds to a closed original bond $e$,
whereas a dual-open edge crossing $e^R$ corresponds to an open original
bond $e$.  Thus $\mathcal U_j^+$ and $\mathcal U_j^-$ are defined by
the required original bond states, regardless of the copy in which a
crossing is drawn.  Where the two joins use the same crossing in
$\mathcal W_1\cap\mathcal W_2$, its original bonds and required states are counted only
once.  The chosen rectangles impose no opposite requirements on any
original bond.
Let $\mathscr F_j$ be the sigma-field generated by $\eta$ and every
other original bond in $B(c;R_1)$.  In particular,
$\mathscr F_j$ determines the configuration in $D$ and $F_a$.

Define $B_j$ as follows.  There must exist one compatible choice of
kept inner paths and of the connected crossings
$\Gamma_i^-,\Gamma_i^+$ such that the two end connections
\[
           A_i\longleftrightarrow\Gamma_i^- ,\qquad
           \Gamma_i^+\longleftrightarrow J_i\qquad(i=1,2,3)
\]
are present, with the colours and identifications prescribed by $j$.
Keep $F_a$, the same complete excursion $\Lambda_x$, and the same
inner endpoint matching in this choice.  The two middle connections
in \eqref{Hk2-sc-joins} are not yet required.  Outside $\mathcal U_j^+\cup\mathcal U_j^-$, all required bond
states belong to $\mathscr F_j$.  Within $\mathcal U_j^+$ the
conditions require the original bonds to be open, while within
$\mathcal U_j^-$ they require them to be closed, including when
the corresponding path is drawn in copy $1$.
This definition is a union over all compatible choices of the kept
paths, not over independent choices for their open and dual-open parts.
In particular,
\[
                             T_j\subseteq B_j\subseteq F_a.
\]
The dependence of $T_j,B_j,\mathscr F_j$ on the fixed scale $a$ is
suppressed in this calculation.

Let $G_j$ be the event of the additional rectangle crossings in both
joining operations, including the connected crossings
$\Gamma_i^0$.  The chosen crossing directions ensure that, on
$B_j\cap G_j$, the end connections are joined as in
\eqref{Hk2-sc-joins} with their original matching.
Conditional on $\mathscr F_j$, use the variables
\[
 Z_e=\begin{cases}
       \omega(e),&e\in\mathcal U_j^+,\\
       1-\omega(e),&e\in\mathcal U_j^-.
     \end{cases}
\]
They are independent Bernoulli$(1/2)$ variables.
The event $B_j$ is increasing in them: changing a variable in the
favourable direction preserves whichever complete, compatible set of
paths already witnesses $B_j$.  The event $G_j$ is increasing too.
The RSW method and positive association for crossings of the same colour give
\[
                      \mu(G_j\mid\mathscr F_j)\ge c_{\rm join}>0.
\]
The conditional FKG inequality therefore gives
\begin{equation}\label{Hk2-sc-fkg}
\begin{aligned}
 \mu(B_j\cap G_j\mid\mathscr F_j)
 &\ge\mu(B_j\mid\mathscr F_j)\mu(G_j\mid\mathscr F_j)\\
 &\ge c_{\rm join}\,\mu(T_j\mid\mathscr F_j).
\end{aligned}
\end{equation}

Let \(E_{a,j}=B_j\cap G_j\) and \(E_a=\bigcup_{j=1}^M E_{a,j}\).
Integrating \eqref{Hk2-sc-fkg} and using
\eqref{Hk2-sc-terminal-separation} yields
\begin{equation}\label{Hk2-sc-completion}
             E_a\subset F_a,\qquad
             \mu(E_a)\ge\kappa\mu(S_a),\qquad
             \kappa=\frac{c_{\rm sep}c_{\rm join}}{M}>0.
\end{equation}
Indeed, the probability of a union of $M$ events is at least the sum
of their probabilities divided by $M$. So,
\[
\mu(E_a)
\ge \frac1M\sum_{j=1}^M\mu(E_{a,j})
\ge \frac{c_{\rm join}}M\sum_{j=1}^M\mu(T_j)
\ge \frac{c_{\rm sep}c_{\rm join}}M\mu(S_a).
\]

\smallskip\noindent\textit{The local join keeps the middle connection.}
The open joining crossing can be chosen along the boundary of the
cluster containing the middle dual-open connection. Consider the
quadrilateral in Figure~\ref{Hk2-local-join-figure}, with sides
$I_1,I_2,I_3,I_4$ in cyclic order. The sides $I_1,I_3$ meet the
open connections to $J_1,J_3$, respectively. Choose $I_2$ along
the prescribed connected dual-open crossing, whose continuation
reaches $J_2$. Let $\mathcal D_2$ be the dual-open cluster attached
to this crossing within the quadrilateral. The open
$I_1$--$I_3$ crossing supplied by the joining event separates
$\mathcal D_2$ from $I_4$.

The primal boundary of $\mathcal D_2$ facing $I_4$ contains an
$I_1$--$I_3$ crossing. Every bond on this boundary is open.
Otherwise its dual bond would connect $\mathcal D_2$ to a dual
vertex outside that cluster. Choose a simple crossing $\beta$
on this boundary and an interior bond $f=[p_1,p_3]$ of $\beta$.
For $i=1,3$, let $\beta_i$ be the portion of $\beta$, not using
$f$, from $p_i$ to $I_i$. Let $d_f\in\mathcal D_2$ be the dual
vertex adjacent to $f$ on the $I_2$ side. The two open paths
$\beta_1,\beta_3$ continue to $J_1,J_3$ along the existing open
connections. The vertex $d_f$ is joined to $I_2$ within
$\mathcal D_2$ and hence has a dual-open connection to $J_2$.
Thus the two open portions at $f$ and the middle dual-open
connection exist simultaneously. The middle path starts at
$d_f$ and does not cross the open bond $f$. Choosing $\beta$
and $f$ changes no further bond state. Interchange open and
dual-open for the other case.

\begin{figure}[!htbp]
\centering
\resizebox{\linewidth}{!}{
\begin{tikzpicture}[
  x=1cm,y=1cm,
  line cap=round,line join=round,
  every node/.style={font=\small},
  lj open/.style={black,line width=1.15pt},
  lj dual/.style={black,line width=.85pt,dash pattern=on 3pt off 2pt},
  lj boundary/.style={black!65,line width=.55pt},
  lj zoom/.style={black!55,line width=.55pt,densely dotted},
  lj label/.style={fill=white,inner sep=1.2pt},
  lj primal/.style={circle,fill=black,inner sep=0pt,minimum size=3.5pt},
  lj dual vertex/.style={circle,draw=black,fill=white,inner sep=0pt,minimum size=4pt}
]

\begin{scope}
  \node at (3,6.05) {\textbf{(a)} The boundary crossing};
  \path[fill=black!9]
    (0,1.8)--(.6,1.8)--(.6,2.4)--(1.2,2.4)--(1.2,1.8)
    --(2.4,1.8)--(2.4,2.4)--(3.6,2.4)--(3.6,1.8)
    --(4.2,1.8)--(4.2,1.2)--(4.8,1.2)--(4.8,1.8)
    --(6,1.8)--(6,4.5)--(0,4.5)--cycle;
  \draw[lj boundary] (0,4.5)--(0,0)--(6,0)--(6,4.5);
  \draw[lj dual] (0,4.5)--(6,4.5);
  \node[left] at (0,.85) {$I_1$};
  \node[above] at (1.2,4.5) {$I_2$};
  \node[right] at (6,.85) {$I_3$};
  \node[below] at (3,0) {$I_4$};

  \draw[lj dual] (2.7,2.7)--(2.7,3.3)--(3.3,3.3)
       --(3.3,3.9)--(2.7,3.9)--(2.7,5.4);
  \draw[lj dual] (2.7,3.3)--(.3,3.3);
  \draw[lj dual] (.9,2.7)--(.9,3.9)--(1.5,3.9)--(1.5,4.5);
  \draw[lj dual] (3.3,3.3)--(4.5,3.3)--(4.5,2.7)
       --(5.1,2.7)--(5.1,2.1);
  \draw[lj dual] (5.1,2.7)--(5.7,2.7)--(5.7,3.9);
  \node at (4.65,3.95) {$\mathcal D_2$};
  \draw[lj dual,line width=1.2pt] (2.43,5.4)--(2.97,5.4);
  \node[above] at (2.7,5.4) {$J_2$};

  \draw[lj open]
    (0,1.8)--(.6,1.8)--(.6,2.4)--(1.2,2.4)--(1.2,1.8)
    --(2.4,1.8)--(2.4,2.4)--(3.6,2.4)--(3.6,1.8)
    --(4.2,1.8)--(4.2,1.2)--(4.8,1.2)--(4.8,1.8)--(6,1.8);
  \draw[lj open] (-.9,1.8)--(0,1.8);
  \draw[lj open] (6,1.8)--(6.9,1.8);
  \draw[lj open] (-.9,1.55)--(-.9,2.05);
  \draw[lj open] (6.9,1.55)--(6.9,2.05);
  \node[above] at (-.9,2.05) {$J_1$};
  \node[above] at (6.9,2.05) {$J_3$};
  \node at (1.65,1.35) {$\beta$};
  \draw[lj zoom] (2.16,2.05) rectangle (3.22,3.03);
  \draw[black,line width=2.6pt] (2.4,2.4)--(3,2.4);
  \node[lj primal] at (2.4,2.4) {};
  \node[lj primal] at (3,2.4) {};
  \node[lj dual vertex] at (2.7,2.7) {};
  \node[below] at (2.7,2.4) {$f$};
  \node[right,inner sep=1.5pt] at (2.75,2.77) {$d_f$};
\end{scope}

\begin{scope}[shift={(9,0)}]
  \node at (2.8,6.05) {\textbf{(b)} Enlargement at $f$};
  \path[fill=black!9]
      (0,1.4)--(1.2,1.4)--(1.2,2)--(4.4,2)--(4.4,1.4)
      --(5.6,1.4)--(5.6,4.9)--(0,4.9)--cycle;
  \draw[lj zoom] (0,.5) rectangle (5.6,4.9);
  \draw[lj open] (0,1.4)--(1.2,1.4)--(1.2,2)--(1.8,2);
  \draw[black,line width=2.6pt] (1.8,2)--(3.8,2);
  \draw[lj open] (3.8,2)--(4.4,2)--(4.4,1.4)--(5.6,1.4);
  \node[lj primal] at (1.8,2) {};
  \node[lj primal] at (3.8,2) {};
  \node[below=3pt] at (1.8,2) {$p_1$};
  \node[below=3pt] at (3.8,2) {$p_3$};
  \node[below=3pt] at (2.8,2) {$f$};
  \node at (.6,1.02) {$\beta_1$};
  \node at (5,1.02) {$\beta_3$};
  \node[above,align=center] at (.55,1.5) {to $I_1,J_1$};
  \node[above,align=center] at (5.05,1.5) {to $I_3,J_3$};

  \draw[lj dual] (2.8,3)--(2.8,3.6)--(3.4,3.6)
                --(3.4,4.2)--(2.8,4.2)--(2.8,5.35);
  \node[lj dual vertex] at (2.8,3) {};
  \node[left=3pt] at (2.8,3) {$d_f$};
  \node at (4.5,4.15) {$\mathcal D_2$};
  \node[above] at (2.8,5.35) {to $I_2,J_2$};
\end{scope}
\end{tikzpicture}
}
\caption{Choosing the local open join while keeping the middle dual-open
connection. Solid paths are primal-open and dashed paths are dual-open.
(a) The shading indicates schematically the dual-open cluster
$\mathcal D_2$ attached to $I_2$, with its continuation to $J_2$.
The crossing $\beta$ follows its boundary facing $I_4$ and joins the
open connections to $J_1,J_3$.
(b) At $f=[p_1,p_3]$, the open paths $\beta_1,\beta_3$ start at
$p_1,p_3$, whereas the middle dual-open path starts at the adjacent
dual vertex $d_f$. It does not cross $f$.
Both panels show the same configuration.}
\label{Hk2-local-join-figure}
\end{figure}

\par\medskip\noindent
\textit{Reference squares and the first common edge.}
Fix $a<R_1/2^{12}$.  To estimate the probability
of $S_a$, we compare the corresponding events at
different reference squares and apply the counting
argument. For convenience, assume that $c=0$. Set
\[
 L=15R_1/16,\qquad R_*=128L=120R_1,
 \qquad
 V=\left[\frac{3R_*}{8},\frac{7R_*}{16}\right]
       \times\left[-\frac{R_*}{64},\frac{R_*}{64}\right].
\]
The auxiliary lattice has the same mesh $\delta$ as the original
lattice.  An admissible centre has the same position relative to the
lattice as $c$, so translation between two such centres preserves the
primal, dual and medial lattices.  Put
\[
 N(q)=q+[-2a,2a]^2,\qquad
 V_a=\{q\in V:\operatorname{dist}_\infty(q,\partial V)\ge2a\}.
\]
Choose a maximal collection $\mathcal Q_a$ of admissible centres in
$V_a$ satisfying $\|q-q'\|_\infty\ge16a$ for distinct centres.
A fixed ordering of the admissible centres makes this choice
deterministic.  The squares $N(q)$ are contained in $V$ and are
pairwise disjoint.  The number $16a$ describes the separation of the
selected centres, not a second percolation mesh.

The side lengths of $V_a$ are
$R_*/16-4a$ and $R_*/32-4a$, which are at least $R_*/32$ and
$R_*/64$ in the stated range.  Hence $V_a$ contains at least
$c(R_*/\delta)^2$ admissible centres.  By maximality, each admissible
centre is within maximum-norm distance $16a$ of a selected centre.
Such a square contains at most $C(a/\delta)^2$ admissible centres.
It follows that
\begin{equation}\label{Hk2-sc-number}
       N_a:=|\mathcal Q_a|\ge
          c\frac{(R_*/\delta)^2}{(a/\delta)^2}
          \ge c(R_1/a)^2.
\end{equation}
Only the small squares must be disjoint.  The larger comparison disks
may overlap.  The fixed auxiliary disk $B(0;R_*/2)$ contains all the
local disks $B(q;2L)$.  Indeed, for $q\in V$,
\[
 |q|\le\frac{29R_*}{64},\qquad 2L=\frac{R_*}{64},\qquad
 B(q;2L)\subseteq B\!\left(0;\frac{30R_*}{64}\right)
                \subseteq B(0;R_*/2).
\]
This is a separate auxiliary configuration, not an extension of the
original exploration into later annuli.

\smallskip\noindent\textit{Translation of the whole local description.}
Take $q_1=c$ for the original square.  On $S_a$, $N(c)$ contains
$g,Rg$ and their adjacent bonds.  Our particular choice of $x$ also
gives $x,Rx\in N(c)$.  The counted square is used to locate the first
common edge, not to contain the complete excursion.  The complete
excursion is contained in $D$.
For $h=q_2-q_1$, put $T_h(z)=z+h$ and
$R_{q_i}(z)=2q_i-z$ for $i=1,2$.  Then
\begin{equation}\label{Hk2-sc-translation}
 T_hN(q_1)=N(q_2),\qquad
 T_hR_{q_1}=R_{q_2}T_h.
\end{equation}
Thus the whole marked configuration, its enclosing disk when required,
and the centre of its local half-turn are translated together.

\smallskip\noindent\textit{The first common edge and its translation.}
For a fixed landing choice $j$ and permutation $\sigma$ of the three
outer intervals, let $P^\sigma,Q^\sigma$ be the incoming and outgoing
paths with their new outer continuations.
The outer completion leaves $Q^\sigma_{<g}=Q^-$ and agrees with $P$ inside $D$.
All additional portions and their half-turns are outside $D$.
By the definition of $g$,
\begin{equation}\label{Hk2-sc-preservation}
 \mathcal E(Q^\sigma_{<g})\cap\mathcal E(RP^\sigma)
   =\mathcal E(Q^-)\cap\mathcal E(RP)=\varnothing.
\end{equation}
The edges $g\in Q^\sigma$ and $Rg\in P^\sigma$ remain present,
so $g$ is still the first common edge.  The contact $x$, the complete
excursion and their half-turns are unchanged.

The translation applies to the paths as well as to the bond states.
Fix a completed configuration at $q_1$ and its constructed pair
$P_1^\sigma,Q_1^\sigma$.  In the corresponding translated configuration
at $q_2$, use
\[
 P_2^\sigma=T_hP_1^\sigma,\qquad
 Q_2^\sigma=T_hQ_1^\sigma,\qquad
 R_{q_2}=T_hR_{q_1}T_h^{-1}.
\]
All specified endpoints, landing intervals and choices of connecting
paths are translated with the pair.  In particular,
\[
 R_{q_2}P_2^\sigma=T_h(R_{q_1}P_1^\sigma),\qquad
 R_{q_2}Q_2^\sigma=T_h(R_{q_1}Q_1^\sigma).
\]
No independent incoming or outgoing path is chosen in making this
correspondence.

Let $e_2=T_he_1$.  Translation preserves the order of edges on a path,
so $(Q_2^\sigma)_{<e_2}=T_h((Q_1^\sigma)_{<e_1})$.  Therefore
\[
\begin{aligned}
 &\mathcal E((Q_2^\sigma)_{<e_2})
             \cap\mathcal E(R_{q_2}P_2^\sigma)\\
 &\qquad=T_h\!\left(
     \mathcal E((Q_1^\sigma)_{<e_1})
             \cap\mathcal E(R_{q_1}P_1^\sigma)\right).
\end{aligned}
\]
Also, $e_2$ is common to $Q_2^\sigma,R_{q_2}P_2^\sigma$
if and only if $e_1$ is common to $Q_1^\sigma,R_{q_1}P_1^\sigma$.
Thus $e_1$ is their first common edge exactly when $e_2$ is the first
common edge of the corresponding translated pair.  The four marked
points and edges and the complete $\mathcal K_x$ are translated together.
The change of the half-turn centre causes no loss of this property.

\smallskip\noindent\textit{The enclosing disk at a reference root.}
For $q=q_2$ and $h=q-c$, write
\[
 \mathcal K_{x,q}=T_h\mathcal K_x,\qquad
 D_q=T_hD=B(q;5R_1/8),\qquad
 R_q(z)=2q-z.
\]
Equation~\eqref{Hk2-sc-confinement} and the half-turn identity give
\[
 \mathcal K_{x,q}\cup R_q\mathcal K_{x,q}
       \subseteq B(q;R_1/2+C\delta)\subseteq D_q.
\]
The adjacent original bonds are included by the same whole-bond
convention.  The reference square $N(q)$ contains the four translated
marks, whereas $D_q$ contains the complete translated loop and the
kept incoming portion.  The loop need not be contained in $N(q)$,
nor need it stay in the original disk $B(c;R_1)$.  The relevant
containment is instead
\[
 D_q\subseteq B(q;2L)
     \subseteq B\!\left(0;\frac{30R_*}{64}\right)
     \subseteq B(0;R_*/2).
\]
The local joining annuli are translated along with $D$ and remain
inside $B(q;R_1)\subseteq B(q;2L)$.  Thus all translated loops and
local joining regions fit in this one auxiliary disk.

\smallskip\noindent\textit{Equal conditional laws and the outer RSW crossings.}
Let $z$ be a possible assignment of the original bonds in the fixed
annulus on which $F_a$ depends, and let $\nu$ be its distribution under
$\mu$.  At the centre $c$, the bond states in $B(c;r)$ are specified
by $\eta$, and the bonds determining $F_a$ have the values $z$.
All other original bonds are independent Bernoulli$(1/2)$ variables
under $\mu(\,\cdot\mid z)$.  At a reference centre $q$, put
$h=q-c$ and specify the values $\eta(e)$ and $z(e)$ at the
corresponding translated bonds $T_he$.  Assign independent
Bernoulli$(1/2)$ states to every other original bond.  Denote this
conditional probability law by $\widehat\mu_{q,z}$.  The half-turned
copy uses these same bonds, with the states complemented as above.
Thus, for any event $H$ determined by bonds in the disk centred at
$c$ and its translated event $H_q=T_hH$,
\begin{equation}\label{Hk2-sc-local-law}
        \widehat\mu_{q,z}(H_q)=\mu(H\mid z).
\end{equation}
The prescribed bonds are translated with the event.  The additional
conditioning on $D$ during arm separation has already been averaged out.
The same arm-separation and RSW estimates hold uniformly at every $q$.

To complete the comparison outside the local disk, prescribe three
common separated outer intervals on the boundary of a fixed auxiliary
region containing all $B(q;2L)$.  Match the separated local intervals
to these outer intervals in their cyclic order.  The finitely many
possible interval choices are already included in $j$.  Let
$\sigma$ specify the assignment of the three separated landing intervals to the prescribed outer intervals.  The compatible
open and dual-open extensions are made through disjoint rectangles
apart from overlapping crossings of the same colour.  The rectangles
are chosen from the separated interval pattern before their states
are examined.  Their widths are bounded below by a fixed multiple
of $R_1$, and their total lengths are at most a fixed multiple of
$R_*$.  Since $R_*/R_1=120$, only a fixed number of rectangles of
bounded aspect ratio is needed.  The extensions keep the complete
configuration in $D_q$ unchanged.

The extensions keep the first common edge by
\eqref{Hk2-sc-preservation} and
\eqref{Hk2-sc-translation}.  The RSW method and conditional
FKG estimates are uniform in $q,a,\eta,z$
and $\delta$.  The half-turned occurrences impose no
additional independent bond requirements.  We apply
this construction to the common-law events defined below.

\par\medskip\noindent
\textit{The common-law comparison and first-common-edge count.}
For the first-common-edge count, fix a common auxiliary law
$\widehat\mu_z:=\widehat\mu_{q,z}$ and, for each $a,j,\sigma$, construct pair
$P^\sigma,Q^\sigma$ in its completed configuration.  Let $q_0$ be the centre of the half-turn used for this pair.
The paths and the map $R_{q_0}(z)=2q_0-z$ are selected before any
reference square is tested.  Define $G_{q,j}^{\sigma}$ as the event requiring the kept inner
connections, the prescribed outer intervals, and an edge $e\in N(q)$
such that
\begin{equation}\label{Hk2-sc-counted-event}
 \begin{split}
 e&\in\mathcal E(Q^\sigma)\cap\mathcal E(R_{q_0}P^\sigma),\\
 \mathcal E(Q^\sigma_{<e})&\cap
                      \mathcal E(R_{q_0}P^\sigma)=\varnothing.
 \end{split}
\end{equation}
Include in this event $\widehat F_a$, the condition that the bond assignment
$z$ gives the event $F_a$, and count no square in the prescribed inner disk.

The equal conditional laws in \eqref{Hk2-sc-local-law} and the outer RSW construction above give the two-disk
comparison for these counted events. The inner configurations agree in law after translation, while the
outer connections are compared by arm separation and the RSW method between scales $R_1/2$ and $R_*$.  Since $R_*/R_1=120$,
the comparison constant is uniform.

The comparison is made for the complete constructed pairs, retaining
the first-common-edge condition and the complete inner return. In the common auxiliary law $\widehat\mu_z$, it gives
\begin{equation}\label{Hk2-sc-comparison}
c_*\,\mu(E_a\mid z)
       \le\sum_{j=1}^{M}\sum_{\sigma\in\mathfrak S_3}
                 \widehat\mu_z(G_{q,j}^{\sigma}),
      \qquad q\in\mathcal Q_a,
\end{equation}
where $\mathfrak S_3$ is the permutation group of $\{1,2,3\}$, containing six elements.

For fixed $a,j,\sigma$, fix one auxiliary completed configuration and its
common pair $P^\sigma,Q^\sigma$, with half-turn $R_{q_0}$.
Suppose two different squares contain edges $e,e'$ satisfying
\eqref{Hk2-sc-counted-event}, and suppose $e$ is earlier along
$Q^\sigma$.  Then
\[
R_{q_0} e\in\mathcal E(R_{q_0}Q^\sigma_{<e'})\cap
                         \mathcal E(P^\sigma),
\]
contradicting the defining condition for $e'$.  There is at most one
such square for this fixed $j,\sigma$. Here we also used the fact that $T_{q-c}\eta$ is conditional if $N(q)$ is another root square.
  Since the squares are disjoint,
there are at most six permutations and $M$ landing choices, and every
successful event includes $\widehat F_a$, the pointwise count is
\begin{equation}\label{Hk2-sc-count}
 \sum_{q\in\mathcal Q_a}\sum_{j=1}^{M}
       \sum_{\sigma\in\mathfrak S_3}\mathbf1_{G_{q,j}^\sigma}
                \le6M\,\mathbf1_{\widehat F_a}.
\end{equation}
This argument concerns first common edges of fixed constructed pairs.
It does not require different open paths to belong to different
clusters, and it does not reconstruct the original edge $g$.
The bound is for each fixed dyadic scale, not simultaneously over all
scales.

Translation preserves the first-common-edge condition and the
pointwise counting bound by \eqref{Hk2-sc-translation}.

Integrate \eqref{Hk2-sc-comparison} with
respect to $\nu$.  The common auxiliary law is then
$\widehat\mu=\int\widehat\mu_z\,\nu(dz)$, and
\[
                    \widehat\mu(\widehat F_a)=\mu(F_a).
\]
Summing over $q$ and applying \eqref{Hk2-sc-count} gives
\begin{align*}
 c_*N_a\mu(E_a)
 &\le\sum_{q,j,\sigma}\widehat\mu(G_{q,j}^\sigma)\\
 &\le6M\widehat\mu(\widehat F_a)
   =6M\mu(F_a).
\end{align*}
Together with \eqref{Hk2-sc-completion}, \eqref{Hk2-sc-number}, and
\eqref{Hk2-sc-inner}, this yields
\begin{equation}\label{Hk2-sc-cell}
 \mu(S_a)\le\frac{6M}{\kappa c_*N_a}\mu(F_a)
       \preceq \left(\frac a{R_1}\right)^2
                \left(\frac r a\right)^2
       =\left(\frac r{R_1}\right)^2.
\end{equation}
The factor $\mu(F_a)$ is kept in the counted indicator.  No independence
of $E_a$ and its inner connections is asserted.

\par\medskip\noindent
\textit{Summation over scales and the second touching event.}
\eqref{Hk2-sc-confinement} implies that the number of
nonempty dyadic classes is at most $C+\log_2(R_1/r)$.
The union of the classes with $a\ge R_1/2^{12}$ forces $Q^-$ to
reach a fixed positive multiple of $R_1$.  The twisted-arm estimate at that deterministic scale gives a total contribution at
most $C(r/R_1)^2$.  Applying \eqref{Hk2-sc-cell} to the other
classes and summing gives
\begin{equation}\label{Hk2-sc-first-alternative}
                \mu(S_1)\preceq (1+\log(R_1/r))(r/R_1)^2.
\end{equation}

For the second touching alternative, let $N$ be a final exploration
index and set $\widehat\gamma(m)=\ga_w(N-m)$, with the same
specified visit to $w$.  The reversed incoming and outgoing times are
\[
\begin{aligned}
 \widehat t_{l,k}&=N-\tau_{r,k},&\widehat\tau_{l,k}&=N-t_{r,k},\\
 \widehat t_{r,k}&=N-\tau_{l,k},&\widehat\tau_{r,k}&=N-t_{l,k}.
\end{aligned}
\]
The definitions by unoriented common edges give
\[
 \widehat{\mathrm{s}}_{r,k}=N-\mathfrak{s}_{l,k},\qquad
 \widehat{\mathrm{s}}_{l,k}=N-\mathfrak{s}_{r,k}.
\]
Because the interval $[\mathfrak{s}_{r,k},\tau_{r,k})$ is half open, its
reversal is $(\widehat t_{l,k},\widehat{\mathrm{s}}_{l,k}]$.  Consequently the
reversal of $S_2$ is contained in the first touching event for
$\widehat\gamma$.  The crossing conditions
$\mathcal H_{k,0},\mathcal H_{k,1}$ are preserved with incoming and
outgoing interchanged.  The complete loop, its half-turn, and the
connections inside $D$ are unchanged.  The same argument, with the
corresponding comparison \eqref{Hk2-sc-comparison} in this orientation,
gives the bound in \eqref{Hk2-sc-first-alternative} for $S_2$.

Adding the two alternatives gives
\eqref{Hk2-logarithmic-target}.  The larger reference domain is auxiliary and does not expose any
later annulus of the original configuration.  This completes the proof. \qed

Put
\[
 \mathcal J_k=\mathcal H_{k,0}\cap\mathcal H_{k,1}\cap\mathcal H_{k,2}.
\]
We call $\mathcal A_k$ locally good when $\mathcal J_k$ occurs.
Equations \eqref{Hk0}, \eqref{Hk1}, and Lemma~\ref{Hk2} give,
for $k\geq 1$,
\begin{equation}
 \Pro(\mathcal J_k^c\mid\hbox{the previously revealed bonds})
 \le C(1+\log\mathrm R_0)\mathrm R_0^{-2}.
 \label{local-good-annulus-bound}
\end{equation}

For $k$ such that $\mathcal J_k$ occurs, put
\[
 \mathcal X_k
 =\ga_w[n_w,\mathrm{s}_{r,k,2}]\cup\overline{B(c;r_w)},
\]
and let $\mathcal O_k$ be the collection of unordered pairs
$\{E,E^R\}$ for which the centre of at least one of the two primal bonds
belongs to $\mathcal X_k$.  The closed disk is used because
$r_w$ is the maximum distance of a point of $\Loop_w$ from $c$;
every primal bond whose midpoint lies on $\Loop_w$ must be included.
Define $\omega_v^{(k)}$ by
\beqn
 \omega_v^{(k)}(E)=1-\omega_w(E^R),\qquad
 \omega_v^{(k)}(E^R)=1-\omega_w(E),
 \qquad \{E,E^R\}\in\mathcal O_k, \label{caseI1}
\eeqn
and leave all other bond states unchanged.  Let $\ga_v^{(k)}$ be the
exploration path in $\omega_v^{(k)}$.  Define $\mathcal H_{k,4}$ to be the
event that $\ga_v^{(k)}$ has the local visit $(C,D)$ at $v$ and
\[
 W_{\ga_w}(A_w,e_b)=W_{\ga_v^{(k)}}(D,e_b),
 \qquad
 W_{\ga_w}(D_w,e_b)=W_{\ga_v^{(k)}}(C,e_b).
\]

\medskip
\noindent
{\bf The outer contact estimate.}
We apply the two-disk comparison and the arm-separation argument from
the proof of Lemma~\ref{Hk2} to the contact configurations below.
Their complete loops, primal-open and dual-open connections, and outer
continuations are included.  The inner and outer scales are now
$\rho_k$ and $\rho_{m_c}$.

We specify the conditioning used here.  First fix the initially
revealed inner configuration $\eta_0$, including the attached loop and
its supporting bonds.  Thus $r_w$ and the radii $\rho_j$ are fixed.
Let $v_E$ denote the midpoint of a primal bond $E$, and define
\begin{equation}\label{Hk3-new-filtration}
 \mathscr F_k
 =\sigma\bigl(\eta_0,\,\omega(E):|v_E-c|\leq\rho_k\bigr).
\end{equation}
Conditional on $\mathscr F_k=\eta$, the unprescribed bonds outside
$B(c;\rho_k)$ are independent critical Bernoulli bonds. Write
$\mu_\eta$ for this conditional law.

\begin{lemma}\label{Hk34-outer-lemma}
For every $1\leq k\leq m_c$, uniformly in the fixed inner bond
states and the lattice mesh,
\begin{equation}\label{Hk3-new-estimate}
 \Pro\bigl(\mathcal H_{k,3}^{c}\cap\mathcal J_k
                     \mid\mathscr F_k\bigr)
 \preceq \biggl(1+\log\frac{\rho_{m_c}}{\rho_k}\biggr)
          \left(\frac{\rho_k}{\rho_{m_c}}\right)^2
 =\bigl(1+(m_c-k)\log\mathrm R_0\bigr)
          \mathrm R_0^{-2(m_c-k)}
\end{equation}
and
\begin{equation}\label{Hk4-new-estimate}
 \Pro\bigl(\mathcal H_{k,4}^{c}\cap\mathcal H_{k,3}
                   \cap\mathcal J_k\mid\mathscr F_k\bigr)
 \preceq \biggl(1+\log\frac{\rho_{m_c}}{\rho_k}\biggr)
          \left(\frac{\rho_k}{\rho_{m_c}}\right)^2
 =\bigl(1+(m_c-k)\log\mathrm R_0\bigr)
          \mathrm R_0^{-2(m_c-k)}.
\end{equation}
In particular,
\begin{equation}\label{Hk34-outer-estimate}
 \Pro\bigl((\mathcal H_{k,3}^{c}\cup\mathcal H_{k,4}^{c})
                   \cap\mathcal J_k\mid\mathscr F_k\bigr)
 \preceq \bigl(1+(m_c-k)\log\mathrm R_0\bigr)
          \mathrm R_0^{-2(m_c-k)}.
\end{equation}
\end{lemma}

\begin{proof}
If $\rho_{m_c}/\rho_k$ is bounded by a fixed constant, the estimates
follow by adjusting the implicit constants. Otherwise there is room in the annulus
$B(c;\rho_k,\rho_{m_c})$ for the RSW constructions in the proof of
Lemma~\ref{Hk2}.  Under $\mu_\eta$, the bonds inside $B(c;\rho_k)$
are fixed and the unprescribed bonds outside this disk are independent.

\smallskip\noindent
{\it The event $\mathcal H_{k,3}^{c}$.}
There are two kinds of touching.  We first consider the incoming
part of the exploration.  The other cases follow by reversing the path.

Suppose first that
\[
 \ga_w[0,\mathrm{s}_{l,k,2}]
 \cap\ga_w[n_w,\mathrm{s}_{r,k,2}]\ne\emptyset.
\]
Choose the last such contact along the incoming path and write
\[
 y_k=\ga_w(u_k^-)=\ga_w(u_k^+),\qquad
 u_k^-\leq\mathrm{s}_{l,k,2}<n_w
 \leq u_k^+\leq\mathrm{s}_{r,k,2}.
\]
The portion $\ga_w[u_k^-,u_k^+]$ is a complete return from $y_k$
to $y_k$ through $w$.  Keep the primal-open and dual-open paths along
its sides, the incoming and outgoing continuations, and their half-turns,
as in the proof of Lemma~\ref{Hk2}. If $\ga_w[u_k^-,\mathrm{s}_{l,k,2}]$ crosses $\partial B(c;\rho_{m_c})$,
use the half-turned complete return in copy $1$ for the inner connections. In copy $0$, follow the incoming
path backwards from $Rg$ and the outgoing path forwards from $g$, stopping each at its first intersection
with $\partial B(c;\rho_{m_c})$. These give the outer connections with the prescribed matching, without requiring
the complete return to be included in copy $0$.

Suppose next that one of the incoming portions appearing in the
definition of $\mathcal H_{k,3}$ touches $\Loop_w$ at $y_k$.
Parameterize $\Loop_w$ starting at $y_k$, and keep both arcs from
$y_k$ to $w$, the incoming path between these vertices, their
primal-open and dual-open sides, and their continuations towards the
boundary.  For a contact with $R\Loop_w$, use the complete loop through
$v$ and its half-turn through $w$.  The half-turned connections are
determined by the original bond states.

In both cases the contact supplies the complete loop and the primal-open and dual-open connections used in the proof of
Lemma~\ref{Hk2}. Keep these connections and their half-turns, including their portions in $B(c;\rho_k)$. Use the enclosing
disk of that comparison, denoted by $D_\ell=B(c;\ell)$. Its radius $\ell$ is distinct from the dyadic arm scale.
The complete return need not lie in the small reference square.

Translate the enclosing disk, the marked paths,
the prescribed inner bond states and the centre
of the half-turn together. This preserves the selected first common
edge and its defining
condition. The two-disk comparison and counting argument are applied
 in the auxiliary configuration. No new common edge
outside $B(c;\rho_k)$ is required. The additional conditioning used for
arm separation is averaged out within the uniform
comparison, without a further summation over $\ell$

For the alternatives in the second line of $\mathcal H_{k,3}$, reverse
the exploration.  If $N$ is its final index, put
$\widehat\gamma(j)=\ga_w(N-j)$ and $\widehat n_w=N-n_w$.  The
corresponding times satisfy
\[
 \widehat{\mathrm s}_{l,k,2}=N-\mathfrak s_{r,k,2},
 \qquad
 \widehat{\mathrm s}_{r,k,2}=N-\mathfrak s_{l,k,2}.
\]
The same loops and primal-open and dual-open paths occur with the
incoming and outgoing directions interchanged.

Apply arm separation, the two-disk comparison, and the
first-common-edge count from Lemma~\ref{Hk2}. As in that proof,
classify the paths by their radial reach before the first common
edge. At each dyadic scale, the twisted-arm estimate and the count
of reference squares give a contribution
$\preceq(\rho_k/\rho_{m_c})^2$. The dyadic sum has at most
$O(1+\log(\rho_{m_c}/\rho_k))$ terms. The classes whose arm scale reaches a fixed positive multiple
of $\rho_{m_c}$ are treated together by the twisted-arm estimate at that scale. The finitely many
contact and colour choices change only the implicit constant. Thus
\[
 \mu_\eta\bigl(\mathcal H_{k,3}^{c}\cap\mathcal J_k\bigr)
 \preceq \biggl(1+\log\frac{\rho_{m_c}}{\rho_k}\biggr)
          \left(\frac{\rho_k}{\rho_{m_c}}\right)^2.
\]
This proves \eqref{Hk3-new-estimate}.

\medskip\noindent
{\it The event $\mathcal H_{k,4}^{c}$.}
Put
\[
 a=\mathrm s_{l,k,2},\qquad b=\mathrm s_{r,k,2},\qquad
 z_k=\ga_w(a)=R\ga_w(b),\qquad Q_k=\ga_w[n_w,b].
\]
For every primal bond $E$ whose midpoint belongs to $Q_k$ or
$\Loop_w$, the pair $\{E,E^R\}$ belongs to $\mathcal O_k$.
By \eqref{caseI1},
\[
 \omega_v^{(k)}(E^R)=1-\omega_w(E).
\]
Thus the primal-open and dual-open sides of these paths are
interchanged under the half-turn in the modified configuration.
In particular, $(RQ_k)^{-1}$ is a directed medial path from $z_k$
to $v$, and $R\Loop_w$ is a closed medial loop through $v$ using $A,B$.
If the original outgoing path has no common unoriented medial edge with
$RQ_k$ other than the prescribed matching, the choice of $a$ at the
first common edge, together with $\mathcal H_{k,3}$, implies that
\[
 (RQ_k)^{-1}*\ga_w[0,a]
\]
is the incoming part of the modified exploration, which enters $v$
through $C$ and leaves through $D$.
For the planar separation argument, temporarily flip only
$E_w$ in the original configuration and follow the two sides of
the attached loop together with $\ga_w[a,n_w]$ and $RQ_k$.
This auxiliary flip is distinct from the bond switch
\eqref{caseI1}. Any additional return preventing this
continuation, or changing its winding, would enclose the outgoing
edge at $w$.
 The original outgoing path reaches the domain boundary
and cannot cross its own incoming part or $\Loop_w$.  Planar
separation would therefore force a common edge with $RQ_k$.
In the absence of such a contact, the prescribed turns at $w,v$ give
\[
 W_{\ga_w}(A_w,e_b)=W_{\ga_v^{(k)}}(D,e_b),\qquad
 W_{\ga_w}(D_w,e_b)=W_{\ga_v^{(k)}}(C,e_b).
\]
Thus failure of $\mathcal H_{k,4}$ forces an additional common edge of
$\ga_w[n_w,\infty)$ and $RQ_k$.

Choose the first such unoriented edge in the outgoing order.  With
consistent endpoint choices there are times $T_k\ge n_w$ and
$t_k\in[n_w,b]$ such that
\[
 \ga_w(T_k)=R\ga_w(t_k).
\]
With path products read from right to left, the complete closed walk is
\begin{equation}\label{Hk4-complete-loop}
 \mathcal C_k=R\ga_w[t_k,b]*\ga_w[n_w,T_k]*\ga_w[a,n_w].
\end{equation}
It runs from $z_k$ back to $z_k$ through $w$.  Keep the entire walk,
its half-turn, the attached loops, and their primal-open and dual-open
sides and continuations to the outer boundary.  Any common portion of
$\ga_w[n_w,T_k]$ and $\ga_w[t_k,b]$ is kept with its original bond
states. Apply the complete-loop arm-separation, two-disk comparison,
and first-common-edge counting argument of Lemma~\ref{Hk2}, using
the same dyadic decomposition. Each class contributes
$\preceq(\rho_k/\rho_{m_c})^2$, and summing the classes gives
\[
 \mu_\eta\bigl(\mathcal H_{k,4}^{c}\cap\mathcal H_{k,3}
                         \cap\mathcal J_k\bigr)
 \preceq \biggl(1+\log\frac{\rho_{m_c}}{\rho_k}\biggr)
          \left(\frac{\rho_k}{\rho_{m_c}}\right)^2.
\]
This proves \eqref{Hk4-new-estimate}. The union bound gives
\eqref{Hk34-outer-estimate}.
\end{proof}

Let
\[
 K=\min\{1\leq k\leq m_c:\mathcal J_k\hbox{ occurs}\}.
\]
If the set is empty, the configuration is exceptional.  If $K$ is defined,
we next check $\mathcal H_{K,3}$ and $\mathcal H_{K,4}$.  The selected
annulus is called good when both events occur.  Thus
\[
 \mathcal H_k=\{K=k\}\cap\mathcal H_{k,3}\cap\mathcal H_{k,4}.
\]
By \eqref{local-good-annulus-bound}, the probability that the first
$k-1$ annuli are not locally good is at most
$\big(C(1+\log \mathrm R_0) \mathrm R_0^{-2}\big)^{k-1}$.  Lemma~\ref{Hk34-outer-lemma} therefore gives
\begin{align*}
 &\Pro\big(K\hbox{ is not defined, or }
       \mathcal H_{K,3}^c\cup\mathcal H_{K,4}^c
       \mid\hbox{the explored loop and its incident bonds}\big)\\
 &\leq
 \big(C(1+\log \mathrm R_0) \mathrm R_0^{-2}\big)^{m_c}
 +C\sum_{k=1}^{m_c}
   \big(C(1+\log \mathrm R_0) \mathrm R_0^{-2}\big)^{k-1}
\big(1+(m_c-k)\log \mathrm R_0\big)   \mathrm R_0^{-2(m_c-k)}\\
 &
 \preceq
 \left(\frac{r_w}{d_{v,w}}\right)^{2-\varepsilon}
\end{align*}
if $d_{v,w}/r_w$ is large enough.
If $d_{v,w}/r_w$ is bounded by a constant, the same estimate follows after changing its multiplicative constant.

\medskip
\noindent
{\bf The bond switch on a good annulus.}
Assume that $\mathcal H_k$ holds.  Thus $k=K$ is the first locally good
annulus and $\mathcal H_{k,3}\cap\mathcal H_{k,4}$ also occurs.  Put
\[
 \mathcal X_k
 =\ga_w[n_w,\mathrm{s}_{r,k,2}]\cup B(c;r_w),
\]
and let $\mathcal O_k$ be the collection of unordered pairs
$\{E,E^R\}$ for which the centre of at least one of the two primal
bonds belongs to $\mathcal X_k$.  For every pair in $\mathcal O_k$,
interchange the two bond states:
\beq
 \omega_v(E)=1-\omega_w(E^R),\qquad
 \omega_v(E^R)=1-\omega_w(E),
 \qquad \{E,E^R\}\in\mathcal O_k. 
\eeq
All other bond states are left unchanged.  Thus the states of the primal
bonds centred on $\ga_w[n_w,\mathrm{s}_{r,k,2}]$, together with those in
the ball containing the attached loop, are carried through the half-turn
$R$.
Since this operation merely interchanges the states within each pair,
it preserves the percolation measure.

The place where the half-turned path joins the unchanged
exploration path is
\[
 z_{k,2}
 =\ga_w(\mathrm{s}_{l,k,2})
 =\ga_w^R(\mathrm{s}_{r,k,2}).
\]
The primal-open and dual-open bonds on the two sides of
$\ga_w[n_w,\mathrm{s}_{r,k,2}]$ are carried to the corresponding two
sides of its half-turn.  The event $\mathcal H_{k,3}$ excludes a
touching with an unused part of the exploration path before
$z_{k,2}$.  The event $\mathcal H_{k,4}$ prevents extra or less winding.
It follows that the exploration path in $\omega_v$ enters $v$ through
$C$, leaves through $D$, and agrees with the old exploration path away
from the part replaced by the half-turn.  The directions at the two
ends of the replaced part are both rotated by $\pi$, so its total
winding is unchanged.  This proves
\eqref{edge-centre-wind-1}--\eqref{edge-centre-wind-2}.

We finally prove that the change is one-to-one.  Suppose that $\omega_w^{(1)}$ and $\omega_w^{(2)}$ give the same
configuration after the half-turn.  Let \(k_i\) be the annulus index
selected for \(\omega_w^{(i)}\).
 Without loss of
generality, suppose
\[
 (\ga_w^{(1)})[n_w^{(1)},\mathrm{s}_{r,k_1,2}^{(1)}]
 \subseteq
 (\ga_w^{(2)})[n_w^{(2)},\mathrm{s}_{r,k_2,2}^{(2)}].
\]
If the inclusion were proper, the corresponding incoming portions
would cross.  Let $[z,z']$ be their first common medial edge in the order
of the shorter incoming portion.  At the endpoint $z$, either the state of
$E_z$ or the state of $E_z^R$ would be different in the two modified
configurations, according as $z$ does not or does belong to their
common modified incoming path.  This contradicts the equality of
the modified configurations.  Hence the two selected outgoing
portions are equal.  The same argument gives equality of the selected
incoming portions.  All bonds outside the corresponding half-turned
pairs are unchanged, and therefore
$\omega_w^{(1)}=\omega_w^{(2)}$.

\medskip
\noindent
{\bf The crossing estimate.}
It remains to prove \eqref{Hk1}.  We state it in a scale-independent form.

\begin{lemma} \label{rotation-event}
Let $0<R_2<R_1\leq d$.  Let $\eta^-$ be the last incoming crossing and
$\eta^+$ the first outgoing crossing of the annulus
$B(c,R_1)\setminus B(c,R_2)$ by the exploration path through $w$.  Then
\[
 \Pro\big(R(\eta^+)\hbox{ does not cross }\eta^-\big)
 \preceq (R_2/R_1)^2.
\]
In particular, \eqref{Hk1} holds.
\end{lemma}

\begin{proof}
If $R_1\leq256R_2$, the estimate follows after increasing the implicit
constant.  We therefore assume $R_1>256R_2$.  After a translation and a
rotation, suppose that the centre $c$ of the half-turn is the origin, that
$E_w$ is horizontal, and that its parallel line $\ell$ containing $c$ is $\mathbb R$.  Let $\widehat\eta^-$ and
$\widehat\eta^+$ be the incoming and outgoing crossings of
$B(c;4R_2,R_1/4)$ obtained from $\eta^-$ and $\eta^+$.  On the event under
consideration, $R(\widehat\eta^+)$ does not cross $\widehat\eta^-$.

Neither trimmed crossing touches its own half-turn.  Indeed, such a touch
for $\widehat\eta^-$ would produce, from a subpath joining $z$ to $R(z)$
and its half-turn, a medial circuit surrounding $c$.  The outgoing
exploration path must cross this circuit.  It cannot cross the part lying
in its own past, and therefore crosses $R(\widehat\eta^-)$.  Applying $R$
gives a crossing of $R(\widehat\eta^+)$ with $\widehat\eta^-$, a
contradiction.  The argument for $\widehat\eta^+$ is identical.

Cut the two crossings at their successive intersections with $\ell$ and
apply $R$ to every portion below $\ell$, identifying $x$ with $R(x)$ on the
cut.  The two folded interface crossings have no common medial edge: a
same-side common edge would be used twice by the original exploration path,
whereas either mixed common edge would give a common edge of
$R(\widehat\eta^+)$ and $\widehat\eta^-$.  Their primal-open and dual-open
sides therefore give three ordered arms across the folded half-annulus.
Keep the usual type on an unrotated portion and reverse it on a rotated
portion.  The three arms then have fixed alternating twisted types, either
$010$ or $101$.  Arms of opposite twisted types cannot cross, since
unfolding a first crossing would give a primal-open path crossing a
dual-open path in the original percolation configuration.

Let $\mathcal A_{01}^{\rm tw}(r,R)$ and
$\mathcal A_{010}^{\rm tw}(r,R)$ denote the corresponding two- and
three-arm events.  The usual percolation proof of the half-plane arm bounds
applies.  After arm separation and fences, prescribe the landing intervals
and require the arms outside a fixed enlargement of their root box to stay
outward.  For three arms, the two arms of the same twisted type from any
root form a same-type crosscut.  Choosing an occurring root nearest the
inner boundary, any separated second root would force the opposite middle
arm from the first root to cross this crosscut.  Thus only a universal
number of root boxes can occur.  The two-arm root count is the same: two
separated roots force a primal-open arm to cross a dual-open arm. The RSW
comparability and the usual count of possible roots now give
\begin{equation}\label{twisted-half-plane-arm-bounds}
 \Pro\big(\mathcal A_{01}^{\rm tw}(r,R)\big)
 \preceq \frac rR,
 \qquad
 \Pro\big(\mathcal A_{010}^{\rm tw}(r,R)\big)
 \preceq \left(\frac rR\right)^2,
 \qquad 16r<R.
\end{equation}

Applying the second estimate in \eqref{twisted-half-plane-arm-bounds} with
$r=4R_2$ and $R=R_1/4$ gives
\[
 \Pro\big(R(\eta^+)\hbox{ does not cross }\eta^-\big)
 \preceq (R_2/R_1)^2.
\]
This proves the lemma.
\end{proof}

\medskip
\noindent
{\bf Fixed-loop conclusion.}
If $r_w\geq d_{v,w}$, we put the whole configuration into the fixed-loop exceptional set.  Since $(r_w/d_{v,w})^{2-\varepsilon}\geq1$, the same bound remains true after increasing its constant.  We have therefore proved the following fixed-loop statement for the visit $(A_w,D_w)$.  Conditional on the explored loop and its incident bonds, the map $\omega_w\mapsto\omega_v$ is defined and one-to-one outside an event whose conditional probability is at most
\beqn
 C\left(\frac{r_w}{\min(d_v,d_w)}\right)^{2-\varepsilon}. \label{fixed-loop-error}
\eeqn
The same construction, after a quarter-turn of the labels, gives the other three correspondences:
\beq
 (A_w,D_w)&\longmapsto&(C,D),\qquad
 (C_w,D_w)\longmapsto(B,C),\\
 (B_w,C_w)&\longmapsto&(A,B),\qquad
 (A_w,B_w)\longmapsto(A,D).
\eeq
For these four correspondences, the winding identities are, respectively,
\[
\begin{array}{c}
(A_w,D_w)\longmapsto(C,D):\quad
 W_{\ga_w}(A_w,e_b)=W_{\ga_v}(D,e_b),\quad
 W_{\ga_w}(D_w,e_b)=W_{\ga_v}(C,e_b);\\[1mm]
(C_w,D_w)\longmapsto(B,C):\quad
 W_{\ga_w}(C_w,e_b)=W_{\ga_v}(B,e_b),\quad
 W_{\ga_w}(D_w,e_b)=W_{\ga_v}(C,e_b);\\[1mm]
(B_w,C_w)\longmapsto(A,B):\quad
 W_{\ga_w}(B_w,e_b)=W_{\ga_v}(A,e_b),\quad
 W_{\ga_w}(C_w,e_b)=W_{\ga_v}(B,e_b);\\[1mm]
(A_w,B_w)\longmapsto(A,D):\quad
 W_{\ga_w}(A_w,e_b)=W_{\ga_v}(D,e_b),\quad
 W_{\ga_w}(B_w,e_b)=W_{\ga_v}(A,e_b).
\end{array}
\]
In every case, the half-turn also preserves which of the three subcases in Section \ref{partI} is seen by the relevant exploration-path segment and the attached loop.

\begin{remark}
Estimate \eqref{fixed-loop-error} is conditional on the attached loop.  It is not an unconditional estimate of order $(\de/d_{v,w})^{2-\varepsilon}$.  Indeed, the event $r_w\geq r$ requires four alternating arms from scale $\de$ to scale $r$, and hence
\[
 \Pro(r_w\geq r)\preceq(\de/r)^{1+\al}.
\]
A dyadic average of \eqref{fixed-loop-error} therefore gives only four-arm order, namely at most $(\de/d_{v,w})^{1+\al}$ up to the usual small loss in the exponent.  Section \ref{partIII} obtains the stronger marginal exponent $2-\varepsilon$ only for the additional failure events which supply a fifth arm.
\end{remark}

\subsection{Revisit rotation around medial vertex} \label{partIII}

We now return to one medial vertex $v$.  The purpose of this subsection is to  combine the reflection--duality bond change of Section \ref{partI} with the fixed-loop half-turn of Section \ref{partII}.  The important point is that we do not discard all configurations for which the fixed-loop construction has no good annulus.  We discard only those failures which, together with the local geometry at $v$, create five arms.  This is why the final exceptional probability is of order $(\de/d_v)^{2-\varepsilon}$.

Let $w$ be the neighboring medial vertex used in Section \ref{partII}, so that $B_w=A$.
Given a configuration $\omega_v$, apply the half-turn construction of Section \ref{partII} with the roles of $v$ and $w$ interchanged, and denote the resulting configuration by $\omega_w$.  Variables associated with these two configurations carry subscripts $v$ and $w$.  In particular, $r_v$ and $r_w$ are the radii of the attached loops.  The half-turn carries the whole attached loop without changing its radius whenever one of the two loops is contained in the ball used in the construction.  Therefore
\beqn
 \min(r_v,r_w)<d_{v,w}\quad\Longrightarrow\quad r_v=r_w. \label{equal-radii}
\eeqn
Thus, on every scale below $d_{v,w}$, the two centered constructions use the same loop radius.

Now we begin the construction for an original configuration with local visit $(A,D)$,
\[
 A,D\in\ga_v,
 \qquad B,C\notin\ga_v.
\]
The new configuration $\omega_v'$ has exploration path $\ga_v'$ and one of the two local visits
\[
 (A,B)\qquad\hbox{or}\qquad(C,D).
\]
In the first case,
\beq
 W_{\ga_v'}(B,e_b)&=&W_{\ga_v}(A,e_b)-\pi/2,\qquad
 W_{\ga_v'}(A,e_b)=W_{\ga_v}(D,e_b)-\pi/2,
\eeq
whereas in the second case,
\beq
 W_{\ga_v'}(D,e_b)&=&W_{\ga_v}(A,e_b)+\pi/2,\qquad
 W_{\ga_v'}(C,e_b)=W_{\ga_v}(D,e_b)+\pi/2.
\eeq
As in Section \ref{partI}, the two signs correspond to the clockwise and counterclockwise quarter-turns at $v$.

\medskip
\noindent
{\bf Class I.}
We place $\omega_v$ in Class I when
\begin{enumerate}
\item $r_v<d_{v,w}$;
\item at least one good annulus $\mathcal H_{v,k}$ of Section \ref{partII} occurs for the construction centered at $c$.
\end{enumerate}
 We choose the first good annulus in the construction.  The half-turn of Section \ref{partII} identifies the future of the path at $w$ and its attached loop with the past of the path at $v$ and its attached loop. Consequently the same one of Subcases (1)--(3) in Section \ref{partI} is determined in the configurations $\omega_v$ and $\omega_w$.

We now apply the corresponding bond change \eqref{subcase(1)}, \eqref{subcase(2)}, or \eqref{subcase(3)} to $\omega_v$.  In summary, the selected primal-open and dual-open path segments are reflected and their states are interchanged up to the meeting points supplied by the good annuli.  The verification is the same in all three subcases.  Before the first meeting point, the new interface follows the old exploration path.  Between the two meeting points, a primal-open path and a dual-open path lie on the two sides of the reflected exploration-path segment, so the interface is forced to follow that segment through $v$.  After the second meeting point, the configuration is unchanged and the new interface rejoins the old path.  The good-annulus conditions exclude a shortcut or an earlier reconnection.  Thus the new configuration has the claimed local visit and winding.  The one-to-one property on Class I follows from the corresponding arguments in Case~I of Section \ref{partI}.

\medskip
\noindent
{\bf Class II.}
Class II is the complement of Class I.

We apply the reflection-duality switch in Section \ref{partI} to each configuration in Class II.

So we are still using the reflection in Section \ref{partI}. The only difference is that we assign one class to each configuration.

\medskip
\noindent
{\bf Revisit the half-turn: $\omega_v \to \omega_w$.} Let $\omega_w$ be the half-turn of $\omega_v$. After removing Class II, we have the one-to-one mapping from $\mathcal E_{j,m}^v\setminus\mathcal E_v\longrightarrow \mathcal E_{j+1,m}^w\setminus\mathcal E_w$.
Note that if $\omega_v$ is in Class I, then $\mathrm f_v(\omega_v)$ may be in Class II due to large loops or failure of $\mathcal H_{v,k}$. However, we have the following proposition which is the key point in deriving Lemma \ref{F1F2}.

\begin{proposition} \label{omegavw}
 Let $\omega_v\in \mathcal E_{1,2}^v\setminus\mathcal E_v$.  Denote by $\omega_v'$ the image  of $\omega_v$ after the reflection–duality switch. Then
 the probability that $\omega_v$ and $\omega_v'$ are in Class II is at most of the order $(\de/d_{v,w})^{2-\ep}$.

 Denote the half-turn map from $\omega_v$ to $\omega_w$ by $R_v$. If both $\omega_v$ and $\omega_v'$ are in Class I, then
 \[
 \mathrm{f}_w\big((\omega_v)^{R_v}\big)=(\omega_v')^{R_v}.
 \]
 \end{proposition}
 \begin{proof}
   Here we join $\ga_v^{R_v}[\mathfrak s_{l,k,2},n_v]$ and $\ga_v[\mathfrak s_{r,k,2},\infty)$, and consider the reflection–duality switch $\omega_v \to \omega_v'$. Define
 $$
 r_c=\max\{|\Loop^{AC}_v(j)-c|: 0\leq j\leq \mathrm{s}^v\},
 $$

Now suppose $r_v\geq d_{v,w}$ and $r_v'\geq d_{v,w}$. Then in $\omega_v'$, there are six disjoint arms crossing the annulus $B(c;r_c,d_{v,w})$. On the other hand, by Lemma \ref{rotation-event} and the definition of $\mathrm{s}_v$, the configuration in $B(c;r_c)$ has probability of the order $(\de/r_c)^2$. This combined with the six-arm estimate in $B(c;r_c,d_{v,w})$ provides the required probability.

Next, suppose $r_v\geq d_{v,w}$ and $r_v'< d_{v,w}$. If $r_v'=r_c$, using Lemma \ref{rotation-event}, the definition of $\mathrm{s}_v$, and the condition that $\omega_v'$ is in Class II, we obtain the required probability. If $r_c<r_v'$, there are six disjoint arms crossing the annulus $B(c;r_c,r_v')$. Inside $B(c;r_c)$, we apply Lemma \ref{rotation-event}; inside $B(c;r_c,r_v')$, we use six-arm estimate; outside $B(c;r_v')$, we resort to the condition that $\omega_v'$ is in Class II and $r_v'<d_{v,w}$. Combining them implies the desired probability bound.

Thirdly, suppose $r_v< d_{v,w}$ and $r_v'\geq d_{v,w}$. This can be handled by noting that there are six disjoint arms crossing the annulus $B(c;r_c,r_v)$ and the condition that $\omega_v$ is in Class II and $r_v<d_{v,w}$. So, the one-to-one property of the map $\mathrm{f}_v$ implies the required probability bound.

Finally, suppose $r_v< d_{v,w}$ and $r_v'< d_{v,w}$. The cases where $r_v'=r_c$ and $r_c<r_v'\leq r_v$ can be treated in the same way as the case where $r_v\geq d_{v,w}$ and $r_v'< d_{v,w}$. If $r_v'>r_v\geq r_c$, there are six disjoint arms crossing the annulus $B(c;r_c,r_v)$. Outside $B(c;r_v)$, we utilize the condition that $\omega_v$ is in Class II and $r_v<d_{v,w}$. By the one-to-one property of the map $\mathrm{f}_v$, we can derive the probability bound.

For the second assertion, note that
there are two parts of $\Loop_v'$ in $\omega_v'$. One part is $\Loop_v^{AC}[0,\mathrm{s}^v]$, and the other part is surrounded by $\ga_v[\mathrm s_v,n_v]$. So, after $R_v$, the fixed part of the outgoing path away from $v$ can not cross $(\Loop_v')^{R_v}$ without crossing the half-turned part of $\ga_v[0,n_v]$. In addition, $\mathcal H_{k,3}$ guarantees that $\mathfrak s_{l,k,2}<\mathrm s_v$.
Hence the second assertion follows from the definition of $\mathfrak s_{r,k,2}$ and $R_v$.

We can conclude the proof now.
 \end{proof}

 \begin{remark}
 1. Proposition \ref{omegavw} also holds when $\omega_v\in \mathcal E_{3,2}^v\setminus\mathcal E_v$. In addition, if we apply the reflection–duality switch to $\omega_v$ via the incoming exploration path to $v$, then after half-turn, the reflection–duality switch should be implemented through the outgoing exploration path away from $w$. 2. For the second assertion of Proposition \ref{omegavw}, since $R_v$ is not the half-turn $R_v'$ corresponding to $\omega_v'$, it may happen that $R_v\omega_v'\not=R_v'\omega_v'$. However, since we assume that $\omega_v'$ is in Class I, and $R_v \omega_v'$ already gives an image of $\omega_v'$, we can regard $R_v\omega_v'$ as $R_v'\omega_v'$ when $\omega_v\in \mathcal E_{1,2}^v\setminus\mathcal E_v$. This change still preserves the one-to-one property of the half-turn.
 \end{remark}

  \section{Modified parafermionic observable and basic relations} \label{modi}

The ordinary parafermionic observable $F$ satisfies the discrete Cauchy--Riemann relation \eqref{CR}, but it does not satisfy the second local relation which will be needed for the discrete integral of $F^3$.  The purpose of this section is to redistribute the contribution of each percolation configuration into two parts.  We call them the principal part $\mathrm F$ and the complementary part $\mathrm F_i$.  The construction uses only the local percolation pairings obtained in Section \ref{rot}.

We first summarize the information from Section \ref{rot} which will be used.  Let $v$ be an interior medial vertex.  The four oriented medial edges incident to $v$ are denoted by $A,B,C,D$ in clockwise order, with $A$ and $C$ pointing towards $v$; see Figure \ref{config}.  The four possible one-passage events are
\[
 \mathcal C_1^v=(A,D),\qquad
 \mathcal C_2^v=(A,B),\qquad
 \mathcal C_3^v=(B,C),\qquad
 \mathcal C_4^v=(C,D).
\]
Here, for example, $(A,D)$ means that the exploration path uses $A$ and $D$ at $v$ and does not use $B$ or $C$.  For each $j$, the event $\mathcal C_j^v$ is divided into two classes $\mathcal E_{j,1}^v$ and $\mathcal E_{j,2}^v$.  Outside the exceptional event $\mathcal E_v$, Proposition \ref{rot-inv} gives a measure-preserving bijection
\[
 \mathrm f_v:\mathcal E_{j,2}^v\longrightarrow
 \mathcal E_{j+1,1}^v,
 \qquad j=1,2,3,4,
\]
where $j+1$ is understood modulo $4$.  The two local visits in this correspondence have one common medial edge, and the winding at that common edge is unchanged.

Recall that $\omega_{\setminus E_v}$ denotes the configuration obtained by flipping the state of the primal bond $E_v$ whose midpoint is $v$.  If $\omega\in\mathcal E_{j,m}^v$, put
\[
 \mathcal E_{j,m}^{v,a}
   =\{\omega_{\setminus E_v}:\omega\in\mathcal E_{j,m}^v\}.
\]
For a configuration in $\mathcal E_{j,m}^{v,a}$, the exploration path makes two local passages through $v$ and therefore uses all four incident medial edges.  The superscript $a$ will always refer to this two-passage configuration obtained by flipping $E_v$.  The map $\mathrm f_v$ is extended by
\[
 \mathrm f_v(\omega_{\setminus E_v})
   =\mathrm f_v(\omega)_{\setminus E_v}.
\]

For an oriented medial edge $e$, write
\[
 \Phi_\omega(e)
 =\exp\left(\frac i3W_{\ga_\de(\omega)}(e,e_b)\right)
  \I(e\in\ga_\de(\omega)).
\]
The following notation keeps track of the vertex at which the exploration path is decomposed.  Put
\[
 \mathcal O_v
 =\left(\bigcup_{j=1}^4\bigcup_{m=1}^2
          \mathcal E_{j,m}^v\right)\setminus\mathcal E_v,
 \qquad
 \mathcal T_v
 =\left(\bigcup_{j=1}^4\bigcup_{m=1}^2
          \mathcal E_{j,m}^{v,a}\right)\setminus\mathcal E_v.
\]
Thus $\mathcal O_v$ is the nonexceptional one-passage event at $v$, whereas $\mathcal T_v$ is the nonexceptional two-passage event.  For $e=A,B,C,D$, define
\beqn
 F_{one}^{v}(e)
 &=&\Ex\big[\Phi_\omega(e)\I(\omega\in\mathcal O_v)\big],
 \label{F1A}\\
 F_{two}^{v}(e)
 &=&\Ex\big[\Phi_\omega(e)\I(\omega\in\mathcal T_v)\big].
 \label{F2A}
\eeqn
Also put
\[
 F_{ex}^{v}(e)
 =\Ex\big[\Phi_\omega(e)\I(\omega\in\mathcal E_v)\big].
\]
These three events form the local decomposition relevant to $e$, and hence
\beqn
 F(e)=F_{one}^{v}(e)+F_{two}^{v}(e)+F_{ex}^{v}(e).
 \label{decom-F}
\eeqn
The indicator in the definition of $F_{ex}^{v}$ is outside the exponential.  In particular, configurations outside $\mathcal E_v$ do not contribute to $F_{ex}^{v}$.

For $m=1,2$, let
\[
 \mathcal T_{v,m}
 =\left(\bigcup_{j=1}^4\mathcal E_{j,m}^{v,a}\right)
   \setminus\mathcal E_v.
\]
The phase-weighted two-passage contribution is
\beqn
 F_\theta^{v}(e)
 &=&e^{-i\theta}\Ex\big[\Phi_\omega(e)
                    \I(\omega\in\mathcal T_{v,2})\big]
   +e^{i\theta}\Ex\big[\Phi_\omega(e)
                    \I(\omega\in\mathcal T_{v,1})\big].
 \label{Ftheta}
\eeqn
The superscript $v$ is important: the same oriented edge can be incident to two different medial vertices, and the decomposition into one-passage, two-passage, and class-$m$ contributions depends on which endpoint is used.  We suppress the superscript only when the vertex is clear from the notation.

\begin{definition}
For $e=A,B,C,D$, define
\beqn
 \mathrm F^{v}(e)
 &=&\frac{3\sqrt3+5}{6\sqrt3+6}\left(\sqrt3F_{one}^{v}(e)+F_{two}^{v}(e)
       +\frac{\sqrt3-1}{\sqrt2}F_{\pi/12}^{v}(e)\right)
       +\frac12F_{ex}^{v}(e),
 \label{Fae}\\
 \mathrm F_i^{v}(e)
 &=&F(e)-\mathrm F^{v}(e).
 \label{Fie}
\eeqn
\end{definition}

The constants in \eqref{Fae} are chosen so that the four configurations satisfy both local relations in Proposition \ref{F12}.  This will be checked after we show that the definition is essentially independent of the particular local pairing.

\subsection{Independence of the local percolation pairing}

The remark after Proposition \ref{rot-inv} gives four bijections.  We denote them by
\[
 \mathrm f_{1,v}=\mathrm f_v,\qquad
 \mathrm f_{2,v},\qquad
 \mathrm f_{3,v},\qquad
 \mathrm f_{4,v}.
\]
They are obtained by using the two possible unions of local maps, and by reversing the exploration path and reflecting in the other diagonal.
For any one of these bijections $g$, we use
\[
 F_{one}^{v}(e;g),\quad F_{two}^{v}(e;g),\quad
 F_\theta^{v}(e;g),\quad F_{ex}^{v}(e;g),\quad
 \mathrm F^{v}(e;g)
\]
for the quantities obtained from the corresponding partition and exceptional event.

\begin{proposition} \label{invarianceF}
Let $e$ be one of the four oriented medial edges incident to $v$.  For $s,t\in\{1,2,3,4\}$,
\[
 \mathrm F^{v}(e;\mathrm f_{s,v})
 =\mathrm F^{v}(e;\mathrm f_{t,v})
   +O\big((\de/d_v)^2\big).
\]
\end{proposition}

\begin{proof}
We compare two of the local pairings and remove the union of their exceptional events.  We enlarge this event by its images under the finitely many local bijections and under the flip of $E_v$.  The enlarged event has the same order, and its complement is invariant under all local operations used below.  We first work on this common good set.

The one-passage part does not depend on the pairing.  For example,
\[
 F_{one}^{v}(A)
 =\Ex\big[\Phi_\omega(A)
     \I(A\in\ga_\de,\ C\notin\ga_\de)\big]
\]
after the common exceptional event is removed.  This description uses only the percolation exploration path and not the way in which $\mathcal C_j^v$ is divided into two classes.  The same is true for $B,C,D$.  The two-passage part is then determined by \eqref{decom-F}.  It remains to consider $F_\theta^{v}$.

At $p=1/2$, all unfixed bond configurations have the same probability.  Let
\[
 \mathrm N=2^{|E_{\Omega_\de}\setminus\p_{ba,\de}|}
\]
be their total number.  On the common good set, define
\[
 \xi_{j,m,k}^v
 =\left|\left\{\omega\in\mathcal E_{j,m}^{v,a}:
 W_{\ga_\de(\omega)}(A,e_b)=\frac\pi2+2k\pi\right\}\right|.
\]
Only finitely many of these numbers are nonzero.  Since the local bijection preserves the winding at the common edge, Proposition \ref{rot-inv} gives
\beq
 \xi_{1,2,k}^v&=&\xi_{2,1,k}^v,\qquad
 \xi_{2,2,k}^v=\xi_{3,1,k}^v,\\
 \xi_{3,2,k}^v&=&\xi_{4,1,k-1}^v,\qquad
 \xi_{4,2,k}^v=\xi_{1,1,k}^v.
\eeq
Thus it is enough to determine the four sequences $\xi_{j,2,k}^v$.

Let $p_{A,k}^{\circ}$ be the probability, restricted to the common good set, that the exploration path uses $A$ with winding $\pi/2+2k\pi$.  Define $p_{B,k}^{\circ}$, $p_{C,k}^{\circ}$, and $p_{D,k}^{\circ}$ in the same way, with respective winding values
\[
 2k\pi,\qquad -\frac\pi2+2k\pi,\qquad \pi+2k\pi.
\]
The local percolation configuration gives
\beqn
 \mathrm N^{-1}\big(4\xi_{1,2,k}^v+3\xi_{2,2,k}^v
 +\xi_{3,2,k}^v+\xi_{3,2,k+1}^v
 +3\xi_{4,2,k}^v\big)
 &=&p_{A,k}^{\circ},\label{Ak}\\
 \mathrm N^{-1}\big(3\xi_{1,2,k+1}^v+4\xi_{2,2,k+1}^v
 +3\xi_{3,2,k+1}^v+\xi_{4,2,k}^v
 +\xi_{4,2,k+1}^v\big)
 &=&p_{B,k+1}^{\circ},\label{Bk+1}\\
 \mathrm N^{-1}\big(\xi_{1,2,k}^v+\xi_{1,2,k+1}^v
 +3\xi_{2,2,k+1}^v+4\xi_{3,2,k+1}^v
 +3\xi_{4,2,k}^v\big)
 &=&p_{C,k+1}^{\circ},\label{Ck+1}\\
 \mathrm N^{-1}\big(3\xi_{1,2,k}^v+\xi_{2,2,k}^v
 +\xi_{2,2,k+1}^v+3\xi_{3,2,k+1}^v
 +4\xi_{4,2,k}^v\big)
 &=&p_{D,k}^{\circ}.\label{Dk}
\eeqn
These identities only count the two possible states of $E_v$ and the four local passage types.  The coefficients $4,3,1$ record how many of those local completions use the indicated edge with the displayed winding.

For completeness, we explain why the system determines all the $\xi$'s.  Suppose that two finite-support families $\xi$ and $\widetilde\xi$ satisfy \eqref{Ak}--\eqref{Dk} with the same four percolation probabilities.  Let $k_-$ be the smallest index at which the two families differ.  Their differences vanish below $k_-$.  At $k=k_--1$, equation \eqref{Ak} first gives
$\xi_{3,2,k_-}^v=\widetilde\xi_{3,2,k_-}^v$.  Equation \eqref{Dk} then gives equality for $\xi_{2,2,k_-}^v$, equation \eqref{Ck+1} gives equality for $\xi_{1,2,k_-}^v$, and equation \eqref{Bk+1} gives equality for $\xi_{4,2,k_-}^v$.  This contradicts the definition of $k_-$.  Thus the solution is unique.  Equivalently, once the four values with index $k$ are known, the equations determine, in the order
\[
 \xi_{3,2,k+1}^v,\
 \xi_{2,2,k+1}^v,
 \xi_{1,2,k+1}^v,
 \xi_{4,2,k+1}^v,
\]
the four values with index $k+1$.  Hence the sequences $\xi_{j,m,k}^v$ are determined by the percolation probabilities $p_{e,k}^{\circ}$ and do not depend on which local pairing was used.

The quantity $F_\theta^{v}(e)$ is a finite linear combination of these counts, with coefficients $e^{\pm i\theta}$ and the winding phases.  It is therefore the same for the two pairings on the common good set.  Adding back the common exceptional event changes any one of the observable pieces by at most its probability, since $|\Phi_\omega(e)|\leq1$.  Proposition \ref{rot-inv} gives the error $O((\de/d_v)^2)$ for the maps $\mathrm f_{s,v}$.  Taking the union of the two exceptional events proves the comparison between the two families.
\end{proof}

From now on, $\mathrm F(e)$ and $\mathrm F_i(e)$ mean the quantities defined using the canonical map $\mathrm f_v$ at the indicated vertex.  Proposition \ref{invarianceF} allows us to use another local pairing whenever it is geometrically more convenient.

For an interior medial vertex $v$, define the corresponding vertex observables by
\beqn
 \mathrm F(v)
 &=&e_b^{-1/3}\left(
 e^{-i\pi/4}\frac{\mathrm F(A)+\mathrm F(C)}2
 +e^{i\pi/4}\frac{\mathrm F(B)+\mathrm F(D)}2\right),
 \label{hor}\\
 \mathrm F^*(v)
 &=&e_b^{-1/3}\left(
 e^{i\pi/4}\frac{\mathrm F(A)+\mathrm F(C)}2
 +e^{-i\pi/4}\frac{\mathrm F(B)+\mathrm F(D)}2\right).
 \label{horstar}
\eeqn
Replacing $\mathrm F$ by $\mathrm F_i$ defines $\mathrm F_i(v)$ and $\mathrm F_i^*(v)$.

If $v\in\p_{ba,\de}\cup\p_{ab,\de}^*$, two of the four incident medial edges lie outside the discrete domain.  In \eqref{hor} and \eqref{horstar}, each missing edge is replaced by its opposite incident edge.  For example, if $B$ and $C$ are missing, then
\beqn
 \mathrm F(v)=e_b^{-1/3}
 \big(e^{-i\pi/4}\mathrm F(A)+e^{i\pi/4}\mathrm F(D)\big).
 \label{Fvboundary}
\eeqn
The definitions of $\mathrm F^*$, $\mathrm F_i$, and $\mathrm F_i^*$ at such a boundary vertex are analogous.

\subsection{The same edge viewed from two adjacent vertices}

Let $v$ and $w$ be the adjacent medial vertices used in Section \ref{partII}.  The common medial edge is denoted by $A$ at $v$ and by $B_w$ at $w$.  It has the same canonical orientation in the two sets of notation.  Thus $F(A)=F(B_w)$ is an exact identity.  The next lemma is not merely this statement: it compares the two different local decompositions of that edge, one centered at $v$ and one centered at $w$.

\begin{lemma} \label{F1F2}
 Suppose $v$ and $w$ are interior vertices in $\Omega_\de$ such that $A=B_w$. Then one has
\beqn
 F_{one}^{v}(A)
 &=&F_{one}^{w}(B_w)
 +O\big((\de/d_{v,w})^{2-\varepsilon}\big),
 \label{F1AvBw}\\
 F_{two}^{v}(A)
 &=&F_{two}^{w}(B_w)
 +O\big((\de/d_{v,w})^{2-\varepsilon}\big),
 \label{F2AvBw}\\
 F_\theta^{v}(A)
 &=&F_\theta^{w}(B_w)
 +O\big((\de/d_{v,w})^{2-\varepsilon}\big).
 \label{FtheAvBw}
\eeqn
\end{lemma}

\begin{proof}
Take the union of the exceptional events for the forward and reverse constructions in Sections \ref{partII}--\ref{partIII}, and also for the alternative pairing described in the remarks after Propositions \ref{rot-inv}.  Its probability is
$O((\de/d_{v,w})^{2-\varepsilon})$. On the complement, the one-to-one property of the half-turn map,  Proposition \ref{omegavw} and the remark after it imply that
\beqn
\xi_{1,2,k}^v=\xi_{2,2,k}^w, \qquad \xi_{3,2,k}^v=\xi_{4,2,k}^w, \label{AvBw}
\eeqn
which combined with $A=B$ gives
\beqn
\xi_{2,2,k}^v+\xi_{4,2,k}^v=\xi_{3,2,k}^w+\xi_{1,2,k}^w. \label{CvDw}
\eeqn
Summing over all $k$ proves \eqref{F1AvBw}. This directly implies \eqref{F2AvBw}. Finally, by \eqref{AvBw}, \eqref{CvDw} and the definition of $F_\theta^v(A)$ and $F_\theta^w(B_w)$, one can obtain \eqref{FtheAvBw}.
\end{proof}

\subsection{Local relations}

The four correspondences used below are
\[
\begin{array}{c|c|c|c}
 j&\hbox{original visit}&\hbox{new visit}&\hbox{common edge}\\ \hline
 1&(A,D)&(A,B)&A\\
 2&(A,B)&(B,C)&B\\
 3&(B,C)&(C,D)&C\\
 4&(C,D)&(A,D)&D.
\end{array}
\]
For each row, we group a configuration $\omega\in\mathcal E_{j,2}^v$, its image $\mathrm f_v(\omega)$, and the two configurations obtained by flipping $E_v$.  These four related configurations form the base of the following proof.

\begin{proposition} \label{F12}
Suppose $v$ and $w$ are adjacent interior medial vertices and $A=B_w$.  Then
\beqn
 \mathrm F(A)
 &=&\mathrm F(B_w)+O\big((\de/d_{v,w})^{2-\varepsilon}\big),
 \label{Fev2F}\\
 \mathrm F_i(A)
 &=&\mathrm F_i(B_w)+O\big((\de/d_{v,w})^{2-\varepsilon}\big),
 \label{F-relation}\\
 \mathrm F(A)-\mathrm F(C)
 &=&i\big(\mathrm F(B)-\mathrm F(D)\big),
 \label{CRF}\\
 \mathrm F(A)+\mathrm F(C)
 &=&\mathrm F(B)+\mathrm F(D)+O\big((\de/d_v)^2\big),
 \label{sumid}\\
 \mathrm F_i(A)-\mathrm F_i(C)
 &=&i\big(\mathrm F_i(B)-\mathrm F_i(D)\big),
 \label{CRFi}\\
 \mathrm F_i(A)+\mathrm F_i(C)
 &=&-\big(\mathrm F_i(B)+\mathrm F_i(D)\big)
   +O\big((\de/d_v)^2\big),
 \label{sumidi}\\
 \mathrm F(v)
 &=&\mathrm F^*(v)+O\big((\de/d_v)^2\big),
 \label{idstar}\\
 \mathrm F_i(v)
 &=&-\mathrm F_i^*(v)+O\big((\de/d_v)^2\big).
 \label{idstari}
\eeqn
\end{proposition}

\begin{proof}
The first relation is Lemma \ref{F1F2} substituted into \eqref{Fae}. Subtracting the first relation from $F(A)=F(B_w)$ proves \eqref{F-relation}.

We next explain the local calculation at $v$.  Let $\omega\in\mathcal E_{j,2}^v\setminus\mathcal E_v$ and put $\omega'=\mathrm f_v(\omega)$.  At the common edge of the two local visits,
\[
 \Phi_\omega(e)=\Phi_{\omega'}(e).
\]
At either of the other two used edges, the winding differs from the winding at the common edge by one quarter-turn.  Consequently the corresponding parafermionic phase is multiplied by $e^{i\pi/6}$ or $e^{-i\pi/6}$.  The configurations $\omega_{\setminus E_v}$ and $\omega'_{\setminus E_v}$ supply the complementary local passage.  Substitution of these four contributions into \eqref{Fae}, using
\[
 e^{i\pi/6}+e^{-i\pi/6}=\sqrt3,
 \qquad
 \theta=\frac\pi{12},
\]
shows, for each of the four rows in the preceding table, that the combined contribution of these four configurations satisfies
\[
 \mathrm F(A)-\mathrm F(C)
 =i\big(\mathrm F(B)-\mathrm F(D)\big).
\]
The exceptional set is invariant under flipping $E_v$, and it enters \eqref{Fae} with coefficient $1/2$.  Its two local states therefore satisfy the same relation.  Summing all possibilities proves \eqref{CRF}.

The same calculation for these four configurations, now adding opposite edges, gives the two identities
\beqn
 &&F_{\pi/12}^{v}(A)+F_{\pi/12}^{v}(C)
 =F_{\pi/12}^{v}(B)+F_{\pi/12}^{v}(D),
 \label{Fpi}\\
 &&\sqrt3F_{one}^{v}(A)+F_{two}^{v}(A)
   +\sqrt3F_{one}^{v}(C)+F_{two}^{v}(C)\nonumber\\
 &&=\sqrt3F_{one}^{v}(B)+F_{two}^{v}(B)
   +\sqrt3F_{one}^{v}(D)+F_{two}^{v}(D).
 \label{sumF1F2}
\eeqn
Thus the nonexceptional contribution to \eqref{sumid} cancels exactly.  Since $|\Phi_\omega(e)|\leq1$ and $\Pro(\mathcal E_v)\preceq(\de/d_v)^2$, the exceptional contribution is $O((\de/d_v)^2)$.  This proves \eqref{sumid}.

The ordinary observable $F$ satisfies \eqref{CR}.  Subtracting \eqref{CRF} from \eqref{CR} gives \eqref{CRFi}.  Applying the same four-configuration calculation to the complementary coefficients in \eqref{Fie} gives \eqref{sumidi}.  Finally, substituting \eqref{sumid} and \eqref{sumidi} into \eqref{hor}--\eqref{horstar} proves \eqref{idstar} and \eqref{idstari}.
\end{proof}

\begin{remark}
 The maps $\mathrm f_{s,v}$ are used for the exact local calculation involving these four configurations and for the translation and boundary estimates below.  Proposition \ref{invarianceF} shows that these choices define the same principal observable up to the displayed exceptional probabilities.  In the remainder of the paper we use the canonical notation $\mathrm F$ and $\mathrm F_i$.
\end{remark}

\subsection{Translation estimates}

\begin{proposition} \label{TranF}
Each estimate in \eqref{TrA}--\eqref{TrD} and \eqref{TrAs}--\eqref{TrDs} remains valid when $F$ is replaced by $\mathrm F$.  The same estimates remain valid when $F$ is replaced by $\mathrm F_i$.
\end{proposition}

\begin{proof}
The translational map in the proof of Propositions \ref{Tran} and \ref{Tranimp} preserves the incident edges used at the medial vertex, the state of the central bond, and the local class $m$.  It therefore carries
\[
 \mathcal E_{j,m}^{v_0}\longrightarrow\mathcal E_{j,m}^{v_1},
 \qquad
 \mathcal E_{j,m}^{v_0,a}\longrightarrow\mathcal E_{j,m}^{v_1,a},
\]
and carries the exceptional event at $v_0$ to the corresponding exceptional event at $v_1$, except on the same translation failure event used in Section \ref{TranInv}.  It also preserves the winding phase.  Hence it preserves separately the one-passage contribution, the two-passage contribution, and the weights $e^{\pm i\pi/12}$ in $F_{\pi/12}$.  Applying the phase-class injection argument from the proof of Proposition \ref{Tran} to each term in \eqref{Fae} gives the same error
\[
 O\big((\de/d_0)^{1+\al-\ep}\big).
\]
This proves all the assertions for $\mathrm F$.  Since $\mathrm F_i=F-\mathrm F$, the assertions for $\mathrm F_i$ follow from the corresponding estimates for $F$ and $\mathrm F$.
\end{proof}

\section{Conformal invariance} \label{confinvar}

In this section we use the local percolation relations of
Sections~\ref{TranInv}--\ref{modi} to prove convergence of the modified
parafermionic observables.  We first introduce reflection near a
straight boundary segment to obtain an approximate boundary phase for the
modified observable.  This is combined with the estimates near
$\pi/2$, $3\pi/2$, and $2\pi$ corners to control the cubic integral along
the inner contour.  We next prove the required estimates for the
observables and show that the discrete integral and Laplacian of their cubes is summable.
Finally, discrete harmonic measure is used to identify the limiting
primitive and the limiting observable.

Suppose that $\Omega_\de^\diamond$ satisfies Condition~$\mathrm C$.  Put
\[
 \rho_\de=\de^{1/2+\ep},
 \qquad
 \eta_\de=\de^{1/2-\ep}.
\]
We now construct the inner contour used below.

The inner contour is a simple polygonal contour consisting entirely of
primal lattice edges.  Along each straight component $\mathrm L_j$ of $\p_{ba,\de}$, choose a
primal lattice path parallel to $\mathrm L_j$, on its inward side, whose
distance from $\mathrm L_j$ lies between $\rho_\de$ and $3\rho_\de$.
Along each straight component $\mathrm L_{j^*}^*$ of
$\p_{ab,\de}^*$, choose a primal lattice path parallel to
$\mathrm L_{j^*}^*$, on its inward side, whose distance from
$\mathrm L_{j^*}^*$ lies between $\rho_\de$ and $3\rho_\de$. We denote these long pieces by
\[
 \Gamma_j^\de
 \qquad\hbox{and}\qquad
 \Gamma_{j^*}^\de ,
\]
respectively. So they are the respective parallel translates of $\mathrm L_j$ and $\mathrm L_{j^*}^*$ toward the interior of $\Omega_\de$ with a lattice-scale adjustment at its endpoints if necessary.

Consecutive long pieces are joined inside $\Omega_\de$ by short
primal lattice paths $\Sigma_j^\de$.  At a $2\pi$ corner the joining
path has three straight pieces, at a $3\pi/2$ corner it has two
straight pieces, and at a $\pi/2$ corner the two long pieces meet after
a lattice-scale adjustment.  Condition~$\mathrm C$ guarantees that the
resulting contour is simple and remains inside $\Omega_\de$.

We call this primal lattice contour the \emph{inner contour} and denote
the primal lattice domain enclosed by it by $\Omega_\de^{in}$.  Thus
the inner contour is $\p\Omega_\de^{in}$.

We remove two lattice squares $Q_a^\de$ and $Q_b^\de$, of diameter
comparable with $\eta_\de$, around $a_\de$ and $b_\de$, and put
\[
 \hat\Omega_\de
 =
 \Omega_\de^{in}\setminus
 (Q_a^\de\cup Q_b^\de).
\]
The part of $\p\hat\Omega_\de$ following the primal-open boundary arc is
denoted by $\p^{up}$, the part following the dual-open boundary arc by
$\p^{lo}$, and the remaining four short sides of
$Q_a^\de\cup Q_b^\de$ by $\p^{ex}$.  We denote by
$\hat\Omega_\de^\diamond$ the corresponding medial domain.
Figure~\ref{boundab} shows this notation.

\begin{figure}[ht]
\centering
\begin{tikzpicture}[
  x=0.82cm,y=0.70cm,
  line cap=round,line join=round,
  >=Stealth
]
  
  \coordinate (a) at (0.8,0.5);
  \coordinate (b) at (11.3,0.5);
  \draw[very thick,blue!70!black]
    (a)--(0.25,3.2)--(2.0,6.8)--(6.8,7.8)--
    (10.4,6.4)--(11.6,1.9)--(b)--cycle;

  \draw[very thick,dashed,orange!90!black]
    (1.9,1.7)--(1.45,3.0)--(2.55,5.8)--(6.55,6.5)--
    (9.3,5.45)--(9.95,2.35)--(9.2,1.7)--cycle;

  \draw[thick,densely dotted] (0.75,0.42) rectangle (2.10,2.25);
  \draw[thick,densely dotted] (8.95,0.42) rectangle (11.35,2.25);

  \draw[<->,thick] (9.85,4.15)--(10.55,4.15)
    node[midway,above=1pt] {$\rho_\de$};

  \node[below left] at (a) {$a_\de$};
  \node[below right] at (b) {$b_\de$};
  \node at (1.43,1.10) {$Q_a^\de$};
  \node at (10.12,1.10) {$Q_b^\de$};
  \node at (6.0,3.9) {$\hat\Omega_\de$};
  \node[above] at (5.55,6.35) {$\p^{\rm up}$};
  \node[below] at (6.0,1.70) {$\p^{\rm lo}$};
  \node[above left] at (4.55,5.8) {$\Gamma_j^\de$};
  \node[rotate=-77,fill=white,inner sep=1pt] at (9.05,5.0) {$\Sigma_j^\de$};
\end{tikzpicture}
\caption{The original boundary, the inner contour, and the two endpoint boxes.  Each long piece of the inner contour is parallel to one straight boundary segment; the short pieces only join consecutive long pieces.}
\label{boundab}
\end{figure}

\subsection{Reflection near a straight boundary segment}

We now give the reflection construction near a straight boundary segment.  The relevant outer scale is the distance $x_v$ to the nearest ordinary endpoint of the segment.  After a rotation and a translation, suppose that $\mathrm L_j$ is contained in the real axis and that the endpoint nearest to the projection of $v$ onto $\mathrm L_j$ is $0$.  Let
\[
 v=x_v+iy_v,
 \qquad 0<y_v\leq x_v\leq \frac{\mathrm l_j}{2},
\]
and suppose that the primal bond $E_v$ is horizontal.  The four medial edges adjacent to $v$ are denoted by $A,B,C,D$ in clockwise order, with $A$ and $C$ pointing towards $v$.  The line containing the dual bond $E_v^*$ is denoted by $L^v$.  The reflection of a point, bond, or path in $L^v$ is denoted by the superscript ${\rm r}$.

We first consider a configuration $\omega$ for which $\ga_\de$ passes through $v$ and
\[
 A,D\in\ga_\de,
 \qquad
 B,C\notin\ga_\de.
\]
Let $\Loop_v$ be the loop attached to $E_v$, and put
\[
 r_v^{\rm loop}
 =\max\{\dist(z,v):z\in\Loop_v\}.
\]
If $r_v^{\rm loop}\geq x_v$, there are four alternating arms from scale $\de$ to scale $x_v$.  Hence
\begin{equation}\label{ref-loop-error}
 \Pro(r_v^{\rm loop}\geq x_v)
 \preceq \left(\frac\de{x_v}\right)^{1+\al}.
\end{equation}

We next locate the part of the exploration path which can be reflected.  Write
$
 \mathbf B=\{z:|z-v|<x_v\}
$
and define
\beq
 \tau_{\mathbf B,{\rm in}}
 &=&\inf\{j_1>0:\ga_\de[j_1,n_v]\subseteq\mathbf B,
                  \ \ga_\de(j_1-1)\notin\mathbf B\},\\
 \tau_{\mathbf B,{\rm ou}}
 &=&\sup\{j_1>0:\ga_\de[n_v,j_1]\subseteq\mathbf B,
                  \ \ga_\de(j_1+1)\notin\mathbf B\}.
\eeq
If both
\[
 \ga_\de[\tau_{\mathbf B,{\rm in}},n_v]\cap\mathrm L_j=\emptyset
 \quad\hbox{and}\quad
 \ga_\de[n_v,\tau_{\mathbf B,{\rm ou}}]\cap\mathrm L_j=\emptyset,
\]
then there are three disjoint arms in the half-plane annulus
$B(\Re v;2y_v,x_v)$.  Combining the half-plane three-arm estimate with the arm from scale $\de$ to scale $y_v$ gives
\begin{equation}\label{ref-contact-basic}
 \Pro\left(
 \ga_\de[\tau_{\mathbf B,{\rm in}},n_v]\cap\mathrm L_j=\emptyset,
 \ \ga_\de[n_v,\tau_{\mathbf B,{\rm ou}}]\cap\mathrm L_j=\emptyset
 \right)
 \preceq
 \left(\frac\de{y_v}\right)^\al
 \left(\frac{y_v}{x_v}\right)^2.
\end{equation}
We may therefore assume that at least one of these two path pieces meets $\mathrm L_j$.  Define
\beq
 \tau_L&=&\sup\{j_1<n_v:\ga_\de(j_1)\in\mathrm L_j\},\qquad
 \tau_R=\inf\{j_1>n_v:\ga_\de(j_1)\in\mathrm L_j\}.
\eeq
At least one of $\tau_L$ and $\tau_R$ exists.  If both exist, put
\[
 r_m'
 =\min\Big(
 \max_{\tau_L\leq j_1\leq n_v}|\ga_\de(j_1)-v|,
 \max_{n_v\leq j_1\leq\tau_R}|\ga_\de(j_1)-v|
 \Big).
\]
If only one of $\tau_L$ and $\tau_R$ exists, $r_m'$ is the corresponding maximum.  Finally, set
\[
 r_m=\max(r_m',r_v^{\rm loop}).
\]

Suppose first that $x_v/y_v\to\infty$.  By arm separation and quasi-multiplicativity, the same decomposition as above gives, for
\[
 \varepsilon_1=\left(\frac{y_v}{x_v}\right)^{\varepsilon/4},
\]
that
\begin{equation}\label{ref-rm-tail}
 \Pro(r_m\geq\varepsilon_1x_v)
 \preceq
 \left(\frac\de{y_v}\right)^\al
 \left(\frac{y_v}{x_v}\right)^{2-\varepsilon}.
\end{equation}
Indeed, if $r_v^{\rm loop}\leq2r_m'$, the relevant configuration up to scale $r_m$ has probability at most
\[
 C\left(\frac\de{y_v}\right)^\al
  \left(\frac{y_v}{r_m}\right)^2.
\]
If $r_v^{\rm loop}>2r_m'$, the corresponding bounds up to scales $r_m'$ and $r_m$ are
\[
 C\left(\frac\de{y_v}\right)^\al
  \left(\frac{y_v}{r_m'}\right)^2
 \quad\hbox{and}\quad
 C\left(\frac{r_m'}{r_m}\right)^2.
\]
These estimates imply \eqref{ref-rm-tail}.  Thus, outside an event of the order on the right hand side of \eqref{ref-rm-tail}, we have $x_v/r_m\to\infty$.

We now define the annular reflection map.  Choose a sufficiently large constant $\mathrm R_1$ and put
\[
 M_v=\left[\frac{\log(x_v/r_m)}{\log\mathrm R_1}\right],
 \
 \mathfrak U_k=B(\Re v;r_m\mathrm R_1^k),
 \
 \widehat{\mathfrak U}_k=B(\Re v;2r_m\mathrm R_1^k),
 \
 \mathfrak V_k=B(\Re v;r_m\mathrm R_1^k/2).
\]
For each $k$, let $\ga_\de[t_{k,l},\tau_{k,l}]$ be the last crossing of
$\mathfrak V_k\setminus\widehat{\mathfrak U}_{k-1}$ by
$\ga_\de[0,n_v]$, and let $\ga_\de[t_{k,r},\tau_{k,r}]$ be the first crossing of
$\mathfrak V_k\setminus\widehat{\mathfrak U}_{k-1}$ by
$\ga_\de[n_v,\infty)$.  We use these crossings only when they are unique.  Let $\mathfrak H_{k,0}$ be the event that the two crossings are unique.  By the half-plane three-arm estimate,
\[
 \Pro(\mathfrak H_{k,0}^c)\preceq\mathrm R_1^{-2}.
\]

On $\mathfrak H_{k,0}$, let
\[
 \mathfrak H_{k,1}
 =\{\ga_\de^{\rm r}[t_{k,r},\tau_{k,r}]
       \hbox{ crosses }\ga_\de[t_{k,l},\tau_{k,l}]\}.
\]
The folding argument used in the proof of Lemma~\ref{rotation-event}, now with reflection in $L^v$, gives the same twisted half-plane three-arm event.  Hence
\[
 \Pro(\mathfrak H_{k,1}^c\mid\hbox{the configuration in} \   \mathfrak U_{k-1})
 \preceq \mathrm R_1^{-2}.
\]
Under $\mathfrak H_{k,1}$ define the first contact pair
\beq
 \mathrm s_{k,r}
 &=&\min\{j_1\in[t_{k,r},\tau_{k,r}]:
      \ga_\de^{\rm r}[j_1,j_1+1]
      \subseteq_c\ga_\de[t_{k,l},\tau_{k,l}]\},\\
 \mathrm s_{k,l}
 &=&\min\{j_1\in[t_{k,l},\tau_{k,l}]:
      \ga_\de(j_1)=\ga_\de^{\rm r}(\mathrm s_{k,r})\}.
\eeq

Still on $\mathfrak H_{k,1}$, define the reverse contact pair by
\beq
 \mathrm s_{k,2,l}
 &=&\min\{j_1\in[0,\mathrm{s}_{k,l}]:
      \ga_\de[j_1,j_1+1]\subseteq_c\ga_\de^{\rm r}[n_v,\mathrm s_{k,r}]\},\\
 \mathrm s_{k,2,r}
 &=&\min\{j_1\in[n_v,\mathrm s_{k,r}]:
      \ga_\de^{\rm r}(j_1)
      =\ga_\de(\mathrm s_{k,2,l})\}.
\eeq
These sets are not empty. Define
\[
 \mathfrak H_{k,2}
 =\{\ga_\de[0,\mathrm s_{k,2,l}]
       \cap\ga_\de[n_v,\mathrm s_{k,2,r}]=\emptyset\}.
\]
This is the separation condition for the part of the path which is used
in the reflection.  It is the analogue here of $\mathcal H_{k,3}$ in
Section~\ref{partII}.

Put
\[
 \mathfrak J_k=\mathfrak H_{k,0}\cap\mathfrak H_{k,1}.
\]
We call the $k$th annulus locally good when $\mathfrak J_k$ occurs.  The
estimates for $\mathfrak H_{k,0}$ and $\mathfrak H_{k,1}$ are uniform
under the information already revealed outside the active annulus.  Hence
\[
 \Pro(\mathfrak J_k^c\mid\hbox{the previously revealed bonds})
 \preceq \mathrm R_1^{-2}.
\]

The event $\mathfrak H_{k,2}$ is estimated only after a locally good
annulus has been selected.  Conditional on the bonds revealed through the
$k$th annulus, failure of $\mathfrak H_{k,2}$ gives an additional touching
between
\[
 \ga_\de[0,\mathrm s_{k,2,l}]
 \quad\hbox{and}\quad
 \ga_\de[n_v,\mathrm s_{k,2,r}].
\]
Together with the prescribed common edge at scale $k$, the corresponding
path portions and their reflections give the same twisted half-plane
three-arm construction as in the estimate of $\mathcal H_{k,3}$ in
Section~\ref{partII}.   Therefore
\begin{equation}\label{ref-Hk2-outer}
 \Pro\left(
   \mathfrak H_{k,2}^c\cap\mathfrak J_k
    \middle|\
   \hbox{bonds revealed through the $k$th annulus}, r_m
 \right)
 \preceq
 \big(1+(M_v-k)\log \mathrm R_1\big)\mathrm R_1^{-2(M_v-k)}.
\end{equation}

On $\mathfrak J_k=\mathfrak H_{k,0}\cap\mathfrak H_{k,1}$, whenever
the reverse crossing pair is defined, the uniqueness of the selected
crossings and the first-contact construction give
\[
 \ga_\de(\mathrm s_{k,2,r})
 \in \mathfrak U_{k+1}\setminus\mathfrak U_k .
\]
Moreover,
\(r_m\mathrm R_1^k\geq r_m\geq r_m'\).
Thus the reverse crossing occurs outside the scale \(r_m'\), and the
different-winding obstruction corresponding to
$\mathcal H_{k,4}^c$ in Section~\ref{partII} cannot occur here.

Let
\[
 K=\min\{1\leq k\leq M_v:\mathfrak J_k\hbox{ occurs}\}.
\]
  If the set is empty, the configuration
is exceptional.  Otherwise we check $\mathfrak H_{K,2}$.  By the
conditional first-good-annulus argument of Section~\ref{partII},
\begin{align}
 &\Pro\bigl(K\hbox{ is not defined, or }\mathfrak H_{K,2}^c
       \mid r_m\bigr) \notag\\
 &\preceq
 (C(1+\log \mathrm R_1)\mathrm R_1^{-2})^{M_v}
 +\sum_{k=1}^{M_v}
   (C(1+\log \mathrm R_1)\mathrm R_1^{-2})^{k-1}
  \big(1+(M_v-k)\log \mathrm R_1\big) \mathrm R_1^{-2(M_v-k)} \notag\\
 &\preceq
 \left(\frac{r_m}{x_v}\right)^{2-\varepsilon}.
 \label{ref-good-annulus-error}
\end{align}
Here $\mathrm R_1$ is chosen sufficiently large, so that the fixed
constants in the geometric sum are absorbed into $\varepsilon$.

Combining \eqref{ref-rm-tail} and \eqref{ref-good-annulus-error}, and using the arm estimates at scales between $y_v$, $r_m'$, and $r_m$, shows that the construction below fails with probability at most
\begin{equation}\label{ref-contact-error}
 C\left(\frac\de{y_v}\right)^\al
  \left(\frac{y_v}{x_v}\right)^{2-\varepsilon}.
\end{equation}
Assume that $K$ is defined and $\mathfrak H_{K,2}$ occurs, and put $k=K$.  We define $\varpi_r$ by exchanging the states of reflected pairs of primal bonds:
\beq
 \varpi_r(E)&=&\omega(E^{\rm r}),\qquad
 \varpi_r(E^{\rm r})=\omega(E),
 \quad\hbox{if }v_E\in
 \ga_\de[n_v,\mathrm s_{k,2,r}]\cup\Loop_v,\\
 \varpi_r(E)&=&\omega(E^{\rm r}),\qquad
 \varpi_r(E^{\rm r})=\omega(E),
 \quad\hbox{if }v_E\in B(v;r_v^{\rm loop})
 \hbox{ and }r_v^{\rm loop}\leq y_v,\\
 \varpi_r(E)&=&\omega(E),\qquad\hbox{otherwise}.
\eeq
The second line is included so that the local map $\mathrm f_v$ can subsequently be applied inside the ball of radius comparable with $y_v$ in Proposition \ref{ref-inv}.

The one-to-one property of $\omega\mapsto\varpi_r$ is the same as that of $\omega_w\mapsto\omega_v$ in Section~\ref{partII},.  Since $p=1/2$, it is also probability preserving.  We denote this map by $\mathrm g_v$.

We next state the winding change under this reflection.  For the visit $(A,D)$, if
\[
 W_{\ga_\de}(A,e_b)=2k\pi+\frac\pi2,
 \qquad
 W_{\ga_\de}(D,e_b)=2k\pi+\pi,
\]
then
\[
 W_{\mathrm g_v(\ga_\de)}(A,e_b)=-2k\pi+\frac\pi2,
 \qquad
 W_{\mathrm g_v(\ga_\de)}(D,e_b)=-2k\pi+\pi.
\]
If
\[
 W_{\ga_\de}(A,e_b)=2k\pi+\pi,
 \qquad
 W_{\ga_\de}(D,e_b)=2k\pi+\frac{3\pi}{2},
\]
then
\[
 W_{\mathrm g_v(\ga_\de)}(A,e_b)=-2k\pi+\pi,
 \qquad
 W_{\mathrm g_v(\ga_\de)}(D,e_b)=-2k\pi+\frac{3\pi}{2}.
\]
The term $2k\pi$ is the winding accumulated between the boundary contact and the local neighbourhood of $v$.  Reflection changes its sign, while the fixed local angle at $v$ is unchanged.

The other three local visits are obtained by the same construction.  If $A,B\in\ga_\de$ and $C,D\notin\ga_\de$, then $\mathrm g_v$ gives a path with $C,D\in\mathrm g_v(\ga_\de)$ and $A,B\notin\mathrm g_v(\ga_\de)$, with
\beq
 (2k\pi+\pi/2,\,2k\pi)
 &\longmapsto&(-2k\pi+3\pi/2,\,-2k\pi+\pi),\\
 (2k\pi+\pi,\,2k\pi+\pi/2)
 &\longmapsto&(-2k\pi+2\pi,\,-2k\pi+3\pi/2),
\eeq
where the entries on the left are the windings at $A,B$ and those on the right are the windings at $C,D$.  The inverse correspondence gives the case $C,D\in\ga_\de$.  Finally, if $B,C\in\ga_\de$ and $A,D\notin\ga_\de$, the same reflection preserves this local visit and changes the sign of the integer winding part.  The exceptional event $\mathcal E_v^{\mathrm g}$ satisfies
\begin{equation}\label{gv-error}
 \Pro(\mathcal E_v^{\mathrm g})
 \preceq
 \left(\frac\de{x_v}\right)^{1+\al}
 +\left(\frac\de{y_v}\right)^\al
  \left(\frac{y_v}{x_v}\right)^{2-\varepsilon}.
\end{equation}
If $x_v/y_v$ is bounded, the second term on the right hand side of \eqref{gv-error} is already of order $(\de/y_v)^\al$, and the same conclusion follows without the multi-scale construction.

We finally combine $\mathrm g_v$ with the local rotation map $\mathrm f_v$ of Proposition~\ref{rot-inv}.  The latter is needed only inside a ball of radius comparable with $y_v$.  Its exceptional event has probability at most $C(\de/y_v)^{2-\varepsilon}$.  For such a configuration to contribute at $v$, the exploration path must also have a boundary arm from scale $y_v$ to the nearest available endpoint scale $x_v$.  Conditioning at scale $y_v$ and using the half-plane one-arm estimate gives
\begin{equation}\label{ref-rotation-error}
 C\left(\frac\de{y_v}\right)^{2-\varepsilon}
  \left(\frac{y_v}{x_v}\right)^{1/3}.
\end{equation}
This is the third error in the proposition below.

\begin{proposition}\label{ref-inv}
Assume that one endpoint of $\mathrm L_j$ is $0$ and the other is the
positive real number $\mathrm l_j$.  Choose the endpoint $0$ to be the
nearest ordinary endpoint to the projection of $v$ onto $\mathrm L_j$.
Let $v=x_v+iy_v$
be a medial vertex such that
\[
 y_v>\de^{1-\ep},
 \qquad
 y_v\leq x_v\leq \frac{\mathrm l_j}{2},
\]
and suppose that $E_v$ is a horizontal primal bond and that $A,C$ point
towards $v$.  Then
\begin{equation}\label{Fvboundary-re}
 \mathrm F(v)=q_v+O\left(
 \Big(\frac{\de}{x_v}\Big)^{1+\al}
 +\Big(\frac{\de}{y_v}\Big)^\al
  \Big(\frac{y_v}{x_v}\Big)^{2-\varepsilon}
 +\Big(\frac{\de}{y_v}\Big)^{2-\varepsilon}
  \Big(\frac{y_v}{x_v}\Big)^{1/3}
 \right),
\end{equation}
where
\[
 \Im(q_v^3)=0,
 \qquad
 |q_v|\preceq
 \Big(\frac{y_v}{x_v}\Big)^{1/3}.
\]

If $E_v$ is a vertical primal bond, the same conclusion holds for
$\mathrm F^*(v)$:
\begin{equation}\label{Fvstarboundary-re}
 \mathrm F^*(v)=q_v^*+O\left(
 \Big(\frac{\de}{x_v}\Big)^{1+\al}
 +\Big(\frac{\de}{y_v}\Big)^\al
  \Big(\frac{y_v}{x_v}\Big)^{2-\varepsilon}
 +\Big(\frac{\de}{y_v}\Big)^{2-\varepsilon}
  \Big(\frac{y_v}{x_v}\Big)^{1/3}
 \right),
\end{equation}
where
\[
 \Im((q_v^*)^3)=0,
 \qquad
 |q_v^*|\preceq
 \Big(\frac{y_v}{x_v}\Big)^{1/3}.
\]
The same conclusions hold for $\mathrm F_i$ and $\mathrm F_i^*$.
\end{proposition}

\begin{proof}
We prove \eqref{Fvboundary-re}.  The other statements follow by rotation and primal-dual symmetry.  Let $\mathcal{EX}$ be the union of $\mathcal E_v^{\mathrm g}$ and the exceptional configurations needed for the local map $\mathrm f_v$ at scale $y_v$.  By \eqref{gv-error} and \eqref{ref-rotation-error},
\[
 \Pro(\mathcal{EX})
 \preceq
 \left(\frac\de{x_v}\right)^{1+\al}
 +\left(\frac\de{y_v}\right)^\al
  \left(\frac{y_v}{x_v}\right)^{2-\varepsilon}
 +\left(\frac\de{y_v}\right)^{2-\varepsilon}
  \left(\frac{y_v}{x_v}\right)^{1/3}.
\]
Outside $\mathcal{EX}$, both $\mathrm f_v$ and $\mathrm g_v$ are one-to-one and probability preserving.  We use the classes $\mathcal E_{j,m}^v$ from Proposition~\ref{rot-inv} and suppress the superscript $v$ in the remainder of the proof.

We write out the winding comparison because this is the point at which reflection is used.  Suppose first that $\omega(\ga_\de)\in\mathcal E_{1,2}$ and
\[
 W_{\ga_\de}(A,e_b)=2k\pi+\frac\pi2,
 \qquad
 W_{\ga_\de}(D,e_b)=2k\pi+\pi.
\]
Then
\begin{align*}
 W_{\mathrm f_v(\ga_\de)}(A,e_b)
 &=2k\pi+\frac\pi2,
 &
 W_{\mathrm f_v(\ga_\de)}(B,e_b)
 &=2k\pi,\\
 W_{\mathrm g_v(\ga_\de)}(A,e_b)
 &=-2k\pi+\frac\pi2,
 &
 W_{\mathrm g_v(\ga_\de)}(D,e_b)
 &=-2k\pi+\pi,\\
 W_{\mathrm g_v(\mathrm f_v(\ga_\de))}(C,e_b)
 &=-2k\pi+\frac{3\pi}{2},
 &
 W_{\mathrm g_v(\mathrm f_v(\ga_\de))}(D,e_b)
 &=-2k\pi+\pi.
\end{align*}
Together with the configurations obtained by changing the state of $E_v$, these are the eight configurations
\[
 \ga_\de,\quad
 \ga_\de(\omega_{\setminus E_v}),\quad
 \mathrm f_v(\ga_\de),\quad
 \mathrm f_v(\ga_\de(\omega_{\setminus E_v})),
\]
\[
 \mathrm g_v(\ga_\de),\quad
 \mathrm g_v(\ga_\de(\omega_{\setminus E_v})),\quad
 \mathrm g_v(\mathrm f_v(\ga_\de)),\quad
 \mathrm g_v(\mathrm f_v(\ga_\de(\omega_{\setminus E_v}))).
\]
Their contributions to $F_\theta^v(A)$ and $F_\theta^v(D)$ occur in conjugate pairs.  Summing over
\[
 \mathcal E_{4,2}\cup\mathcal E_{1,1}
 \cup\mathcal E_{1,2}\cup\mathcal E_{2,1}
\]
gives
\begin{equation}\label{boundary-conj1}
 e^{-i\pi/4}F_\theta^v(A)
 =\overline{e^{-i\pi/4}F_\theta^v(D)}
 +O\big(\Pro(\mathcal{EX})\big).
\end{equation}
The same relation holds with $F_\theta^v$ replaced by
$\sqrt3F_{one}^v+F_{two}^v$.

For completeness, the remaining four classes give the same relation.  If
$\omega(\ga_\de')\in\mathcal E_{3,1}$ with
\[
 W_{\ga_\de'}(B,e_b)=2k\pi,
 \qquad
 W_{\ga_\de'}(C,e_b)=2k\pi-\frac\pi2,
\]
then
\begin{align*}
 W_{\mathrm g_v(\ga_\de')}(B,e_b)
 &=-2k\pi+2\pi,
 &
 W_{\mathrm g_v(\ga_\de')}(C,e_b)
 &=-2k\pi+\frac{3\pi}{2},\\
 W_{\mathrm f_v(\mathrm g_v(\ga_\de'))}(C,e_b)
 &=-2k\pi+\frac{3\pi}{2},
 &
 W_{\mathrm f_v(\mathrm g_v(\ga_\de'))}(D,e_b)
 &=-2k\pi+\pi,\\
 W_{\mathrm g_v(\mathrm f_v(\mathrm g_v(\ga_\de')))}(A,e_b)
 &=2k\pi+\frac\pi2,
 &
 W_{\mathrm g_v(\mathrm f_v(\mathrm g_v(\ga_\de')))}(B,e_b)
 &=2k\pi.
\end{align*}
Together with the corresponding configurations with the state of $E_v$ changed, these contributions cover
\[
 \mathcal E_{2,2}\cup\mathcal E_{3,1}
 \cup\mathcal E_{3,2}\cup\mathcal E_{4,1},
\]
and \eqref{boundary-conj1} again follows.  Hence \eqref{boundary-conj1} holds on all nonexceptional local classes.  The same is true for $\sqrt3F_{one}^v+F_{two}^v$.

Substituting these two conjugacy relations into the definition \eqref{Fae}, and then into the vertex definition \eqref{hor}, gives
\[
 \mathrm F(v)=q_v+O\big(\Pro(\mathcal{EX})\big),
 \qquad
 \Im(q_v^3)=0.
\]
Moreover, every contribution to $q_v$ has a boundary arm from scale $y_v$ to scale $x_v$.  The half-plane one-arm estimate therefore gives
\[
 |q_v|\preceq\left(\frac{y_v}{x_v}\right)^{1/3}.
\]
This proves \eqref{Fvboundary-re} when $2y_v\leq x_v$.

If $x_v/2<y_v\leq x_v$, the second term in \eqref{Fvboundary-re} is bounded by a constant multiple of $(\de/y_v)^\al$.  The same argument applies without the multi-scale choice.  This completes the proof.
\end{proof}

\begin{remark}
The first two errors in Proposition~\ref{ref-inv} belong to the reflection construction.  The last error comes from the exceptional event for the local rotation map at scale $y_v$, together with the boundary one-arm event from scale $y_v$ to the nearest available endpoint scale $x_v$.  This is why the factor
$(y_v/x_v)^{1/3}$ appears.
\end{remark}

\subsection{Boundary values of the cubic integral}

Let
\[
 \soup=z_0\sim z_1\sim\cdots\sim z_m
\]
be an oriented lattice path in $\Omega_\de$, and let $v_k$ be the
centre of $[z_k,z_{k+1}]$.  We define
\[
 \int_\soup \mathrm F(z)^3\,dz
 =
 \sum_{k=0}^{m-1}
 (z_{k+1}-z_k)\mathrm F(v_k)^3.
\]
The dual integral is defined in the same way with $\mathrm F^*$ and
a path in $\Omega_\de^*$.

On a primal-open straight boundary segment, the winding phase is fixed.
Hence, for a boundary bond $[z_1,z_2]$ with centre $v$,
\begin{equation}\label{F3boundary}
 \Im\bigl(\mathrm F(v)^3(z_2-z_1)\bigr)=0.
\end{equation}
The analogous relation holds for $\mathrm F_i$.  On a dual-open
straight boundary segment, the corresponding relations hold for
$\mathrm F^*$ and $\mathrm F_i^*$.

Put
\begin{equation}\label{corner-block-scale}
 L_\de=\de^\ep.
\end{equation}
An ordinary endpoint of the polygonal boundary will be called
\emph{reentrant} if its local angle is $2\pi$ or $3\pi/2$.

Let $q$ be a reentrant ordinary endpoint.  The joining path of the
inner contour near $q$, together with the portions of the two adjacent
long pieces whose tangential distance from $q$ is at most
$L_\de$, will be denoted by
\[
 \mathcal K_q^\de.
\]
We call $\mathcal K_q^\de$ the reentrant corner block at $q$.
Since every straight boundary segment has length at least $r_\de$ and
$L_\de=o(r_\de)$, distinct corner blocks are disjoint when $\de$ is
small.  The three local corner geometries are shown in
Figure~\ref{fig:corner-geometries}.

\begin{figure}[ht]
\centering
\begin{tikzpicture}[
  x=0.56cm,y=0.56cm,
  line cap=round,line join=round,
  >=Stealth
]

\begin{scope}[shift={(0,0)}]
  \node[font=\small] at (2.35,5.05) {$\pi/2$ corner};
  \coordinate (q1) at (0,0);

  \draw[very thick,blue!70!black] (q1)--(5,0);
  \draw[very thick,blue!70!black] (q1)--(0,5);
  \fill[blue!70!black] (q1) circle (1.5pt);
  \node[font=\scriptsize,below left] at (q1) {$q$};
  \node[font=\scriptsize,left] at (0,3.9) {$\mathrm L_{j-1}$};
  \node[font=\scriptsize,below] at (4.15,0) {$\mathrm L_j$};

  \draw[very thick,orange!90!black] (0.72,0.58)--(4.55,0.58);
  \draw[very thick,orange!90!black] (0.58,0.72)--(0.58,4.55);
  \draw[very thick,orange!90!black]
        (0.72,0.58)--(0.58,0.58)--(0.58,0.72);
  \node[font=\scriptsize,above] at (3.05,0.58) {$\Gamma_j^\de$};
  \node[font=\scriptsize,right] at (0.58,3.05) {$\Gamma_{j-1}^\de$};
  \node[font=\scriptsize,left] at (0.20,0.58) {$\Sigma_j^\de$};

  \node at (2.35,-1.60) {(a)};
\end{scope}

\begin{scope}[shift={(7.45,0)}]
  \node[font=\small] at (2.35,5.05) {$3\pi/2$ corner};
  \coordinate (q2) at (0,0);

  \draw[very thick,blue!70!black] (q2)--(5,0);
  \draw[very thick,blue!70!black] (q2)--(0,5);
  \fill[blue!70!black] (q2) circle (1.5pt);
  \node[font=\scriptsize,above right] at (q2) {$q$};

  \draw[very thick,orange!90!black] (0.62,-0.58)--(3.15,-0.58);
  \draw[very thick,orange!90!black] (-0.58,0.62)--(-0.58,3.15);
  \draw[very thick,orange!90!black]
        (0.62,-0.58)--(-0.58,-0.58)--(-0.58,0.62);

  \draw[densely dashed,thick,orange!90!black] (3.15,-0.58)--(5,-0.58);
  \draw[densely dashed,thick,orange!90!black] (-0.58,3.15)--(-0.58,5);

  \node[font=\scriptsize,below] at (2.00,-0.58) {$\Gamma_j^\de$};
  \node[font=\scriptsize,left] at (-0.58,2.45) {$\Gamma_{j-1}^\de$};
  \node[font=\scriptsize,below left] at (-0.56,-0.56) {$\Sigma_j^\de$};

  \node[font=\scriptsize,below left] at (-0.95,-1.16) {$\mathcal K_q^\de$};
  \node[font=\scriptsize] at (1.15,-0.28) {$L_\de$};

  \node[font=\scriptsize,left] at (0,3.90) {$\mathrm L_{j-1}$};
  \node[font=\scriptsize,above] at (4.15,0.02) {$\mathrm L_j$};

  \node at (2.35,-1.95) {(b)};
\end{scope}

\begin{scope}[shift={(14.90,0)}]
  \node[font=\small] at (2.35,5.05) {$2\pi$ corner};
  \coordinate (q3) at (0,0);

  \draw[very thick,blue!70!black] (q3)--(5,0);
  \fill[blue!70!black] (q3) circle (1.5pt);
  \node[font=\scriptsize,left] at (q3) {$q$};

  \fill[black!7] (0.62,0.10)--(3.15,0.10)--(3.15,0.58)--(0.62,0.58)--cycle;
  \fill[black!7] (0.62,-0.10)--(3.15,-0.10)--(3.15,-0.58)--(0.62,-0.58)--cycle;

  \draw[very thick,orange!90!black] (0.62,0.58)--(3.15,0.58);
  \draw[very thick,orange!90!black] (0.62,-0.58)--(3.15,-0.58);
  \draw[very thick,orange!90!black]
        (0.62,0.58)--(-0.58,0.58)--(-0.58,-0.58)--(0.62,-0.58);

  \draw[densely dashed,thick,orange!90!black] (3.15,0.58)--(5,0.58);
  \draw[densely dashed,thick,orange!90!black] (3.15,-0.58)--(5,-0.58);

  \node[font=\scriptsize,above] at (2.00,0.58) {$\Gamma_{j-1}^\de$};
  \node[font=\scriptsize,below] at (2.00,-0.58) {$\Gamma_j^\de$};
  \node[font=\scriptsize,left] at (-0.58,0) {$\Sigma_j^\de$};

  \node[font=\scriptsize,below left] at (-0.93,-1.16) {$\mathcal K_q^\de$};
  \node[font=\scriptsize] at (0.85,1.02) {$L_\de$};

  \node[font=\scriptsize,above] at (4.15,0.02) {$\mathrm L_j$};

  \node at (2.35,-1.95) {(c)};
\end{scope}

\end{tikzpicture}
\caption{The three corner geometries used for the inner contour.
In (a), at a $\pi/2$ corner, the two long pieces are joined after
a lattice-scale adjustment.  In (b), at a $3\pi/2$ corner, the thick
orange part is the reentrant corner block $\mathcal K_q^\de$ and the
dashed portions continue along the two long pieces.  In (c), at
a $2\pi$ corner, the same geometric boundary is traversed twice and the
two corresponding pieces lie on opposite sides of it.  The two
shaded strips represent the two lattice-square regions used when
$\mathcal D_q^\de$ is decomposed according to the two boundary
traversals.}
\label{fig:corner-geometries}
\end{figure}

A vertex of the inner contour is called \emph{admissible} if it does
not belong to the relative interior of any reentrant corner block.

Finally, put \(\vartheta=3\al\ep-(\frac12-\ep)\varepsilon\).  Our choice
\(\varepsilon\leq \al\ep/100\) gives \(\vartheta>0\).

\begin{proposition}\label{boundofin}
Let $\mathbf z_0,\mathbf z_0'$ be admissible vertices on the part of
$\p\Omega_\de^{in}$ near $\p_{ba,\de}$.  Assume that both
vertices are at distance at least a constant multiple of $L_\de$ from
$\p_{ab,\de}^*$.  Then
\begin{equation}\label{boundofin-new}
 \left|
 \Im\int_{(\mathbf z_0\to\mathbf z_0')
          \subseteq\p\Omega_\de^{in}}
 \de^{-1}\mathrm F(z)^3\,dz
 \right|
 \preceq
 \de^{\ep-\varepsilon}r_\de^{-4}
 +
 \de^\vartheta r_\de^{-2-3\al}.
\end{equation}
In particular, the right-hand side tends to zero.
The same estimate holds for the part of the inner contour near
$\p_{ab,\de}^*$.  Analogous estimates hold for $\mathrm F_i$,
$\mathrm F^*$ and $\mathrm F_i^*$.
\end{proposition}

\begin{proof}
We only prove the estimate for $\mathrm F$ on the part of the inner
contour near $\p_{ba,\de}$.  The estimate near $\p_{ab,\de}^*$ is
obtained by applying the same argument to $\mathrm F^*$ and then using
Proposition~\ref{F12}.  The other assertions follow in the same way.

We divide the contour into three types of pieces: the portions of the
long pieces at tangential distance at least $L_\de$ from every
ordinary endpoint, the neighbourhoods of the $\pi/2$ corners, and the
complete reentrant corner blocks $\mathcal K_q^\de$.  Admissibility is
used only for the last type, where it guarantees that every reentrant
corner block met by the path is traversed completely.

\medskip
\noindent
{\it 1. The portions away from ordinary endpoints.}

Consider a long piece $\Gamma_j^\de$.  Choose coordinates in
which $\mathrm L_j$ is horizontal.  Let $x$ be the distance from the
projection of $v$ to the nearer ordinary endpoint and let
$y\asymp\rho_\de$.  Suppose that $x\geq L_\de$.  Put
$h=(y/x)^{1/3}$.  Proposition~\ref{ref-inv} gives
\[
 \mathrm F(v)=q_v+e_v,
 \qquad
 \Im(q_v^3)=0,
 \qquad
 |q_v|\preceq h,
\]
where $|e_v|\preceq E_1+E_2+E_3$ with $E_1=(\de/x)^{1+\al}$, $E_2=(\de/ y)^\al( y/x)^{2-\varepsilon}$ and $E_3=(\de /y)^{2-\varepsilon}
( y/x)^{1/3}$.

For $x\geq L_\de$ and $y\asymp\rho_\de$, each $E_i\preceq h$ for
small $\de$.  Hence
$|\Im\mathrm F(v)^3|\preceq h^2(E_1+E_2+E_3)$.  Summing over the
lattice values of $x$ gives
\begin{align*}
 \sum_{x\geq L_\de}h^2E_1
 &\preceq
 \de^\al\rho_\de^{2/3}L_\de^{-2/3-\al}
 =\de^{1/3+\al(1-\ep)},\\
 \sum_{x\geq L_\de}h^2E_2
 &\preceq
 \de^{\al-1}\rho_\de^{8/3-\varepsilon-\al}
 L_\de^{-5/3+\varepsilon}
 =\de^{1/3+\al/2-\varepsilon/2+\ep(1-\al)},\\
 \sum_{x\geq L_\de}h^2E_3
 &\preceq
 \de^{1-\varepsilon}\rho_\de^{-1+\varepsilon}
 \log\de^{-1}
 =\de^{1/2-\varepsilon/2-\ep(1-\varepsilon)}
 \log\de^{-1}.
\end{align*}
By the choice of $\ep$ and $\varepsilon$, these terms are bounded by
$C\de^{\ep-\varepsilon}r_\de^{-2}$.  This is the required estimate
for one long piece away from its ordinary endpoints.

\medskip
\noindent
{\it 2. A local square estimate.}

Let $\square=z_1z_2z_3z_4$ be an elementary counterclockwise primal
square, let $v_j$ be the centre of $[z_j,z_{j+1}]$, with $z_5=z_1$,
and put $F_j=\mathrm F(v_j)$.  With
\[
 S_{13}=F_1^2+F_1F_3+F_3^2,
 \qquad
 S_{24}=F_2^2+F_2F_4+F_4^2,
\]
we have
\[
 \de^{-1}\int_{\p\square}\mathrm F(z)^3\,dz
 =(F_1-F_3)S_{13}+i(F_2-F_4)S_{24}.
\]
Put $R_\square=(F_1-F_3)+i(F_2-F_4)$,
$T=(\de/ x)^{1+\al-\ep}$,
and
 $Q=(\de /y)^{2-\varepsilon}$.

Suppose that $P$ is an upper bound for the arm contribution at the four
medial vertices.  Proposition~\ref{TranF}, Proposition~\ref{F12}, and
the arm estimate outside the local exceptional ball give
\[
 |F_1-F_3|\preceq T,
 \qquad
 |R_\square|\preceq QP,
 \qquad
 |F_1-F_2|+|F_3-F_4|\preceq T+QP,
 \qquad
 \max_j|F_j|\preceq P.
\]
Hence $|S_{24}|\preceq P^2$ and
$|S_{13}-S_{24}|\preceq P(T+QP)$.  Using
\[
 (F_1-F_3)S_{13}+i(F_2-F_4)S_{24}
 =(F_1-F_3)(S_{13}-S_{24})+R_\square S_{24},
\]
we obtain
\begin{equation}\label{corner-square-bound}
 \left|
 \de^{-1}\int_{\p\square}\mathrm F(z)^3\,dz
 \right|
 \preceq T^2P+TQP^2+QP^3.
\end{equation}
The factor $P$ in the terms containing $Q$ is important.  A
configuration contributing to the same-edge local defect must also
carry the exploration arm from the local exceptional ball to the
available outer scale, exactly as in Proposition~\ref{ref-inv}.

\medskip
\noindent
{\it 3. The $\pi/2$ corners.}

Let $q$ be an ordinary endpoint with local angle $\pi/2$.  The local
domain is contained in a half-plane, so \eqref{halfone} gives
$P_0(y)=(y/r_\de)^{1/3}$ as an upper bound for the observable at distance $y$ from the boundary line segment.  The portion of the inner contour within
distance $C\rho_\de$ of $q$ has length $O(\rho_\de)$ and therefore
contributes at most
\[
 \frac{\rho_\de}{\de}\frac{\rho_\de}{r_\de}
 =\de^{2\ep}r_\de^{-1}.
\]
For the two remaining strips, where
$\rho_\de\leq x\leq L_\de$ and $\de\leq y\leq C\rho_\de$, put
$\beta=1+\al-\ep$.  Applying \eqref{corner-square-bound} with
$P=P_0(y)$, the sums of the
three terms on the right-hand side of \eqref{corner-square-bound},
respectively, gives
\begin{align*}
 T_1\preceq
 \de^{2\beta-2}\rho_\de^{7/3-2\beta}r_\de^{-1/3},\
 T_2\preceq
 \de^{\beta-1/3}\rho_\de^{1-\beta}r_\de^{-2/3},\
 T_3\preceq
 L_\de\left(\frac{\rho_\de}{\de}\right)^\varepsilon r_\de^{-1}.
\end{align*}
The corresponding powers of $\de$ are
\[
 \frac16+\al-\frac{2\ep}{3}-2\al\ep+2\ep^2,
 \qquad
 \frac23+\frac\al2-\frac\ep2-\al\ep+\ep^2,
 \qquad
 \ep-\left(\frac12-\ep\right)\varepsilon.
\]
They are all at least $\ep-\varepsilon$ for our choice of parameters.
Thus every $\pi/2$ corner contributes at most
$C\de^{\ep-\varepsilon}r_\de^{-2}$.

\medskip
\noindent
{\it 4. Reentrant corner blocks.}

Let now $q$ be a $2\pi$ or $3\pi/2$ ordinary endpoint.  We control the
complete block $\mathcal K_q^\de$ rather than estimating its joining
path pointwise.  Let $\mathcal D_q^\de$ be the lattice region between
$\mathcal K_q^\de$ and the corresponding pieces of the actual polygonal
boundary.  At tangential distance $L_\de$ we close this region by one
or two transverse lattice paths $N_q^{(1)},N_q^{(2)}$.  In the $2\pi$
case the geometric boundary is traversed twice, and we regard
$\mathcal D_q^\de$ as a bounded union of lattice-square regions, one
for each traversal.  This changes only the absolute constants below.
Since the actual primal-open boundary has zero imaginary cubic integral
by \eqref{F3boundary}, discrete Stokes gives
\begin{equation}\label{reentrant-stokes}
 \left|
 \Im\int_{\mathcal K_q^\de}\de^{-1}\mathrm F(z)^3\,dz
 \right|
 \preceq
 \sum_{\square\subseteq\mathcal D_q^\de}
 \left|
 \int_{\p\square}\de^{-1}\mathrm F(z)^3\,dz
 \right|
 +
 \sum_m
 \left|
 \int_{N_q^{(m)}}\de^{-1}\mathrm F(z)^3\,dz
 \right|.
\end{equation}

In a side region let $x$ be tangential distance from $q$ and $y$ the
normal distance from the corresponding straight boundary piece.  The
 two-scale arm estimate gives
\[
 P\preceq M(x,y)
 :=\left(\frac yx\right)^{1/3}
   \left(\frac{x}{r_\de}\right)^\al.
\]
Thus \eqref{corner-square-bound} becomes
\begin{equation}\label{reentrant-square}
 \left|
 \de^{-1}\int_{\p\square}\mathrm F(z)^3\,dz
 \right|
 \preceq T^2M+TQM^2+QM^3,
\end{equation}
where $T=(\de/x)^{1+\al-\ep}$ and
$Q=(\de/y)^{2-\varepsilon}$.

First consider $\rho_\de\leq x\leq L_\de$ and
$\de\leq y\leq C\rho_\de$.  The sums of the three terms on the right hand side of
\eqref{reentrant-square} respectively give
\begin{align*}
 S_1^+
 &\preceq
 \de^{2\al-2\ep}\rho_\de^{-\al+2\ep}r_\de^{-\al},\\
 S_2^+
 &\preceq
 \de^{2/3+\al-\ep}r_\de^{-2\al}
 \int_{\rho_\de}^{L_\de}x^{\al+\ep-5/3}\,dx,\\
 S_3^+
 &\preceq
 \left(\frac{\rho_\de}{\de}\right)^\varepsilon
 \left(\frac{L_\de}{r_\de}\right)^{3\al}
 =\de^\vartheta r_\de^{-3\al}.
\end{align*}
For $S_2^+$, since $\al+\ep\leq2/3$, the $x$-integral is controlled
by its lower endpoint. The resulting power of $\de$, as well as that in
$S_1^+$, is strictly larger than $\vartheta$. Hence
\[
 S_1^++S_2^++S_3^+
 \preceq \de^\vartheta r_\de^{-3\al}.
\]

For the part of a side region with $\de\leq x\leq\rho_\de$ and
$\de\leq y\leq cx$, let $S_1^-,S_2^-,S_3^-$ denote the corresponding
total contributions. The same calculation gives
\[
 S_1^-\preceq\left(\frac\de{r_\de}\right)^\al,
 \qquad
 S_2^-\preceq\left(\frac\de{r_\de}\right)^{2\al},
 \qquad
 S_3^-\preceq
 \left(\frac{\rho_\de}{\de}\right)^\varepsilon
 \left(\frac{\rho_\de}{r_\de}\right)^{3\al}.
\]
Each term is bounded by $C\de^\vartheta r_\de^{-3\al}$.

It remains to consider the central part of the corner, where the
normal and tangential scales are comparable.  Divide it into dyadic
regions of scale $s$, with $\de\leq s\leq C\rho_\de$.  At this scale
use
\[
 M_s=\left(\frac s{r_\de}\right)^\al,
 \qquad
 T_s=\left(\frac\de s\right)^{1+\al-\ep},
 \qquad
 Q_s=\left(\frac\de s\right)^{2-\varepsilon}.
\]
Since there are $O((s/\de)^2)$ elementary squares, the three terms in
\eqref{reentrant-square} contribute at most
\begin{align*}
 C
 \left(\frac\de s\right)^{2\al-2\ep}
 \left(\frac s{r_\de}\right)^\al, \qquad
 C
 \left(\frac\de s\right)^{1+\al-\ep-\varepsilon}
 \left(\frac s{r_\de}\right)^{2\al}, \qquad
 C
 \left(\frac s\de\right)^\varepsilon
 \left(\frac s{r_\de}\right)^{3\al}.
\end{align*}
Their dyadic sums are respectively bounded by
\[
 C\left(\frac\de{r_\de}\right)^\al,
 \qquad
 C\left(\frac\de{r_\de}\right)^{2\al},
 \qquad
 C\left(\frac{\rho_\de}{\de}\right)^\varepsilon
  \left(\frac{\rho_\de}{r_\de}\right)^{3\al}.
\]
Thus the central region contributes at most
$C\de^\vartheta r_\de^{-3\al}$.

Finally, on a transverse closing path at $x=L_\de$, the two-scale arm
bound gives
\[
 |\mathrm F(v)|^3
 \preceq
 \frac y{L_\de}
 \left(\frac{L_\de}{r_\de}\right)^{3\al}.
\]
Since the path has length $O(\rho_\de)$, its cubic integral is bounded
by
\[
 \frac{\rho_\de^2}{\de L_\de}
 \left(\frac{L_\de}{r_\de}\right)^{3\al}
 =\de^{(1+3\al)\ep}r_\de^{-3\al}
 \preceq \de^\vartheta r_\de^{-3\al}.
\]
Combining these estimates with \eqref{reentrant-stokes} gives
\begin{equation}\label{reentrant-block-bound}
 \left|
 \Im\int_{\mathcal K_q^\de}
 \de^{-1}\mathrm F(z)^3\,dz
 \right|
 \preceq
 \de^\vartheta r_\de^{-3\al}.
\end{equation}

\medskip
\noindent
{\it 5. Summation over the boundary.}

Condition $\mathrm C$ gives $\mathrm n+\mathrm n^*=O(r_\de^{-2})$.
The long pieces and the $\pi/2$ corners therefore contribute at most
$C\de^{\ep-\varepsilon}r_\de^{-4}$.  There are also at most
$Cr_\de^{-2}$ reentrant corner blocks, whose total contribution is at
most $C\de^\vartheta r_\de^{-2-3\al}$ by
\eqref{reentrant-block-bound}.  This proves \eqref{boundofin-new}.
\end{proof}

\begin{remark}\label{reentrant-admissible-remark}
The restriction to admissible endpoints is deliberate.  At a $2\pi$
or $3\pi/2$ ordinary endpoint, the proof controls the complete corner
block by discrete Stokes, not an arbitrary initial part of its joining
path.  Later, exits in the interiors of these short corner blocks are
treated by discrete harmonic measure.
\end{remark}

\subsection{Preliminary estimates for the observables} \label{Pre-est-ob}

We begin with the original parafermionic observable $F$.  Its local
Cauchy--Riemann relation is exact.  We first check the corresponding
closed-square relation because the four values around a primal square
are not the four medial edges incident to one medial vertex.

Let $\soup=z_0\sim z_1\sim\cdots\sim z_m=z_0$ be an edge-avoiding
lattice contour and let $v_k$ be the centre of $[z_k,z_{k+1}]$.  Then
\begin{equation}\label{DC}
 \oint_\soup F(z)\,dz
 :=\sum_{k=0}^{m-1}(z_{k+1}-z_k)F(v_k)=0.
\end{equation}
To verify this for one elementary primal square, let
$v_1,v_2,v_3,v_4$ be the centres of its four sides, listed
counterclockwise.  At $v_j$, let $A_j$ be the medial edge directed towards $v_j$ inside the square, and let $D_j$ be the medial edge
 directed away from $v_j$ inside the square.  By \eqref{CR} and
the definition of $F(v_j)$,
\begin{equation}\label{F-square-original}
 F(v_j)=e_b^{-1/3}e^{-i\pi/4}
 \big(F(A_j)+iF(D_j)\big).
\end{equation}
The four medial edges inside the square satisfy
\[
 A_1=D_2,\qquad A_2=D_3,\qquad
 A_3=D_4,\qquad A_4=D_1
\]
as oriented medial edges.  Hence
\begin{align*}
 F(v_1)+iF(v_2)-F(v_3)-iF(v_4)&=e_b^{-1/3}e^{-i\pi/4}
 \big[
 F(A_1)+iF(D_1)+iF(A_2)-F(D_2)\\
 &\hspace{20mm}
 -F(A_3)-iF(D_3)-iF(A_4)+F(D_4)
 \big]=0.
\end{align*}
This is the integral of $F$ around the elementary square divided by
$\de$.  Every edge-avoiding contour in the simply connected lattice
domain is the oriented boundary of a union of elementary squares.
Summing the square identities proves \eqref{DC}.

Consequently, for $z\in V_{\Omega_\de}$ we may define
\begin{equation}\label{primitive-F}
 \mathbb F(z)=\int_{a_\de\to z}\de^{-1/3}F(x)\,dx.
\end{equation}
The value does not depend on the path.  We define $\mathbb F^*$ in the
same way on the dual lattice.

For a medial vertex $v$, write
\[
 d_v=\dist\bigl(v,\p_{ba,\de}\cup\p_{ab,\de}^*\bigr)
\]
and let
\[
 \mathrm d_1(v)=
 \min\big\{\dist(v,\mathrm v_j),\dist(v,\mathrm v_{j^*}^*):
 1\leq j\leq\mathrm n+1,\ 1\leq j^*\leq\mathrm n^*+1\big\}.
\]
Thus $\mathrm d_1(v)$ is the distance from $v$ to the nearest endpoint
of a primal-open or dual-open straight boundary segment.
\begin{lemma} \label{initialb}
Under Condition $\mathrm C$,
\begin{eqnarray}
 |F(v)|&\preceq&\de^{1/3}(\log\de^{-1})^2,
 \label{Finboun}\\
 |F(v+\de\be)-F(v)|&\preceq&
 \de^{1+1/3}(\log\de^{-1})^2,
 \label{Fdiffinboun}
\end{eqnarray}
uniformly for $v$ in a fixed compact subset of $\Omega_0$, where
$\be=\be_1$ or $\be_2$.

If $d_v\to0$, then
\begin{eqnarray}
 |F(v)|&\preceq&
 \Big(\frac\de{d_v}\Big)^{1/3}(\log\de^{-1})^2,
 \label{Finboun-b}\\
 |F(v+\de\be)-F(v)|&\preceq&
 \Big(\frac\de{d_v}\Big)^{1+1/3}(\log\de^{-1})^2.
 \label{Fdiffinboun-b}
\end{eqnarray}
The same conclusions hold for $F^*$.
\end{lemma}

\begin{proof}
We first control the boundary values of the primitive without using an
estimate across an ordinary endpoint.  Consider one straight boundary
segment $\mathrm L_j$.  If a boundary medial vertex is $k$ lattice
steps from the nearer endpoint of $\mathrm L_j$, then every configuration
which contributes to $F(v)$ has a half-plane arm from scale comparable
with $\de$ to scale comparable with $k\de$.  The half-plane one-arm
estimate \eqref{halfone} therefore gives
\[
 |F(v)|\preceq k^{-1/3}.
\]
For the first bounded number of lattice vertices we use the trivial
bound.  If $m_j$ is the number of lattice bonds in $\mathrm L_j$, then
\begin{align}
 \int_{\mathrm L_j}|\de^{-1/3}F(z)|\,|dz|
 \preceq
 \de^{2/3}\sum_{k=1}^{\lceil m_j/2\rceil}k^{-1/3}
 \preceq \de^{2/3}m_j^{2/3}
 \preceq 1.
 \label{primitive-one-segment}
\end{align}
The same argument applies to every straight segment on the other
boundary arc.  Condition $\mathrm C$ gives $O(r_\de^{-2})$ straight
segments.  Since $\mathbb F$ is path independent,
\eqref{primitive-one-segment} gives
\begin{equation}\label{Fboun1}
 \sup_{z\in\p\Omega_\de}|\mathbb F(z)|
 \preceq(\log\de^{-1})^2.
\end{equation}

We shall also use a local form of the same calculation.  Let
$z,z'\in\p\Omega_\de$ be joined along the polygonal boundary by a piece
of length $R\leq r_\de/10$ which lies in one straight segment or in two
consecutive straight segments.  Then
\begin{equation}\label{primitive-boundary-holder}
 |\mathbb F(z)-\mathbb F(z')|\preceq R^{2/3}.
\end{equation}
Indeed, if the boundary piece contains an endpoint, split there and use
the sum in \eqref{primitive-one-segment} on the two pieces.  If it does
not contain an endpoint and its distance to the nearer endpoint is
$s\geq R$, then \eqref{halfone} gives
\[
 |F|\preceq\Big(\frac\de s\Big)^{1/3}
\]
on this piece.  Its contribution to the primitive is at most
\[
 \de^{-1/3}R\Big(\frac\de s\Big)^{1/3}
 =Rs^{-1/3}\leq R^{2/3}.
\]
If $s<R$, enlarge the boundary piece to the endpoint and apply the
first case.  This proves \eqref{primitive-boundary-holder}.  The same
argument applies at a marked endpoint $a_\de$, $a_\de^*$, $b_\de$ or $b_\de^*$.  If two consecutive boundary
traversals overlap geometrically, path independence of $\mathbb F$
removes the possible ambiguity at the common lattice vertices.

Every point of $\p\hat\Omega_\de$ can be joined to a corresponding
point of $\p\Omega_\de$ by a lattice path of length $O(\eta_\de)$.
The trivial estimate $|F(v)|\preceq1$ gives
\begin{align}
 |\mathbb F(z)-\mathbb F(z')|
 &\preceq \de^{-1/3}\eta_\de
 =\de^{1/6-\ep}=o(1).
 \label{primitive-normal-path}
\end{align}
Therefore
\begin{equation}\label{Fboun3}
 \sup_{z\in\p\hat\Omega_\de}|\mathbb F(z)|
 \preceq(\log\de^{-1})^2.
\end{equation}

We next bound the discrete Laplacian of the primitive.  Put
\[
 \widetilde F^*(v)=\mathrm F^*(v)-\mathrm F_i^*(v).
\]
Since $F=\mathrm F+\mathrm F_i$, Proposition~\ref{F12} gives
\begin{align}
 F(v)
 &=\mathrm F(v)+\mathrm F_i(v)\nonumber\\
 &=\mathrm F^*(v)-\mathrm F_i^*(v)
   +O\left(\Big(\frac\de{d_v}\Big)^2\right)\nonumber\\
 &=\widetilde F^*(v)
   +O\left(\Big(\frac\de{d_v}\Big)^2\right).
 \label{F-dual-comparison}
\end{align}
The exact relations \eqref{CRF} and \eqref{CRFi}, applied on the dual
lattice, show that the integral of $\widetilde F^*(v)$ around every elementary dual square is of the order $(\de/d_v)^2$.  Let $z$ be a primal
vertex and put
\[
 d(z)=\dist(z,\p_{ba,\de}\cup\p_{ab,\de}^*).
\]
The four medial vertices incident to $z$ are the four side centres of
the elementary dual square around $z$.  The four increments of
$\mathbb F$ entering $\Delta\mathbb F(z)$ are the corresponding values
of $\de^{2/3}F$.  Thus \eqref{F-dual-comparison} and the integral of $\widetilde F^*(v)\,dz$ around every elementary dual square give
\begin{equation}\label{primitive-local-laplacian}
 |\Delta\mathbb F(z)|
 \preceq
 \de^{2/3}\Big(\frac\de{d(z)}\Big)^2.
\end{equation}
Changing $d(z)$ by $O(\de)$ only changes the implicit constant.

By the construction of the inner contour,
\[
 d(z)\geq c\rho_\de,
 \qquad z\in V_{\hat\Omega_\de}.
\]
For $k\geq c\rho_\de/\de$, put
\[
 \mathcal V_k=
 \{z\in V_{\hat\Omega_\de}:
   k\de\leq d(z)<(k+1)\de\}.
\]
A distance layer of width $\de$ is covered by strips associated with
the $O(r_\de^{-2})$ straight boundary segments and by a bounded number
of lattice-scale corrections for each endpoint.  Hence
\begin{equation}\label{distance-layer-count}
 |\mathcal V_k|\preceq r_\de^{-2}\de^{-1}.
\end{equation}
Using \eqref{primitive-local-laplacian} and
\eqref{distance-layer-count}, we obtain
\begin{align}
 \sum_{z\in V_{\hat\Omega_\de}}|\Delta\mathbb F(z)|
 &\preceq
 r_\de^{-2}\de^{-1}\de^{2/3}
 \sum_{k\geq c\rho_\de/\de}k^{-2}\nonumber\\
 &=\de^{1/6-\ep}r_\de^{-2}=o(1).
 \label{primitive-laplacian-small}
\end{align}

Let $H$ be the discrete harmonic function in $\hat\Omega_\de$ with
boundary values $\mathbb F$.  The Green representation,
\eqref{Fboun3}, and \eqref{primitive-laplacian-small} give
\[
 \sup_{z\in V_{\hat\Omega_\de}}|\mathbb F(z)-H(z)|=o(1).
\]
Indeed, the maximal Green function is $O(\log\de^{-1})$, while the
right-hand side of \eqref{primitive-laplacian-small} is a positive
power of $\de$ times a power of $\log\de^{-1}$.  It follows that
\begin{equation}\label{primitive-global-bound}
 \sup_{z\in V_{\hat\Omega_\de}}|\mathbb F(z)|
 \preceq(\log\de^{-1})^2.
\end{equation}

We first derive the estimates on a fixed compact subset of $\Omega_0$.
In a fixed disk contained in $\Omega_0$, \eqref{primitive-local-laplacian}
gives
\[
 |\Delta\mathbb F|\preceq \de^{8/3}.
\]
Since $8/3>5/2$, Lemma~\ref{asyhar} may be applied with exponent
$1/4$.  Together with \eqref{primitive-global-bound}, it gives, for
one lattice step $u$,
\[
 |\mathbb F(z+u)-\mathbb F(z)|
 \preceq \de(\log\de^{-1})^2
\]
and
\[
 |\mathbb F(z+u)+\mathbb F(z-u)-2\mathbb F(z)|
 \preceq \de^2(\log\de^{-1})^2.
\]
The proof of Lemma~\ref{asyhar}, using the corresponding mixed second
difference of the Green function, gives the same bound for mixed
second differences.  An increment of $\mathbb F$ equals
$\de^{2/3}F$ on the corresponding medial vertex.  Dividing the first
and second difference estimates by $\de^{2/3}$ proves
\eqref{Finboun} and \eqref{Fdiffinboun}.

It remains to prove the estimates when $d_v\to0$.  Rescale space by
$d_v^{-1}$.  The mesh of the rescaled lattice is $\de/d_v$.  After
subtracting a constant from the primitive and multiplying it by
$d_v^{-2/3}$, the boundary oscillation estimate
\eqref{primitive-boundary-holder} has the same form on unit scales.
Moreover, in a ball of radius comparable with $d_v$ around $v$,
\eqref{primitive-local-laplacian} becomes a bound of order
$
 (\de/d_v)^{8/3}
$
for the Laplacian of the rescaled primitive.  Therefore the same Green
representation and the same asymptotic-harmonic estimate used above,
now with mesh $\de/d_v$, give
\[
 |F(v)|\preceq
 \Big(\frac\de{d_v}\Big)^{1/3}(\log\de^{-1})^2, \qquad
 |F(v+\de\be)-F(v)|\preceq
 \Big(\frac\de{d_v}\Big)^{1+1/3}(\log\de^{-1})^2.
\]
This proves \eqref{Finboun-b} and \eqref{Fdiffinboun-b}.  The proof
for $F^*$ is the primal-dual version of the same argument.
\end{proof}

The bound \eqref{Finboun-b} can be improved if $\mathrm{d}_1\geq \de^{3\ep}$.

\begin{lemma} \label{imitialb-imp}
Under Condition $\mathrm C$, suppose that
\[
 \mathrm d_1(v)\geq\de^{3\ep}.
\]
If $d_v\geq\de^{1/3-\ep}$, then
\[
 |F(v)|+|F^*(v)|
 \preceq
 \Big(\frac\de{\de^{3\ep}}\Big)^{1/3}
 (\log\de^{-1})^2.
\]
If $d_v\geq\de^{1/2+\ep}$ and $E_v$ is parallel to
$\mathrm L_j$ or $\mathrm L_{j^*}^*$, which is the closest straight
boundary segment to $v$, then
\[
 |F(v)|
 \preceq
 \Big(\frac\de{\de^{3\ep}}\Big)^{1/3}
 (\log\de^{-1})^2.
\]
If $d_v\geq\de^{1/2+\ep}$ and $E_v^*$ is parallel to
$\mathrm L_j$ or $\mathrm L_{j^*}^*$, which is the closest straight
boundary segment to $v$, then
\[
 |F^*(v)|
 \preceq
 \Big(\frac\de{\de^{3\ep}}\Big)^{1/3}
 (\log\de^{-1})^2.
\]
\end{lemma}

\begin{proof}
We first consider $d_v\geq\de^{1/3-\ep}$.  If
$d_v\geq\de^{3\ep}/4$, the conclusion follows from
Lemma~\ref{initialb}.  If $d_v$ stays bounded away from zero, we use
\eqref{Finboun} instead of \eqref{Finboun-b}.  It remains to consider
\[
 \de^{1/3-\ep}\leq d_v<\frac14\de^{3\ep}.
\]
Since $\mathrm d_1(v)\geq\de^{3\ep}$ and
$r_\de\gg\de^{3\ep}$, a disk of radius $c\de^{3\ep}$ around $v$
meets only the interior of one closest straight boundary segment.
Translate and rotate the picture so that this segment is on the real
axis and the domain is locally above it.  After multiplication by the
fixed boundary phase, define
\[
 F_{\rm ext}(w)=
 \begin{cases}
  F(w),&\Im w\geq0,\\
  \overline{F(\overline w)},&\Im w<0.
 \end{cases}
\]
The exact Cauchy--Riemann relation \eqref{CR} and the boundary phase
show that the corresponding primitive is path independent across the
real axis.  The rescaling argument in the proof of
Lemma~\ref{initialb}, now at scale $c\de^{3\ep}$, gives the first
assertion.  The dual reflection gives the estimate for $F^*$.

For the second assertion it remains only to consider
\[
 \de^{1/2+\ep}\leq d_v<\de^{1/3-\ep}.
\]
Translate $v$ in normal lattice steps away from the closest boundary
segment until its distance from that segment is $\de^{1/3-\ep}$.
Since $\de^{1/3-\ep}=o(\de^{3\ep})$, the whole translation stays a
distance at least $c\de^{3\ep}$ from the two endpoints of this segment
and from every other boundary segment.  There are
$O(\de^{-2/3-\ep})$ steps.  Proposition~\ref{Tranimp} gives a total
error at most
\[
 C\de^{-2/3-\ep}
 \Big(\frac\de{\de^{3\ep}}\Big)^{1+\al-\ep}.
\]
Relative to the desired bound, this is
\[
 \de^{\al-4\ep-3\al\ep+3\ep^2}=o(1),
\]
because $\ep\leq\al/100$.  The first assertion applies at the terminal
point and proves the second assertion.  The third assertion follows in
the same way from the dual version of Proposition~\ref{Tranimp}.
\end{proof}

\begin{lemma} \label{initial-b}
Under Condition $\mathrm C$,
\begin{eqnarray}
 |\mathrm F(v)|&\preceq&\de^{1/3}(\log\de^{-1})^2,
 \label{Finboun-M}\\
 |\mathrm F(v+\de\be)-\mathrm F(v)|&\preceq&
 \de^{1+1/3}(\log\de^{-1})^2,
 \label{Fdiffinboun-M}
\end{eqnarray}
uniformly on compact subsets of $\Omega_0$.

If $d_v\to0$, then
\begin{eqnarray}
 |\mathrm F(v)|&\preceq&
 \Big(\frac\de{d_v}\Big)^{1/3}(\log\de^{-1})^2
 +\Big(\frac\de{d_v}\Big)^2
  \Big(\frac{d_v}{\sqrt{\mathrm d_1(v)^2-d_v^2}}\Big)^{1/3},
 \label{Finboun-bM}\\
 |\mathrm F(v+\de\be)-\mathrm F(v)|&\preceq&
 \Big(\frac\de{d_v}\Big)^{1+1/3}(\log\de^{-1})^2
 +\Big(\frac\de{d_v}\Big)^2
  \Big(\frac{d_v}{\sqrt{\mathrm d_1(v)^2-d_v^2}}\Big)^{1/3}.
 \label{Fdiffinboun-bM}
\end{eqnarray}
If
\[
 \mathrm d_1(v)\geq\de^{3\ep},
 \qquad
 d_v\geq\de^{1/2+\ep},
\]
and $E_v$ is parallel to a closest straight boundary segment, then
\begin{equation}\label{Finboun-imp}
 |\mathrm F(v)|
 \preceq
 \Big(\frac\de{\de^{3\ep}}\Big)^{1/3}
 (\log\de^{-1})^2.
\end{equation}
The same conclusions hold for $\mathrm F_i$.
\end{lemma}

\begin{proof}
The compact estimates follow from Lemma~\ref{initialb} and
Proposition~\ref{F12}.  The local exceptional contribution in the
bulk is $O((\de/d_v)^2)$ and is smaller than the displayed terms.

We next consider $d_v\to0$.  The nonexceptional part is bounded by
\eqref{Finboun-b} and \eqref{Fdiffinboun-b}.  For the exceptional
part, the local rotation construction fails inside a ball of radius
comparable with $d_v$, which contributes
\[
 O\left(\Big(\frac\de{d_v}\Big)^2\right).
\]
Let $\mathrm L$ be a closest straight boundary segment.  Up to a fixed
factor, the distance along $\mathrm L$ from the projection of $v$ to
the nearer endpoint is at least
\[
 \sqrt{\mathrm d_1(v)^2-d_v^2}.
\]
If this quantity is at most a fixed multiple of $d_v$, the displayed
second term in \eqref{Finboun-bM} and \eqref{Fdiffinboun-bM} is already
at least a constant multiple of the local exceptional bound.  Hence
there is nothing more to prove.  Otherwise, for an exceptional
configuration to contribute at $v$, the exploration path must continue
from scale comparable with $d_v$ to scale comparable with
$\sqrt{\mathrm d_1(v)^2-d_v^2}$.  Condition on the configuration in a
slightly larger ball around $v$ and on the exit point of the selected
arm.  After changing the two radii by fixed factors, the remaining
continuation lies in a straight half-plane neighbourhood.  The
half-plane one-arm estimate \eqref{halfone} gives
\[
 C\Big(\frac{d_v}{\sqrt{\mathrm d_1(v)^2-d_v^2}}\Big)^{1/3}.
\]
  Multiplying the local
exceptional bound by this continuation bound gives the second terms in
\eqref{Finboun-bM} and \eqref{Fdiffinboun-bM}.

It remains to prove \eqref{Finboun-imp}.  If
$d_v\geq\de^{1/3-\ep}$, Lemma~\ref{imitialb-imp} and
Proposition~\ref{F12} give the result.  Suppose that
\[
 \de^{1/2+\ep}\leq d_v<\de^{1/3-\ep}.
\]
Translate the modified observable in normal lattice steps until the
distance from the closest segment is $\de^{1/3-\ep}$.  As in the proof
of Lemma~\ref{imitialb-imp}, the distance to the endpoints and to every
other boundary segment is at least $c\de^{3\ep}$ throughout the
translation.  Proposition~\ref{TranF} gives a total error
\[
 O\left(
 \de^{-2/3-\ep}
 \Big(\frac\de{\de^{3\ep}}\Big)^{1+\al-\ep}
 \right)
 =o\left(
 \Big(\frac\de{\de^{3\ep}}\Big)^{1/3}
 \right).
\]
At the terminal point the first case applies.  This proves
\eqref{Finboun-imp}.  The proof for $\mathrm F_i$ is identical because
Proposition~\ref{TranF} also holds for $\mathrm F_i$.
\end{proof}

\subsection{The cubic integral and discrete Laplacian}

Let $z_1z_2z_3z_4$ be an elementary counterclockwise square of the primal lattice, with $z_5=z_1$, and let $v_j$ be the centre of $[z_j,z_{j+1}]$.  Put
\[
 d_0=\min_{1\leq j\leq4}
 \dist(v_j,\p_{ba,\de}\cup\p_{ab,\de}^*).
\]
The next proposition is the local calculation which makes the cubic integral useful.  The first identity says that the discrete integral of $\mathrm F^3$ along a lattice face is small, while the second is its discrete Laplacian form.

\begin{proposition}
For every $\varepsilon>0$,
\begin{equation}\label{unitsq}
 \de^{-1}\sum_{j=1}^4
 \mathrm F(v_j)^3(z_{j+1}-z_j)
 =O\left(\Big(\frac\de{d_0}\Big)^{2+2\al-\varepsilon}\right).
\end{equation}
Let $v_1'=z_1-\de\be_1/2$ and $v_4'=z_1-\de\be_2/2$.  Then
\begin{equation}\label{Lap}
 \big(\mathrm F(v_1)^3-\mathrm F(v_1')^3\big)
 +i\big(\mathrm F(v_4)^3-\mathrm F(v_4')^3\big)
 =O\left(\Big(\frac\de{d_0}\Big)^{2+2\al-\varepsilon}\right).
\end{equation}
The same conclusions hold for $\mathrm F_i$.
\end{proposition}

\begin{proof}
Write $F_j=\mathrm F(v_j)$.  Directly,
\begin{align}
 \de^{-1}\int_{\p(z_1z_2z_3z_4)}\mathrm F(z)^3\,dz
 &=(F_1^3-F_3^3)+i(F_2^3-F_4^3)\nonumber\\
 &=(F_1-F_3)(F_1^2+F_1F_3+F_3^2)\nonumber\\
 &\quad+i(F_2-F_4)(F_2^2+F_2F_4+F_4^2).
 \label{int-square}
\end{align}
Proposition~\ref{F12} gives
\[
 F_1-F_3+i(F_2-F_4)=
 O\left(\Big(\frac\de{d_0}\Big)^{2-\varepsilon}\right),
\]
while Proposition~\ref{TranF} and the second local relation in Proposition~\ref{F12} give
\[
 |F_1-F_2|+|F_2-F_4|
 \preceq
 \Big(\frac\de{d_0}\Big)^{1+\al-\varepsilon}.
\]
Finally, the arm bound gives $|F_j|\preceq(\de/d_0)^\al$.  Substituting these three estimates into \eqref{int-square} yields \eqref{unitsq}.

For \eqref{Lap}, apply \eqref{unitsq} to the adjacent elementary square of the dual lattice.  Its four medial values are $\mathrm F^*$ at $v_1,v_4,v_1',v_4'$.  Proposition~\ref{F12} replaces each $\mathrm F^*$ by $\mathrm F$ with error $O((\de/d_0)^{2-\varepsilon})$; after multiplication by the two remaining observable factors, this error is of the order in \eqref{Lap}.  The proof for $\mathrm F_i$ is the same.
\end{proof}

The following result illustrates another important property of $\mathrm{F}$ and $\mathrm{F}_i$.
\begin{proposition} \label{sum-int-lap}
Let $z_1\in V_{\hat\Omega_\de}$.
Suppose $z_2=z_1+\de \be_1\in V_{\Omega_\de}$, $z_3=z_2+\de \be_2\in V_{\Omega_\de}$ and $z_4=z_3-\de \be_1$. Denote the center of each edge with ends $z_j$ and $z_{j+1}$ by $v_j$. (Here we suppose $z_5=z_1$.) Then
\beqn
\sum_{z_1z_2z_3z_4 \subseteq \hat \Omega_\de} | \int_{z_1z_2z_3z_4} \de^{-1}\mathrm{F}(z)^3dz| \preceq \de^{2\ep/3-\varepsilon} r_\de^{-6}. \label{sum-of-int}
\eeqn
Write $v_1'=z_1-\de\be_1/2$ and $v_4'=z_1-\de\be_2/2$. Then
\beqn
\sum_{z_1\in V_{\hat\Omega_\de}} |\Delta \mathrm{F}^3(z_1)| \preceq \de^{2\ep/3-\varepsilon} r_\de^{-6}, \label{sum-of-lap}
\eeqn
where $\Delta \mathrm{F}^3(z_1)=\mathrm{F}(v_1)^3-\mathrm{F}(v_1')^3+i\big(\mathrm{F}(v_4)^3-\mathrm{F}(v_4')^3\big)$.

The same conclusion holds for $\mathrm{F}_i(v)$ too.
\end{proposition}

\begin{proof}
We divide the squares according to their distance from the boundary.
If $d_0\geq r_\de/2$, \eqref{unitsq} gives
\begin{equation}\label{class1}
 \sum_{d_0\geq r_\de/2}
 \Big(\frac\de{d_0}\Big)^{2+2\al-\varepsilon}
 \preceq
 \de^{2\al-\varepsilon}
 r_\de^{-2-2\al}|\Omega_0|.
\end{equation}

We next consider the boundary layer.  We first work near a
primal-open straight boundary segment.  Let $y$ be the normal
distance to the segment and let $x$ be the tangential distance to
the nearer endpoint.  Put
\[
 m=\max\{x,y\}.
\]
All distances below may be changed by fixed factors without affecting
the estimates.

The local exceptional event has probability
$O\big((\de/ y)^{2-\varepsilon}\big)$.
If $y\leq x$, a configuration in this event which contributes to the
observable must also have an arm from scale comparable with $y$ to
scale comparable with $x$.  The half-plane one-arm estimate gives the
additional factor $ (y/x)^{1/3}$.
This estimate stops before the endpoint of the straight segment.  If
$y>x$, we use the trivial bound $1$.  In the estimates below we may
bound both cases by $(y/x)^{1/3}$, with $x$ replaced by $\de$ when
$x<\de$.

Put
\[
 T=\Big(\frac\de m\Big)^{1+\al-\ep},
 \qquad
 Q=\Big(\frac\de y\Big)^{2-\varepsilon}.
\]
Using \eqref{int-square}, Proposition~\ref{F12},
Proposition~\ref{TranF}, and
\eqref{Finboun-bM}--\eqref{Fdiffinboun-bM}, we obtain
\begin{align}
 \left|
 \int_{\p\square}\de^{-1}\mathrm F(z)^3\,dz
 \right|
 \preceq{}&
 T\left[T+Q\Big(\frac yx\Big)^{1/3}\right]
 \Big(\frac\de y\Big)^{1/3}(\log\de^{-1})^2
 \nonumber\\
 &+Q\Big(\frac yx\Big)^{1/3}
 \Big(\frac\de y\Big)^{2/3}(\log\de^{-1})^4.
 \label{square-boundary}
\end{align}
Notice that the half-plane estimate used in this bound is applied only
when $y\leq x$, and only up to scale $x$.

We split the boundary rectangle into
\[
 \mathrm{REC}_1
 =
 \{x\geq\de^{3\ep},
   \ \rho_\de\leq y\leq r_\de/2\}, \qquad
 \mathrm{REC}_2
 =
 \{x<\de^{3\ep},
   \ \rho_\de\leq y\leq r_\de/2\}.
\]

We first consider $\mathrm{REC}_1$.  Here the distance to the nearest
endpoint is at least $\de^{3\ep}$ up to a fixed factor.  Since
$y\geq\rho_\de=\de^{1/2+\ep}$, the appropriate assertion of
Lemma~\ref{imitialb-imp}, and hence \eqref{Finboun-imp}, applies to
the two medial bonds parallel to the boundary segment.  Therefore, in the last term of \eqref{square-boundary}, the two
remaining observable factors may be bounded by
\[
 \Big(\frac\de{\de^{3\ep}}\Big)^{2/3}
 (\log\de^{-1})^4.
\]
Thus on $\mathrm{REC}_1$ we use
\begin{align}
 \left|
 \int_{\p\square}\de^{-1}\mathrm F(z)^3\,dz
 \right|
 \preceq{}&
 T\left[T+Q\Big(\frac yx\Big)^{1/3}\right]
 \Big(\frac\de y\Big)^{1/3}(\log\de^{-1})^2
 \nonumber\\
 &+Q\Big(\frac yx\Big)^{1/3}
 \Big(\frac\de{\de^{3\ep}}\Big)^{2/3}
 (\log\de^{-1})^4.
 \label{square-boundary-imp}
\end{align}

The first two terms in \eqref{square-boundary-imp} give smaller
powers of $\de$.  For the last term,
\begin{align}
 &\sum_{\square\subseteq\mathrm{REC}_1}
 Q\Big(\frac yx\Big)^{1/3}
 \Big(\frac\de{\de^{3\ep}}\Big)^{2/3}(\log\de^{-1})^4
 \nonumber\\
 &\quad\preceq
 \de^{-2}
 \int_{\rho_\de}^{r_\de/2}
 \int_{\de^{3\ep}}^{O(1)}
 \Big(\frac\de y\Big)^{2-\varepsilon}
 \Big(\frac yx\Big)^{1/3}
 \Big(\frac\de{\de^{3\ep}}\Big)^{2/3}
 \,dx\,dy
 (\log\de^{-1})^4
 \nonumber\\
 &\quad\preceq
 \de^{1/3-8\ep/3-\varepsilon}
 (\log\de^{-1})^4.
 \label{sum-REC1}
\end{align}
For completeness,
\[
 \sum_{\square\subseteq\mathrm{REC}_1}
 T^2\Big(\frac\de y\Big)^{1/3}(\log\de^{-1})^2
 \preceq
 \de^{1/3+2\al-O(\ep)}
 (\log\de^{-1})^2
\]
and
\[
 \sum_{\square\subseteq\mathrm{REC}_1}
 TQ\Big(\frac yx\Big)^{1/3}
 \Big(\frac\de y\Big)^{1/3}(\log\de^{-1})^2
 \preceq
 \de^{5/6+\al-O(\ep)-\varepsilon}
 (\log\de^{-1})^2.
\]
Both are smaller than the right-hand side of
\eqref{sum-REC1} for our choice of $\ep$ and $\varepsilon$.

We next consider $\mathrm{REC}_2$.  Here we use
\eqref{square-boundary} without the improved estimate.  Splitting the
integration into $x\leq y$ and $x>y$ gives
\begin{align}
 \sum_{\square\subseteq\mathrm{REC}_2}
 T^2\Big(\frac\de y\Big)^{1/3}(\log\de^{-1})^2
 &\preceq
 \Big(\frac\de{\rho_\de}\Big)^{
   1/3+2\al-2\ep}
 (\log\de^{-1})^2,
 \label{REC2-first}\\
 \sum_{\square\subseteq\mathrm{REC}_2}
 TQ\Big(\frac yx\Big)^{1/3}
 \Big(\frac\de y\Big)^{1/3}(\log\de^{-1})^2
 &\preceq
 \Big(\frac\de{\rho_\de}\Big)^{
   4/3+\al-\ep-\varepsilon}
 (\log\de^{-1})^2.
 \label{REC2-second}
\end{align}
Both quantities are
$o(\de^{2\ep/3-\varepsilon})$.

For the last term,
\begin{align}
 \sum_{\square\subseteq\mathrm{REC}_2}
 Q\Big(\frac yx\Big)^{1/3}
 \Big(\frac\de y\Big)^{2/3}(\log\de^{-1})^4
 &\preceq
 \de^{-2}
 \int_{\rho_\de}^{r_\de/2}
 \int_0^{\de^{3\ep}}
 \Big(\frac\de y\Big)^{2-\varepsilon}
 \Big(\frac yx\Big)^{1/3}
 \Big(\frac\de y\Big)^{2/3}
 \,dx\,dy
 (\log\de^{-1})^4
 \nonumber\\
 &\preceq
 \de^{2\ep/3-\varepsilon}
 (\log\de^{-1})^4.
 \label{sum-REC2}
\end{align}
Together with \eqref{REC2-first} and
\eqref{REC2-second}, this proves
\[
 \sum_{\square\subseteq\mathrm{REC}_2}
 \left|
 \int_{\p\square}\de^{-1}\mathrm F^3\,dz
 \right|
 \preceq
 \de^{2\ep/3-\varepsilon}
 (\log\de^{-1})^4.
\]

Combining the two regions gives, for one primal-open straight
boundary segment,
\begin{equation}\label{sum-REC}
 \sum_{\square\text{ in its boundary layer}}
 \left|
 \int_{\p\square}\de^{-1}\mathrm F^3\,dz
 \right|
 \preceq
 \de^{2\ep/3-\varepsilon}
 (\log\de^{-1})^4.
\end{equation}
The calculation in $\mathrm{REC}_2$ includes the neighbourhoods of
the ordinary endpoints.  The terminal portions near
$a_\de^\diamond$ and $b_\de^\diamond$ are included in exactly the
same calculation.

On a dual-open boundary segment we first apply the same argument to
$\mathrm F^*$.  Proposition~\ref{F12} then replaces $\mathrm F^*$
by $\mathrm F$ on the corresponding medial bonds.  The replacement error has the same exceptional factor
$Q(y/x)^{1/3}$ as above.  The third assertion
of Lemma~\ref{imitialb-imp} gives the improved bound for the
boundary-parallel dual bonds in $\mathrm{REC}_1$.  Hence
\eqref{sum-REC} also holds on the dual-open side.

Condition $\mathrm C$ gives
$\mathrm n+\mathrm n^*=O(r_\de^{-2})$.
Since
$ (\log\de^{-1})^4=r_\de^{-4}$,
summing \eqref{sum-REC} over all boundary segments gives the upper bound
 $\de^{2\ep/3-\varepsilon}r_\de^{-6}$.
The interior contribution \eqref{class1} is smaller than this
quantity.  This proves \eqref{sum-of-int}.

It remains to prove \eqref{sum-of-lap}.  The proof of \eqref{Lap}
expresses $\Delta\mathrm F^3(z)$ as the cubic contour term around the
adjacent dual square together with the same edge
$\mathrm F^*-\mathrm F$ replacement errors.  In the boundary layer
these replacement errors have the same local exceptional factor
$Q(y/x)^{1/3}$ as above.  In $\mathrm{REC}_1$, the two remaining factors are bounded
by \eqref{Finboun-imp}.  In $\mathrm{REC}_2$, they are bounded by
\eqref{Finboun-bM}.  Consequently the estimates
\eqref{sum-REC1}--\eqref{sum-REC} apply without change.  The bulk
part follows from \eqref{Lap}.  This proves \eqref{sum-of-lap}.

The proof for $\mathrm F_i$ is identical.
\end{proof}

\subsection{Harmonic measure and convergence of the cubic integral} \label{con-of-int}

Let $S=(S_n)_{n\geq0}$ be simple random walk on $\de\Z^2$, with increments uniformly distributed on
\[
 \{\de\be_1,-\de\be_1,\de\be_2,-\de\be_2\}.
\]
For $S_0=z\in V_{\hat\Omega_\de}$, let
\[
 \tau=\inf\{n\geq0:S_n\in\p\hat\Omega_\de\}.
\]
Define
\[
 \mathbf F_\de(z)
 =\Ex_z\int_{z\to S_\tau}\de^{-1}\mathrm F(\zeta)^3\,d\zeta,
 \qquad
 \mathbf F_{i\de}(z)
 =\Ex_z\int_{z\to S_\tau}\de^{-1}\mathrm F_i(\zeta)^3\,d\zeta,
\]
and set both functions equal to zero on $\p\hat\Omega_\de$.  The integrals are taken along the random-walk path, with multiplicity.

\begin{lemma} \label{zeroint}
Suppose that $(\Omega_\de^\diamond,a_\de^\diamond,b_\de^\diamond)$ converges to $(\Omega_0,a,b)$ in the Carath\'eodory sense and satisfies Condition $\mathrm C$.  Then
\[
 \sup_{z\in V_{\hat\Omega_\de}}|\mathbf F_\de(z)|
 +\sup_{z\in V_{\hat\Omega_\de}}|\mathbf F_{i\de}(z)|
 \longrightarrow0.
\]
\end{lemma}

\begin{proof}
Condition on the first step of the walk.  Write the four neighbours of $z$ as
\[
 z+\de\be_1,\quad z+\de\be_2,\quad
 z-\de\be_1,\quad z-\de\be_2,
\]
and let $v_1,v_4,v_1',v_4'$ be the corresponding medial vertices, with the notation of \eqref{Lap}.  Then
\begin{align*}
 \Delta\mathbf F_\de(z)
 &=\frac14\Big[
  \mathrm F(v_1)^3-\mathrm F(v_1')^3
  +i\big(\mathrm F(v_4)^3-\mathrm F(v_4')^3\big)
 \Big].
\end{align*}
Thus $4\Delta\mathbf F_\de(z)$ is exactly the expression on the left-hand side of \eqref{Lap}.  Proposition~\ref{sum-int-lap} gives
\begin{equation}\label{Lapest}
 \sum_{z\in V_{\hat\Omega_\de}}
 |\Delta\mathbf F_\de(z)|
 \preceq\de^{2\ep/3-\varepsilon}r_\de^{-6}.
\end{equation}
By Theorem~1.4.6 of~\cite{Law91}, or equivalently by the stopped Green representation,
\[
 \mathbf F_\de(z)
 =-\sum_{w\in V_{\hat\Omega_\de}}
 G_{\hat\Omega_\de}(z,w)\Delta\mathbf F_\de(w).
\]
In a bounded planar domain,
\[
 \sup_{z,w}G_{\hat\Omega_\de}(z,w)
 \preceq\log\de^{-1}.
\]
The right-hand side of \eqref{Lapest}, multiplied by this logarithm, tends to zero.  This proves the assertion for $\mathbf F_\de$.  The proof for $\mathbf F_{i\de}$ is identical.
\end{proof}

We next fix the normalization constant.  Let $\mathcal C_\de$ be the lattice approximation, inside $\hat\Omega_\de$, of the component of
\[
 \p B(b_\de,\de^{1/3})\cap\Omega_\de
\]
which joins $\p^{lo}$ to $\p^{up}$.  We orient it from $\p^{lo}$ to $\p^{up}$.  If two nearest lattice approximations are possible, choose any of them.  Define
\begin{equation}\label{c-de-ci}
 \mathrm c_\de
 =\Im\int_{\mathcal C_\de}\de^{-1}\mathrm F(z)^3\,dz,
 \qquad
 \mathrm c_{i\de}
 =\Im\int_{\mathcal C_\de}\de^{-1}\mathrm F_i(z)^3\,dz.
\end{equation}
Pass to a sequence along which
\[
 \mathrm c_\de\longrightarrow\mathrm c,
 \qquad
 \mathrm c_{i\de}\longrightarrow\mathrm c_i
\]
in the extended real line.  We first work under the assumption that both limits are finite.  This assumption will be removed in Subsection~\ref{Con-F-boudary}.

Let $\mathbf F:\Omega_0\to\R\times(0,1)$ be the conformal map with $\mathbf F(a)=-\infty$, $\mathbf F(b)=\infty$, chosen so that $\mathbf F(\p^{up})$ converges to the lower side $\R$ and $\mathbf F(\p^{lo})$ converges to the upper side $\R+i$.  Then
$\Im\mathbf F(z)$ is the limit of $\operatorname{hm}_{\Omega_0}(z,\p^{lo})$ as $\de\to 0$,
which is the harmonic measure of $\p^{lo}$ from $z$ in $\Omega_0$,

\begin{proposition}\label{conformal1}
Let $\mathcal P_\de$ be an edge-avoiding lattice path in
$\hat\Omega_\de$ from $z_\de$ to $\p^{up}$, where
$z_\de\to z\in\Omega_0$.  Assume that the terminal vertex of
$\mathcal P_\de$ is admissible and is at distance at least a constant
multiple of $L_\de$ from $\p^{lo}$.

Under Condition $\mathrm C$ and the assumption $|\mathrm c|<\infty$,
\[
 \Im\int_{\mathcal P_\de}
 \de^{-1}\mathrm F(\zeta)^3\,d\zeta
 \longrightarrow
 \mathrm c\,\Im\mathbf F(z),
\]
uniformly when $z$ ranges over a compact subset of $\Omega_0$.

Similarly, under Condition $\mathrm C$ and the assumption
$|\mathrm c_i|<\infty$,
\[
 \Im\int_{\mathcal P_\de}
 \de^{-1}\mathrm F_i(\zeta)^3\,d\zeta
 \longrightarrow
 \mathrm c_i\,\Im\mathbf F(z),
\]
uniformly when $z$ ranges over a compact subset of $\Omega_0$.
\end{proposition}

\begin{proof}
We give the proof for $\mathrm F$.  Let $LS$ be the chronological loop erasure of the stopped random-walk path $S[0,\tau]$. One can refer to Section~7.2 of~\cite{Law91}.  Erasing one loop changes the path integral by the sum of the cubic integrals around the elementary squares enclosed by that loop, counted with winding number.  Hence
\begin{align}
 \left|\mathbf F_\de(z)
 -\Ex_z\int_{LS}\de^{-1}\mathrm F(\zeta)^3\,d\zeta\right|
 &\preceq
 \sup_w\Ex_z\sup_{n\leq\tau}|\Theta(n,w;S)|
 \sum_{\square\subseteq\hat\Omega_\de}
 \left|\int_{\p\square}\de^{-1}\mathrm F^3dz\right|.
 \label{looperase}
\end{align}
The random-walk winding estimate of Lemma~\ref{winding} gives
\begin{equation}\label{windexp}
 \sup_{z,w}\Ex_z\sup_{n\leq\tau}|\Theta(n,w;S)|
 \preceq\log\de^{-1}.
\end{equation}
Together with \eqref{sum-of-int},
\begin{equation}\label{difloop}
 \left|\mathbf F_\de(z)
 -\Ex_z\int_{LS}\de^{-1}\mathrm F^3dz\right|
 \preceq\de^{2\ep/3-\varepsilon}r_\de^{-7}=o(1).
\end{equation}

Put
\[
 I_\de(z)=
 \Im\int_{\mathcal P_\de}\de^{-1}\mathrm F^3dz.
\]
We first control random-walk exits on $\p^{up}\cup\p^{lo}$ which are
close to $a_\de$ or $b_\de$.   We begin with
the terminal portion of $\p^{up}$ adjacent to $b_\de$.  Let $x$ be
the tangential distance from $b_\de$.  On this portion the normal
distance from the corresponding boundary segment is comparable with
$\rho_\de$.   For
\[
 \eta_\de\preceq x\leq\de^{1/3},
\]
Lemma~\ref{initial-b} gives
\[
 |\mathrm F(v)|
 \preceq
 \left(\frac\de{\rho_\de}\right)^{1/3}
 (\log\de^{-1})^2.
\]
Therefore
\begin{equation}\label{int-lk-1}
 \int_{\eta_\de\preceq x\leq\de^{1/3}}
 \de^{-1}|\Im\mathrm F(z)^3|\,|dz|
 \preceq
 \de^{-1/6-\ep}(\log\de^{-1})^6.
\end{equation}

For
\[
 \de^{1/3}\leq x\leq L_\de,
\]
write $\mathrm F(v)=q_v+e_v$ as in Proposition~\ref{ref-inv}.  Lemma~\ref{initial-b} and the three error terms in \eqref{Fvboundary-re} give
\[
 |\mathrm F(v)|+|q_v|
 \preceq
 \left(\frac\de{\rho_\de}\right)^{1/3}
 (\log\de^{-1})^2.
\]
Since $\Im(q_v^3)=0$,
\[
 |\Im\mathrm F(v)^3|
 \preceq
 |e_v|\bigl(|\mathrm F(v)|+|q_v|\bigr)^2.
\]
Consequently, summing over the lattice values of $x$,
\begin{align}
 &\int_{\de^{1/3}\leq x\leq L_\de}
 \de^{-1}|\Im\mathrm F(z)^3|\,|dz| \nonumber\\
 &\quad\preceq
 \sum_{\de^{1/3}\leq x\leq L_\de}
 \Bigg[
  \left(\frac\de x\right)^{1+\al}
  +\left(\frac\de{\rho_\de}\right)^\al
   \left(\frac{\rho_\de}x\right)^{2-\varepsilon}
  +\left(\frac\de{\rho_\de}\right)^{2-\varepsilon}
   \left(\frac{\rho_\de}x\right)^{1/3}
 \Bigg]
 \left(\frac\de{\rho_\de}\right)^{2/3}
 (\log\de^{-1})^4
 \preceq \de^{\al/2}.
 \label{int-lk-2}
\end{align}

The estimate \eqref{int-lk-2}, together with Proposition~\ref{sum-int-lap}, shows that moving a boundary endpoint between tangential distances $L_\de$ and $\de^{1/3}$ changes the imaginary cubic integral by $o(1)$.  The remaining part at $b_\de$ is exactly the crosscut $\mathcal C_\de$, whose imaginary integral is $\mathrm c_\de$.  On the closer part, Lemma~\ref{dhm}, through \eqref{hm4} and \eqref{hm5}, gives
\[
 \Pro_z\bigl(S_\tau\hbox{ exits with }x\leq\de^{1/3}\bigr)
 \preceq \frac{\de^{1/3}}{r_\de}.
\]
Multiplying this by \eqref{int-lk-1} gives
\[
 \de^{1/6-\ep}r_\de^{-1}(\log\de^{-1})^6=o(1).
\]
Thus the close terminal part is negligible after averaging over the random-walk exit point.
The same argument applies to the terminal portion adjacent to $a_\de$.
On $\p^{lo}$ we first apply the corresponding estimate to $\mathrm F^*$
and then use Proposition~\ref{F12}.

Since Proposition~\ref{boundofin} is stated for admissible endpoints, it
remains to treat exits in the relative interiors of reentrant corner
blocks.  Fix one reentrant corner \(q\).  If \(w\in\mathcal K_q^\de\) is
an inadmissible exit point, let \(\widehat w\) be the admissible endpoint
obtained by completing \(w\) along the same block.
 It is enough to show
that
\[
 \Ex_z\left[
 \left|
 \Im\int_{S_\tau}^{\widehat S_\tau}
 \de^{-1}\mathrm F(\zeta)^3\,d\zeta
 \right|
 \I(S_\tau\in\mathcal K_q^\de)
 \right]
\]
is sufficiently small.

We first note a standard harmonic measure estimate.  For
\(\rho_\de\leq s\leq L_\de\), let \(\mathcal K_q^\de(s)\) denote the part
of the block within distance \(Cs\) of \(q\), where \(C\) is fixed.
Choose \(c>0\) small.  One of the two boundary segments adjacent to \(q\)
crosses the annulus between radii \(Cs\) and \(cr_\de\).  The discrete
Beurling estimate gives a factor
$
 O\big((s/r_\de)^{1/2}\big)
$
for reaching the inner scale without hitting this boundary segment.
Applying the same estimate between radii \(cr_\de\) and \(Mcr_\de\), with
\(M\) fixed sufficiently large, and using the strong Markov property to
sum over possible returns, we obtain
\[
 \Pro_z\bigl(S_\tau\in\mathcal K_q^\de(s)\bigr)
 \preceq
 \left(\frac{s}{r_\de}\right)^{1/2}.
\]

Put
\[
 s_\de=\de^{1/3+\ep}.
\]
Suppose first that \(S_\tau\in\mathcal K_q^\de(s_\de)\).  The part of the
inner contour needed to move \(S_\tau\) to the point at tangential
distance \(s_\de\) has length \(O(s_\de)\).  The rough bound from the
proof of Lemma~\ref{initial-b} gives
\[
 |\mathrm F(v)|
 \preceq
 \left(\frac{\de}{\rho_\de}\right)^{1/3}
 (\log\de^{-1})^2
\]
on this part.  Hence
\[
 \left|
 \Im\int
 \de^{-1}\mathrm F(\zeta)^3\,d\zeta
 \right|
 \preceq
 \frac{s_\de}{\rho_\de}
 (\log\de^{-1})^6.
\]
After multiplication by the preceding harmonic-measure bound, its
expected contribution is at most
\[
 C
 \frac{s_\de}{\rho_\de}
 \left(\frac{s_\de}{r_\de}\right)^{1/2}
 (\log\de^{-1})^6
 =
 C\de^{\ep/2}r_\de^{-1/2}
 (\log\de^{-1})^6
 =
 o(1).
\]

For the remaining part of the completion path, the tangential distance
from the corner is at least \(s_\de\).  The calculation used to obtain
\eqref{int-lk-2}, with the lower cutoff \(s_\de\) in place of
\(\de^{1/3}\), applies without change.  It gives, up to powers of
\(\log\de^{-1}\),
\[
 \left(\frac{\de}{\rho_\de}\right)^{2/3}
 \left(\frac{\de}{s_\de}\right)^\al
 +
 \left(\frac{\de}{\rho_\de}\right)^{\al-1/3}
 \left(\frac{\rho_\de}{s_\de}\right)^{1-\varepsilon}
 +
 \left(\frac{\de}{\rho_\de}\right)^{5/3-\varepsilon}
 \left(\frac{L_\de}{\rho_\de}\right)^{2/3},
\]
which is \(o(1)\).  Thus the expected difference between the imaginary
line integrals ending at an inadmissible exit point and at its admissible
completion is bounded, for one reentrant corner, by a positive power of
\(\de\), up to logarithmic factors.

There are \(O(r_\de^{-2})\) reentrant corners.  Since \(r_\de^{-1}\) is
logarithmic, summing over them still gives \(o(1)\).  Hence replacing an
inadmissible random-walk exit by its admissible completion changes the
expected imaginary cubic integral by \(o(1)\).

We may now compare the loop-erased path with $\mathcal P_\de$.  Proposition~\ref{boundofin}, Proposition~\ref{sum-int-lap}, and the preceding endpoint estimates give
\begin{align}
 \Ex_z\left[
  \Im\int_{LS}\de^{-1}\mathrm F^3dz\,
  \I(S_\tau\in\p^{up})
 \right]
 &=I_\de(z)\Pro_z(S_\tau\in\p^{up})+o(1),
 \label{up-exit-average}\\
 \Ex_z\left[
  \Im\int_{LS}\de^{-1}\mathrm F^3dz\,
  \I(S_\tau\in\p^{lo})
 \right]
 &=\bigl(I_\de(z)-\mathrm c_\de\bigr)
   \Pro_z(S_\tau\in\p^{lo})+o(1).
 \label{lo-exit-average}
\end{align}
Indeed, for an exit on $\p^{lo}$ we append the boundary part leading to $\mathcal C_\de$ and then traverse $\mathcal C_\de$ from $\p^{lo}$ to $\p^{up}$.  The resulting path ends on $\p^{up}$, so its integral is $I_\de(z)+o(1)$.

It remains to treat $\p^{ex}$.  For each side $J$ of an endpoint box, Lemma~\ref{initial-b}, its dual version, and Proposition~\ref{F12} give
\[
 \left|\int_J\de^{-1}\mathrm F(z)^3\,dz\right|
 \preceq r_\de^{-7}.
\]
For example, after a rotation the relevant sum is bounded by
\[
 C(\log\de^{-1})^6
 \sum_{c\rho_\de\leq x\leq C\eta_\de}\frac\de x
 \preceq r_\de^{-7}.
\]
Equations \eqref{hm1}--\eqref{hm3} of Lemma~\ref{dhm} give harmonic measure $O(\eta_\de/r_\de)$ for these sides.  Combining these bounds with \eqref{int-lk-1}, \eqref{int-lk-2}, and the square-curl estimate \eqref{sum-of-int} gives
\begin{equation}\label{exboundary}
 \left|
 \Ex_z\left[
  \Im\int_{LS}\de^{-1}\mathrm F^3dz\,
  \I(S_\tau\in\p^{ex})
 \right]
 -I_\de(z)\Pro_z(S_\tau\in\p^{ex})
 \right|=o(1).
\end{equation}
Here an exit through $\p^{ex}$ is first completed to the adjacent part of $\p^{up}$ or $\p^{lo}$.  If the latter choice is used, the reference integral differs by $\mathrm c_\de+o(1)$.  Since $\mathrm c_\de$ is bounded along the present subsequence and Lemma~\ref{dhm} gives
\[
 \Pro_z(S_\tau\in\p^{ex})
 \preceq \frac{\eta_\de}{r_\de},
\]
this distinction contributes only $o(1)$.

Adding \eqref{up-exit-average}, \eqref{lo-exit-average}, and \eqref{exboundary}, we obtain
\[
 \Ex_z\Im\int_{LS}\de^{-1}\mathrm F^3dz
 =I_\de(z)-\mathrm c_\de\Pro_z(S_\tau\in\p^{lo})+o(1).
\]
By Lemma~\ref{zeroint} and \eqref{difloop}, the left-hand side is $o(1)$.  Therefore
\begin{equation}\label{Idelta-hm}
 I_\de(z)
 =\mathrm c_\de\Pro_z(S_\tau\in\p^{lo})+o(1).
\end{equation}
Discrete harmonic measure converges to its continuum counterpart under the present inner approximation.  Hence
\[
 \Pro_z(S_\tau\in\p^{lo})
 \longrightarrow
 \Im\mathbf F(z)
\]
uniformly for $z$ in compact subsets of $\Omega_0$.  Since $\mathrm c_\de\to\mathrm c$, this proves the first assertion.  The proof for $\mathrm F_i$ is identical.
\end{proof}

\subsection{Convergence of $\mathrm F$ inside $\Omega_0$}

 We first remove the logarithmic loss in Lemma~\ref{initial-b}.  The argument compares the cubic integrals started from two neighbouring primal vertices by coupling the corresponding random walks.

\begin{lemma}\label{Fimboun}
Under Condition~$\mathrm C$ and the assumption $|\mathrm c|<\infty$,
\[
 |\de^{-1/3}\mathrm F(v)|
 \preceq 1
\]
uniformly for \(v\) in any fixed compact subset of \(\Omega_0\).

Similarly, under Condition~$\mathrm C$ and the assumption
$|\mathrm c_i|<\infty$,
\[
 |\de^{-1/3}\mathrm F_i(v)|
 \preceq 1
\]
uniformly for \(v\) in any fixed compact subset of \(\Omega_0\).
\end{lemma}

\begin{proof}
Fix a compact set \(\mathcal K\subseteq\Omega_0\).  Choose \(d>0\) so
that the \(4d\)-neighbourhood of \(\mathcal K\) is contained in
\(\Omega_0\).  Let \(z_1,z_2=z_1+\de\be\) be adjacent primal vertices
in \(\mathcal K\), where \(\be=\be_1\) or \(\be_2\), and let
\[
 v=\frac{z_1+z_2}{2}.
\]
Put
\[
 n_\de=\left\lfloor\frac{\log(d/\de)}{\log2}\right\rfloor-2.
\]

  Start a
simple random walk \(S\) from \(z_1\) and a simple random walk \(S'\)
from \(z_2\).  If the first step of \(S\) is \(z_2\), the two walks are
coupled from that step onward.  Otherwise, reflect the initial part of
\(S\) across the perpendicular bisector, and couple the two walks when they first meet.  Let
\(\sigma\) be the coupling time.  For \(j\geq0\), put
\[
 \phi_j=\inf\{n\geq0:S_n\notin B(z_1,2^j\de)\}.
\]
The reflection coupling gives
\begin{equation}\label{coupling-tail}
 \Pro(\sigma>\phi_j)\preceq2^{-j}.
\end{equation}

We first provide the local estimate needed below.  By
\eqref{Finboun-M}--\eqref{Fdiffinboun-M}, Proposition~\ref{F12}, and
the calculation in \eqref{int-square}, every elementary square
\(\square\subseteq B(z_1,d)\) satisfies
\begin{equation}\label{Fimboun-local-curl}
 \left|
 \int_{\p\square}\de^{-1}\mathrm F(z)^3\,dz
 \right|
 \preceq
 \de^{8/3-\varepsilon}
 (\log\de^{-1})^4.
\end{equation}
Indeed, the error in the local Cauchy--Riemann relation contributes at
most
\[
 C\de^{2-\varepsilon}
 \de^{2/3}(\log\de^{-1})^4,
\]
and the terms containing differences of two nearby values of
\(\mathrm F\) are smaller.

Let \(\tau\) and \(\tau'\) be the respective hitting times of
\(\p\hat\Omega_\de\).  By the definition of \(\mathbf F_\de\),
\begin{align}
 &\mathbf F_\de(z_1)-\mathbf F_\de(z_2)
 -\mathrm F(v)^3\be                                      \nonumber\\
 &\quad=
 \Ex\left[
 \left(
 \int_{z_1\to S_\tau}\de^{-1}\mathrm F(z)^3\,dz
 -
 \int_{z_1\to z_2\to S'_{\tau'}}
       \de^{-1}\mathrm F(z)^3\,dz
 \right)
 \I(\sigma\leq\phi_{n_\de})
 \right]                                                  \nonumber\\
 &\qquad+
 \Ex\left[
 \left(
 \int_{z_1\to S_\tau}\de^{-1}\mathrm F(z)^3\,dz
 -
 \int_{z_1\to z_2\to S'_{\tau'}}
       \de^{-1}\mathrm F(z)^3\,dz
 \right)
 \I(\sigma>\phi_{n_\de})
 \right].
 \label{Fimboun-coupling}
\end{align}

We first consider the first expectation.  On
\[
 \{\phi_j<\sigma\leq\phi_{j+1}\},
\]
the two walks agree after coupling.  Hence their difference is the
integral around the closed walk formed before the coupling time.  It
is a sum of the elementary cubic integral in
\(B(z_1,2^{j+1}\de)\), counted with winding number.

Let \(\Theta(n,w)\) be the winding of \(S[0,n]\) around \(w\).
Lemma~\ref{winding}, the standard exit-time moment estimate for a
random walk in a disk, and H\"older's inequality give, for every
\(q>5/4\),
\[
 \Ex\left[
 \max_{m\leq\phi_{j+1}}|\Theta(m,w)|
 \I(\phi_j<\sigma)
 \right]
 \preceq
 (j+1)\Pro(\phi_j<\sigma)^{1/(2q)}.
\]
Using \eqref{coupling-tail} and \eqref{Fimboun-local-curl}, we obtain
\begin{align*}
 &\left|
 \Ex\left[
 \left(
 \int_{z_1\to S_\tau}\de^{-1}\mathrm F(z)^3\,dz
 -
 \int_{z_1\to z_2\to S'_{\tau'}}
       \de^{-1}\mathrm F(z)^3\,dz
 \right)
 \I(\phi_j<\sigma\leq\phi_{j+1})
 \right]
 \right|                                                  \\
 &\qquad\preceq
 \de^{8/3-\varepsilon}
 (\log\de^{-1})^4
 (j+1)\,
 2^{2(j+1)-j/(2q)}.
\end{align*}
The case \(\sigma\leq\phi_0\) is bounded in the same way.

Take \(q=11/8\).  Since \(2^{n_\de}\asymp\de^{-1}\),
\begin{align}
 \sum_{j=0}^{n_\de-1}
 \de^{8/3-\varepsilon}
 (\log\de^{-1})^4
 (j+1)2^{2(j+1)-j/(2q)}
 \preceq
 \de^{34/33-\varepsilon}
 (\log\de^{-1})^5
 =o(\de),
 \label{Fimboun-small}
\end{align}
after decreasing \(\varepsilon\), if necessary.

It remains to consider \(\{\sigma>\phi_{n_\de}\}\).  By
\eqref{coupling-tail},
\[
 \Pro(\sigma>\phi_{n_\de})\preceq\de.
\]
Apply the strong Markov property at \(\phi_{n_\de}\).  Lemma~\ref{zeroint}
shows that the conditional expected contribution from the part of the
walk after \(\phi_{n_\de}\) is \(o(1)\), uniformly in the exit point
from \(B(z_1,2^{n_\de}\de)\).  Its contribution on
\(\{\sigma>\phi_{n_\de}\}\) is therefore \(o(\de)\).

For the part before \(\phi_{n_\de}\), replace the walk by its
chronological loop erasure.  Proposition~\ref{conformal1}, together
with the path-comparison estimate used in its proof, gives a uniform
\(O(1)\) bound for the imaginary cubic integral along such an
edge-avoiding path contained in the fixed interior disk.  Hence its
contribution on \(\{\sigma>\phi_{n_\de}\}\) is \(O(\de)\).  The
difference between the walk and its loop erasure is again a sum of
elementary cubic integrals counted with winding number.  Applying the preceding
estimate with \(j=n_\de\) gives
\[
 O\left(
 \de^{34/33-\varepsilon}
 (\log\de^{-1})^5
 \right)
 =o(\de).
\]
The same argument applies to the corresponding part of \(S'\).
Consequently,
\begin{equation}\label{bfF}
 \left|
 \Im\left(
 \mathbf F_\de(z_1)-\mathbf F_\de(z_2)
 -\mathrm F(v)^3\be
 \right)
 \right|
 \preceq\de.
\end{equation}

We next estimate the first two terms in \eqref{bfF}.  In the fixed
interior disk, \eqref{Lap} gives
\[
 |\Delta\mathbf F_\de|
 \preceq
 \de^{2+2\al-\varepsilon}.
\]
Moreover, Lemma~\ref{zeroint} gives
\[
 \sup|\mathbf F_\de|=o(1).
\]
Applying Lemma~\ref{asyhar} in this disk therefore gives
\[
 |\mathbf F_\de(z_1)-\mathbf F_\de(z_2)|
 \preceq\de.
\]
Together with \eqref{bfF},
\begin{equation}\label{Fimboun-one-component}
 |\Im(\mathrm F(v)^3\be)|\preceq\de.
\end{equation}

To obtain the other component, let \(\widetilde v\) be the adjacent
medial vertex corresponding to the perpendicular primal edge.  Thus
\[
 |\widetilde v-v|=\frac{\de}{\sqrt2}.
\]
Proposition~\ref{TranF} gives
\[
 |\mathrm F(\widetilde v)-\mathrm F(v)|
 \preceq
 \de^{1+\al-\ep}.
\]
Using the preliminary bound \eqref{Finboun-M},
\[
 |\mathrm F(\widetilde v)^3-\mathrm F(v)^3|
 =o(\de).
\]
Applying \eqref{Fimboun-one-component} at \(\widetilde v\), where the
corresponding primal direction is perpendicular to \(\be\), controls
the other component of \(\mathrm F(v)^3\).  Hence $ |\mathrm F(v)^3|\preceq\de$,
and therefore
$ |\mathrm F(v)|\preceq\de^{1/3}$.

The proof for \(\mathrm F_i\) is identical.  This proves the lemma.
\end{proof}

Now with the help of Lemma \ref{Fimboun}, we can remove the log term in \eqref{Fdiffinboun-M}. This is presented in the following result.

\begin{proposition} \label{improveb}
Under Condition $\mathrm C$ and the assumption $|\mathrm c|+|\mathrm c_i|<\infty$,
\begin{eqnarray}
 |\mathrm F(v+\de\be)-\mathrm F(v)|
 &\preceq&\de^{1+1/3},
 \label{Fdiffimboun}\\
 |\mathrm F_i(v+\de\be)-\mathrm F_i(v)|
 &\preceq&\de^{1+1/3},
 \label{Fdiffimbouni}
\end{eqnarray}
uniformly on compact subsets, where $\be=\be_1$ or $\be_2$.  In addition, if
\[
 \mathrm F_+(v)=
 \mathrm F(v)+\mathrm F^*(v+e^{-i\pi/4}\de/\sqrt2),
\]
then
\begin{equation}\label{Fpi4}
 |\mathrm F_+(v+\de e^{\pm i\pi/4}/\sqrt2)-\mathrm F_+(v)|
 \preceq\de^{1+1/3}.
\end{equation}
\end{proposition}

\begin{proof}
Fix a compact set \(\mathcal K\subseteq\Omega_0\), and choose \(d>0\)
such that the \(d\)-neighbourhood of \(\mathcal K\) is contained in
\(\Omega_0\).  Fix \(z\in\mathcal K\), and in \(B(z,d/2)\) put
\[
 G(w)=\mathbb F(w)-\mathbb F(z),
\]
where \(\mathbb F\) is the primitive in \eqref{primitive-F}.
By Lemma~\ref{Fimboun} and \(F=\mathrm F+\mathrm F_i\),
\[
 |F(v)|\preceq\de^{1/3}
\]
uniformly in \(B(z,d/2)\).  Hence every edge increment of \(G\) is
\(O(\de)\).  Summing along a lattice path of bounded Euclidean length
gives
$\|G\|_\p\preceq1$.

Since \(B(z,d/2)\) stays a fixed positive distance from
\(\p\Omega_0\), \eqref{primitive-local-laplacian} gives
\[
 |\Delta G(w)|\preceq\de^{8/3}.
\]
After rescaling the disk by a fixed factor, the second-difference
estimate \eqref{asyhar-second} in Lemma~\ref{asyhar} therefore gives
\[
 |\nabla_{\de\be}^2G(w)|\preceq\de^2,
 \qquad
 \be=\be_1,\be_2.
\]
If \(v\) is the midpoint of the primal edge from \(w\) to
\(w+\de\be\), then \eqref{primitive-F} gives
\[
 \nabla_{\de\be}^2G(w)
 =
 \de^{2/3}\be
 \bigl(F(v+\de\be)-F(v)\bigr).
\]
Consequently,
\begin{equation}\label{Diff-F}
 |F(v+\de\be)-F(v)|
 \preceq
 \de^{1+1/3}.
\end{equation}

Proposition~\ref{F12} transfers \eqref{Diff-F} to \(\mathrm F\), which
gives \eqref{Fdiffimboun}.  Since
$F=\mathrm F+\mathrm F_i$,
equations \eqref{Diff-F} and \eqref{Fdiffimboun} also give
\eqref{Fdiffimbouni}.

It remains to prove \eqref{Fpi4}.  Put
$
 h=\de e^{-i\pi/4}/\sqrt2
$
and define
$ \mathrm F_+^*(v)
 =
 \mathrm F^*(v)+\mathrm F(v+h)$.
By Proposition~\ref{F12},
\[
 \mathrm F_+^*(v+h)
 =
 \mathrm F_+(v+h)+O(\de^2).
\]
On the other hand,
\[
 \mathrm F_+(v)-\mathrm F_+^*(v+h)
 =
 \mathrm F(v)-\mathrm F(v+\de(1-i)).
\]
The right-hand side is the sum of two coordinate differences, so
\eqref{Fdiffimboun} gives
\[
 |\mathrm F_+(v)-\mathrm F_+^*(v+h)|
 \preceq
 \de^{1+1/3}.
\]
Since \(O(\de^2)\) is smaller, this proves \eqref{Fpi4} for the
displacement \(\de e^{-i\pi/4}/\sqrt2\).  The other diagonal
displacement is treated in the same way.
\end{proof}

Finally, we can present the convergence of $\mathrm{F}$ and $\mathrm{F}_i$ in any compact subset of $\Omega_0$,

\begin{proposition} \label{convF}
Suppose that $\Omega_\de^\diamond$ converges to $(\Omega_0,a,b)$ in the Carath\'eodory sense and satisfies Condition $\mathrm C$.  If $|\mathrm c|<\infty$, then
\[
 \de^{-1/3}\mathrm F(v)
 \longrightarrow
 \mathrm c^{1/3}\big(\mathbf F'(v)\big)^{1/3},
 \qquad
 \de^{-1/3}\mathrm F_i(v)\longrightarrow0
\]
uniformly on every compact subset of $\Omega_0$.  The cube root is the branch determined by the orientation of the medial edges and the boundary winding convention.
\end{proposition}

\begin{proof}
Pass, if necessary, to a subsequence along which
\[
 \mathrm c_{i\de}\longrightarrow \mathrm c_i
\]
in the extended real line.

We first note that the normalization constants are stable under a change
of the domain away from \(b_\de\).  Suppose that two Dobrushin domains,
together with their boundary conditions, agree in
\(B(b_\de;\de^\ep)\).  For a medial vertex \(v\) on the normalization
crosscut \(\mathcal C_\de\), restrict the corresponding expectation to
configurations for which the reversed exploration path reaches \(v\)
before crossing \(\p B(b_\de;\de^\ep)\).  If the restricted and
unrestricted observables differ, the reversed exploration path must
cross from scale \(\de^{1/3}\) to scale \(\de^\ep\) before reaching
\(v\).  Hence
\[
 |\mathrm F(v)-\mathrm F_{\rm res}(v)|
 +
 |\mathrm F_i(v)-\mathrm F_{i,{\rm res}}(v)|
 \preceq
 \de^{(1+\al)(1/3-\ep)}.
\]
Together with Lemma~\ref{initial-b},
\[
 |x^3-y^3|\leq3(|x|+|y|)^2|x-y|,
\]
and summation over the boundary layers
\(\rho_\de\leq d_v\leq\de^{1/3}\), this gives
\[
 |\mathrm c_\de-\widetilde{\mathrm c}_\de|
 +
 |\mathrm c_{i\de}-\widetilde{\mathrm c}_{i\de}|
 \preceq
 \de^{\al/3-(1+\al)\ep}
 (\log\de^{-1})^4
 =
 o(1).
\]
Thus the limiting constants are unchanged if the domain is modified
outside \(B(b_\de;\de^\ep)\).

Suppose first that \(|\mathrm c_i|<\infty\).  Lemma~\ref{Fimboun} and
Proposition~\ref{improveb} imply precompactness of
\(\de^{-1/3}\mathrm F\) and \(\de^{-1/3}\mathrm F_i\) on compact
subsets.  Let \(f_0\) be a subsequential limit of
\(\de^{-1/3}\mathrm F\).  The local contour estimate \eqref{unitsq}
and Morera's theorem show that \(f_0^3\) is analytic.

For two points \(v,w\) in a compact set, choose lattice approximations
and a lattice path joining them.  Proposition~\ref{conformal1} gives
\[
 \Im\int_w^v f_0(z)^3\,dz
 =
 \mathrm c\,\Im\bigl(\mathbf F(v)-\mathbf F(w)\bigr).
\]
Hence
\[
 z\longmapsto
 \int_w^z f_0(\zeta)^3\,d\zeta
 -
 \mathrm c\bigl(\mathbf F(z)-\mathbf F(w)\bigr)
\]
is analytic and has constant imaginary part.  It is therefore constant,
and
\[
 f_0(z)^3=\mathrm c\,\mathbf F'(z).
\]
The winding convention fixes the cube-root branch.  Thus every
subsequential limit is the same.

We next consider \(\mathrm F_i\).  Let \(V_{\rm h}\) and \(V_{\rm v}\)
be the medial vertices for which \(E_v\) is respectively horizontal
and vertical.  Along a further subsequence, the restrictions of
\(\de^{-1/3}\mathrm F_i\) to these two sets converge to limits
\(f_{iH}\) and \(f_{iV}\).  Proposition~\ref{F12} and
Proposition~\ref{TranF} give
\[
 f_{iH}=-f_{iV}.
\]
The horizontal and vertical cubic integrals and the second assertion
of Proposition~\ref{conformal1} give
\begin{equation}\label{imag-real}
 \Im(f_{iH}^3)=\Im(\mathrm c_i\mathbf F'),
 \qquad
 \Re(f_{iV}^3)=\Re(\mathrm c_i\mathbf F').
\end{equation}

It remains to identify \(\mathrm c_i\).  By the stability proved above,
for this purpose we may modify the domain outside
\(B(b_\de;\de^\ep)\) and take a square which is symmetric about the
diagonal joining \(a_\de^\diamond\) to \(b_\de^\diamond\).  Reflection
of the percolation configuration across this diagonal gives, for a
medial vertex \(v\) on the diagonal,
\begin{equation}\label{conjugate}
 F(v)e_b^{1/3}
 =
 \overline{F^*(v)e_b^{1/3}}.
\end{equation}
Using \(F=\mathrm F+\mathrm F_i\), the relation between
\(\mathrm F\) and \(\mathrm F^*\) in Proposition~\ref{F12}, and
\eqref{imag-real}, we obtain
\begin{equation}\label{formula-FiH}
 f_{iH}(v)
 =
 -\mathrm c_i^{1/3}
 \bigl(\overline{\mathbf F'(v)}\bigr)^{1/3}.
\end{equation}
Substituting the limits of \(\mathrm F\) and \(\mathrm F_i\) into
\eqref{conjugate} gives
\begin{align}
 \bigl(\mathrm c\mathbf F'(v)e_b\bigr)^{1/3}
 -
 \overline{\bigl(\mathrm c\mathbf F'(v)e_b\bigr)^{1/3}}
=
 \bigl(\mathrm c_i\overline{\mathbf F'(v)}e_b\bigr)^{1/3}
 +
 \overline{
 \bigl(\mathrm c_i\overline{\mathbf F'(v)}e_b\bigr)^{1/3}}.
 \label{ci-relation1}
\end{align}
The left-hand side is purely imaginary and the right-hand side is real.
With the cube-root branches fixed by the boundary winding convention,
the right-hand side is nonzero when \(\mathrm c_i\neq0\).  Hence
\[
 \mathrm c_i=0.
\]
Equation \eqref{imag-real}, together with \(f_{iH}=-f_{iV}\), then gives
\[
 f_{iH}=f_{iV}=0.
\]

It remains to exclude \(|\mathrm c_i|=\infty\).  In that case divide
all the observables by \(|\mathrm c_{i\de}|^{1/3}\).  The corresponding
normalization constants satisfy, after passing to a further subsequence,
\[
 \frac{\mathrm c_\de}{|\mathrm c_{i\de}|}\longrightarrow0,
 \qquad
 \frac{\mathrm c_{i\de}}{|\mathrm c_{i\de}|}\longrightarrow\pm1.
\]
The proofs of Lemma~\ref{Fimboun}, Proposition~\ref{improveb}, and
Proposition~\ref{conformal1} are homogeneous under this normalization.
Hence the preceding compactness and square-reflection argument applies
to the normalized observables.  Equation \eqref{ci-relation1} then has
zero left-hand side and a nonzero real right-hand side, a contradiction.
Therefore \(|\mathrm c_i|<\infty\).

We have proved that every subsequential limit of
\(\mathrm c_{i\de}\) is zero.  Thus $
 \mathrm c_{i\de}\to 0$.
The preceding compactness argument now shows that every subsequential
limit of \(\de^{-1/3}\mathrm F\) is
$\mathrm c^{1/3}\bigl(\mathbf F'\bigr)^{1/3}$,
while every subsequential limit of
\(\de^{-1/3}\mathrm F_i\) is zero.  Hence
\[
 \de^{-1/3}\mathrm F(v)
 \longrightarrow
 \mathrm c^{1/3}\bigl(\mathbf F'(v)\bigr)^{1/3},
 \qquad
 \de^{-1/3}\mathrm F_i(v)\longrightarrow0
\]
uniformly on compact subsets of \(\Omega_0\).
\end{proof}

\subsection{Convergence of $\mathrm F$ near $\mathrm L_1$} \label{Con-F-boudary}
We need the following stronger condition in this subsection.

\medskip
\noindent{\it Condition} $\mathrm C'$.  There exists a fixed number $\mathrm l_1>0$ such that the length of $\mathrm L_1$ is at least $\mathrm l_1$.  Let $\mathrm v_{1,c}$ be the midpoint of $\mathrm L_1$.  For every point $u$ in the relative interior of the subsegment part joining $b_\de$ to $\mathrm v_{1,c}$, the inward normal segment of length $\mathrm l_1/4$ starting from $u$ is contained in $\Omega_\de$ and is disjoint from
\[
 \bigcup_{j=2}^{\mathrm n}\mathrm L_j\cup\p_{ab,\de}^*.
\]

\begin{remark}
The constant $1/4$ in the length of the line segment $[v_1,v_1']$ can be any other fixed positive constant. This can be accomplished by modifying $\mathrm{l}_1$. For example, if $\mathrm{l}_1/4$ is changed to $\mathrm{a}\mathrm{l}_1$ for fixed positive constant $\mathrm{a}<1/4$, then we can change $\mathrm{l}_1$ to a smaller number $(4\mathrm{a})\mathrm{l}_1 $ since
$$\mathrm{a}\mathrm{l}_1=\frac 14 (4\mathrm{a})\mathrm{l}_1.$$
\end{remark}

Now we turn to the extension of $\mathrm{F}$ near $\mathrm{L}_1$.
Put
\[
 h_\de=\left\lfloor\de^{-2/3-\ep}\right\rfloor\de.
\]
After a rotation and a translation, we assume that $\mathrm L_1$ is horizontal, that the domain is locally below $\mathrm L_1$, and that
\[
 \mathrm v_{1/3}:=\mathrm v_{1,c}-ih_\de=0.
\]
Thus $\mathrm L_1$ is locally contained in the line $\Im z=h_\de$. In $B(0;\mathrm{l}_1/5)$, define
\begin{equation} \label{extension}
\mathrm{F}_{ext}(v)= \begin{cases}
      \mathrm{F}(v)& \Im v\leq 0,\\
      \overline{\mathrm{F}(\bar v)}& \Im v>0,
    \end{cases}
\end{equation}
and
\begin{equation} \label{extensionstar}
\mathrm{F}_{ext}^*(v)= \begin{cases}
      \mathrm{F}^*(v)& \Im v\leq \de/2,\\
      \overline{ \mathrm{F}^*(\bar{v}+i\de)}& \Im v>\de/2,
    \end{cases}
\end{equation}
where $\bar v$ is the conjugate of $v$. The line $\Im v=\de/2$ in $\mathrm{F}_{ext}^*$  plays the role of $\Im v=0$ in $\mathrm{F}_{ext}$.

Similarly, in \eqref{extension} after replacing $ \mathrm{F}$ by $ \mathrm{F}_i$ and $F$ respectively, we can obtain the respective $ \mathrm{F}_{i,ext}$ and $F_{ext}$. Also in \eqref{extensionstar}, after replacing $ \mathrm{F}^*$ by $ \mathrm{F}_i^*$ and $F^*$ respectively, we can obtain the respective $ \mathrm{F}^*_{i,ext}$ and $F^*_{ext}$.

In order to obtain a parallel result of Lemma \ref{Fimboun}, we introduce
$$
\mathbb{F}_{ext}(z)= \begin{cases}
      \int_{(0 \to z)\subseteq B(0;\mathrm{l}_1/5)\cap (-\Half^{-}) }\de^{-1/3}F_{ext}(x)dx& \Im z\leq 0,\\
      \int_{(0 \to z)\subseteq B(0;\mathrm{l}_1/5)\cap \Half }\de^{-1/3}F_{ext}(x)dx& \Im z> 0,
    \end{cases}
$$
where $(0 \to z)\subseteq B(0;\mathrm{l}_1/5)\cap (-\Half^{-})$ is any path from $0$ to $z$ consisting of primal edges in $B(0;\mathrm{l}_1/5)\cap (-\Half^{-})\cap \de\E$,  $(0 \to z)\subseteq B(0;\mathrm{l}_1/5)\cap \Half$ is any path from $0$ to $z$ consisting of primal edges in $B(0;\mathrm{l}_1/5)\cap \Half\cap \de\E$.

 Corresponding to $\mathbb{F}_{ext}(z)$, we introduce
$$
\mathbb{F}^*_{ext}(z^*)= \begin{cases}
      \int_{(\de/2+i\de/2 \to z^*)\subseteq B(0;\mathrm{l}_1/5)\cap (-\Half^{-}+i\de/2) }\de^{-1/3}F^*_{ext}(x)dx& \Im z^*\leq \de/2,\\
      \int_{(\de/2+i\de/2 \to z^*)\subseteq B(0;\mathrm{l}_1/5)\cap (\Half+i\de/2) }\de^{-1/3}F^*_{ext}(x)dx& \Im z^*> \de/2,
    \end{cases}
$$
where $(\de/2+i\de/2 \to z^*)\subseteq B(0;\mathrm{l}_1/5)\cap (-\Half^{-}+i\de/2)$ is any path from $\de/2+i\de/2$ to $z^*$ consisting of dual edges in $B(0;\mathrm{l}_1/5)\cap (-\Half^{-} +i\de/2)\cap \de\E^*$, $(\de/2+i\de/2 \to z^*)\subseteq B(0;\mathrm{l}_1/5)\cap (\Half+i\de/2)$ is any path from $\de/2+i\de/2$ to $z^*$ consisting of dual edges in $B(0;\mathrm{l}_1/5)\cap (\Half+i\de/2)\cap \de\E^*$.

Next, let $H_{im}$ be the discrete harmonic function on $B(0;\mathrm{l}_1/5)\cap (-\Half^{-})\cap V_{\Omega_\de}$ such that
\begin{equation} \label{Himboundary}
H_{im}(z)= \begin{cases}
     \Im \mathbb{F}_{ext}(z)& \Im z=0,\\
       0& z\in \p B(0;\mathrm{l}_1/5)\cap (-\Half)\cap V_{\Omega_\de}.
    \end{cases}
\end{equation}

Now we extend $H_{im}(z) $ to $\Half$ as follows. For this purpose, define
\begin{equation} \label{himv}
h_{im}(v)= \begin{cases}
     \de^{1/3}\big(H_{im}(z+\de)-H_{im}(z)\big)/\de&  v=z+\de/2,\\
      \de^{1/3}\big(H_{im}(z+i\de)-H_{im}(z)\big)/(i\de)& v=z+i\de/2,
    \end{cases}
\end{equation}
and
\begin{equation} \label{himvexten}
h_{im,ext}(v)= \begin{cases}
      h_{im}(v)& \Im v\leq 0,\\
      \overline{h_{im}(\bar v)}& \Im v>0.
    \end{cases}
\end{equation}
Note that $\int_{z_1z_2z_3z_4} h_{im,ext}(z)dz=0$ for any square with four corners $z_1,z_2,z_3,z_4\in B(0;\mathrm{l}_1/5)\cap\de\Z^2$ and side length $\de$. Based on $h_{im,ext}(v)$, we can define the extension $H_{im,ext}(z) $ as follows,
\begin{equation} \label{Himextension}
H_{im,ext}(z)= \begin{cases}
      H_{im}(z)& \Im z\leq 0,\\
       \int_{(0 \to z)\subseteq B(0;\mathrm{l}_1/5)\cap \Half }h_{im,ext}(x)dx& \Im z> 0.
    \end{cases}
\end{equation}

$h_{im,ext}$ satisfies the following property.
\begin{lemma} \label{Himdiff}
If $|z|\leq \mathrm{l}_1(6^{-1}+60^{-1})$, then $h_{im,ext}(z) \preceq \de^{\al/2}$.
\end{lemma}
\begin{proof}
Note that $F(v)$ is real when $v\in \mathrm{L}_1$. It follows from Proposition \ref{Tranimp} that
$$
|F(v)-F(v+i[\de^{-2/3-\ep}]\de)|\preceq \de^{1/3+\al-2\ep},
$$
where $v\in B(0;\mathrm{l}_1/5)$ and $\Im v=0$. This implies
\beqn
H_{im}(z) \preceq \de^{\al-2\ep} \  \text{when} \ \Im z=0. \label{Himbound}
\eeqn
Let $z$ be any primal vertex such that $|z|\leq \mathrm{l}_1(6^{-1}+60^{-1})$ and $\Im z=0$. We consider the simple random walk $\mathrm{S}_{H_{im}}[0,\tau_{H_{im}}]$, where $\mathrm{S}_{H_{im}}(0)=z-i\de$ and $\tau_{H_{im}}$ is the smallest $n$ such that $\mathrm{S}_{H_{im}}(n)\in \p \big(B(0;\mathrm{l}_1/5)\cap (-\Half^{-})\cap V_{\Omega_\de}\big)$.
The probability that $\mathrm{S}_{H_{im}}(\tau_{H_{im}})=z+k\de$ is at most of the order $k^{-2}$ by Theorem 1.2 in~\cite{Uch10}.
Since $\Pro\big(\Im \mathrm{S}_{H_{im}}(\tau_{H_{im}}) \not=0\big) \preceq \de$, it follows from \eqref{Himbound}, Proposition \ref{Tranimp} and the above probability estimate that
\beq
H_{im}(z)-H_{im}(z-i\de) \preceq \de H_{im}(z)+\de^{2/3}\sum_{k=0}^{[\de^{-1}]} k^{-2}k\de^{1/3+\al-2\ep}\preceq \de^{1+\al/2}.
\eeq
So $H_{im,ext}$ satisfies \eqref{weakLap} of Lemma \ref{asyhar} in the appendix. Now we can apply Lemma \ref{asyhar} to complete the proof.
\end{proof}

However, we may have $\int_{z_1z_2z_3z_4} \big(F_{ext}(z)-ih_{im,ext}(z)\big)dz\not=0$ when $\Im z_1=0$, $z_2=z_1+\de$, $z_3=z_2+i\de$ and $z_4=z_1+i\de$. To overcome this difficulty, we extend $iH_{im}(z) $ to $\Half$ by the following method,
\begin{equation} \label{HIMEXT}
H_{im,EXT}(z)= \begin{cases}
     i H_{im}(z)& \Im z\leq 0,\\
       \int_{(0 \to z)\subseteq B(0;\mathrm{l}_1/5)\cap \Half }h_{im,EXT}(x)dx& \Im z> 0,
    \end{cases}
\end{equation}
where
\begin{equation} \label{himvEXT}
h_{im,EXT}(v)= \begin{cases}
      h_{Im}(v)& \Im v\leq 0,\\
    -h_{Im}(\bar v)& \Im v>0, \ v \ \text{is the center of horizontal edge},\\
    h_{Im}(\bar v)& \Im v>0, \ v \ \text{is the center of vertical edge}
    \end{cases}
\end{equation}
with
$$
h_{Im}(v)= \begin{cases}
     i\de^{1/3}\big(H_{im}(z+\de)-H_{im}(z)\big)/\de&  v=z+\de/2,\\
      \de^{1/3}\big(H_{im}(z+i\de)-H_{im}(z)\big)/\de& v=z+i\de/2.
    \end{cases}
$$

Now we can make use of $\mathbb{F}_{ext}(z)-H_{im,EXT}(z)$, whose imaginary part is zero when $\Im z=0$, to obtain the following lemma.
\begin{lemma} \label{initialbb}
There exist functions $r_F$ and $r_{F^*}$ such that
\beq
&&|r_F| \preceq \de^{1+\al/2}, \ |r_{F^*}| \preceq \de^{1+\al/2},\\
&&|F_{ext}(v)| \preceq \de^{1/3}(\log\de^{-1})^2, \ |F_{ext}^*(v)| \preceq \de^{1/3}(\log\de^{-1})^2 \\
&&|F_{ext}(v+\de\be)+r_F(v+\de\be)-F_{ext}(v)-r_F(v)| \preceq \de^{1+1/3}(\log\de^{-1})^2,\\
&&|F_{ext}^*(v+\de\be)+r_{F^*}(v+\de\be)-F_{ext}^*(v)-r_{F^*}(v)| \preceq \de^{1+1/3}(\log\de^{-1})^2,
\eeq
uniformly over  $v\in \big(B(0;\mathrm{l}_1/6)\cup \p B(0;\mathrm{l}_1/6)\big)$. In addition, $\Delta \big(\mathbb{F}_{ext}(z)-R_F(z)\big)=0$ and $\Delta \big(\mathbb{F}_{ext}^*(z^*)-R_{F^*}^*(z^*)\big)=0$ when $z\in B(0;\mathrm{l}_1/5)\cap \de\Z^2$ and $z^*\in B(\de/2+i\de/2;\mathrm{l}_1/5)\cap \big(\de\Z^2+(\de/2,\de/2)\big)$, where $R_F(z)$ and $R_{F^*}^*(z^*)$ are defined in the same way as $\mathbb{F}_{ext}(z)$ and $\mathbb{F}_{ext}^*(z^*)$ except that $F_{ext}(v)$ and $F_{ext}^*(v)$ are replaced by $r_F$ and $r_{F^*}$ respectively.
\end{lemma}
\begin{proof}
The proof about the properties of $F$ is similar to that of \eqref{Finboun} and \eqref{Fdiffinboun} by considering $\mathbb{F}_{ext}(z)-H_{im,EXT}$. However we have to point out two things. The first thing is that for any point $z\in \big(\p B(0;\mathrm{l}_1/5)\cap(-\Half)\big) \cup [-\mathrm{l}_1/5,\mathrm{l}_1/5]$, there exists $z_0\in \Omega_0$ whose distance to $\p \Omega_0$ is exactly of constant order such that
$\mathbb{F}_{ext}(z)-\mathbb{F}_{ext}(z)\preceq \de^{2/3}\sum_{k=[\de^{-2/3-\ep}]}^{[\de^{-1}]} k^{-1/3}\log^2\de $ by \eqref{Finboun-b}. This combined with the definition of $H_{im,EXT}(z)$ shows that $\mathbb{F}_{ext}(z)-H_{im,EXT}(z)\preceq \log^2 \de$ for $z\in \big(\p B(0;\mathrm{l}_1/5)\cap(-\Half)\big) \cup [-\mathrm{l}_1/5,\mathrm{l}_1/5]$.
 The second one is to check the condition in Lemma \ref{asyhar}, especially $|\Delta \big(\mathbb{F}_{ext}(z)-H_{im,EXT}(z)\big)| \preceq \de^{1+\al/2}$  when $\Im z=0$. Note that $F(v)$ is real when $v\in \mathrm{L}_1$. It follows from Proposition \ref{Tranimp} that
$$
|F(v)-F(v+i[\de^{-2/3-\ep}]\de)|\preceq \de^{1/3+\al-2\ep},
$$
where $v\in B(0;\mathrm{l}_1/6)\cap V_{\Omega_\de^\diamond}$ and $\Im v=0$. On the other hand, the proof of Proposition \ref{ref-inv} shows that $\Im \mathrm{F}(v)=o(\de^{1/3+\al-2\ep})$. Using $F(v)=\mathrm{F}(v)+\mathrm{F}_i(v)$ and $F^*(v)=\mathrm{F}(v)-\mathrm{F}_i(v)+O(\de^{(2/3+\ep)(2-\varepsilon)})$, we can obtain that when $\Im v=0$,
$
\Im F(v+\de e^{-i\pi/4}/\sqrt 2) \preceq \de^{1/3+\al-2\ep}
$
by Proposition \ref{Tranimp}. This combined with Lemma \ref{Himdiff} implies $|\Delta \big(\mathbb{F}_{ext}(z)-H_{im,EXT}(z)\big)| \preceq \de^{1+\al/2}$ when $\Im z=0$. Applying the same argument leading to Lemma \ref{asyhar} in the appendix, we can complete the proof about the properties of $F$. Similarly, we can derive the properties of $F^*$.
\end{proof}

Through $\phi=\inf\{j\geq 0: \mathrm{S}_j\not\in B(0;\mathrm{l}_1/5)\}$ and
$$
\mathbf{F}_{ext}(z)=\Ex \int_{z\to \mathrm{S}_{\phi}} \de^{-1}\big( \mathrm{F}_{ext}(z_1)+2^{-1}r_F(z_1)+2^{-1}r_{F^*}(z_1)+r_{FF^*}(z_1)\big)^3dz_1
$$
when $z\in \de \Z^2\cap \big(B(0;\mathrm{l}_1/5)\cup \p B(0;\mathrm{l}_1/5)\big)$, where $r_{FF^*}(v)=\big(F_{ext}(v)+F_{ext}^*(v)\big)/2- \mathrm{F}_{ext}(v)$ and $z\to \mathrm{S}_\phi$ is the path of $\{\mathrm{S}_n\}$ from $\mathrm{S}_0=z$ to $\mathrm{S}_{\phi}$, we can obtain the following result.

\begin{proposition} \label{Fboundary}
Under {\it Condition} $\mathrm{C}$ and {\it Condition} $\mathrm{C}'$, and the assumption that $|\mathrm{c}|+|\mathrm{c}_i|\not=\infty$,
$$|\de^{-1/3}\mathrm{F}_{ext}(v)| \preceq 1, \  |\de^{-1/3}\mathrm{F}_{i,ext}(v)| \preceq 1$$
uniformly over $v\in \big(B(0;\mathrm{l}_1/6)\cup \p B(0;\mathrm{l}_1/6)\big)$.
\end{proposition}
\begin{proof}
Following the proof of Lemma \ref{zeroint} and making use of Lemma \ref{initialbb}, we can show that $\lim_{\de\to 0}\mathbf{F}_{ext}(z)=0$ uniformly over $z\in \de \Z^2\cap \big(B(0;\mathrm{l}_1/5)\cup \p B(0;\mathrm{l}_1/5)\big)$. Then by Lemma \ref{initialbb} and the same argument in the proof of Lemma \ref{Fimboun} we can complete the proof.
\end{proof}

 So based on Proposition \ref{Fboundary} we have the following result whose proof is omitted since it is exactly the same as that of Proposition \ref{improveb}.

\begin{proposition} \label{difFboundary}
Under {\it Condition} $\mathrm{C}$, {\it Condition} $\mathrm{C}'$ and the assumption that $|\mathrm{c}|+|\mathrm{c}_i|\not=\infty$,
\beqn
&&|\mathrm{F}(v+\de\be)-\mathrm{F}(v)| \preceq \de^{1+1/3}, \label{difFboundaryf1}\\
&&|\mathrm{F}_i(v+\de\be)-\mathrm{F}_i(v)| \preceq \de^{1+1/3}, \label{difFboundaryf1i}\\
&&|\mathrm{F}_{+}(v+ \de e^{\pm i\pi/4}/\sqrt 2)-\mathrm{F}_{+}(v)| \preceq \de^{1+1/3}, \label{difFboundary2}
\eeqn
uniformly over $v\in \{z:|z-\mathrm{v}_{1/3}|<\mathrm{l}_1/6,\Im(z-\mathrm{v}_{1/3})\leq 0\}\cap V_{\Omega_\de^\diamond}\cap \mathcal{K}_0$, where $\be=\be_1$ or $\be_2$, (in terms of complex number, $\be=1$ or $i$), and $\mathcal{K}_0$ is any compact subset of $B(\mathrm{v}_{1/3};\mathrm{l}_1/6)$.
\end{proposition}

\begin{remark}
Rigorously, in Proposition \ref{difFboundary}, all $\mathrm{F}$, $\mathrm{F}_i$ and $\mathrm{F}_{+}$ should be replaced with the sum of the respective $\mathrm{F}$, $\mathrm{F}_i$ and $\mathrm{F}_{+}$ and a remainder function with order $\de^{1/3+\al/2}$. However, we only apply Proposition \ref{difFboundary} to derive the precompactness of $\mathrm{F}$ and $\mathrm{F}_i$, so we can ignore this remainder function.
\end{remark}

Now we can present the convergence of $\mathrm{F}$ in any compact subset $\mathcal{K}$ of $\overline{\Omega_0}$ whose boundary contains part of $\mathrm{L}_1$,

\begin{proposition} \label{convFb}
Suppose $\Omega_0$ is a bounded simply connected domain in $\C$ with two points $a$ and $b$ on its boundary. Assume that $(\Omega_\de^\diamond,a_\de^\diamond,b_\de^\diamond)$ is a family of Dobrushin domains converging to $(\Omega_0,a,b)$ in the Carath\'eodory sense as $\de\to 0$. Under {\it Condition} $\mathrm{C}$ and {\it Condition} $\mathrm{C}'$,
as $\de\to 0$,
\beqn
\de^{-1/3}\mathrm{F}(v) - \mathrm{c}^{1/3}\big(\mathbf{F}'(v)\big)^{1/3}\to 0, \label{convFb1} \\
\de^{-1/3}\mathrm{F}_i(v)\to 0, \label{convFb1i}\\
\de^{-1/3}F(v) - \mathrm{c}^{1/3}\big(\mathbf{F}'(v)\big)^{1/3} \to 0, \label{convFb2}
\eeqn
uniformly over $v\in  \{z:|z-\mathrm{v}_{1/3}|<\mathrm{l}_1/6,\Im(z-\mathrm{v}_{1/3})\leq 0\}\cap \overline{\Omega_0}$. In addition,
$$0<\mathrm{c}<\infty,\ \mathrm{c}_i=0$$ and
$$ \de^{-1/3}F(v) \to \mathrm{c}^{1/3}\big(\mathbf{F}'(v)\big)^{1/3}$$
uniformly over any $v$ in the line segment $\mathrm{v}_{1,c}b_\de$ such that $ \liminf_{\de\to 0}|v-b_\de|>0$.
\end{proposition}
\begin{proof}
First, we assume that  $|\mathrm{c}|\not=\infty$.
The proof of \eqref{convFb1} and \eqref{convFb1i}   is the same as that of Proposition \ref{convF}, so we omit it. The proof of \eqref{convFb2} follows from the identity $F(v)=\mathrm{F}(v) +\mathrm{F}_i(v) $. We focus on the proof of the second part. Since
\beqn
|\de^{-1/3}F(\mathrm{v}_{1/3})-\de^{-1/3}F(\mathrm{v}_{1,c})| \preceq \de^{-1/3-2/3-\ep}\de^{1+\al-\ep}=\de^{\al-2\ep}  \label{estimateFv}
\eeqn
 by Proposition \ref{Tranimp}, we obtain that
$$
\de^{-1/3}F(\mathrm{v}_{1,c}) \to \mathrm{c}^{1/3}\big(\mathbf{F}'(\mathbf{v}_{1,c})\big)^{1/3}.
$$
By \eqref{halfone} and the following Lemma \ref{boundarm}, we can get that $\mathrm{c}>0$. Since $ \liminf_{\de\to 0}|v-b_\de|>0$, we can get that $\de^{-1/3}F(v) \to \mathrm{c}^{1/3}\big(\mathbf{F}'(v)\big)^{1/3}$ by the same argument.

Next, we assume that  $|\mathrm{c}|=\infty$. If $\limsup_{\de\to 0} |\mathrm{c}_{i\de}|/|\mathrm{c}_\de|=\infty$, by considering a subsequence we can suppose $\lim_{\de\to 0} |\mathrm{c}_{i\de}|/|\mathrm{c}_\de|=\infty$. We re-scale $\mathrm{F}$ and $\mathrm{F}_i$ by $\mathrm{c}_{i\de}$. So after re-scaling, $\mathrm{c}=0$ and $\mathrm{c}_i=1$. By the argument for $|\mathrm{c}|\not=\infty$ and $|\mathrm{c}_i|\not=\infty$, we can see that this is impossible. So $\limsup_{\de\to 0} |\mathrm{c}_{i\de}|/|\mathrm{c}_\de|<\infty$. Now we re-scale all observables by $\mathrm{c}_\de$. So after re-scaling, $\mathrm{c}=1$ and $\mathrm{c}_i=0$. This implies that $\de^{-1/3}\Prob(\mathrm{v}_{1,c}\in \ga_\de)\to \infty$ as $\de\to 0$,, which contradicts with the following Lemma \ref{boundarm}. In other words, we always have
$$0<\liminf_{\de\to 0} \mathrm{c}_\de \leq \limsup_{\de\to 0}\mathrm{c}_\de<\infty, \ \lim_{\de\to 0} \mathrm{c}_{i\de}=0.$$
\end{proof}

\begin{lemma} \label{boundarm}
Suppose $v\in \mathrm{L}_1$ and $\min(\dist(v,\mathrm{v}_1),\dist(v,\mathrm{v}_2))>\ep \mathrm{l}_1$. Then
$$
\Pro(v\in \ga_\de) \asymp \de^{1/3}.
$$
\end{lemma}
\begin{proof}
Suppose $\mathrm{v}_2=0$, $\mathrm{L}_{1}$ is in the positive horizontal axis, and $\mathrm{L}_1^*$ is in $-\Half$. Let $x\rightsquigarrow_{dual}y$
represent the event that there is a dual-open cluster from $x$ to $y$.
Consider the event $\{v-\de/2\rightsquigarrow_{dual} \p B(v;\ep\mathrm{l}_1/3)\}\cap \{\p B(v;\ep\mathrm{l}_1/3)\rightsquigarrow_{dual}\p\Omega_{ab,\de}^* \}\cap\{$ there is a dual-open cluster in $B(v;\ep\mathrm{l}_1/3,2\ep\mathrm{l}_1/3)\cap \Omega_\de$ connecting $\{z: \ep\mathrm{l}_1/3<\Re z-v< 2\ep\mathrm{l}_1/3, \Im z=-\de/2\}$ and $\{z: \ep\mathrm{l}_1/3<v-\Re z< 2\ep\mathrm{l}_1/3, \Im z=-\de/2\}\}$. The RSW-method shows that both $\Pro\big(\p B(v,\ep\mathrm{l}_1/3)\rightsquigarrow_{dual}\p\Omega_{ab,\de}^* \big)$ and $\Pro\big($there is a dual-open cluster in $B(v;\ep\mathrm{l}_1/3,2\ep\mathrm{l}_1/3)\cap \Omega_\de$ connecting $\{z: \ep\mathrm{l}_1/3<\Re z-v< 2\ep\mathrm{l}_1/3, \Im z=-\de/2\}$ and $\{z: \ep\mathrm{l}_1/3<v-\Re z< 2\ep\mathrm{l}_1/3, \Im z=-\de/2\}\big)$ are bounded away from zero. This combined with \eqref{halfone} implies $\de^{1/3}\preceq \Pro(v\in \ga_\de)$. The other relation $\Pro(v\in \ga_\de) \preceq \de^{1/3}$ is straightforward. So we complete the proof.
\end{proof}

\section{Convergence to $SLE_6$} \label{ConSLE6}

Let $\mathcal E_n$ be the set of primal edges which $\ga_\de[0,n]$ has
touched or crossed.  Let $\Omega_{\de,n}$ be the simply
connected subdomain of $\Omega_\de\setminus\mathcal E_n$ which contains
$b_\de$, and let $\Omega_{\de,n}^\diamond$ and $\Omega_{\de,n}^*$ be the
associated medial and dual Dobrushin domains.  We define the boundary arcs
$\p^n_{ba,\de}$ and $\p^n_{ab,\de}$ by replacing $\Omega_\de$ by
$\Omega_{\de,n}$ and $a_\de^\diamond$ by $\ga_\de(n)$ in the definition of
the Dobrushin boundary.  Let $\Omega_n$ be the bounded domain in $\R^2$
enclosed by $\p^n_{ba,\de}$ and $\p^n_{ab,\de}$.  We define $\Omega_n^*$
similarly.

Put
\[
 \Omega^n=
 \left\{z\in\Omega_n:
 \dist\left(z,
 (\p^n_{ba,\de}\setminus\p_{ba,\de})
 \cup
 (\p^n_{ab,\de}\setminus\p_{ab,\de})\right)
 \geq100r_\de\right\},
\]
and let $\hat\Omega^n$ be the simply connected component whose closure
contains $b_\de$.  Set
\[
 \de_1=
 \left(\de\left\lfloor\de^{-1}\log\de^{-1}\right\rfloor\right)^{-1}.
\]
Thus $\de_1\asymp r_\de$.  By considering the coarse square lattice
$\de_1\Z^2$ in $\hat\Omega^n$, we can approximate $\p\hat\Omega^n$ by a
simple polygonal path $\p_1$ consisting of horizontal or vertical edges of
length $\de_1$.  We choose $\p_1$ to start on $\p_{ba,\de}$ and end on
$\p_{ab,\de}$.

Let $\Omega_{(n)}$ be the simply connected domain in $\R^2$ enclosed by
$\p_1$ and the corresponding parts of $\p_{ba,\de}$ and $\p_{ab,\de}$, and
put
\[
 \hat\Omega_n=\de\Z^2\cap\Omega_{(n)}.
\]
Define
\[
 a_n=\operatorname*{arg\,min}\Big\{
 |z-\ga_\de(n)|:
 z-\frac\de2\be_1\in V_{\hat\Omega_n}
 \ \hbox{or}\
 z-\frac\de2\be_2\in V_{\hat\Omega_n}
 \Big\}.
\]
Let $\hat\Omega_{\de,n}^\diamond$ be the Dobrushin domain associated with
$\hat\Omega_n$, with marked points $a_n$ and $b_\de^\diamond$.  The
corresponding primal and dual Dobrushin domains are denoted by
$\hat\Omega_{\de,n}$ and $\hat\Omega_{\de,n}^*$.

Let $a_{\de,n}$ be the endpoint of the primal edge with center $a_n$ which
is also an endpoint of another primal boundary edge, and put
$b_{\de,n}=b_\de$.  Define $a_{\de,n}^*$ analogously on the dual lattice and
put $b_{\de,n}^*=b_\de^*$.  Denote the two primal boundary arcs by
$\hat\p_{ba,\de}^n$ and $\hat\p_{ab,\de}^n$, and the two dual boundary arcs
by $\hat\p_{ba,\de}^{*,n}$ and $\hat\p_{ab,\de}^{*,n}$.

The boundary $\hat\p_{ba,\de}^n$ consists of part of $\p_{ba,\de}$ and a
simple polygonal chain
\[
 w_{1,n},w_{2,n},\ldots,w_{m_n,n},
\]
where $w_{1,n}=a_{\de,n}$, $w_{j,n}\in\Omega_n\cap\de\Z^2$ for
$2\leq j\leq m_n-1$, and $w_{m_n,n}\in\p_{ba,\de}$.  Similarly,
$\hat\p_{ab,\de}^{*,n}$ consists of part of $\p_{ab,\de}^*$ and a
polygonal chain
\[
 w_{1,n}^*,w_{2,n}^*,\ldots,w_{m_n^*,n}^*,
\]
where $w_{1,n}^*=a_{\de,n}^*$,
$w_{j,n}^*\in\Omega_n^*\cap\de(\Z^2)^*$ for
$2\leq j\leq m_n^*-1$, and $w_{m_n^*,n}^*\in\p_{ab,\de}^*$.
After modifying these two polygonal chains near the two ends of $\p_1$, we
may assume that the nearby Dobrushin domain satisfies Condition $\mathrm C$.
Condition $\mathrm C'$ is inherited from the unchanged boundary near
$b_\de$.

In $\Omega_{\de,n}^\diamond$, denote the exploration path from
$\ga_\de(n)$ to $b_\de^\diamond$ by $\ga_{\de,n}$, so that
\[
 \ga_{\de,n}(j)=\ga_\de(n+j).
\]
In $\hat\Omega_{\de,n}^\diamond$, denote the exploration path from $a_n$ to
$b_\de^\diamond$ by $\hat\ga_{\de,n}$.  Let $\mathcal F_{\de,n}$ be the
$\sigma$-field generated by $\ga_\de[0,n]$ and by the bond states revealed
while constructing this initial path.  Conditionally on $\mathcal F_{\de,n}$,
the unrevealed bonds are independent Bernoulli bonds with parameter $1/2$.
For an oriented medial edge $e$, define
\[
 F_{\de,n}(e)=\Ex_{\de,n}\Big[
  \exp\big(\frac i3W_{\ga_{\de,n}}(e,e_b)\big)
  \I(e\in\ga_{\de,n})\Big],
\]
and define $\hat F_{\de,n}(e)$ in the same way for a critical percolation
configuration in $\hat\Omega_{\de,n}^\diamond$.  By the domain Markov
property,
\begin{equation}\label{marting}
 F_{\de,n}(e)=\Ex\left[
  \exp\big(\frac i3W_{\ga_\de}(e,e_b)\big)
  \I(e\in\ga_\de)\,\middle|\,\mathcal F_{\de,n}\right].
\end{equation}

Let $v$ be the center of a horizontal primal edge in the line segment
$[b_\de,\mathrm v_{1,c}]$, with
$\liminf_{\de\to0}|v-b_\de|>0$.  Suppose neither $B_v$ nor $C_v$ belongs to
the remaining or nearby medial domain.  Put
\beq
F_{\de,n}(v)&=&e_b^{-1/3}\big(
 e^{-i\pi/4}F_{\de,n}(A_v)+e^{i\pi/4}F_{\de,n}(D_v)\big),\\
\hat F_{\de,n}(v)&=&e_b^{-1/3}\big(
 e^{-i\pi/4}\hat F_{\de,n}(A_v)+e^{i\pi/4}\hat F_{\de,n}(D_v)\big).
\eeq

The next lemma shows that the two boundary observables have the same scaling
limit.

\begin{lemma}\label{samelimit}
Suppose $\hat\p_{ba,\de}^n$ and $\hat\p_{ab,\de}^{*,n}$ satisfy Conditions
$\mathrm C$ and $\mathrm C'$.  Let $v\in[b_\de,\mathrm v_{1,c}]$ satisfy
$\liminf_{\de\to0}|v-b_\de|>0$, and put
\[
 d_n=\dist\Big(v,\bigcup_{j=1}^{m_n-1}[w_{j,n},w_{j+1,n}]\Big),
 \qquad
 d_n^*=\dist\Big(v,\bigcup_{j=1}^{m_n^*-1}
 [w_{j,n}^*,w_{j+1,n}^*]\Big).
\]
If $\liminf_{\de\to0}d_n>0$ and
$\liminf_{\de\to0}d_n^*>0$, then, uniformly over such $v$,
\[
 \de^{-1/3}\big(F_{\de,n}(v)-\hat F_{\de,n}(v)\big)\longrightarrow0.
\]
\end{lemma}

\begin{proof}
Fix a configuration $\hat\omega_{\de,n}$ in \(\hat\Omega_{\de,n}\).  Let
$\mathcal W_n$ be the set of primal bonds corresponding to the two artificial
boundaries of $\hat\Omega_{\de,n}^\diamond$.  For any assignment
$\tilde\omega_{\de,n}$ of states to the remaining primal bonds of
$\Omega_{\de,n}$, including the bonds in $\mathcal W_n$, define an extension
$\omega_{\de,n}$ by
\[
 \omega_{\de,n}(E)=
 \begin{cases}
  \hat\omega_{\de,n}(E),
   &E\in E_{\hat\Omega_{\de,n}}\setminus\mathcal W_n,\\
  \tilde\omega_{\de,n}(E),
   &E\in E_{\Omega_{\de,n}}\setminus
          (E_{\hat\Omega_{\de,n}}\setminus\mathcal W_n).
 \end{cases}
\]
Every configuration in the remaining domain is obtained in this way.

On the unchanged straight boundary segment, the winding factor associated
with a visit of $v$ is fixed.  It is therefore enough to compare the
probabilities that the two exploration paths visit $v$.

Suppose first that $\hat\ga_{\de,n}$ visits $v$ but the exploration path
$\ga_{\de,n}$ in the extended configuration does not.  Then the reversed
path $\hat\ga_{\de,n}^{-1}$ from $b_\de^\diamond$ to $a_n$ visits $v$ only
after touching the artificial dual-open boundary
\[
 \bigcup_{j=1}^{m_n^*-1}[w_{j,n}^*,w_{j+1,n}^*].
\]
Moreover, $v$ lies on a loop $\widehat{\Loop}$ in the extended
configuration.  Define
\[
 \hat s_n=\inf\Big\{j>0:
 \hat\ga_{\de,n}^{-1}(j)\in
 \widehat{\Loop}\cap
 \bigcup_{k=1}^{m_n^*-1}[w_{k,n}^*,w_{k+1,n}^*]\Big\}.
\]
Conditionally on $\hat\ga_{\de,n}^{-1}[0,\hat s_n]$, the dual-open side of
$\widehat{\Loop}$ gives a dual-open arm across the annulus centered at
$\hat\ga_{\de,n}^{-1}(\hat s_n)$ with inner radius
\[
 \hat d_n^*=\dist\big(
 \hat\ga_{\de,n}^{-1}(\hat s_n),\p_{ab,\de}^n\big)-\de/2
 \asymp r_\de
\]
and outer radius comparable with
$|v-\hat\ga_{\de,n}^{-1}(\hat s_n)|$.  Together with the boundary one-arm
event needed for the visit of $v$, \eqref{one} and \eqref{halfone} give
\[
 \Pro(\hbox{this case})
 \preceq
 \de^{1/3}
 \Big(
  \frac{\hat d_n^*}
       {|v-\hat\ga_{\de,n}^{-1}(\hat s_n)|}
 \Big)^\al
 =o(\de^{1/3}).
\]
Here the last relation follows from $\hat d_n^*\asymp r_\de$ and the
assumption $\liminf_{\de\to0}d_n^*>0$.

Next, suppose that $\ga_{\de,n}$ visits $v$ but $\hat\ga_{\de,n}$ does not.
Then the reversed path $\ga_{\de,n}^{-1}$ from $b_\de^\diamond$ to
$\ga_\de(n)$ visits $v$ only after touching or crossing the artificial
primal boundary
\[
 \bigcup_{j=1}^{m_n-1}[w_{j,n},w_{j+1,n}].
\]
Let
\[
 s_n=\inf\Big\{j\geq0:
 \ga_{\de,n}^{-1}(j)\in
 \bigcup_{k=1}^{m_n-1}[w_{k,n},w_{k+1,n}]\Big\}.
\]
Since the reversed path must subsequently reach $v$, there is a dual-open
arm from scale
\[
 d_n'=\dist\big(\ga_{\de,n}^{-1}(s_n),\p_{ba,\de}^n\big)
 \asymp r_\de
\]
to a scale comparable with
$|v-\ga_{\de,n}^{-1}(s_n)|$.  Again, \eqref{one} and \eqref{halfone} give
\[
 \Pro(\hbox{this case})
 \preceq
 \de^{1/3}
 \Big(
  \frac{d_n'}{|v-\ga_{\de,n}^{-1}(s_n)|}
 \Big)^\al
 =o(\de^{1/3}),
\]
by $\liminf_{\de\to0}d_n>0$.

Consequently,
\[
 \left|\Pro(v\in\ga_{\de,n})-
       \Pro(v\in\hat\ga_{\de,n})\right|=o(\de^{1/3}).
\]
This
proves the assertion.
\end{proof}

Based on Lemma~\ref{samelimit}, we now identify any subsequential limit of
$\{\ga_\de\}$.  We use the martingale method of~\cite{LSWlerw}; see
also~\cite{CDHKS} and Proposition 9.20 of~\cite{DC13}.

\begin{proposition}\label{subseq}
Suppose $\p_{ba,\de}$ and $\p_{ab,\de}^*$ satisfy Conditions $\mathrm C$
and $\mathrm C'$.  Then every subsequential limit of $\{\ga_\de\}$ is
chordal $SLE_6$.
\end{proposition}

\begin{proof}
Let $\ga$ be a subsequential limit of $\{\ga_\de\}$ in $\Omega_0$.
By~\cite{DST17}, $\ga$ is almost surely generated by a continuous Loewner
driving function.  Let
\[
 \varphi:\Omega_0\longrightarrow\Half,
 \qquad \varphi(a)=0,\quad\varphi(b)=\infty,
\]
and put $\tilde\ga=\varphi(\ga)$.

Passing to a further subsequence if necessary, assume that
$\mathrm c_\de\to\mathrm c\in(0,\infty)$.  We use the normalization fixed
in Section~\ref{confinvar} and multiply the vertex observables by
$\mathrm c^{-1/3}$.  This does not change their martingale property, and we
keep the same notation for the normalized observables below.

Parametrize $\tilde\ga$ by half-plane capacity.  Let $W_t$ be its continuous
driving process and let $g_t$ be the conformal map from
$\Half\setminus\tilde\ga[0,t]$ to $\Half$ with
\[
 g_t(z)=z+\frac{2t}{z}+O(z^{-2})
 \qquad\hbox{as }z\to\infty.
\]
We can choose $w>0$ sufficiently large so that $v=\varphi^{-1}(w)$
lies on the unchanged straight boundary segment near $b$.  For the discrete
domains, choose boundary bond centers converging to $v$ and denote them again
by $v$.
Let $\tau_t$ be the first index for which the hull of
$\varphi(\ga_\de[0,\tau_t])$ has half-plane capacity at least $t$.  Let
$\xi_\de$ be the first index for which this hull contains $w$ or
$\ga_\de$ touches $\p B(b_\de^\diamond;\de^\ep)$.

By \eqref{marting},
\[
 \big(F_{\de,\tau_t\wedge\xi_\de}(v),
      \mathcal F_{\de,\tau_t\wedge\xi_\de}\big)
\]
is a martingale.  Lemma~\ref{samelimit} and Proposition~\ref{convFb}
identify its scaling limit, and the boundary one-arm bound gives uniform
integrability.  Hence
\[
 F_{t\wedge\xi}(v)
 =\lim_{\de\to0}\de^{-1/3}
   F_{\de,\tau_t\wedge\xi_\de}(v)
\]
is a martingale with respect to the $\sigma$-field generated by
$\tilde\ga[0,t\wedge\xi]$, where $\xi$ is the first time at which the hull
of $\tilde\ga$ contains $w$.

The conformal map from the slit domain to $\R\times(0,1)$ which sends the
tip to $-\infty$ and $b$ to $\infty$ is unique up to a real additive
constant.  Proposition~\ref{convFb} and the chain rule therefore give
\beq
\pi^{1/3}F_{t\wedge\xi}^w
&:=&\pi^{1/3}F_{t\wedge\xi}(v)=\big((\varphi^{-1})'(w)\big)^{1/3}
      \\
&=&\left(
 \frac{g_{t\wedge\xi}'(w)}
      {g_{t\wedge\xi}(w)-W_{t\wedge\xi}}
 \right)^{1/3}.
\eeq

Using
\[
 g_t(w)=w+2tw^{-1}+O(w^{-2}),
 \qquad
 g_t'(w)=1-2tw^{-2}+O(w^{-3})
\]
as $w\to\infty$ along the real axis, we obtain, for $s\leq t$,
\beq
&&\pi^{1/3}\Ex\left(
 F_{t\wedge\xi}^w\,\middle|\,\mathcal F_{s\wedge\xi}\right)\\
&=&\Ex\left[
 \left(
 \frac{1-2(t\wedge\xi)w^{-2}+O(w^{-3})}
 {w-W_{t\wedge\xi}+2(t\wedge\xi)w^{-1}+O(w^{-2})}
 \right)^{1/3}
 \middle|\mathcal F_{s\wedge\xi}\right]\\
&=&w^{-1/3}\left(
 1+\frac13\Ex\left(
 W_{t\wedge\xi}\,\middle|\,\mathcal F_{s\wedge\xi}\right)w^{-1}
 +\frac29\Ex\left(
 W_{t\wedge\xi}^2-6(t\wedge\xi)
 \middle|\mathcal F_{s\wedge\xi}\right)w^{-2}
 +O(w^{-3})\right).
\eeq
The required moments of $W_t$ are provided by~\cite{DST17}.  Taking $s=t$
gives
\[
 \pi^{1/3}F_{s\wedge\xi}^w
 =w^{-1/3}\left(
 1+\frac13W_{s\wedge\xi}w^{-1}
 +\frac29\big(W_{s\wedge\xi}^2-6(s\wedge\xi)\big)w^{-2}
 +O(w^{-3})\right).
\]
Consequently,
\begin{eqnarray}
&&w^{-1/3}\Big(
 1+\frac13W_{s\wedge\xi}w^{-1}
 +\frac29\big(W_{s\wedge\xi}^2-6(s\wedge\xi)\big)w^{-2}
 +O(w^{-3})\Big)\nonumber\\
&=&w^{-1/3}\Big(
 1+\frac13\Ex\left(
 W_{t\wedge\xi}\,\middle|\,\mathcal F_{s\wedge\xi}\right)w^{-1}
 +\frac29\Ex\left(
 W_{t\wedge\xi}^2-6(t\wedge\xi)
 \middle|\mathcal F_{s\wedge\xi}\right)w^{-2}
 +O(w^{-3})\Big). \qquad
\label{equalmartigale}
\end{eqnarray}
Letting $w\to\infty$, so that $\xi\to\infty$ for every fixed time, and
comparing the coefficients in \eqref{equalmartigale}, we obtain
\[
 \Ex(W_t\mid\mathcal F_s)=W_s,
 \qquad
 \Ex(W_t^2-6t\mid\mathcal F_s)=W_s^2-6s.
\]
L\'evy's characterization now gives
\[
 W_t=\sqrt6\,B_t,
\]
where $B_t$ is a standard one-dimensional Brownian motion.  Hence
$\tilde\ga$ is chordal $SLE_6$ in $\Half$, and $\ga$ is chordal $SLE_6$ in
$\Omega_0$ from $a$ to $b$.
\end{proof}

We can now state the convergence under Conditions $\mathrm C$ and
$\mathrm C'$.

\begin{proposition}\label{pathconvergeC'}
Suppose $\Omega_0$ is a bounded simply connected domain with two marked
boundary points $a$ and $b$.  Let
$(\Omega_\de^\diamond,a_\de^\diamond,b_\de^\diamond)$ be Dobrushin domains
converging to $(\Omega_0,a,b)$ in the Carath\'eodory sense.  Under Conditions
$\mathrm C$ and $\mathrm C'$, $\ga_\de$ converges weakly to chordal $SLE_6$
in $(\mathcal M,\mathrm d)$.
\end{proposition}

\begin{proof}
Theorem 6 of~\cite{DST17}, with $q=1$, gives tightness of $\{\ga_\de\}$ in
$(\mathcal M,\mathrm d)$.  Proposition~\ref{subseq} shows that every limit
of a convergent subsequence is chordal $SLE_6$.  Hence the full family
converges weakly to chordal $SLE_6$.
\end{proof}

\section{Proofs of the main results}\label{ProofMain}

The main work was done in Sections~\ref{confinvar} and~\ref{ConSLE6}.  In this section we remove the extra boundary condition used near $b_\de^\diamond$ and combine the preceding propositions.

\medskip
\noindent
{\bf Proof of Theorem~\ref{paraconverge}.}
By the stability proved in Proposition~\ref{convF}, we may modify the
domain away from \(B(b_\de;\de^\ep)\) without changing the limiting
normalization constants.  We may therefore assume Condition
$\mathrm C'$.  Proposition~\ref{convFb} then gives
\[
 0<\liminf_{\de\to0}\mathrm c_\de
 \leq
 \limsup_{\de\to0}\mathrm c_\de
 <\infty.
\]

Now take an arbitrary sequence $\de_k\to0$.  Passing to a subsequence, we may assume that
$\mathrm c_{\de_k}\to\mathrm c$ for some $\mathrm c\in(0,\infty)$.  Proposition~\ref{convF} gives, uniformly on compact subsets of $\Omega_0$,
\[
 \de_k^{-1/3}\mathrm F(v)
 \longrightarrow
 \mathrm c^{1/3}\big(\mathbf F'(v)\big)^{1/3},
 \qquad
 \de_k^{-1/3}\mathrm F_i(v)\longrightarrow0.
\]
Since $F=\mathrm F+\mathrm F_i$,
\[
 \de_k^{-1/3}\mathrm c_{\de_k}^{-1/3}F(v)
 \longrightarrow
 \big(\mathbf F'(v)\big)^{1/3}.
\]
The limit does not depend on the chosen subsequence.  Therefore the whole family converges, and Theorem~\ref{paraconverge} follows. \qed

\medskip
\noindent
{\bf Proof of Theorem~\ref{pathconverge}.}
Let $(\Omega_\de^\diamond)'$ be a Dobrushin domain satisfying
Conditions $\mathrm C$ and $\mathrm C'$.  We choose it so that the
length $\mathrm l_1'$ of the corresponding segment $\mathrm L_1'$ is a
positive constant and
\[
 \mathrm l_1'
 <
 \min\bigl(2^{-1},10^{-2}\dist(a,b)\bigr).
\]
Suppose that $(\Omega_\de^\diamond)'$ converges in the
Carath\'eodory sense to a bounded simply connected domain
$\Omega_0'$.  We may choose the modification so that the original and
modified domains agree outside a ball of radius $2\mathrm l_1'$ near
$b$.

Couple the percolation configurations in the two domains by using the
same bond states on their common part outside this ball.  Let
$\ga_\de$ and $\ga_\de'$ be the corresponding exploration paths.
The two paths are identical until their first entrance into the ball
of radius $2\mathrm l_1'$.

Consequently,
\[
 \mathrm d(\ga_\de,\ga_\de')>2\sqrt{\mathrm l_1'}
\]
can occur only if one of the two paths, after entering the ball of
radius $2\mathrm l_1'$, reaches the boundary of the ball of radius
$\sqrt{\mathrm l_1'}$.  By \eqref{one},
\begin{equation}\label{couple}
 \Pro\bigl(
   \mathrm d(\ga_\de,\ga_\de')>2\sqrt{\mathrm l_1'}
 \bigr)
 \preceq
 (\mathrm l_1')^{\al/2}.
\end{equation}

For fixed $\mathrm l_1'$, Proposition~\ref{pathconvergeC'} gives the
weak convergence of $\ga_\de'$ to chordal $SLE_6$ in $\Omega_0'$.
Since the original and modified domains differ only in a ball whose
radius tends to zero with $\mathrm l_1'$, we have
$\Omega_0'\to\Omega_0$ in the Carath\'eodory sense as
$\mathrm l_1'\to0$.  Letting first $\de\to0$ in \eqref{couple} and
then $\mathrm l_1'\to0$ proves the theorem. \qed

\medskip
\noindent
{\bf Proof of Corollary~\ref{card}.}
This is a direct consequence of Theorem \ref{pathconverge}.\qed

\section{Universality}
\label{sec:universality}

Universality predicts that critical percolation models on different
planar lattices have the same continuum behaviour after a suitable
choice of embedding. The convergence established in this paper gives
critical bond percolation on the square lattice the same chordal
\(\mathrm{SLE}_6\) interface limit as critical site percolation on the
triangular lattice \cite{CN}. A natural question is whether our result
and method extend to more general planar graphs.

Isoradial graphs provide a natural setting for this question.
Their geometry determines canonical bond probabilities compatible
with the star--triangle transformation. Under the bounded-angle and
square-grid assumptions, Grimmett and Manolescu \cite{GM14} proved
criticality and the box-crossing property. They also established
universality of the one-arm and alternating even-arm exponents,
conditional on their existence. These results provide essential
percolation estimates, although equality of critical exponents alone
does not identify the limiting interfaces.

At the level of macroscopic geometry, Duminil-Copin et al.\
\cite{DKKMO20} proved large-scale rotational invariance for critical
random-cluster models with \(1\leq q\leq4\), together with universality
for isoradial rectangular lattices. Independent bond percolation
corresponds to \(q=1\).
Extending our observable method directly presents additional
difficulties. The transformations in Sections~2 and~3 use
square-lattice symmetries and equal bond probabilities. On a general
isoradial graph, reflected bonds need not belong to the graph,
and interchanging bond states need not preserve probability.
Moreover, macroscopic rotational invariance does not directly
provide the quantitative local passage estimates needed for our
winding-weighted observable and its cubic contour integrals.
Star--triangle transformations offer a possible replacement,
but their effect on these quantities must be controlled.

Chelkak and Smirnov \cite{CS12} proved universality and conformal
invariance of critical Ising fermionic observables on isoradial
graphs with bounded angles. Their work shows that discrete analytic
methods can accommodate nonuniform local geometry. Establishing
corresponding estimates for the percolation observable is a natural
direction for further work.

\section*{Appendix}
\addcontentsline{toc}{section}{Appendix}

We collect three discrete random-walk estimates used in Section~\ref{confinvar}.  The constants below are uniform under translations, rotations, and changes of lattice scale.

\addcontentsline{toc}{subsection}{Windings of random walks in the plane}
\subsection*{Windings of random walks in the plane}

Let $Z(t)=X(t)+iY(t)$ be a two-dimensional Brownian motion with
$Z(0)=z_0\neq0$, and let $\theta(t)$ be the total angle wound by $Z$
around the origin up to time $t$.  Define
\[
 \theta_+(t)=\int_0^t \I(|Z(u)|>1)\,d\theta(u).
\]
Let $X_1,X_2,\ldots$ be independent and identically distributed
two-dimensional random vectors such that
\[
 \Pro(X_1=\sqrt2\be_1)=\Pro(X_1=-\sqrt2\be_1)
 =\Pro(X_1=\sqrt2\be_2)=\Pro(X_1=-\sqrt2\be_2)=\frac14,
\]
and let
\[
 S_n=z_0+\sum_{j=1}^nX_j.
\]
Let $\Theta(n)$ be the total angle wound by the random-walk path around
the origin up to time $n$.  More precisely,
\[
 \Theta(n)=\sum_{k=1}^n\phi(k),
\]
where $\phi(k)=0$ if $S_{k-1}$, $S_k$ and the origin are collinear.
Otherwise, $\phi(k)$ is the unique number in $(-\pi,\pi)$ such that
\[
 \frac{S_{k-1}}{|S_{k-1}|}e^{i\phi(k)}
 =\frac{S_k}{|S_k|}.
\]
The proof below follows the Brownian-embedding comparison in~\cite{Be91}
and uses the Brownian winding estimate in~\cite{Sh98}.

\begin{lemma}\label{winding}
For every $n\geq2$,
\[
 \Ex\sup_{1\leq k\leq n}|\Theta(k)|^2\preceq\log^2 n.
\]
\end{lemma}

\begin{proof}
We first assume that $z_0\in\R$.  Define $\tau_0=0$ and, for $k\geq1$,
\[
 \tau_k=\inf\{t\geq\tau_{k-1}:
 |Z(t)-Z(\tau_{k-1})|=\sqrt2\}.
\]
Then $\tau_1,\tau_2-\tau_1,\ldots$ are independent and identically
distributed random variables with mean $1$.  We have
\[
 \Ex\sup_{1\leq k\leq n}\Theta(k)^2
 \preceq
 \Ex\sup_{1\leq k\leq n}
 \bigl(\Theta(k)-\theta_+(\tau_k)\bigr)^2
 +
 \Ex\sup_{1\leq k\leq n}\theta_+(\tau_k)^2.
\]
We estimate the two terms on the right separately.

For a fixed $n$, put
\[
 M(0)=\Ex\bigl(\Theta(n)\mid S_0\bigr),
\]
and, for $1\leq k\leq n$,
\[
 M(k)=
 \Ex\bigl(\Theta(n)\mid S_0,S_1,\ldots,S_k\bigr)
 -
 \Ex\bigl(\Theta(n)\mid S_0,S_1,\ldots,S_{k-1}\bigr).
\]
By symmetry, $M(0)=0$ when $z_0\in\R$.  Following the notation of
\cite{Be91}, write
\[
 \lambda_j=M(j),
 \qquad
 \eta_j=\theta_+(\tau_j)-\theta_+(\tau_{j-1}),
 \qquad
 \Delta_j=\lambda_j-\eta_j.
\]
The same calculation as in Proposition~2 of~\cite{Be91} gives
\begin{equation}\label{mart1}
 \Ex\bigl(\Theta(n)-\theta_+(\tau_n)\bigr)^2
 \preceq\log^2 n.
\end{equation}
For completeness, the only point in that calculation which needs to be
checked here is the estimate for $\lambda_j$.  With the notation of
\cite{Be91},
\begin{align*}
 \Ex\Delta_j^2
 &=
 \int_{[0,1]}\int_{[0,\infty)}
 \Ex\bigl(\Delta_j^2\mid R_j=r,|S_{j-1}|=u\bigr)
 \,\mu(dr)\mu^{(j-1)}(du)\\
 &\quad+
 \int_{[1,\infty)}\int_{[u-1,\infty)}
 \Ex\bigl(\Delta_j^2\mid R_j=r,|S_{j-1}|=u\bigr)
 \,\mu(dr)\mu^{(j-1)}(du),
\end{align*}
where $R_j=|X_j|=\sqrt2$, $\mu$ is the distribution of $R_j$, and
$\mu^{(j-1)}$ is the distribution of $|S_{j-1}|$.  Moreover,
\begin{equation}\label{lambdaeta}
 \Ex\bigl(\Delta_j^2\mid R_j=r,|S_{j-1}|=u\bigr)
 \leq
 2\Ex\bigl(\lambda_j^2\mid R_j=r,|S_{j-1}|=u\bigr)
 +2\Ex\bigl(\eta_j^2\mid R_j=r,|S_{j-1}|=u\bigr).
\end{equation}
The first expectation on the right of \eqref{lambdaeta} is bounded by
$\pi^2$, since $|M(j)|\leq\pi$.  This follows by coupling the two simple
random walks appearing in
\[
 \Ex\bigl(\Theta(n)\mid S_0,\ldots,S_j\bigr)
 \quad\hbox{and}\quad
 \Ex\bigl(\Theta(n)\mid S_0,\ldots,S_{j-1}\bigr).
\]
The rest of the calculation is the same as in~\cite{Be91}, and yields
\eqref{mart1}.

We next estimate $\theta_+(\tau_k)$.  By page~121 of~\cite{Sh98}, for
$xt\geq\exp(16\pi)$,
\[
 \Pro\Big(\sup_{0\leq u\leq t}|\theta_+(u)|>x\Big)
 \leq
 11\exp\Big(-\frac{\pi x}{\log(xt)}\Big).
\]
Consequently, for all sufficiently large $y$,
\begin{equation}\label{diff-equ}
 \Pro\Big(
 \sup_{0\leq u\leq t}\frac{|\theta_+(u)|}{\log t}>y
 \Big)
 \leq 11e^{-\sqrt y}.
\end{equation}
In particular,
\begin{equation}\label{supthe}
 \Ex\sup_{0\leq u\leq n}\theta_+(u)^2\preceq\log^2 n.
\end{equation}

On the other hand,
\begin{align*}
 \Ex\sup_{1\leq k\leq n}\theta_+(\tau_k)^2
 &=
 \Ex\left[
 \sup_{1\leq k\leq n}\theta_+(\tau_k)^2\I(\tau_n<n)
 \right]\\
 &\quad+
 \sum_{j=1}^{\infty}
 \Ex\left[
 \sup_{1\leq k\leq n}\theta_+(\tau_k)^2
 \I\bigl(\tau_n\in[n+(jn)^3,n+(j+1)^3n^3)\bigr)
 \right].
\end{align*}
For the $j$th term in the sum, H\"older's inequality gives
\begin{align*}
 &\Ex\left[
 \sup_{1\leq k\leq n}\theta_+(\tau_k)^2
 \I\bigl(\tau_n\in[n+(jn)^3,n+(j+1)^3n^3)\bigr)
 \right]\\
 &\qquad\leq
 \Big(
 \Ex\sup_{0\leq u\leq n+(j+1)^3n^3}|\theta_+(u)|^4
 \Big)^{1/2}
 \Pro\bigl(\tau_n\in[n+(jn)^3,n+(j+1)^3n^3)\bigr)^{1/2}.
\end{align*}
It follows from \eqref{diff-equ} that
\[
 \Ex\sup_{0\leq u\leq n+(j+1)^3n^3}|\theta_+(u)|^4
 \preceq
 \bigl(\log(n+(j+1)^3n^3)\bigr)^4.
\]
Also, by Markov's inequality,
\[
 \Pro\bigl(\tau_n\in[n+(jn)^3,n+(j+1)^3n^3)\bigr)
 \leq
 \frac{\Ex\tau_n}{n+(jn)^3}
 =\frac1{1+j^3n^2}.
\]
Combining these estimates with \eqref{supthe}, we obtain
\[
 \Ex\sup_{1\leq k\leq n}\theta_+(\tau_k)^2
 \preceq\log^2 n.
\]
Together with \eqref{mart1}, this gives
\[
 \Ex\Theta(n)^2\preceq\log^2 n
\]
when $z_0\in\R$.

We now consider $z_0\notin\R$.  Let
\[
 \tau_\R=\inf\{j\geq0:S_j\in\R\}.
\]
If $\tau_\R\geq n$, then $|\Theta(n)|\leq\pi$.  If $\tau_\R<n$, then
\[
 |\Theta(n)|
 \leq |\Theta(\tau_\R)|+|\Theta(n)-\Theta(\tau_\R)|
 \leq \pi+|\Theta(n)-\Theta(\tau_\R)|.
\]
Conditionally on $S_0,S_1,\ldots,S_{\tau_\R}$, the second term is the
winding of a walk which starts on the real axis.  Hence the preceding
case gives
\begin{equation}\label{EThetan2}
 \Ex\Theta(n)^2\preceq\log^2 n,
 \qquad z_0\neq0.
\end{equation}

For the fixed terminal time $n$, the process
\[
 \sum_{k=1}^{m}M(k),
 \qquad 1\leq m\leq n,
\]
is a martingale.  By Doob's maximal inequality and \eqref{EThetan2},
\begin{equation}\label{doob}
 \Ex\max_{1\leq m\leq n}
 \left|\sum_{k=1}^{m}M(k)\right|^2
 \preceq\log^2 n.
\end{equation}
Finally, the triangle inequality gives
\begin{align*}
 &\Big|
 \Big(
 \Ex\max_{1\leq m\leq n}
 \big|\sum_{k=1}^{m}M(k)\big|^2
 \Big)^{1/2}
 -
 \Big(
 \Ex\max_{1\leq m\leq n}|\Theta(m)|^2
 \Big)^{1/2}
 \Big|\\
 &\qquad\leq
 \left(
 \Ex\max_{1\leq m\leq n}
 \left|
 \Ex\bigl(\Theta(n)-\Theta(m)\mid S_0,\ldots,S_m\bigr)
 -\Ex\bigl(\Theta(n)\mid S_0\bigr)
 \right|^2
 \right)^{1/2}
 \preceq\log n,
\end{align*}
where the last estimate follows from \eqref{EThetan2} and conditional
Jensen's inequality.  Combining this with \eqref{doob} proves the lemma.
\end{proof}

\addcontentsline{toc}{subsection}{Discrete harmonic measure}
\subsection*{Discrete harmonic measure}

We first specify the coordinates used near $Q_a^\de$.  Let
$\Gamma_{\mathrm n}^\de$ and $(\Gamma_1^*)^\de$ be the two long pieces of
the inner contour adjacent to $a_\de^\diamond$.  The first is parallel to the
terminal primal-open segment $\mathrm L_{\mathrm n}$ and the second is
parallel to the initial dual-open segment $\mathrm L_1^*$.  By
Condition~$\mathrm C$, their supporting lines are perpendicular.  Let
$o_a^\de$ be the intersection of these two supporting lines.  Translate
$o_a^\de$ to the origin and rotate the picture so that
$(\Gamma_1^*)^\de$ lies on the positive real axis and
$\Gamma_{\mathrm n}^\de$ lies on the positive imaginary axis.  Thus the
local part of $\Omega_\de^{in}$ is on the first-quadrant side of these two
lines.

Put
\[
 \widehat\eta_\de
 =\de\left\lfloor\de^{-1/2-\ep}\right\rfloor
 \asymp\eta_\de.
\]
We take $Q_a^\de$ to be the lattice square of side
$\widehat\eta_\de$ cut off by these two long pieces.  In the coordinates
above its four corners are
\[
 0,\qquad i\widehat\eta_\de,\qquad
 \widehat\eta_\de,\qquad
 (1+i)\widehat\eta_\de,
\]
up to a lattice-scale adjustment.  The two sides of $Q_a^\de$ which belong
to $\p\hat\Omega_\de$ are therefore
\[
 E_a^{\rm ex}
 = [i\widehat\eta_\de,(1+i)\widehat\eta_\de]
   \cup
   [\widehat\eta_\de,(1+i)\widehat\eta_\de].
\]
The remaining part of $\Gamma_{\mathrm n}^\de$ starts from
$i\widehat\eta_\de$ and goes in the positive imaginary direction, while
the remaining part of $(\Gamma_1^*)^\de$ starts from
$\widehat\eta_\de$ and goes in the positive real direction.  Define
\[
 N_\de=\left\lfloor\frac{r_\de}{40\widehat\eta_\de}\right\rfloor
\]
and, for $0\leq k\leq2N_\de$, put
\[
 I_k^{\rm v}
 =i[(k+1)\widehat\eta_\de,(k+2)\widehat\eta_\de],
 \qquad
 I_k^{\rm h}
 =[(k+1)\widehat\eta_\de,(k+2)\widehat\eta_\de].
\]
Thus $I_k^{\rm v}$ is a subinterval of $\Gamma_{\mathrm n}^\de$ and
$I_k^{\rm h}$ is a subinterval of $(\Gamma_1^*)^\de$.  Our choice of
$N_\de$ guarantees that all the intervals used below stay on these two long
pieces.  The same coordinates are used near $Q_b^\de$ after a rotation.

Let $\mathrm S$ be simple random walk in $\hat\Omega_\de$, and let
\[
 \tau=\inf\{m\geq0:\mathrm S_m\in\p\hat\Omega_\de\}.
\]

\begin{lemma} \label{dhm}
Suppose that the starting point $z$ stays at a distance bounded below by a
positive constant from $\p\hat\Omega_\de$.  Then, uniformly for
$0\leq k\leq N_\de$,
\begin{eqnarray}
 \Pro_z(\mathrm S_\tau\in E_a^{\rm ex})
 &\preceq&\frac{\eta_\de}{r_\de}, \label{hm1}\\
 \Pro_z(\mathrm S_\tau\in I_k^{\rm v})
 &\preceq&\frac{\eta_\de}{r_\de}, \label{hm2}\\
 \Pro_z(\mathrm S_\tau\in I_k^{\rm h})
 &\preceq&\frac{\eta_\de}{r_\de}. \label{hm3}
\end{eqnarray}
If
\[
 J^{\rm v}=i[\widehat\eta_\de,\widehat\eta_\de+\de^{1/3}],
 \qquad
 J^{\rm h}=[\widehat\eta_\de,\widehat\eta_\de+\de^{1/3}],
\]
then
\begin{eqnarray}
 \Pro_z(\mathrm S_\tau\in J^{\rm v})
 &\preceq&\frac{\de^{1/3}}{r_\de}, \label{hm4}\\
 \Pro_z(\mathrm S_\tau\in J^{\rm h})
 &\preceq&\frac{\de^{1/3}}{r_\de}. \label{hm5}
\end{eqnarray}
The analogous estimates hold near $Q_b^\de$.
\end{lemma}

\begin{proof}
We first prove \eqref{hm2}.  Fix $0\leq k\leq N_\de$ and let
\[
 \mathbf L_k
 =\{x+i(k+2)\widehat\eta_\de:x\in\R\}
\]
be the horizontal line through the common endpoint of $I_k^{\rm v}$ and
$I_{k+1}^{\rm v}$.  For all sufficiently small $\de$, the starting point $z$
lies on the side of $\mathbf L_k$ opposite to $I_k^{\rm v}$.  Hence every
walk path with $\mathrm S_\tau\in I_k^{\rm v}$ crosses $\mathbf L_k$ at least
once.

On the event $\{\tau<\infty\}$, decompose the random-walk paths according
to the number of crossings of $\mathbf L_k$.  A crossing means that, after
deleting a possibly nonempty consecutive part of the path lying on
$\mathbf L_k$, the path passes from one of the two open half-planes determined
by $\mathbf L_k$ to the other.  Let $M_k$ be the number of such crossings
before time $\tau$, and put
\[
 \mathcal C_{k,m}=\{\tau<\infty,\ M_k=m\},
 \qquad m\geq0.
\]
For the estimate of $I_k^{\rm v}$, set
\[
 \mathcal A_{k,m}
 =\mathcal C_{k,m}\cap\{\mathrm S_\tau\in I_k^{\rm v}\}.
\]
Since a path which exits through $I_k^{\rm v}$ has at least one crossing,
\[
 \{\mathrm S_\tau\in I_k^{\rm v}\}
 =\bigcup_{m\geq1}\mathcal A_{k,m},
\]
and the union is disjoint.

Fix a path in $\mathcal A_{k,m}$.  After its last crossing of
$\mathbf L_k$, let $\sigma_k$ be its last visit to $\mathbf L_k$.  Thus
$\mathrm S[\sigma_k+1,\tau]$ lies strictly on the side of $\mathbf L_k$
which contains $I_k^{\rm v}$.  Reflect this terminal part in $\mathbf L_k$ and
leave $\mathrm S[0,\sigma_k]$ unchanged.  Since
$\mathrm S_{\sigma_k}\in\mathbf L_k$, the reflected path is again a
nearest-neighbour walk path.  It exits through $I_{k+1}^{\rm v}$.  The final
crossing of $\mathbf L_k$ has been removed, so the reflected path has exactly
$m-1$ crossings of $\mathbf L_k$.

For a fixed $m$ this reflection is one-to-one.  Indeed, in the reflected path
$\sigma_k$ is the last visit to $\mathbf L_k$.  Reflecting the part after this
last visit recovers the original path.  Moreover, the images corresponding to
different values of $m$ are disjoint, since they have different numbers of
crossings of $\mathbf L_k$.  Reflection preserves the probability of every
simple random-walk path.  Summing over $m$ therefore gives
\[
 \Pro_z(\mathrm S_\tau\in I_k^{\rm v})
 \leq
 \Pro_z(\mathrm S_\tau\in I_{k+1}^{\rm v}).
\]

The same construction can be applied successively to
$I_k^{\rm v},I_{k+1}^{\rm v},\ldots,I_{k+N_\de}^{\rm v}$.  All these
intervals remain on $\Gamma_{\mathrm n}^\de$.  Since the corresponding exit
events are disjoint,
\[
 (N_\de+1)\Pro_z(\mathrm S_\tau\in I_k^{\rm v})\leq1.
\]
As $N_\de\asymp r_\de/\eta_\de$, this proves \eqref{hm2}.  After a
rotation by $\pi/2$, the same argument on $(\Gamma_1^*)^\de$ proves
\eqref{hm3}. We can also follow the same method of deriving \eqref{hm2} to obtain \eqref{hm4} and \eqref{hm5}.

We finally prove \eqref{hm1}.  For $m=1,2$, put
\[
 \mathfrak l_m
 =\{x+im\widehat\eta_\de:0\leq x\leq m\widehat\eta_\de\},
 \qquad
 \mathfrak l_m'
 =\{m\widehat\eta_\de+iy:0\leq y\leq m\widehat\eta_\de\}.
\]
Thus $E_a^{\rm ex}=\mathfrak l_1\cup\mathfrak l_1'$.  Let
\[
 t_2=\inf\{j\geq0:\mathrm S_j\in\mathfrak l_2\cup\mathfrak l_2'\}.
\]
If $\mathrm S_\tau\in E_a^{\rm ex}$, then $t_2<\tau$.  Suppose first that
$\mathrm S_{t_2}\in\mathfrak l_2$.  Starting from
$\mathrm S_{t_2}$, a simple random walk has probability bounded below by a
positive absolute constant to leave $\hat\Omega_\de$ through
$I_0^{\rm v}$.  This is a fixed-scale random-walk estimate after rescaling by
$\widehat\eta_\de$.  By the strong Markov property and \eqref{hm2} with
$k=0$,
\[
 \Pro_z(\mathrm S_{t_2}\in\mathfrak l_2,\ t_2<\tau)
 \preceq \frac{\eta_\de}{r_\de}.
\]
The same argument, using \eqref{hm3}, applies when
$\mathrm S_{t_2}\in\mathfrak l_2'$.  Hence \eqref{hm1} follows.
\end{proof}

\addcontentsline{toc}{subsection}{Asymptotic harmonic functions}
\subsection*{Asymptotic harmonic functions}

Let
\[
 \mathrm U=\{x\in\R^2:|x|<1\},
 \qquad
 \mathrm U_\eta=\mathrm U\cap\eta\Z^2,
\]
and let $\p\mathrm U_\eta$ be the outer vertex boundary of $\mathrm U_\eta$.  For a function $G$ on $\mathrm U_\eta\cup\p\mathrm U_\eta$, write
\[
 \Delta_\eta G(v)
 =\frac14\sum_{e\in\{\be_1,-\be_1,\be_2,-\be_2\}}
       G(v+\eta e)-G(v).
\]
Let $u\in\eta\Z^2$ satisfy $|u|\leq C_0\eta$, and define, whenever the displayed vertices belong to the domain,
\[
 \nabla_uG(v)=G(v+u)-G(v),
 \qquad
 \nabla_u^2G(v)=G(v+u)+G(v-u)-2G(v).
\]
Put
\[
 \|G\|_\p=\max_{x\in\p\mathrm U_\eta}|G(x)|.
\]

\begin{lemma} \label{asyhar}
Suppose that, for every interior vertex $v$,
\begin{equation}\label{LapG}
 |\Delta_\eta G(v)|\preceq\eta^{2+2\al}.
\end{equation}
Then, uniformly for $|v|\leq1/2$,
\begin{eqnarray}
 |\nabla_uG(v)|
 &\preceq&\eta\big(\|G\|_\p+\eta^{2\al}\big), \label{asyhar-first}\\
 |\nabla_u^2G(v)|
 &\preceq&\eta^2\big(\|G\|_\p+\eta^\al\big). \label{asyhar-second}
\end{eqnarray}
If instead
\begin{equation}\label{weakLap}
 |\Delta_\eta G(x,y)|
 \preceq
 \begin{cases}
  \eta^{2+2\al},&y\neq0,\\
  \eta^{1+\al},&y=0,
 \end{cases}
\end{equation}
then, uniformly for $|v|\leq1/2$,
\begin{eqnarray}
 |G(v)-H(v)|&\preceq&\eta^\al, \label{asyhar-weak-value}\\
 |\nabla_uG(v)|
 &\preceq&\eta\big(\|G\|_\p+\eta^\al\log\eta^{-1}\big)
 \preceq\eta\big(\|G\|_\p+\eta^{\al/2}\big), \label{asyhar-weak-first}
\end{eqnarray}
where $H$ is the discrete harmonic function in $\mathrm U_\eta$ with boundary values $G$.
\end{lemma}

\begin{proof}
Let $H$ be the discrete harmonic extension of $G|_{\p\mathrm U_\eta}$ and put $R=G-H$.  Then $R=0$ on $\p\mathrm U_\eta$ and
\begin{equation}\label{green-representation}
 R(v)=-\sum_{w\in\mathrm U_\eta}
       \mathrm G_{\mathrm U_\eta}(v,w)\Delta_\eta G(w),
\end{equation}
where $\mathrm G_{\mathrm U_\eta}(v,w)$ is the Green function of simple random walk killed on leaving $\mathrm U_\eta$.
Using the potential kernel representation of the Green function and
Theorem~1.6.5 of~\cite{Law91}, together with the interior estimates for
discrete harmonic functions, we obtain, for $|v|\leq2/3$,
\begin{eqnarray}
 \mathrm G_{\mathrm U_\eta}(v,w)
 &\preceq& \log\frac{C}{|v-w|+\eta}, \label{green-value}\\
 |\nabla_u\mathrm G_{\mathrm U_\eta}(v,w)|
 &\preceq&\frac{\eta}{|v-w|+\eta}, \label{green-first}\\
 |\nabla_u^2\mathrm G_{\mathrm U_\eta}(v,w)|
 &\preceq&\frac{\eta^2}{(|v-w|+\eta)^2}. \label{green-second}
\end{eqnarray}
Here $C$ is an absolute constant chosen large enough.  The factors of
$\eta$ in the denominators make these estimates valid also when
$|v-w|=O(\eta)$.

Under \eqref{LapG}, insert \eqref{green-first} and
\eqref{green-second} into \eqref{green-representation} and sum over
dyadic annuli centered at $v$.  A dyadic annulus of radius $r$ contains
$O(r^2/\eta^2)$ lattice vertices.  Hence
\[
 |\nabla_uR(v)|\preceq\eta^{1+2\al}
\]
and
\[
 |\nabla_u^2R(v)|
 \preceq\eta^{2+2\al}\log\eta^{-1}
 \preceq\eta^{2+\al}.
\]
The discrete harmonic function $H$ satisfies the usual interior estimates
\[
 |\nabla_uH(v)|\preceq\eta\|G\|_\p,
 \qquad
 |\nabla_u^2H(v)|\preceq\eta^2\|G\|_\p.
\]
Combining the estimates for $H$ and $R$ proves
\eqref{asyhar-first} and \eqref{asyhar-second}.

Assume now \eqref{weakLap}.  The contribution in
\eqref{green-representation} from vertices off the line $y=0$ is bounded
as above.  In particular, its contribution to $R(v)$ is
$O(\eta^{2\al})$, and its contribution to $\nabla_uR(v)$ is
$O(\eta^{1+2\al})$.  On the line $y=0$, \eqref{green-value} and a
one-dimensional summation give
\[
 |R(v)|
 \preceq
 \eta^{1+\al}
 \sum_{w\in\mathrm U_\eta\cap(\eta\Z\times\{0\})}
 \log\frac{C}{|v-w|+\eta}
 \preceq\eta^\al,
\]
which proves \eqref{asyhar-weak-value}.  Similarly, \eqref{green-first}
gives
\[
 |\nabla_uR(v)|
 \preceq
 \eta^{1+\al}
 \sum_{w\in\mathrm U_\eta\cap(\eta\Z\times\{0\})}
 \frac{\eta}{|v-w|+\eta}
 \preceq\eta^{1+\al}\log\eta^{-1}.
\]
Adding the harmonic estimate for $H$ proves the first inequality in
\eqref{asyhar-weak-first}.  The second follows from
$\eta^\al\log\eta^{-1}\preceq\eta^{\al/2}$.
\end{proof}

\begin{remark}
Under \eqref{weakLap}, the contribution from the line $y=0$ gives in
general the logarithmic factor in \eqref{asyhar-weak-first}.  Without
additional cancellation, this factor cannot be removed by the above
Green-function estimate.  Moreover, \eqref{weakLap} alone does not give
the second-difference estimate \eqref{asyhar-second}.  The bounds
\eqref{asyhar-weak-value} and \eqref{asyhar-weak-first} are sufficient
for the boundary argument leading to Proposition~\ref{convFb}.
\end{remark}

\addcontentsline{toc}{section}{References}

\bibliographystyle{halpha}

\end{document}